\documentclass[12pt, reqno]{amsart}
\usepackage{mathrsfs}
\usepackage{amssymb,amsthm,amsmath}
\usepackage[numbers,sort&compress]{natbib}
\usepackage{amssymb,amsmath}
\usepackage{amsfonts}
\usepackage{mathrsfs}
\usepackage{latexsym}
\usepackage{amssymb}
\usepackage{amsthm}
\usepackage{color}
\usepackage{pdfsync}
\usepackage{indentfirst}
\usepackage{graphics}
\usepackage{subfigure}
\usepackage{epsfig}
\usepackage{float}
\usepackage{cases}
\date{today}

\usepackage{hyperref}
\hypersetup{hypertex=true,
colorlinks=true,
linkcolor=blue,
anchorcolor=g,
citecolor=red}

\usepackage{amsmath}
\usepackage{amsthm}

\newtheorem{remark}{Remark}[section]

\newtheorem{theorem}{Theorem}[section]
\newtheorem{proposition}{Proposition}[section]
\newtheorem{lemma}{Lemma}[section]
\newtheorem{corollary}{Corollary}[section]

\newcommand{\beq}{\begin{equation}}
\newcommand{\eeq}{\end{equation}}
\newcommand{\ben}{\begin{eqnarray}}
\newcommand{\een}{\end{eqnarray}}
\newcommand{\beno}{\begin{eqnarray*}}
\newcommand{\eeno}{\end{eqnarray*}}

\numberwithin{equation}{section}

\begin{document}
\title[\vspace{-0.2cm}Suppression of blow-up for PKS-NS system]{Suppression of blow-up \\for the 3D Patlak-Keller-Segel-Navier-Stokes system \\via the Couette flow}
\author{Shikun~Cui}
\address[Shikun~Cui]{School of Mathematical Sciences, Dalian University of Technology, Dalian, 116024,  China}
\email{cskmath@163.com}
\author{Lili~Wang}
\address[Lili~Wang]{School of Mathematical Sciences, Dalian University of Technology, Dalian, 116024,  China}
\email{wanglili\_@mail.dlut.edu.cn}
\author{Wendong~Wang}
\address[Wendong~Wang]{School of Mathematical Sciences, Dalian University of Technology, Dalian, 116024,  China}
\email{wendong@dlut.edu.cn}
\date{\today}

\vspace*{-0.6cm}
\maketitle

\vspace{-0.4cm}
\begin{abstract}
	As is well known, for the 3D Patlak-Keller-Segel system, regardless of whether they are parabolic-elliptic or parabolic-parabolic forms,  finite-time blow-up  may occur for arbitrarily small values of the initial mass. In this paper, it is proved for the first time that one can prevent the finite-time blow-up when the initial mass is less than a certain critical threshold via the stabilizing effect of the moving Navier-Stokes flows.
	In details, we investigate the nonlinear stability 
	of the Couette flow $(Ay, 0, 0)$ in the
	Patlak-Keller-Segel-Navier-Stokes system 
	and \textcolor[rgb]{0,0,0}{show} that if the Couette flow is sufficiently strong (A is large enough), then the solutions for Patlak-Keller-Segel-Navier-Stokes system are global in time provided that the initial velocity is sufficiently small and  the initial cell mass is less than $\frac{24}{5} \pi^2$. 
	
	
\end{abstract}

\tableofcontents

\parskip5pt
\parindent=1.5em
\section{Introduction}

Consider the following three-dimensional parabolic-elliptic Patlak-Keller-Segel (PKS) system coupled with Navier-Stokes equations in $(x,y,z)\in\mathbb{T}\times\mathbb{R}\times\mathbb{T}$ with $ \mathbb{T}=[0,2\pi] $:
\begin{equation}\label{ini}
	\left\{
	\begin{array}{lr}
		\partial_tn+v\cdot\nabla n=\triangle n-\nabla\cdot(n\nabla c), \\
		\triangle c+n-c=0, \\
		\partial_tv+v\cdot\nabla v+\nabla P=\triangle v+n\nabla\phi, \\
		\nabla\cdot v=0, 
	\end{array}
	\right.
\end{equation}
along with initial conditions
$$(n,v)\big|_{t=0}=(n_{\rm in},v_{\rm in}),$$
where $n$ represents the cell density, $c$ denotes the chemoattractant density, and $v$ denotes the velocity of fluid. In addition, $P$ is the pressure and $\phi$ is the given potential function.

If $v=0$ and $\phi=0$,  the system (\ref{ini}) is reduced to  the classical 3D parabolic-elliptic Patlak-Keller-Segel system, which is a mathematical model used to describe the diffusion and chemotactic movement of chemical substances in a population of cells (or microorganisms),
and it was jointly developed by Patlak \cite{Patlak1}, Keller and Segel \cite{Keller1}.
This system has wide applications in the fields of biology, ecology, and medicine. It helps us understand phenomena such as cell migration, aggregation, and diffusion. 
As long as the dimension of space is higher than one, the solutions of the classical PKS system may blow up in finite time.
In the 2D space, the parabolic-parabolic PKS model ($\triangle c$ is replaced by $\triangle c-\partial_tc$ in $\eqref{ini}_2$) has a critical mass of $8\pi$, if the cell mass $M:=||n_{\rm in}||_{L^1}$ is less than $8\pi$,  the solutions of the system are global in time proved by Calvez-Corrias \cite{Calvez1},
if the cell mass is greater than $8\pi$, the solutions will blow up in finite time proved by Schweyer \cite{Schweyer1}.
In the 2D space, the parabolic-elliptic PKS system is globally well-posed if and only if the total mass $M\leq8\pi$ by Wei \cite{wei11}.
When the spatial dimension is higher than two, the PKS system is supercritical and its solutions  will blow up for any initial mass, see Nagai \cite{Na2000} or Souplet-Winkler \cite{SW2019} for the parabolic-elliptic case, and see Winkler \cite{winkler1} for the parabolic-parabolic case.  For more results on this topic,  we refer to  \cite{BCM2008,DDDMW,TW2017} and the references therein. As said in \cite{zeng}:

{\it An interesting question is to consider whether the stabilizing effect of the moving fluid can suppress the finite time blow-up?}

Firstly, 
let us recall  some of the results obtained in 2D briefly.
For the  parabolic-elliptic PKS system of $(\ref{ini})_1-(\ref{ini})_2$, Kiselev-Xu \cite{Kiselev1} suppressed the  blow-up by stationary relaxation
enhancing flows and time-dependent Yao-Zlatos flows in $\mathbb{T}^d$ with $d=2,3$.
Bedrossian-He \cite{Bedro2} also studied the suppression of blow-up by non-degenerate shear flows ($v\cdot\nabla n=A u(y)\partial_x n$) in
$\mathbb{T}^2$ for the 2D parabolic-elliptic case. He \cite{he0} investigated the suppression of blow-up by a large strictly monotone shear flow for the parabolic-parabolic PKS model in $\mathbb{T}\times\mathbb{R}$ when $A$ lies in an interval. For the coupled PKS-NS system, Zeng-Zhang-Zi \cite{zeng} firstly considered the 2D PKS-NS system near the Couette flow in $\mathbb{T}\times\mathbb{R}$, and
proved that if the Couette flow is sufficiently strong, the solution stays globally regular.
He \cite{he05} considered the blow-up suppression for the parabolic-elliptic PKS-NS system in
$\mathbb{T}\times\mathbb{R}$ with the coupling of buoyancy effects for a class of initial data with small vorticity.
The suppression of blow-up in PKS-NS system via the Poiseuille flow was obtained by 
Li-Xiang-Xu  \cite{Li0} in $\mathbb{T}\times\mathbb{R}$ and Cui-Wang in \cite{cui1} $\mathbb{T}\times I$, respectively. Besides, Hu \cite{Hu2023} proved the solution remains regular for all times in the regime of sufficiently
large buoyancy and viscosity, see also the recent results by Hu-Kiselev-Yao \cite{Hu0} and Hu-Kiselev \cite{Hu1}.

For the 3D PKS system of parabolic-elliptic case, Bedrossian-He \cite{Bedro2} investigated the suppression of blow-up by shear flows in $\mathbb{T}^3$ and $\mathbb{T}\times\mathbb{R}^2$ with the initial mass $M<8\pi$. 
Feng-Shi-Wang \cite{Feng1} used the planar helical flows as transport
flow to research the advective Kuramoto-Sivashinsky and
Keller-Segel equations. 
Shi-Wang \cite{wangweike2} considered the suppression effect of the flow $(z,z^2,0)$ in $\mathbb{T}^2\times\mathbb{R}$, and Deng-Shi-Wang \cite{wangweike1} proved the Couette flow with a sufficiently large amplitude prevents
the blow-up of solutions in the whole space for exponential decay data.
For the parabolic-parabolic PKS system, He \cite{he24-1} considered an alternating flow and proved the solution remains globally regular  in $\mathbb{T}^3$ as long as the
flow is sufficiently strong without a mass threshold.
For a time-dependent shear flow, He \cite{he24-2} demonstrated that when the total
mass of the cell density is below a specific threshold ($8\pi|\mathbb{T}|$), the solution remains globally regular  in $\mathbb{T}^3$ as long as the
flow is sufficiently strong.

For the 3D PKS-NS system, the following questions are still open:\\
{\it 1. Whether the solutions are globally well-posed provided that the amplitude of shear flows is sufficiently large;}\\
{\it 2.  Whether there exists a critical threshold for the total mass of the cell density.}

Note that the critical threshold may depend on the shear flow or the domain. Recently, the authors \cite{CWW1} investigated the linear stability of the Couette flow $(Ay,0,0)$ in the 3D PKS system coupled with the linearized NS equations. However, it's still unknown for the 3D PKS-NS system. We will investigate this issue in this paper.
At this time, this is related the stability problem of the 3D Navier-Stokes equations:
\begin{equation}\label{eq:3DNS}
	\left\{
	\begin{array}{lr}
		\partial_tu
		-\frac{1}{Re}\triangle u+u\cdot\nabla u+\nabla P=0, \\
		\nabla \cdot u=0.
	\end{array}
	\right.
\end{equation}
Due to Reynolds's famous work in 1883 \cite{Re1883}, the stability and transition
to turbulence of the laminar flows at high Reynolds number have been an important
field in fluid mechanics. It is well-known that the plane
Couette flow is linearly stable for any Reynolds number  \cite{DR1981}. However,
the experiments show that it can be unstable and transition to turbulence for
small but finite perturbations at high Reynolds number \cite{C2002}.   
Then an important mathematical question 
	formulated by Bedrossian, Germain, and
	Masmoudi  \cite{Bedro1} is that:
{\it Given a norm
	$\|\cdot\|_X$, find a $\beta=\beta(X)$ so that
	$$~\|u_{\rm in}\|_X\leq Re^{-\beta}\Longrightarrow {\rm stability},$$
	$$\quad\|u_{\rm in}\|_X\gg Re^{-\beta}\Longrightarrow {\rm instability}.$$}
The exponent $\beta$ is referred to as the transition threshold in the applied literature.
\textcolor[rgb]{0,0,0}{
	Bedrossian-Germain-Masmoudi proved $\beta\leq \frac{3}{2}$ for the 3D Couette flow \cite{Bedro1} in Sobolev space and $\beta\leq 1$ in Gevrey class \cite{BGM2022}.}
In Sobolev space, we refer to Wei-Zhang \cite{wei2} and Chen-Wei-Zhang \cite{Chen1} for  recent results of $\beta\leq 1$. More references on MHD, Boussinesq equations or other models, we refer to \cite{HSX2024-1,HSX2024-2,LISS2020,NW1,ZZZ2022} and the references therein.
For the stability of the 2D Navier-Stokes equations and related models, there are  very rich research progress on this topic, and we refer to \cite{BMV2016,BVW2018,WZ2023,MZ2022,DWZ2021,BHIW2023} and the references therein.


Motivated by the above \textcolor[rgb]{0,0,0}{transition threshold problem} \cite{Bedro1} and \cite{wei2}, our main goal is to investigate the suppression of blow-up \textcolor[rgb]{0,0,0}{and} the nonlinear stability of the system (\ref{ini}) via the 3D Couette flow. Introduce a perturbation $u=(u_1,u_2,u_3)$ around the Couette flow $( Ay,0,0 )$, which $u(t,x,y,z)=v(t,x,y,z)-( Ay,0,0 )$ satisfying $u\big|_{t=0}=u_{\rm in}=(u_{1,\rm in}, u_{2,\rm in}, u_{3,\rm in})$. Assume $\phi=y$. Then we get
\begin{equation}\label{ini1}
	\left\{
	\begin{array}{lr}
		\partial_tn+Ay\partial_x n+u\cdot\nabla n-\triangle n=-\nabla\cdot(n\nabla c), \\
		\triangle c+n-c=0, \\
		\partial_tu+Ay\partial_x u+\left(
		\begin{array}{c}
			Au_2 \\
			0 \\
			0 \\
		\end{array}
		\right)
		-\triangle u+u\cdot\nabla u+\nabla P^{N_1}+\nabla P^{N_2}+\nabla P^{N_3}=\left(
		\begin{array}{c}
			0 \\
			n \\
			0 \\
		\end{array}
		\right), \\
		\nabla \cdot u=0,
	\end{array}
	\right.
\end{equation}
where the pressure $P^{N_1}$, $P^{N_2}$ and $P^{N_3}$ are determined by
\begin{equation}\label{pressure_1}
	\left\{
	\begin{array}{lr}
		\triangle P^{N_1}=-2A\partial_xu_2, \\
		\triangle P^{N_2}=\partial_yn, \\
		\triangle P^{N_3}=-{\rm div}~(u\cdot\nabla u).
	\end{array}
	\right.
\end{equation}
After the time rescaling $t\mapsto\frac{t}{A}$, we get
\begin{equation}\label{ini11}
	\left\{
	\begin{array}{lr}
		\partial_tn+y\partial_x n+\frac{1}{A}u\cdot\nabla n-\frac{1}{A}\triangle n=-\frac{1}{A}\nabla\cdot(n\nabla c), \\
		\triangle c+n-c=0, \\
		\partial_tu+y\partial_x u+\left(
		\begin{array}{c}
			u_2 \\
			0 \\
			0 \\
		\end{array}
		\right)
		-\frac{1}{A}\triangle u+\frac{1}{A}u\cdot\nabla u+\frac{1}{A}(\nabla P^{N_1}+\nabla P^{N_2}+\nabla P^{N_3})=\left(
		\begin{array}{c}
			0 \\
			\frac{n}{A} \\
			0 \\
		\end{array}
		\right), \\
		{\nabla \cdot u=0.}
	\end{array}
	\right.
\end{equation}


Before stating the result,  define the following modes
\begin{equation*}
	\begin{aligned}
		&P_0f=f_0=\frac{1}{|\mathbb{T}|}\int_{\mathbb{T}}f(t,x,y,z)dx,~~~\qquad
		\qquad \
		P_{\neq}f=f_{\neq}=f-f_0,\\
		&P_{(0,0)}f=f_{(0,0)}=\frac{1}{|\mathbb{T}|^2}\int_{\mathbb{T}\times\mathbb{T}}
		f(t,x,y,z)dxdz,\quad
		P_{(0,\neq)}f=f_{(0,\neq)}=f_0-f_{(0,0)}.
	\end{aligned}
\end{equation*} 
Throughout this paper, $f_0$ and $f_{\neq}$ respectively  represent the zero mode and non-zero mode  of $f$. Moreover, $f_{(0,0)}$ and $f_{(0,\neq)}$ denote the z-part zero mode and the z-part non-zero mode of $f_0$,
respectively.

Our first theorem  is stated as follows.
\begin{theorem}\label{result0}
	Assume that $ u_{\rm in}(x,y,z)\in H^2
	(\mathbb{T}\times\mathbb{R}\times\mathbb{T})$.
	Moreover, the non-negative initial data $ n_{\rm in}(x,y,z)=n_{\rm in}(x,y)\in H^2\cap L^1(\mathbb{T}\times\mathbb{R}\times\mathbb{T})$  and $$M=   \int_{\mathbb{T}\times\mathbb{R}\times\mathbb{T}}n_{\rm in}dxdydz < \frac{24}{5}\pi^2.$$
	Then there exists a positive constant $C_{(0)}$ depending on $||n_{\rm in}||_{H^2\cap L^1(\mathbb{T}\times\mathbb{R}\times\mathbb{T})}$ 
	and $||u_{\rm in}||_{H^2(\mathbb{T}\times\mathbb{R}\times\mathbb{T})}$, such that if $A>C_{(0)}$ and $A^{\epsilon_0}\|u_{\rm in}\|_{H^2(\mathbb{T}\times\mathbb{R}\times\mathbb{T})}\leq C$, the solutions of {\rm (\ref{ini11})} are global in time,
	where $\epsilon_0$ is a positive constant satisfying $\epsilon_0>\frac{1}{3}$.
	
\end{theorem}

In fact, the above result is an immediate corollary of the following theorem for the general initial data $n_{\rm in}$.
\begin{theorem}\label{result}
	Assume that the non-negative initial data $ n_{\rm in}(x,y,z)\in H^2\cap L^1(\mathbb{T}\times\mathbb{R}\times\mathbb{T})$
	and $ u_{\rm in}(x,y,z)\in H^2
	(\mathbb{T}\times\mathbb{R}\times\mathbb{T})$.
	Then there exists a positive constant $C_{(0)}$ depending on $||n_{\rm in}||_{H^2\cap L^1(\mathbb{T}\times\mathbb{R}\times\mathbb{T})}$ 
	and $||u_{\rm in}||_{H^2(\mathbb{T}\times\mathbb{R}\times\mathbb{T})}$, such that if $A>C_{(0)}$, and	
	\begin{equation*}
		\left\{
		\begin{array}{lr}
			A^{\epsilon_0}\|u_{\rm in}\|_{H^2(\mathbb{T}\times\mathbb{R}\times\mathbb{T})}
			+A^{\epsilon_0}\|(n_{\rm in})_{(0,\neq)}\|_{L^2(\mathbb{T}\times\mathbb{R}\times\mathbb{T})}\leq C, \\
			\\
			M = \int_{\mathbb{T}\times\mathbb{R}\times\mathbb{T}}n_{\rm in}dxdydz < \frac{24}{5}\pi^2,
		\end{array}
		\right.
	\end{equation*}
	the solutions of {\rm (\ref{ini11})} are global in time, where $\epsilon_0$ is a positive constant satisfying $\epsilon_0>\frac13$.
\end{theorem}
Here are some remarks for the above results.
\begin{remark}
	The suppression of blow-up for the 3D PKS-NS system near the Couette flow was obtained in Theorem \ref{result}, which seems to be the first result for the 3D PKS-NS system.  For the 2D PKS-NS system, stability analysis or suppression of blow-up was obtained in many references, such as \cite{cui1, he05, Li0, Wanglili, zeng}, where enhanced dissipation plays a crucial role in stabilizing solutions or suppressing blow-up. However, there are significant differences in complexity and structures between the 2D and 3D PKS-NS systems. For the 3D cases, one needs a comprehensive consideration of multiple effects, including enhanced dissipation, inviscid damping, the 3D lift-up effect, nonlinear interactions and energy transfer mechanisms etc.,  which bring substantial challenges. For finite channels, the boundary layer effect can also cause some trouble, and we refer to some recent works \cite{CLTZ, Chen0, Chen1}.
	
\end{remark}

\begin{remark}
	It's an open question that  whether there exists a sharp 
	threshold for initial cell mass $M$ in the 3D PKS-NS system.
	We guess it seems to be $16\pi^2$.
	In fact, for the 3D PKS system without coupling the Navier-Stokes flow, the zero mode $n_0$ satisfies
$$\partial_t n_{0}-\frac{1}{A}\triangle n_{0}=-\frac{1}{A}\left[\nabla\cdot(n_{\neq}\nabla c_{\neq})_{0}+\partial_{y}(n_{0}\partial_{y}c_{0})+\partial_{z}(n_{0}\partial_{z}c_{0}) \right]$$
	and it is similar as the 2D  PKS system except for $\frac{1}{A}\nabla\cdot(n_{\neq}\nabla c_{\neq})_{0}$ as a perturbation.
	Recently, He in \cite{he24-2} proved the well-posed result of  $M<16\pi^2$  for the time-dependent shear flow, where the free energy functional of the zero mode plays an important role. Moreover, this critical value limit can be removed when considering other special flows, such as stationary relaxation
	enhancing flows and time-dependent Yao-Zlatos flows in \cite{Kiselev1}, and an alternating flow in \cite{he24-1}.
	
	For the 3D PKS-NS system, there are coupled terms $\partial_{y}(u_{2,0}n_{0})$ and $\partial_{z}(u_{3,0}n_{0})$ for the zero mode $n_0$, which cannot be ignored. Especially, \textcolor[rgb]{0,0,0}{the estimates of $u_{2,0}$ and $u_{3,0}$ still depend} on $n_0$, and for more details we refer to the estimates of $T_{1,4}$ and $T_{1,5}$ in Lemma \ref{lemma_n002}.
	Hence, the value of the critical mass threshold for the 3D PKS-NS system is still unknown.
	
	
\end{remark}

\begin{remark}
It is very challenging to remove the restrictions on 
$\|(u_{\rm in})_{0}\|_{H^2}$ and  $\|(n_{\rm in})_{(0,\neq)}\|_{L^2}$ due to \textcolor[rgb]{,0,0}{the 3D lift-up effect} and the energy transfer mechanism. 
In fact, it follows from (\ref{eq:u10decom}) that $u_{2,0}$ needs to be small, since $\widehat{u_{1,0}}$ should be small by the space-time estimates (Prop \ref{timespace2}, \ref{timespace4}). 
	\textcolor[rgb]{0,0,0}{
		Moreover, $u_{2,0}$ and 
		$u_{3,0}$ are related to $n_{0}$ by}
	\begin{equation*}
		\left\{
		\begin{array}{lr}
			\partial_tu_{2,0}-\frac{1}{A}\triangle u_{2,0}
			+\frac{1}{A}(u\cdot\nabla u_{2})_0
			+\frac{1}{A}\partial_yP^{N_1}_0
			+\frac{1}{A}\partial_y P^{N_2}_0+\frac{1}{A}\partial_y P^{N_3}_0=\frac{n_0}{A}, \\
			\partial_tu_{3,0}-\frac{1}{A}\triangle u_{3,0}
			+\frac{1}{A}(u\cdot\nabla u_3)_0
			+\frac{1}{A}\partial_zP^{N_1}_0
			+\frac{1}{A}\partial_z P^{N_2}_0
			+\frac{1}{A}\partial_z P^{N_3}_0=0,
		\end{array}
		\right.
	\end{equation*}	
	which imply that the norm of  $n_0$ needs to be small to close the energy estimate. 
\end{remark}


\begin{remark}
	For a three-dimensional domain with physical boundary, the stability or instability of the PKS-NS system is still open. In a recent work, the authors investigated the linear stability of the Couette flow $(Ay,0,0)$ in \cite{CWW1}. 
	An interesting question is whether one can prove the nonlinear stability of PKS-NS system in a bounded domain.
\end{remark}

	Here are some notations used in this paper.
	
	\noindent\textbf{Notations}:
	\begin{itemize}
		\item 
		For given $f$, the Fourier transform can be defined by
		\begin{equation}
			f(t,x,y,z)=\sum_{k_1,k_3\in\mathbb{Z}}\frac{1}{2\pi}\int_{k_2\in \mathbb{R}}\widehat{f}_{k_1,k_2,k_3}(t){\rm e}^{ik_2y}dk_2{\rm e}^{i(k_1x+k_3z)}, \nonumber
		\end{equation}
		where $\widehat{f}_{k_1,k_2,k_3}(t)=\frac{1}{|\mathbb{T}|^2}\int_{\mathbb{T}\times\mathbb{T}}\int_{\mathbb{R}}{f}(t,x,y,z){\rm e}^{-ik_2y}dy{\rm e}^{-i(k_1x+k_3z)}dxdz$.

		%
		\item 
		For any given function $f$,  the zero mode and the non-zero mode
		are defined by $f_0$ and $f_{\neq}.$
		Especially, we use $u_{j,0}$, and $u_{j,\neq}$ to represent the zero mode
		and non-zero mode of the velocity $u_{j} (j=1,2,3)$, respectively.
		Similarly, we use $\omega_{2,0}$ and $\omega_{2,\neq}$ to represent the zero mode
		and non-zero mode of the vorticity $\omega_2$, respectively.
		\item The z-part zero and z-part non-zero modes for $f_0$ are defined by
		$f_{(0,0)}$ and $f_{(0,\neq)}.$
		Similarly, we use $u_{j,(0,0)}$, and $u_{j,(0,\neq)}$ to represent the z-part zero mode
		and z-part non-zero mode of the velocity $u_{j,0} (j=1,2,3)$, respectively.
		\item  We denote the partial derivatives $\partial_x,$ $\partial_y,$ and $\partial_z$ by  
		$\partial_1,$ $\partial_2,$ and $\partial_3,$  respectively.
		\item The norm of the $L^p$ space is defined by
		$$\|f\|_{L^p(\mathbb{T}\times\mathbb{R}\times\mathbb{T})}=\left(\int_{\mathbb{T}\times\mathbb{R}\times\mathbb{T}}|f|^p dxdydz\right)^{\frac{1}{p}},$$
		and $\langle\cdot,\cdot\rangle$ denotes the standard $L^2$ scalar product. For simplicity, write $\|f\|_{L^p(\mathbb{T}\times\mathbb{R}\times\mathbb{T})}$ as $\|f\|_{L^p}.$
		\item The time-space norm  $\|f\|_{L^qL^p}$ is defined by
		$$\|f\|_{L^qL^p}=\big\|  \|f\|_{L^p(\mathbb{T}\times\mathbb{R}\times\mathbb{T})}\ \big\|_{L^q(0,t)}.$$
		\item We define the following norms 
		\begin{equation*}
			\begin{aligned}
				&\|f\|_{X_{a}}^2
				=\|{\rm e}^{aA^{-\frac{1}{3}}t}f\|^2_{L^{\infty}L^{2}}
				+\|{\rm e}^{aA^{-\frac{1}{3}}t}\nabla\triangle^{-1}\partial_{x} f\|^2_{L^{2}L^{2}} 
				+\frac{\|{\rm e}^{aA^{-\frac{1}{3}}t}f\|^2_{L^{2}L^{2}}}{A^{\frac{1}{3}}}
				+\frac{\|{\rm e}^{aA^{-\frac{1}{3}}t}\nabla f\|^2_{L^{2}L^{2}}}{A},\\
				&\|f\|_{Y_{0}}^2
				=\|f\|^2_{L^{\infty}L^{2}}
				+\frac{1}{A}\|\nabla f\|^2_{L^{2}L^{2}},
			\end{aligned}
		\end{equation*}
		where $a$ is a positive constant.
		
		\item The total mass $ \|n(t)\|_{L^{1}}$ is denoted by $ M .$ Clearly, $ M:=\|n(t)\|_{L^{1}}=\|n_{\rm in}\|_{L^{1}}.$
		
		\item Throughout this paper, we denote $C$ by  a positive constant independent of $A$, $t$ and the initial data, and it may be different from line to line.
	\end{itemize}
	
	\section{Key ingredients in the proof}
	
	For the 2D PKS-NS system, the enhanced dissipation plays an
	important role in stabilizing solutions and suppressing blow-up (for example, see \cite{he05,zeng}).
	In 2D finite channel, one needs to consider the impact of the boundary layer effect in addition to the enhanced dissipation  (for example, see \cite{cui1}).
	For the 3D PKS system in either unbounded or bounded domains (for example, see \cite{Bedro2,wangweike2}), it is necessary to consider the enhanced dissipation and the boundary layer effect, which are similar to the 2D case. It should be noted that the zero mode of $n$ is dependent on $y$ and $z$ in the 3D case, which is different from the 2D case and will bring new difficulties to estimate $\|n_0\|_{L^{\infty}L^2}.$
	Generally, it's more complex in the 3D PKS-NS system, and one needs to address other complex phenomena such as the 3D lift-up effect, nonlinear interactions, and  energy transfer mechanisms etc., all of which contribute to the system's behavior and stability. It is necessary to fully understand the connections between these factors in order to analyze this system clearly, which are stated as follows.

	\subsection{Some mechanisms affecting the nonlinear stability}
	\
	
	\noindent$\bullet$ \textbf{3D lift-up effect.}
	Consider a simplified equation for the zero mode part of the velocity
	\begin{equation*}
		\left\{
		\begin{array}{lr}
			\partial_tu_0-\frac{1}{A}\triangle u_0+\left(
			\begin{array}{c}
				u_{2,0} \\
				0 \\
				0 \\
			\end{array}
			\right)
			=0, \\
			\nabla \cdot u_0=0.
		\end{array}
		\right.
	\end{equation*} 
	The solution of this linear problem is given by
	\begin{equation*}
		\left\{
		\begin{array}{lr}
			u_{1,0}(t)=e^{{A}^{-1} t \triangle }\big( (u_{1,{\rm in}})_0-t(u_{2,{\rm in}})_0 \big),\\
			u_{2,0}(t)=e^{{A}^{-1} t \triangle }(u_{2,{\rm in}})_0,\\
			u_{3,0}(t)=e^{{A}^{-1} t \triangle }(u_{3,{\rm in}})_0.
		\end{array}
		\right.
	\end{equation*} 
	When $t\lesssim A,$ there is linearized growth for $u_{1,0}(t)$, which is called the ``3D lift-up effect''. It is an important factor leading to the 3D instability in the PKS-NS system and more details we refer to \cite{Bedro1} or \cite{wei2}. 
	
	\noindent$\bullet$ \textbf{Nonlinear interactions}
	
	In 3D space, there are complex nonlinear interactions among different modes of the solution, which influence the stability of the system.
	For a given function $f$, we decompose $f$ into $f=f_{\neq}+f_{(0,0)}+f_{(0,\neq)},$ then  nonlinear interactions can be classified as follows:
	\begin{itemize}
		\item $\neq \cdot \neq \rightarrow ~\neq~{\rm or}~0,$
		\item $0~\cdot \neq \rightarrow ~\neq,$
		\item $ 0\cdot 0\rightarrow~0, $
		\item $(0,\neq) \cdot (0,\neq) \rightarrow ~(0,\neq)~{\rm or}~(0,0),$
		\item $(0,0) \cdot (0,\neq) \rightarrow ~(0,\neq),$
		\item $(0,0) \cdot (0,0) \rightarrow ~(0,0).$
	\end{itemize}
	These nonlinear interactions can be directly observed with the help of Fourier series.
	
	\noindent$\bullet$ \textbf{Energy transfer mechanisms}
	
	In the PKS-NS system, energy transfer mechanisms exist either between the cell density and the velocity field or within each of them individually. These mechanisms, involving both linear and nonlinear interactions, directly impact the nonlinear stability of the system.
	Therefore, analyzing these transfer mechanisms is essential for understanding and predicting the stability of the system. 
	
	Energy transfer mechanisms are classified as follows:
	\begin{itemize}
		\item Linear transfer between the same Fourier modes,
		\item Nonlinear transfer between the same Fourier modes,
		\item Nonlinear transfer between the different Fourier modes.
	\end{itemize}
	It is crucial to note that there is no linear energy transfer between different modes, which is important for us to estimate the zero modes. 
	
	%

	\subsection{Main ideas for constructing the energy functional}\
	
	$\bullet$ {\bf Constructing the energy functional of zero modes:}
	Firstly, the zero modes for $u_{2}$ and $u_3$ satisfy 
	\begin{equation*}
		\left\{
		\begin{array}{lr}
			\partial_tu_{2,0}-\frac{1}{A}\triangle u_{2,0}
			+\frac{1}{A}(u\cdot\nabla u_{2})_0
			+\frac{1}{A}\partial_yP^{N_1}_0
			+\frac{1}{A}\partial_y P^{N_2}_0+\frac{1}{A}\partial_y P^{N_3}_0=\frac{n_0}{A}, \\
			\partial_tu_{3,0}-\frac{1}{A}\triangle u_{3,0}
			+\frac{1}{A}(u\cdot\nabla u_3)_0
			+\frac{1}{A}\partial_zP^{N_1}_0
			+\frac{1}{A}\partial_z P^{N_2}_0
			+\frac{1}{A}\partial_z P^{N_3}_0=0,\\
			\partial_yu_{2,0}+\partial_zu_{3,0}=0.
		\end{array}
		\right.
	\end{equation*}	
	Thus $u_{2,0}$ and $u_{3,0}$ can be regarded as only affected by $n_{(0,\neq)}$
	and not being influenced by $n_{(0,0)},$ since  $$n_0-\partial_yP^{N_2}_{0}
	=\partial_z^2\triangle^{-1}n_{0}=\partial_z^2\triangle^{-1}n_{(0,\neq)}$$
	and
	$$\partial_zP^{N_2}_{0}=\partial_y\partial_z\triangle^{-1}n_0
	=\partial_y\partial_z\triangle^{-1}n_{(0,\neq)}.$$
	As a consequence, we can decompose the zero mode $n_{0}$ into $n_0=n_{(0,0)}+n_{(0,\neq)},$ which satisfy
	\begin{equation}\label{eq:n0'}
		\left\{
		\begin{array}{lr}
			\partial_t n_{(0,0)}-\frac{1}{A}\partial_{yy} n_{(0,0)}=-\frac{1}{A}\partial_{y}\left(n_{(0,0)}
			\partial_{y}c_{(0,0)}\right)+``{\rm good~terms}",\\
			\partial_t n_{(0,\neq)}-\frac{1}{A}\triangle n_{(0,\neq)}
			=-\frac{1}{A}\Big(\partial_y\big(n_{(0,0)}\partial_{y}c_{(0,\neq)})+\partial_z\big(n_{(0,0)}\partial_{z}c_{(0,\neq)}\big)
			+\partial_y(n_{(0,\neq)}\partial_{y}c_{(0,0)})\\
			\qquad\qquad+\partial_y\big(u_{2,(0,\neq)}n_{(0,0)}\big)
			+\partial_z(u_{3,(0,\neq)}n_{(0,0)})\Big)
			+``{\rm good~terms}".
		\end{array}
		\right.
	\end{equation}
	
	{\bf Key observation 1. The introduction of energy estimates for  $\frac{\partial_z}{\sqrt{1-\triangle}}n_{(0,\neq)}$.} Direct energy estimate for $n_{(0,\neq)}$ shows that
\begin{equation*}\label{n0_sub}
		\begin{aligned}
		&\quad
		\frac{A^{2\epsilon}}{2}\|n_{(0,\neq)}\|^2_{L^{\infty}L^2}+
		A^{2\epsilon-1}\|\nabla n_{(0,\neq)}\|_{L^2L^2}^2\\
\leq& ~C+A^{2\epsilon-1}\big(\|n_{(0,0)}\partial_{z}c_{(0,\neq)}\|_{L^2L^2}+\|n_{(0,0)}u_{3,(0,\neq)}\|_{L^2L^2}\big)
		\|\partial_z n_{(0,\neq)}\|_{L^2L^2}\\
		&+A^{2\epsilon-1}\big(\|n_{(0,0)}\partial_{y}c_{(0,\neq)}\|_{L^2L^2}+\|n_{(0,0)}u_{2,(0,\neq)}\|_{L^2L^2}+\|n_{(0,\neq)}\partial_{y}c_{(0,0)}\|_{L^2L^2}\big)
		\|\partial_y n_{(0,\neq)}\|_{L^2L^2}+\cdots
		\end{aligned}
	\end{equation*}
and for example, for the last term by Lemma \ref{lem:ellip_3} and (\ref{eq:1DGN-1}) we have
\begin{equation*}\label{nz_temp3}
		\begin{aligned}
			\|n_{(0,\neq)}\partial_{y}c_{(0,0)}\|_{L^2L^2}&\leq 
			\|\partial_{y}c_{(0,0)}\|_{L^{\infty}L^2}
			\|n_{(0,\neq)}\|_{L^2_{t,x,z}L^{\infty}_y}\\
			&\leq  \frac{1}{\sqrt{2}}\|n_{(0,0)}\|_{L^{\infty}L^2}
			\|\partial_y n_{(0,\neq)}\|^{\frac12}_{L^2L^2}
			\|\partial_z n_{(0,\neq)}\|^{\frac12}_{L^2L^2},
		\end{aligned}
	\end{equation*}
which implies that a smallness condition on $\|n_{(0,0)}\|_{L^{\infty}L^2}$ seems to be necessary for 
	the estimate of $\|\nabla n_{(0,\neq)}\|_{L^2L^2}$.  That is to say, in addition to the condition of the initial mass $M$ in Theorem \ref{result}, we also need to add a condition that requires a smallness of $\|(n_{\rm in})_{(0,0)}\|_{L^2}$. To remove this assumption,
 a new  observation is that  $\|\partial_zn_{(0,\neq)}\|_{L^2L^2}$ is enough to close the estimates of both $\|\triangle u_{2,0}\|_{Y_0}$ and 
	$\|\triangle u_{3,0}\|_{Y_0}$, rather than $\|\nabla n_{(0,\neq)}\|_{L^2L^2}.$ 
In order to estimate $\|\partial_zn_{(0,\neq)}\|_{L^2L^2}$, multiplying $(\ref{eq:n0'})_2$ by  $\frac{\partial_{zz}}{{1-\triangle}}n_{(0,\neq)}$ or $\partial_{zz} c_{(0,\neq)}$ one can obtain the estimate of $\|\partial_zn_{(0,\neq)}\|_{L^2L^2}$.
	With the help of this,  we are able to completely remove the restriction on  $\|(n_{\rm in})_{(0,0)}\|_{L^2}$ (see Lemma \ref{lemma_n002}) and only need a restriction on initial cell mass $M.$

	The velocity $u_{1,0}$ will be affected by the 3D lift-up effect.
	Inspired by \cite{Chen1}, we decompose $u_{1,0}$ into 
	$u_{1,0}=\widehat{u_{1,0}}+\widetilde{u_{1,0}},$ satisfying
	\begin{equation}\label{eq:u10decom}
		\left\{
		\begin{array}{lr}
			\partial_t\widehat{u_{1,0}}-\frac{1}{A}\triangle\widehat{u_{1,0}}
			=-\frac{1}{A}\left(u_{2,0}\partial_y\widehat{u_{1,0}}+u_{3,0}\partial_z\widehat{u_{1,0}}\right)-u_{2,0},\\
			\partial_t\widetilde{u_{1,0}}-\frac{1}{A}\triangle\widetilde{u_{1,0}}
			=-\frac{1}{A}\left(u_{2,0}\partial_y\widetilde{u_{1,0}}+u_{3,0}\partial_z\widetilde{u_{1,0}}\right)-\frac{1}{A}(u_{\neq}\cdot \nabla u_{1,\neq})_0,\\
			\widehat{u_{1,0}}|_{t=0}=0,\quad\widetilde{u_{1,0}}|_{t=0}=(u_{1,\rm in})_0.
		\end{array}
		\right.
	\end{equation}
	In this way, $\widetilde{u_{1,0}}$ will not be affected by the 3D lift-up effect and will also remain unaffected by the linear energy transfer mechanism. 
	More importantly, if $A$ is big enough, it can be regard as a perturbation, and we will explain it in estimating non-zero modes (Lemma \ref{result_0_1}).
	
	Given the above decompositions, it is natural to introduce the first energy functional
	$E_1(t)=E_{1,1}(t)+E_{1,2}(t)+E_{1,3}(t)$ with
	\begin{equation}\label{eq:E1}
		\begin{aligned}
			E_{1,1}(t)=&\|n_{(0,0)}\|_{L^{\infty}L^2}
			+A^{\epsilon}\|\partial_z\nabla c_{(0,\neq)}\|_{L^{\infty}L^2}
			+\frac{\|\partial_zn_{(0,\neq)}\|_{L^2L^2} }{A^{\frac12-\epsilon}}
			+\|\partial_z^2 n_{(0,\neq)}\|_{Y_0},\\
			E_{1,2}(t)=&A^{\epsilon}
			\big(\|u_{2,0}\|_{Y_0}+\|u_{3,0}\|_{Y_0}+\|\nabla u_{2,0}\|_{Y_0}
			+\|\nabla u_{3,0}\|_{Y_0}+\|\triangle u_{2,0}\|_{Y_0}
			\\
			&\quad+\|\min\{(A^{-\frac{2}{3}}+A^{-1}t)^{\frac{1}{2}},1\}\triangle u_{3,0}\|_{Y_0}\big),\\
			E_{1,3}(t)=&A^{\epsilon}(A^{-1}{\|\widehat{u_{1,0}}\|_{L^{\infty}H^4}}
			+A^{-\frac32}{\|\nabla\widehat{u_{1,0}}\|_{L^{2}H^4}}+\|\partial_t\widehat{ u_{1,0}}\|_{L^{\infty}L^{2}}+\|\triangle\partial_t\widehat{ u_{1,0}}\|_{L^{\infty}L^{2}})\\
			&\quad+A^{-\frac13+\epsilon}(\|\widetilde{u_{1,0}}\|_{Y_0}
			+\|\triangle\widetilde{u_{1,0}}\|_{Y_0}),
		\end{aligned}
	\end{equation}
	where $\epsilon$ is a constant defined by $$\epsilon=\left\{\begin{array}{lr}
		\epsilon_0,\quad \epsilon_0\in(\frac13,\frac49],\\
		\frac{4}{9},\quad \epsilon_0>\frac49.
	\end{array}\right.$$
	
	$\bullet$ {\bf Constructing the energy functional of non-zero modes:}
	To facilitate the estimates of non-zero modes, we use the new vorticity  $\omega_2=\partial_zu_1-\partial_xu_3$ and new velocity $\triangle u_2,$ then 
	\begin{equation}\label{ini2}
		\left\{
		\begin{array}{lr}
			\partial_tn+y\partial_x n-\frac{1}{A}\triangle n=-\frac{1}{A}\nabla\cdot(u n)-\frac{1}{A}\nabla\cdot(n\nabla c), \\
			\triangle c+n-c=0, \\
			\partial_t\omega_2+y\partial_x\omega_2
			-\frac{1}{A}\triangle\omega_2+\partial_zu_2=
			-\frac{1}{A}\partial_z(u\cdot\nabla u_1)+\frac{1}{A}\partial_x(u\cdot\nabla u_3), \\
			\partial_t\triangle u_2+y\partial_x \triangle u_2
			-\frac{1}{A}\triangle(\triangle u_2) =\frac{1}{A}\partial_x^2n+\frac{1}{A}\partial_z^2n 
			-\frac{1}{A}(\partial_x^2+\partial_z^2)(u\cdot\nabla u_2)\\
			\qquad
			+\frac{1}{A}\partial_y[\partial_x(u\cdot\nabla u_1)+\partial_z(u\cdot\nabla u_3)], \\
			\nabla \cdot u=0.
		\end{array}
		\right.
	\end{equation}
	Therefore, we can introduce the second energy functional 
	$E_2(t)=E_{2,1}(t)+E_{2,2}(t)$ with 
	\begin{equation*}\label{eq:E2}
		\begin{aligned}
			E_{2,1}(t)=&\|\partial_x^2n_{\neq}\|_{X_a}
			+\|\partial_z^2n_{\neq}\|_{X_a},\\
			E_{2,2}(t)=&A^{\frac{3}{4}\epsilon}(\|\triangle u_{2,\neq}\|_{X_a}+\|\partial_x\omega_{2,\neq}\|_{X_a})
			+A^{-\frac{1}{3}+\frac{3}{4}\epsilon}(\|\partial_y\omega_{2,\neq}\|_{X_a}
			+\|\partial_z\omega_{2,\neq}\|_{X_a}).
		\end{aligned}
	\end{equation*}
	To estimate  non-zero modes of the cell density, the third energy functional is introduced as follows:
	\begin{equation*}\label{eq:E3}
		\begin{aligned}
			E_{3}(t)=\|n\|_{L^{\infty}L^{\infty}}.
		\end{aligned}
	\end{equation*}
	
	$\bullet$ {\bf Constructing the energy functional of non-zero modes with higher weight:}
	Considering 
	\begin{equation*}
		\begin{aligned}
			\partial_t\partial_z^2n_{\neq}+\Big(y+\frac{\widehat{u_{1,0}}}{A}
			\Big)\partial_x\partial_z^2 n_{\neq}-\frac{\triangle \partial_z^2n_{\neq}}{A}=
			-\frac{\partial_z^2\widehat{u_{1,0}}\partial_x n_{\neq}}{A}
			-\frac{2\partial_z\widehat{u_{1,0}}\partial_x\partial_z n_{\neq}}{A}
			+``{\rm good~terms}".
		\end{aligned}
	\end{equation*}  
	Note that it is difficult to close the energy  estimates by using $E_1(t), E_2(t), E_3(t)$ due to the bad terms $\partial_z^2\widehat{u_{1,0}}\partial_x n_{\neq}$, 
	$\partial_z\widehat{u_{1,0}}\partial_x\partial_z n_{\neq}$ and  the 3D lift-up effect. To overcome it, we introduce the fourth energy functional: 
	\begin{equation}\label{eq:E4}
		\begin{aligned}
			E_{4}(t)=\|\partial_{x}^2n_{\neq}\|_{X_{\frac{3}{2}a}}+\|\partial_x\partial_{z}n_{\neq}\|_{X_{\frac{3}{2}a}}.
		\end{aligned}
	\end{equation}
	Lastly, one still needs the fifth energy functional to close the estimates of $E_{4}(t)$:
	\begin{equation}\label{eq:E5}
		\begin{aligned}
			E_{5}(t)=A^{\frac{3}{4}\epsilon}\left(\|\partial^2_{x}u_{2,\neq}\|_{X_{\frac{3}{2}a}}+
			\|\partial_{x}^2u_{3,\neq}\|_{X_{\frac{3}{2}a}}\right).
		\end{aligned}
	\end{equation}
	The fifth energy functional $E_{5}(t)$ serves two important purposes: first, it is used to close the estimates of $E_{4}(t)$, and second, it is used to deal with the 3D lift-up terms in $(\ref{ini2})_3$ and $(\ref{ini2})_4$.
	
	{\bf Key observation 2. A  new quasi-linearized decomposition.} One needs to note that $E_5(t)$ is also related to $E_4(t)$, as there exists linear energy transfer mechanisms between $E_4(t)$ and $E_5(t).$
	To estimate $E_5,$ it is important to
	use the energy functional $E_4$ and the new quantity $W=u_{2,\neq}+ \frac{\partial_z\widehat{u_{1,0}}}{A+\partial_y\widehat{u_{1,0}}} u_{3,\neq}$ as in \cite{wei2}.
	We introduce a quasi-linearized decomposition $W=W^{(1)}+W^{(2)},$
	satisfying 
	\begin{equation*}
		\left\{
		\begin{array}{lr}
			\mathcal{L}_VW^{(1)}
			-2(\partial_y+\kappa\partial_z)
			\triangle^{-1}(\partial_yV\partial_xW^{(1)})
			=2(\partial_y+\kappa\partial_z)\triangle^{-1}
			(\partial_yV\partial_xW^{(2)})
			\\
			+\frac{n_{\neq}-\partial_yP^{N_2}_{\neq}}{A}-\frac{\kappa\partial_zP^{N_2}_{\neq}}{A}+``{\rm  good~part~1}",\\
			\mathcal{L}_VW^{(2)}
			=(\partial_t\kappa-\frac{\triangle\kappa}{A})u_{3,\neq}
			-\frac{2}{A}\nabla\kappa\cdot\nabla u_{3,\neq}+``{\rm  good~part~2}",\\
			W^{(1)}_{\rm in}=W_{\rm in},~~~ W^{(2)}_{\rm in}=0.
		\end{array}
		\right.
	\end{equation*}	
	In this way, we can  estimate $E_5$ directly without any other auxiliary terms: 
	$$E_{5}(t)\leq C
	A^{\frac{3}{4}\epsilon}
	\left(\|\partial_x\nabla W\|_{X_{\frac{3}{2}a}}+
	\|\partial_{x}^2u_{3,\neq}\|_{X_{\frac{3}{2}a}}\right)
	\leq C\big(\|(\partial_x^2n_{\rm in})_{\neq}\|^2_{L^2}
	+\|(\partial_z^2n_{\rm in})_{\neq}\|^2_{L^2}+1\big).$$
	(More details, see Lemma \ref{lemmaw21} and \ref{lemmaw22})	
	\subsection{Main steps}\
	
	\begin{proof}

	{\bf Proof of Theorem \ref{result}.}
	$\bullet$~\textbf{Step~1:} Let's designate $T$ as the terminal point of the largest range $[0, T]$ such that the following
	hypothesis hold
	\begin{equation} \label{assumption}	
		\begin{aligned}
			E_1(t)\leq 2E_1, \
			E_2(t)\leq 2E_2, \
			E_3(t)\leq 2E_3, \
			E_4(t)\leq 2E_4, \
			E_5(t)\leq 2E_5,
		\end{aligned}
	\end{equation}
	for any $t\in[0, T]$, where $E_1$, $E_2$, $E_3$, $E_4$ and $E_5$ are constants independent of
	$t$ and $A$ and
	will be decided during the calculation.
	
	$\bullet$~\textbf{Step~2:} 
	We need to prove the following propositions:
	\begin{proposition}{}\label{pro0}
		Under the conditions of  Theorem \ref{result} and  the assumptions (\ref{assumption}),
		as long as $$M< \frac{24\pi^2}{5},$$
		there exists a positive constant $C_{(1)}$ independent of $A$ and $t,$
		such that if $A>C_{(1)},$
		$$E_1(t)\leq E_1,$$
		for all $t\in(0,T].$
	\end{proposition}
	\begin{proposition}{}\label{pro1}
		Under the conditions of Theorem \ref{result} and  the assumptions (\ref{assumption}),
		there exists a positive constant $C_{(2)}$ independent of $A$ and $t,$
		such that if $A>C_{(2)},$
		$$E_2(t)\leq E_2,$$
		for all $t\in(0,T].$
	\end{proposition}
	
	Due to $n=n_{\neq}+n_{(0,\neq)}+n_{(0,0)},$ 
	by Proposition \ref{pro0} and Proposition \ref{pro1}, 
	when $A>\max\{C_{(1)},C_{(2)}\}$,
	we have $$\|n\|_{L^{\infty}L^2}\leq 
	\|n_{\neq}\|_{L^{\infty}L^2}
	+\|n_{(0,\neq)}\|_{L^{\infty}L^2}+\|n_{(0,0)}\|_{L^{\infty}L^2}\leq E_1+E_2.$$
	Using the Moser's iteration reported as in \cite{CWW1} and elliptic
	estimates in Section \ref{sec_estimate_1},  we can prove that $$E_3(t)\leq C(E_{1}^{6}+E_{2}^{6}+1)(E_1+E_2+||n_{\rm in}||_{L^\infty}+1):=E_3,$$  
	for all $t\in(0,T].$
	\begin{proposition}{}\label{pro3}
		Under the conditions of Theorem \ref{result} and  the assumptions (\ref{assumption}),
		there exists a positive constant $C_{(3)}$ independent of $A$ and $t,$
		such that if $A>C_{(3)},$
		$$E_4(t)\leq E_4,$$
		for all $t\in(0,T].$
	\end{proposition}
	
	\begin{proposition}{}\label{pro4}
		Under the conditions of  Theorem \ref{result} and  the assumptions (\ref{assumption}),
		there exists a positive constant $C_{(4)}$ independent of $A$ and $t,$
		such that if $A>C_{(4)},$
		$$E_5(t)\leq E_5,$$
		for all $t\in(0,T].$
	\end{proposition}
	
	$\bullet$ \textbf{Step~3:} 
By the local well-posedness result in Theorem \ref{thm:local existence},
	 there exists a time $T^*>0$ such that a unique strong solution $(n,u)$ to the system {\rm (\ref{ini11})} exists on the interval $[0,T^*).$ If $T^*<\infty$, by Proposition \ref{pro0}-\ref{pro4}, we have
\beno\label{eq:blow-up condi'}
\sup_{0<t< T^*}
	\|(\partial_{x}^{2},\partial_{z}^{2})n(t,\cdot)\|_{L^{2}}+\|n(t,\cdot)\|_{L^{\infty}}\leq C,
\eeno
and next we will prove 
\ben\label{eq:nabla u2}\sup_{0<t< T^*}A^{-\frac{1}{12}}\|(\nabla u)(t,\cdot)\|_{L^{2}}\leq C.
\een
It follows from Lemma \ref{lemma_u} and the bounded-ness of $E_2$ that
\beno&&\left\|(\partial_x,
				\partial_z)u_{\neq}\right\|_{L^2}\leq C(\|\omega_{2,\neq}\|_{L^2}
				+\|\nabla u_{2,\neq}\|_{L^2}),\\
&&
\left\|(\partial_x,
				\partial_z)\partial_yu_{\neq}\right\|_{L^2}\leq C(\|\partial_y\omega_{2,\neq}\|_{L^2}
				+\|\triangle u_{2,\neq}\|_{L^2}),
\eeno
which imply 
\beno
\left\|\nabla u_{\neq}\right\|_{L^2}\leq C A^{\frac{1}{12}}.
\eeno
By \eqref{eq:u10} we have
\beno\label{eq:u10}
				\|{u_{1,0}}\|_{H^2}
				&\leq \|\widehat{u_{1,0}}\|_{H^2}+\|\widetilde{u_{1,0}}\|_{H^2}\leq \int_0^t \|\partial_s\widehat{u_{1,0}}(s)\|_{H^2}ds
				+\|\widetilde{u_{1,0}}\|_{H^2}\nonumber\\
				&\leq CE_{1,3}A^{\frac13-\epsilon}(1+A^{-\frac13}t),
			\eeno
which and the bounded-ness of $E_1$ imply that $\left\|\nabla u_{0}\right\|_{L^2}\leq C.$ Hence (\ref{eq:nabla u2}) holds.
Consequently, by the blow-up criterion \eqref{eq:blow-up condi} in Theorem \ref{thm:local existence}, the proof of Theorem \ref{result} is complete. \end{proof}

	
	
	\section{Anisotropic Sobolev embeddings and a priori estimates}\label{sec_estimate_1}
	
	\subsection{Anisotropic Sobolev embeddings}
	Here, we give some anisotropic Sobolev inequalities based on the Fourier analysis. These inequalities play important roles in estimating nonlinear interaction terms.
	
	Before the beginning, several important facts must be stated, which can be  proven directly.
	For a given function $f=f(x,y,z)$, by Fourier series there hold
	\begin{equation}\label{eq:fourier ine}
		\begin{aligned}
			&\|f_{\neq}\|^2_{L^2}\leq \|\partial_x^jf_{\neq}\|^2_{L^2},\quad
			\|\partial_xf\|^2_{L^2}\leq \|\partial_x^jf\|^2_{L^2},\\
			&\|f_{(0,\neq)}\|^2_{L^2}\leq \|\partial_z^jf_{(0,\neq)}\|^2_{L^2},\quad
			\|\partial_zf\|^2_{L^2}\leq \|\partial_z^jf\|^2_{L^2},
		\end{aligned}
	\end{equation}
	where $j$ is a positive constant with $j\geq1$ and $\|f\|_{L^2}$ denotes $\|f\|_{L^2(\mathbb{T}\times\mathbb{R}\times\mathbb{T})}$. 
	
	
	\subsubsection{Sobolev inequalities for the $L^{\infty}$ norm}
	
	The following lemma can be used to estimate the $L^{\infty}$ norm for the zero mode.
	\begin{lemma}\label{sob_inf_1}
		For a given function $f(x,y,z)$ and $f_0=\frac{1}{|\mathbb{T}|}\int_{\mathbb{T}}{f}(t,x,y,z)dx,$ we have
		\begin{equation}\label{sob_result_1}
			\begin{aligned}
				&\|f_0\|_{L^{\infty}}\leq C\left(\|\partial_yf_0\|^{\frac{1}{2}}_{L^2}\|f_0\|^{\frac{1}{2}}_{L^2}+\|\partial_y\partial_zf_0\|^{\frac{1}{2}}_{L^2}
				\|\partial_zf_0\|^{\alpha-\frac{1}{2}}_{L^2}
				\|f_0\|^{1-\alpha}_{L^2}\right),\\
				&\|f_0\|_{L^{\infty}}
				\leq 
				C\left(\|\partial_yf_0\|^{\frac{1}{2}}_{L^2}\|f_0\|^{\frac{1}{2}}_{L^2}+\|\partial_y\partial_zf_0\|^{\alpha-\frac{1}{2}}_{L^2}
				\|\partial_zf_0\|^{\frac{1}{2}}_{L^2}
				\|\partial_yf_0\|^{1-\alpha}_{L^2}\right),\\
				&\|f_{0}\|_{L^{\infty}_{z}L^2_y}\leq C\left(\|f_0\|_{L^2}+\|\partial_zf_0\|_{L^2}^{\alpha}\|f_0\|_{L^2}^{1-\alpha}\right),\\
				&\|f_{0}\|_{L^{\infty}_yL^{2}_{z}}
				\leq \|\partial_yf_0\|_{L^2}^{\frac{1}{2}}
				\|f_0\|_{L^2}^{\frac{1}{2}},
			\end{aligned}
		\end{equation}	
		where $\alpha$ is a constant with $\alpha\in(\frac{1}{2},1].$
	\end{lemma}
	\begin{proof}
		
		{\bf Estimate of $\eqref{sob_result_1}_1$.}
		Thanks to the Fourier series $f_0=\sum_{k_3\in \mathbb{Z}}\widehat{f}_{0,k_3}(t,y)e^{ik_3z},$ there holds
		\begin{equation*}
			\begin{aligned}
				\|f_0\|_{L^{\infty}}&\leq \sum_{k_3\in \mathbb{Z}}\|\widehat{f}_{0,k_3}(t,y)\|_{L^{\infty}}\leq 
				\sum_{k_3\in \mathbb{Z}}\|\partial_y\widehat{f}_{0,k_3}(t,y)\|^{\frac{1}{2}}_{L^{2}}
				\|\widehat{f}_{0,k_3}(t,y)\|^{\frac{1}{2}}_{L^{2}}	\\
				&=\sum_{k_3\in \mathbb{Z}}(1+|k_3|^{\alpha})
				\|\partial_y\widehat{f}_{0,k_3}(t,y)\|^{\frac{1}{2}}_{L^{2}}
				\|\widehat{f}_{0,k_3}(t,y)\|^{\frac{1}{2}}_{L^{2}}\frac{1}{1+|k_3|^{\alpha}},
			\end{aligned}
		\end{equation*}
		where we used (\ref{eq:1DGN-1}) and  $\alpha$ is a constant with $\alpha\in(\frac{1}{2},1].$
		Using H$\rm \ddot{o}$lder's inequality, we obtain 
		\begin{equation}\label{zero_proof_1}
			\begin{aligned}
				\|f_0\|_{L^{\infty}}
				&\leq C\Big(\sum_{k_3\in \mathbb{Z}}
				\|\partial_y\widehat{f}_{0,k_3}(t,y)\|_{L^{2}}
				\|\widehat{f}_{0,k_3}(t,y)\|_{L^{2}}
				+\sum_{k_3\in \mathbb{Z}}|k_3|^{2\alpha}
				\|\partial_y\widehat{f}_{0,k_3}(t,y)\|_{L^{2}}
				\|\widehat{f}_{0,k_3}(t,y)\|_{L^{2}}\Big)^{\frac{1}{2}}\\
				&\leq C\Big(\|\partial_yf_0\|_{L^2}\|f_0\|_{L^2}
				+\sum_{k_3\in \mathbb{Z}}
				\|k_3\partial_y\widehat{f}_{0,k_3}(t,y)\|_{L^{2}}
				\|k_3\widehat{f}_{0,k_3}(t,y)\|^{2\alpha-1}_{L^{2}}
				\|\widehat{f}_{0,k_3}(t,y)\|^{2-2\alpha}_{L^{2}}
				\Big)^{\frac{1}{2}}.
			\end{aligned}
		\end{equation}
		Furthermore, by  H$\rm \ddot{o}$lder's inequality again we get $\eqref{sob_result_1}_1$.
		
		{\bf Estimate of $\eqref{sob_result_1}_2$.}
		If we estimate (\ref{zero_proof_1}) in another way:
		\begin{equation*}
			\begin{aligned}
				\|f_0\|_{L^{\infty}}
				\leq C\Big(\|\partial_yf_0\|_{L^2}\|f_0\|_{L^2}
				+\sum_{k_3\in \mathbb{Z}}
				\|\partial_y\widehat{f}_{0,k_3}(t,y)\|_{L^{2}}^{2-2\alpha}
				\|k_3\partial_y\widehat{f}_{0,k_3}(t,y)\|^{2\alpha-1}_{L^{2}}
				\|k_3\widehat{f}_{0,k_3}(t,y)\|_{L^{2}}
				\Big)^{\frac{1}{2}},
			\end{aligned}
		\end{equation*}
		then one can prove $\eqref{sob_result_1}_2$ directly.
		
		{\bf Estimate of $\eqref{sob_result_1}_3$.}		
		For  $f_{0}=\sum_{k_3\in \mathbb{Z}}\widehat{f}_{0,k_3}(t,y)e^{ik_3z},$
		we have $$\|f_{0}\|_{L^{\infty}_{z}L^2_y}
		\leq \sum_{k_3\in \mathbb{Z}}\|\widehat{f}_{0,k_3}(t,y)\|_{L^2_y}
		= \sum_{k_3\in \mathbb{Z}}(1+|k_3|^{\alpha})\|\widehat{f}_{0,k_3}(t,y)\|_{L^2_y}\frac{1}{1+|k_3|^{\alpha}},$$
		where $\alpha\in(\frac{1}{2},1].$
		Using H$\rm \ddot{o}$lder's inequality, one immediately obtained $\eqref{sob_result_1}_3$.

		{\bf Estimate of $\eqref{sob_result_1}_4$.}
		Due to $\|f_{0}\|_{L^{2}_{z}}^2
		\leq|\mathbb{T}| \sum_{k_3\in \mathbb{Z}}|\widehat{f}_{0,k_3}(t,y)|^2,$ there holds
		\begin{align*}
			\|f_{0}\|_{L^{\infty}_yL^{2}_{z}}^2
			&\leq |\mathbb{T}|\sum_{k_3\in \mathbb{Z}}||\widehat{f}_{0,k_3}(t,y)||^2_{L^{\infty}_y}
			\leq |\mathbb{T}|\sum_{k_3\in \mathbb{Z}}
			\|\partial_y\widehat{f}_{0,k_3}(t,y)\|_{L^2_y}
			\|\widehat{f}_{0,k_3}(t,y)\|_{L^2_y}\\
			&\leq |\mathbb{T}|\Big(\sum_{k_3\in \mathbb{Z}}
			\|\partial_y\widehat{f}_{0,k_3}(t,y)\|_{L^2_y}^2\Big)^{\frac{1}{2}}
			\Big(\sum_{k_3\in \mathbb{Z}}
			\|\widehat{f}_{0,k_3}(t,y)\|_{L^2_y}^2\Big)^{\frac{1}{2}}
			=\|\partial_yf_0\|_{L^2}\|f_0\|_{L^2},
		\end{align*} 
		where we use (\ref{eq:1DGN-1}) and
		$\|\partial_yf_0\|_{L^2}^{2}=|\mathbb{T}|\sum_{k_3\in \mathbb{Z}}
		\|\partial_y\widehat{f}_{0,k_3}(t,y)\|_{L^2_y}^2,~
		\|f_0\|_{L^2}^{2}=|\mathbb{T}|\sum_{k_3\in \mathbb{Z}}
		\|\widehat{f}_{0,k_3}(t,y)\|_{L^2_y}^2.$
		
		The proof is complete.
	\end{proof}
	
	The following lemma can be used to estimate the $L^{\infty}$ norm for the non-zero mode.		
	\begin{lemma}\label{sob_inf_2}
		For a given function $g=g(x,y,z)$, if $g_0=\frac{1}{|\mathbb{T}|}\int_{\mathbb{T}}{g}(t,x,y,z)dx=0,$ then we have
		\begin{equation}\label{sob_result_2}
			\begin{aligned}
				&\|g\|_{L^{\infty}}\leq 
				C\Big(\|\partial_y\partial_zg\|^{\frac{1}{2}}_{L^2}
				\|\partial_x\partial_zg\|^{\alpha-\frac{1}{2}}_{L^2}
				\|\partial_x^2g\|^{\alpha-\frac{1}{2}}_{L^2}
				\|\partial_xg\|^{\frac{3}{2}-2\alpha}_{L^2}+
				\|\partial_x\partial_yg\|^{\frac{1}{2}}_{L^2}
				\|\partial_xg\|^{\alpha-\frac{1}{2}}_{L^2}
				\|g\|^{1-\alpha}_{L^2}
				\Big),\\
				&\|g\|_{L^{\infty}_{y,z}L^2_{x}}\leq C\Big(\|\partial_yg\|^{\frac{1}{2}}_{L^2}\|g\|^{\frac{1}{2}}_{L^2}
				+\|\partial_zg\|^{\frac{1}{2}}_{L^2}
				\|\partial_z\partial_yg\|^{\alpha-\frac{1}{2}}_{L^2}
				\|\partial_yg\|^{1-\alpha}_{L^2}\Big),\\
				&\|g\|_{L^{\infty}_{x,y}L^2_{z}}\leq C
				\|\partial_xg\|^{\frac{1}{2}}_{L^2}
				\|\partial_x\partial_yg\|^{\alpha-\frac{1}{2}}_{L^2}
				\|\partial_yg\|^{1-\alpha}_{L^2},\\
				&\|g\|_{L^{\infty}_{x,z}L^2_{y}}\leq  C\big(\|\partial_xg\|^{\alpha}_{L^2}\|g\|^{1-\alpha}_{L^2}
				+\|\partial_x\partial_zg\|^{\alpha}_{L^2}\|g\|^{1-\alpha}_{L^2}\big),\\
				&\|g\|_{L^{\infty}_{x}L^2_{y,z}}\leq C\|\partial_xg\|_{L^2}^{\alpha}\|g\|_{L^2}^{1-\alpha},\\
				&\|g\|_{L^{\infty}_{z}L^2_{y,x}}\leq C(\|g\|_{L^2}+\|\partial_zg\|_{L^2}^{\alpha}\|g\|_{L^2}^{1-\alpha}),\\
				&\|g\|_{L^{\infty}_{y}L^2_{x,z}}\leq  \|\partial_yg\|_{L^2}^{\frac{1}{2}}\|g\|^{\frac{1}{2}}_{L^2},\\
				&\|g\|_{L^{\infty}_{x,z}L^2_{y}}\leq  C\big(\|\partial_xg\|^{\alpha}_{L^2}\|g\|^{1-\alpha}_{L^2}
				+\|\partial_x\partial_zg\|^{\frac{1}{2}}_{L^2}\|\partial_xg\|^{\alpha-\frac{1}{2}}_{L^2}\|\partial_zg\|^{\alpha-\frac{1}{2}}_{L^2}\|g\|^{\frac{3}{2}-2\alpha}_{L^2}\big),
			\end{aligned}
		\end{equation}	
		where $\alpha\in(\frac{1}{2},\frac{3}{4}]$ for $(\ref{sob_result_2})_1$ and $(\ref{sob_result_2})_8$, and 
		$\alpha\in(\frac{1}{2},1]$ for others.
	\end{lemma}
	\begin{proof}
		
		{\bf Estimates of $\eqref{sob_result_2}_1$ and $\eqref{sob_result_2}_8$.}
		Due to $g_0=0$, we denote $	g(x,y,z)$ by
		$$\sum_{k_1,k_3\in\mathbb{Z},k_1\neq0}\widehat{g}_{k_1,k_3}(y){\rm e}^{i(k_1x+k_3z)},$$
		then by (\ref{eq:1DGN-1})
		\begin{equation*}	
			\begin{aligned}
				&\qquad\|g\|_{L^{\infty}}\leq \sum_{k_1\neq0,k_3\in\mathbb{Z}}\|\widehat{g}_{k_1,k_3}(y)\|_{L^{\infty}_y}
				\leq \sum_{k_1\neq0 ,k_3\in\mathbb{Z}}\|\widehat{g}_{k_1,k_3}(y)\|^{\frac{1}{2}}_{L^2_y}
				\|\partial_y\widehat{g}_{k_1,k_3}(y)\|^{\frac{1}{2}}_{L^2_y}\\
				&= \sum_{k_1\neq0,k_3\in\mathbb{Z}}
				\frac{\|k_1\widehat{g}_{k_1,k_3}(y)\|^{\alpha-\frac{1}{2}}_{L^2_y}
					\|\widehat{g}_{k_1,k_3}(y)\|^{1-\alpha}_{L^2_y}
					\|k_1\partial_y\widehat{g}_{k_1,k_3}(y)\|^{\frac{1}{2}}_{L^2_y}}{|k_1|^{\alpha}(1+|k_3|^{\alpha})}\\
				&+\sum_{k_1\neq0,k_3\in\mathbb{Z}}
				\frac{		\|k_3\partial_y\widehat{g}_{k_1,k_3}(y)\|^{\frac{1}{2}}_{L^2_y}	\|k_1k_3\widehat{g}_{k_1,k_3}(y)\|^{\alpha-\frac{1}{2}}_{L^2_y}
					\|k_1^2\widehat{g}_{k_1,k_3}(y)\|^{\alpha-\frac{1}{2}}_{L^2_y}
					\|k_1\widehat{g}_{k_1,k_3}(y)\|^{\frac{3}{2}-2\alpha}_{L^2_y}}
				{|k_1|^{\alpha}(1+|k_3|^{\alpha})},
			\end{aligned}
		\end{equation*}	
		where $\alpha$ is a constant with $\alpha\in(\frac{1}{2},\frac{3}{4}].$
		Using H$\rm \ddot{o}$lder's inequality, we get
		\begin{equation*}	
			\begin{aligned}
				\|g\|_{L^{\infty}}&\leq 
				C\Big(\|\partial_y\partial_zg\|^{\frac{1}{2}}_{L^2}
				\|\partial_x\partial_zg\|^{\alpha-\frac{1}{2}}_{L^2}
				\|\partial_x^2g\|^{\alpha-\frac{1}{2}}_{L^2}
				\|\partial_xg\|^{\frac{3}{2}-2\alpha}_{L^2}+
				\|\partial_x\partial_yg\|^{\frac{1}{2}}_{L^2}
				\|\partial_xg\|^{\alpha-\frac{1}{2}}_{L^2}
				\|g\|^{1-\alpha}_{L^2}
				\Big).
			\end{aligned}
		\end{equation*}	
		Similarly, one can prove $\eqref{sob_result_2}_8$ without the step of  using (\ref{eq:1DGN-1}).

		{\bf Estimates of $\eqref{sob_result_2}_2$ and $\eqref{sob_result_2}_3$.}
		Rewrite $g$ into $g=\sum_{k_3\in \mathbb{Z}}\widehat{g}_{k_3}(x,y)e^{ik_3z},$
		and we have 
		\begin{align*}
			\|g\|_{L^{\infty}_{y,z}L^2_{x}}
			&\leq \sum_{k_3\in \mathbb{Z}}\|\widehat{g}_{k_3}(x,y)\|_{L^2_{x,y}}^{\frac{1}{2}}
			\|\partial_y\widehat{g}_{k_3}(x,y)\|_{L^2_{x,y}}^{\frac{1}{2}}
			\leq   
			\sum_{k_3\in \mathbb{Z}}\frac{\|\widehat{g}_{k_3}(x,y)\|_{L^2_{x,y}}^{\frac{1}{2}}
				\|\partial_y\widehat{g}_{k_3}(x,y)\|_{L^2_{x,y}}^{\frac{1}{2}}}{1+|k_3|^{\alpha}}\\
			&+\sum_{k_3\in \mathbb{Z}}
			\frac{\|k_3\widehat{g}_{k_3}(x,y)\|_{L^2_{x,y}}^{\frac{1}{2}}
				\|k_3\partial_y\widehat{g}_{k_3}(x,y)\|_{L^2_{x,y}}^{\alpha-\frac{1}{2}}\|\partial_y\widehat{g}_{k_3}(x,y)\|_{L^2_{x,y}}^{1-\alpha}}{1+|k_3|^{\alpha}},
		\end{align*}
		where $\alpha\in(\frac{1}{2},1]$.
		Using H$\rm \ddot{o}$lder's inequality three times, we obtain 
		\begin{align*}
			&\sum_{k_3\in \mathbb{Z}}
			\frac{\|k_3\widehat{g}_{k_3}(x,y)\|_{L^2_{x,y}}^{\frac{1}{2}}
				\|k_3\partial_y\widehat{g}_{k_3}(x,y)\|_{L^2_{x,y}}^{\alpha-\frac{1}{2}}\|\partial_y\widehat{g}_{k_3}(x,y)\|_{L^2_{x,y}}^{1-\alpha}}{1+|k_3|^{\alpha}} \\
			\leq&\Big(\sum_{k_3\in \mathbb{Z}}\|k_3\widehat{g}_{k_3}(x,y)\|_{L^2_{x,y}}^{2}\Big)^{\frac{1}{4}}\Big(\sum_{k_3\in \mathbb{Z}}\|k_3\partial_y\widehat{g}_{k_3}(x,y)\|_{L^2_{x,y}}^{2}\Big)^{\frac{2\alpha-1}{4}}
			\Big(\sum_{k_3\in \mathbb{Z}}\|\partial_y\widehat{g}_{k_3}(x,y)\|_{L^2_{x,y}}^{2}\Big)^{\frac{1-\alpha}{2}},
		\end{align*}
		thus
		$\|g\|_{L^{\infty}_{y,z}L^2_{x}}\leq C\left(\|\partial_yg\|^{\frac{1}{2}}_{L^2}\|g\|^{\frac{1}{2}}_{L^2}
		+\|\partial_zg\|^{\frac{1}{2}}_{L^2}
		\|\partial_z\partial_yg\|^{\alpha-\frac{1}{2}}_{L^2}
		\|\partial_yg\|^{1-\alpha}_{L^2}\right),$
		which gives $(\ref{sob_result_2})_2$.	
		
		By rewriting $g$ into $g(x,y,z)=\sum_{k_1\in\mathbb{Z},k_1\neq 0}\widehat{g}_{k_1}(y,z)e^{ik_1x},$ 
		one can prove $(\ref{sob_result_2})_3$, similarly.

		{\bf Estimate of $\eqref{sob_result_2}_4$.}
		Denote $	g(x,y,z)$ by
		$\sum_{k_1,k_3\in\mathbb{Z}, k_1\neq0}\widehat{g}_{k_1,k_3}(y){\rm e}^{i(k_1x+k_3z)},$
		then
		\begin{equation*}	
			\begin{aligned}
				\|g\|_{L^{\infty}_{x,z}L^2_y}&\leq \sum_{k_1\neq 0,k_3\in\mathbb{Z}}\|\widehat{g}_{k_1,k_3}(y)\|_{L^{2}_y}
				=\sum_{k_1\neq 0,k_3\in\mathbb{Z}}\frac{|k_1|^{\alpha}(1+|k_3|^{\alpha})\|\widehat{g}_{k_1,k_3}(y)\|_{L^2_y}}{|k_1|^{\alpha}(1+|k_3|^{\alpha})}\\
				&=\sum_{k_1\neq 0,k_3\in\mathbb{Z}}\frac{\|k_1\widehat{g}_{k_1,k_3}(y)\|^{\alpha}_{L^2_y}\|\widehat{g}_{k_1,k_3}(y)\|^{1-\alpha}_{L^2_y}
					+\|k_1k_3\widehat{g}_{k_1,k_3}(y)\|^{\alpha}_{L^2_y}\|\widehat{g}_{k_1,k_3}(y)\|^{1-\alpha}_{L^2_y}}
				{|k_1|^{\alpha}(1+|k_3|^{\alpha})}.
			\end{aligned}
		\end{equation*}	
		By H$\rm \ddot{o}$lder's inequality, we get
		$\|g\|_{L^{\infty}_{x,z}L^2_y}
		\leq C\left(\|\partial_xg\|^{\alpha}_{L^2}\|g\|^{1-\alpha}_{L^2}
		+\|\partial_x\partial_zg\|^{\alpha}_{L^2}\|g\|^{1-\alpha}_{L^2}\right).$
		
		{\bf Estimates of $\eqref{sob_result_2}_5$ and $\eqref{sob_result_2}_6$.}
		Rewrite $g$ into $g=\sum_{k_1\in \mathbb{Z},k_1\neq 0}\widehat{g}_{k_1}(y,z)e^{ik_1x},$
		then 
		$$\|g\|_{L^{\infty}_{x}L^2_{y,z}}
		\leq \sum_{k_1\neq 0}\|\widehat{g}_{k_1}(y,z)\|_{L^2_{y,z}}
		= \sum_{k_1\neq0}
		\frac{	|k_1|^{\alpha}\|\widehat{g}_{k_1}(y,z)\|_{L^2_{y,z}}}{|k_1|^{\alpha}},$$
		where $\alpha\in(\frac{1}{2},1].$
		Using H$\rm \ddot{o}$lder's inequality, we obtain 
		$\|g\|_{L^{\infty}_{x}L^2_{y,z}}
		\leq C\|\partial_xg\|_{L^2}^{\alpha}\|g\|_{L^2}^{1-\alpha}.$ 
		
		Similarly, by rewriting $g$ into $g=\sum_{k_3\in \mathbb{Z}}\widehat{g}_{k_3}(x,y)e^{ik_3z},$ one can prove $(\ref{sob_result_2})_6$.

		{\bf Estimate of $\eqref{sob_result_2}_7$.}
		Denote $g(x,y,z)$ by
		$\sum_{k_1\neq 0,k_3\in\mathbb{Z}}\widehat{g}_{k_1,k_3}(y){\rm e}^{i(k_1x+k_3z)},$
		then
		\begin{equation*}	
			\begin{aligned}
				\|g\|^2_{L^{\infty}_yL^{2}_{x,z}}&\leq|\mathbb{T}|^{2} \sum_{k_1\neq0,k_3\in\mathbb{Z}}\|\widehat{g}_{k_1,k_3}(y)\|^2_{L^{\infty}_y}
				&\leq  |\mathbb{T}|^{2}\sum_{k_1\neq 0,k_3\in\mathbb{Z}}\|\widehat{g}_{k_1,k_3}(y)\|_{L^{2}_y}\|\partial_y\widehat{g}_{k_1,k_3}(y)\|_{L^{2}_y}.
			\end{aligned}
		\end{equation*}	
		Using H\"{o}lder's inequality, we have 
		\begin{equation*}	
			\begin{aligned}
				\|g\|^2_{L^{\infty}_yL^{2}_{x,z}}&\leq |\mathbb{T}|^{2} \Big(\sum_{k_1\neq 0,k_3\in\mathbb{Z}}\|\widehat{g}_{k_1,k_3}(y)\|_{L^{2}_y}^2\Big)^\frac{1}{2}\Big(\sum_{k_1\neq0,k_3\in\mathbb{Z}}\|\partial_y\widehat{g}_{k_1,k_3}(y)\|_{L^{2}_y}^2\Big)^\frac{1}{2},
			\end{aligned}
		\end{equation*}	
		which implies $(\ref{sob_result_2})_7$. 
		
		The proof is complete.
	\end{proof}
	
	The following lemma can be used to estimate the $L^{\infty}$ norm for the z-part non-zero mode and the z-part zero mode.		
	\begin{lemma}\label{sob_inf_3}
		For a given function $f(x,y,z)$ and $f_{(0,0)}=\frac{1}{|\mathbb{T}|^2}\int_{\mathbb{T}\times\mathbb{T}}{f}(t,x,y,z)dxdz,$ there hold
		\begin{equation}\label{sob_result_3}
			\begin{aligned}
				&\|f_{(0,\neq)}\|_{L^{\infty}}\leq C\|\partial_y\partial_zf_{(0,\neq)}\|^{\frac{1}{2}}_{L^2}
				\|\partial_zf_{(0,\neq)}\|^{\alpha-\frac{1}{2}}_{L^2}
				\|f_{(0,\neq)}\|^{1-\alpha}_{L^2},\\
				&\|f_{(0,\neq)}\|_{L^{\infty}_{z}L^2_y}\leq C\|\partial_zf_{(0,\neq)}\|_{L^2}^{\alpha}\|f_{(0,\neq)}\|_{L^2}^{1-\alpha},\\
				&\|f_{(0,\neq)}\|_{L^{\infty}_yL^{2}_{z}}
				\leq \|\partial_yf_{(0,\neq)}\|_{L^2}^{\frac{1}{2}}
				\|f_{(0,\neq)}\|_{L^2}^{\frac{1}{2}},\\
				&\|f_{(0,0)}\|_{L^{\infty}}\leq 
				\|\partial_yf_{(0,0)}\|^{\frac{1}{2}}_{L^2}\|f_{(0,0)}\|^{\frac{1}{2}}_{L^2},\\
			\end{aligned}
		\end{equation}	
		where $\alpha$ is a constant with $\alpha\in(\frac{1}{2},1].$
	\end{lemma}
	\begin{proof}
		
		{\bf Estimate of $\eqref{sob_result_3}_1$.}
		Thanks to the Fourier series $f_{(0,\neq)}=\sum_{k_3\in \mathbb{Z},k_3\neq0}\widehat{f}_{0,k_3}(t,y)e^{ik_3z},$ there holds
		\begin{equation*}
			\begin{aligned}
				\|f_{(0,\neq)}\|_{L^{\infty}}&\leq \sum_{k_3\neq0}\|\widehat{f}_{0,k_3}(t,y)\|_{L^{\infty}}
				\leq \sum_{k_3\neq 0}
				\|\partial_y\widehat{f}_{0,k_3}(t,y)\|^{\frac{1}{2}}_{L^{2}}
				\|\widehat{f}_{0,k_3}(t,y)\|^{\frac{1}{2}}_{L^{2}}	\\
				&=\sum_{k_3\neq0}
				\frac{|k_3|^{\alpha-\frac12}\|k_3\partial_y\widehat{f}_{0,k_3}(t,y)\|^{\frac{1}{2}}_{L^{2}}
					\|\widehat{f}_{0,k_3}(t,y)\|^{\frac{1}{2}}_{L^{2}}}{|k_3|^{\alpha}},
			\end{aligned}
		\end{equation*}
		where $\alpha$ is a constant with $\alpha\in(\frac{1}{2},1].$
		Using H\"{o}lder's inequality, there holds
		$$	\|f_{(0,\neq)}\|_{L^{\infty}}\leq C\|\partial_y\partial_zf_{(0,\neq)}\|^{\frac{1}{2}}_{L^2}
		\|\partial_zf_{(0,\neq)}\|^{\alpha-\frac{1}{2}}_{L^2}
		\|f_{(0,\neq)}\|^{1-\alpha}_{L^2},$$ 
		which is $\eqref{sob_result_3}_1$.
		
		{\bf Estimate of $\eqref{sob_result_3}_2$.}
		Due to $\|f_{(0,\neq)}\|_{L^2_y}\leq \sum_{k_3\in \mathbb{Z},k_3\neq0}\|\widehat{f}_{0,k_3}(t,y)\|_{L^2_y},$
		then there holds $$\|f_{(0,\neq)}\|_{L^{\infty}_{z}L^2_y}
		\leq \sum_{k_3\neq{0}}\|\widehat{f}_{0,k_3}(t,y)\|_{L^2_y}
		=\sum_{k_3\neq{0}}
		\frac{\|k_3\widehat{f}_{0,k_3}(t,y)\|^{\alpha}_{L^2_y}\|\widehat{f}_{0,k_3}(t,y)\|^{1-\alpha}_{L^2_y}}{|k_3|^{\alpha}},$$
		where $\alpha\in(\frac{1}{2},1].$
		Using H$\rm \ddot{o}$lder's inequality, we obtain $\eqref{sob_result_3}_2$.
		
		{\bf Estimate of $\eqref{sob_result_3}_3$.}
		Due to $\|f_{(0,\neq)}\|_{L^{2}_{z}}^2
		\leq |\mathbb{T}|\sum_{k_3\in \mathbb{Z},k_3\neq0}|\widehat{f}_{0,k_3}(t,y)|^2,$ there holds
		\begin{align*}
			\|f_{(0,\neq)}\|_{L^{\infty}_yL^{2}_{z}}^2
			&\leq |\mathbb{T}|\sum_{k_3\neq0}||\widehat{f}_{0,k_3}(t,y)||^2_{L^{\infty}_y}
			\leq |\mathbb{T}|\sum_{k_3\neq0}
			\|\partial_y\widehat{f}_{0,k_3}(t,y)\|_{L^2_y}
			\|\widehat{f}_{0,k_3}(t,y)\|_{L^2_y}\\
			&\leq |\mathbb{T}|\Big(\sum_{k_3\neq0}
			\|\partial_y\widehat{f}_{0,k_3}(t,y)\|_{L^2_y}^2\Big)^{\frac{1}{2}}
			\Big(\sum_{k_3\neq0}
			\|\widehat{f}_{0,k_3}(t,y)\|_{L^2_y}^2\Big)^{\frac{1}{2}}
			=\|\partial_yf_{(0,\neq)}\|_{L^2}\|f_{(0,\neq)}\|_{L^2},
		\end{align*}
		which implies $\eqref{sob_result_3}_3$. 
		
		{\bf Estimate of $\eqref{sob_result_3}_4$.} The last result follows from the 1D Gagliardo-Nirenberg inequality \eqref{eq:1DGN-1}.
		
		The proof is complete.\end{proof}
	
	The following lemma can be used to deal with the interaction between zero modes.
	\begin{lemma}\label{sob_12}
		For  given functions $f(x,y,z)$ and $g(x,y,z)$, we have
		\begin{equation}\label{sob_result_4}
			\begin{aligned}
				&\|f_{0}g_{0}\|_{L^{2}}
				\leq C\big(\|\partial_yf_{0}\|^{\frac{1}{2}}_{L^2}\|f_{0}\|^{\frac{1}{2}}_{L^2}+\|\partial_y\partial_zf_{0}\|^{\frac{1}{2}}_{L^2}
				\|\partial_zf_{0}\|^{\alpha-\frac{1}{2}}_{L^2}
				\|f_{0}\|^{1-\alpha}_{L^2}\big)\|g_{0}\|_{L^{2}},\\
				&\|f_{0}g_{0}\|_{L^{2}}
				\leq C\left(\|f_{0}\|_{L^2}+\|\partial_zf_{0}\|_{L^2}^{\alpha}\|f_0\|_{L^2}^{1-\alpha}\right)\|\partial_yg_{0}\|_{L^2}^{\frac{1}{2}}
				\|g_{0}\|_{L^2}^{\frac{1}{2}},\\
				&\|\nabla(f_{0}g_{0})\|_{L^{2}}
				\leq C(\|f_{0}\|_{H^1}+\|\partial_zf_{0}\|_{H^1})\|g_{0}\|_{H^{1}},\\
				&\|\triangle(f_{0}g_{0})\|_{L^{2}}
				\leq C\|f_{0}\|_{H^2}\|g_{0}\|_{H^{2}},\\
				&\|\nabla\triangle(f_{0}g_{0})\|_{L^{2}}
				\leq C\|f_{0}\|_{H^3}\|g_{0}\|_{H^{3}},
			\end{aligned}
		\end{equation}	
		where $ \alpha $ is a constant with $\alpha\in(\frac12,1].$
	\end{lemma}
	\begin{proof}
		{\bf Estimate of $\eqref{sob_result_4}_1$.}
		It follows from  $(\ref{sob_result_1})_1$ that 
		\begin{align*}
			\|f_{0}g_{0}\|_{L^{2}} &\leq C\left(\|\partial_yf_{0}\|^{\frac{1}{2}}_{L^2}\|f_0\|^{\frac{1}{2}}_{L^2}+\|\partial_y\partial_zf_{0}\|^{\frac{1}{2}}_{L^2}
			\|\partial_zf_{0}\|^{\alpha-\frac{1}{2}}_{L^2}
			\|f_0\|^{1-\alpha}_{L^2}\right)\|g_{0}\|_{L^{2}}.
		\end{align*}
		
		{\bf Estimate of $\eqref{sob_result_4}_2$.}	
		By $(\ref{sob_result_1})_3$ and  $(\ref{sob_result_1})_4$, we get
		\begin{equation*}
			\begin{aligned}
				\|f_{0}g_{0}\|_{L^{2}}
				&\leq \big\|\|f_{0}\|_{L^{2}_y}\|g_{0}\|_{L^{\infty}_y}
				\big\|_{L^{2}_{z}}
				\leq  \|f_{0}\|_{L^{2}_yL^{\infty}_z}\|g_{0}\|_{L^{\infty}_yL^{2}_z}	\\
				&\leq  C\left(\|f_{0}\|_{L^2}+\|\partial_zf_{0}\|_{L^2}^{\alpha}\|f_{0}\|_{L^2}^{1-\alpha}\right)\|\partial_yg_{0}\|_{L^2}^{\frac{1}{2}}
				\|g_{0}\|_{L^2}^{\frac{1}{2}},
			\end{aligned}
		\end{equation*}
		which implies $\eqref{sob_result_4}_2$.
		
		{\bf Estimate of $\eqref{sob_result_4}_3$.}	
		Given that $\|\nabla(f_{0}g_{0})\|_{L^{2}}\leq \|\nabla f_{0}g_{0}\|_{L^{2}}+\| f_{0}\nabla g_{0}\|_{L^{2}},$
		using $(\ref{sob_result_1})_1$, we have
		\begin{align*}
			\| f_{0}\nabla g_{0}\|_{L^{2}}
			&\leq C\left(\|\partial_yf_0\|^{\frac{1}{2}}_{L^2}\|f_0\|^{\frac{1}{2}}_{L^2}+\|\partial_y\partial_zf_0\|^{\frac{1}{2}}_{L^2}
			\|\partial_zf_0\|^{\alpha-\frac{1}{2}}_{L^2}
			\|f_0\|^{1-\alpha}_{L^2}\right)\|\nabla g_{0}\|_{L^{2}}\\
			&\leq
			C(\|f_{0}\|_{H^1}+\|\partial_zf_{0}\|_{H^1})\|g_{0}\|_{H^{1}},
		\end{align*}
		using $(\ref{sob_result_1})_3$ and $(\ref{sob_result_1})_4$, there holds
		\begin{align*}
			\|\nabla f_{0}g_{0}\|_{L^{2}}
			&\leq  \big\|\|\nabla f_{0}\|_{L^{2}_y}\|g_{0}\|_{L^{\infty}_y}
			\big\|_{L^{2}_{z}}
			\leq  \|\nabla f_{0}\|_{L^{2}_yL^{\infty}_z}\|g_{0}\|_{L^{\infty}_yL^{2}_z}\\
			&\leq  C\left(\|\nabla f_{0}\|_{L^2}+\|\partial_z \nabla f_{0}\|_{L^2}^{\alpha}\|\nabla f_{0}\|_{L^2}^{1-\alpha}\right)\|\partial_yg_{0}\|_{L^2}^{\frac{1}{2}}
			\|g_{0}\|_{L^2}^{\frac{1}{2}}\\
			&\leq
			C(\|f_{0}\|_{H^1}+\|\partial_zf_{0}\|_{H^1})\|g_{0}\|_{H^{1}},
		\end{align*}
		which implies $\eqref{sob_result_4}_3$.
		
		{\bf Estimate of $\eqref{sob_result_4}_4$.}	
		Thanks to Lemma \ref{sob_inf_1}, we get
		\begin{align*}
			\|\triangle(f_{0}g_{0}) \|_{L^{2}}^2
			&\leq C(\|\triangle f_{0}g_{0} \|_{L^{2}}^2
			+\|\nabla f_{0}\cdot \nabla g_{0} \|_{L^{2}}^2
			+\|f_{0}\triangle g_{0} \|_{L^{2}}^2)\\
			&\leq C(\|\triangle f_{0}\|_{L^{2}}^2\|g_{0}\|_{L^{\infty}}^2
			+\|\nabla f_{0}\|_{L^{2}_yL^{\infty}_z}^2
			\|\nabla g_{0} \|_{L^{\infty}_yL^{2}_z}^2
			+\|\triangle g_{0}\|_{L^{2}}^2\|f_{0} \|_{L^{\infty}}^2)\\
			&\leq C\|f_{0} \|_{H^{2}}^2\|g_{0}\|_{H^{2}}^2.
		\end{align*}

		The proof of $\eqref{sob_result_4}_5$ is similar to $\eqref{sob_result_4}_4$. The proof is complete.
	\end{proof}
	
	The following lemma can be used to deal with the interaction between non-zero modes.
	\begin{lemma}\label{sob_13}
		For  given functions $f(x,y,z)$ and $g(x,y,z)$ satisfying $f_0=\frac{1}{|\mathbb{T}|}\int_{\mathbb{T}}{f}(t,x,y,z)dx=0$ and $g_0=\frac{1}{|\mathbb{T}|}\int_{\mathbb{T}}{g}(t,x,y,z)dx=0$, we have
		\begin{equation}\label{sob_result_5}
			\begin{aligned}
				&\|fg\|_{L^{2}}
				\leq C\big(\|f\|_{L^2}+\|\partial_zf\|_{L^2}^{\alpha}
				\|f\|_{L^2}^{1-\alpha}\big)
				\|\partial_xg\|^{\frac{1}{2}}_{L^2}
				\|\partial_x\partial_yg\|^{\alpha-\frac{1}{2}}_{L^2}
				\|\partial_yg\|^{1-\alpha}_{L^2},\\
				&\|fg\|_{L^{2}}
				\leq C\big(\|\partial_xf\|^{\alpha}_{L^2}\|f\|^{1-\alpha}_{L^2}
				+\|\partial_x\partial_zf\|^{\alpha}_{L^2}\|f\|^{1-\alpha}_{L^2}\big)\|\partial_yg\|_{L^2}^{\frac{1}{2}}\|g\|^{\frac{1}{2}}_{L^2},\\
				&\|fg\|_{L^{2}}
				\leq C\|\partial_xf\|_{L^2}^{\alpha}\|f\|_{L^2}^{1-\alpha}\Big(\|\partial_yg\|^{\frac{1}{2}}_{L^2}\|g\|^{\frac{1}{2}}_{L^2}
				+\|\partial_zg\|^{\frac{1}{2}}_{L^2}
				\|\partial_z\partial_yg\|^{\alpha-\frac{1}{2}}_{L^2}
				\|\partial_yg\|^{1-\alpha}_{L^2}\Big),\\
				&\|fg\|_{L^{2}}\leq C\Big(\|\partial_y\partial_zf\|^{\frac{1}{2}}_{L^2}
				\|(\partial_x,\partial_z)\partial_xf\|^{2\alpha-1}_{L^2}
				\|\partial_xf\|^{\frac{3}{2}-2\alpha}_{L^2}+
				\|\partial_x\partial_yf\|^{\frac{1}{2}}_{L^2}
				\|\partial_xf\|^{\alpha-\frac{1}{2}}_{L^2}
				\|f\|^{1-\alpha}_{L^2}
				\Big)\|g\|_{L^2},
			\end{aligned}
		\end{equation}	
		where $ \alpha\in(\frac12, 1] $ for the first three results and $ \alpha\in(\frac12, \frac34] $ for the last result.
	\end{lemma}
	\begin{proof}
		Using $(\ref{sob_result_2})_3$ and $(\ref{sob_result_2})_6,$
		there holds $$\|fg\|_{L^{2}}\leq \|f\|_{L^{\infty}_{z}L^{2}_{x,y}}
		\|g\|_{L^{\infty}_{x,y}L^{2}_{z}}\leq C\big(\|f\|_{L^2}+\|\partial_zf\|_{L^2}^{\alpha}
		\|f\|_{L^2}^{1-\alpha}\big)
		\|\partial_xg\|^{\frac{1}{2}}_{L^2}
		\|\partial_x\partial_yg\|^{\alpha-\frac{1}{2}}_{L^2}
		\|\partial_yg\|^{1-\alpha}_{L^2},$$
		which implies $\eqref{sob_result_5}_1$.
		
		Moreover, the inequality $\eqref{sob_result_5}_2$ follows from $(\ref{sob_result_2})_4$ and $(\ref{sob_result_2})_7.$
		$\eqref{sob_result_5}_3$ follows from $(\ref{sob_result_2})_2$ and $(\ref{sob_result_2})_5.$
		$\eqref{sob_result_5}_4$  follows from $(\ref{sob_result_2})_1$. 
		
		The proof is complete.	
	\end{proof}
	
	\begin{lemma}\label{lemma_non_zz0}
		For given functions $f(x,y,z)$ and $g(x,y,z)$, it holds
		\begin{equation}\label{sob_result_6}
			\begin{aligned}
				&\|(fg)_{(0,0)}\|_{L^2}\leq C\|f\|_{L^2}
				\|g\|^{\frac{1}{2}}_{L^2}
				\|\partial_yg\|_{L^2}^{\frac{1}{2}},\\
				&\|(fg)_{0}\|_{L^2}\leq C \big(\|f\|_{L^2}+\|\partial_zf\|_{L^2}^{\alpha}\|f\|_{L^2}^{1-\alpha}\big)
				\|g\|^{\frac{1}{2}}_{L^2}
				\|\partial_yg\|_{L^2}^{\frac{1}{2}},\\
				&\|(fg)_{0}\|_{L^2}\leq C \left(\|\partial_yf\|^{\frac{1}{2}}_{L^2}\|f\|^{\frac{1}{2}}_{L^2}+\|\partial_y\partial_zf\|^{\frac{1}{2}}_{L^2}
				\|\partial_zf\|^{\alpha-\frac{1}{2}}_{L^2}
				\|f\|^{1-\alpha}_{L^2}\right)
				\|g\|_{L^2},	
			\end{aligned}	
		\end{equation}
		where $ \alpha\in(\frac12, 1]. $
	\end{lemma}
	\begin{proof}
		{\bf Estimate of $\eqref{sob_result_6}_1$.}	
		Let $f=\sum_{k_1,k_3\in\mathbb{Z}}\widehat{f}_{k_1,k_{3}}(y)e^{ik_1x+ik_3z}$	
		and
		$g=\sum_{k_1,k_3\in\mathbb{Z}}\widehat{g}_{k_1,k_{3}}(y)e^{ik_1x+ik_3z},$	
		then 
		$(fg)_{(0,0)}=\sum_{k_1,k_3\in\mathbb{Z}}f_{-k_1,-k_3}(y)g_{k_1,k_3}(y).$
		Notice that by (\ref{eq:1DGN-1})
		$$\|\widehat{g}_{k_1,k_{3}}(y)\|_{L^{\infty}}\leq \|\widehat{g}_{k_1,k_{3}}(y)\|_{L^{2}}^{\frac12}\|\partial_{y}\widehat{g}_{k_1,k_{3}}(y)\|_{L^{2}}^{\frac12},$$
		then direct calculations yield that
		\begin{equation*}
			\begin{aligned}
				\|(fg)_{(0,0)}\|_{L^2}\leq& \sum_{k_1,k_3\in\mathbb{Z}}\|\widehat{f}_{-k_1,-k_{3}}(y)\widehat{g}_{k_1,k_{3}}(y)\|_{L^2}
				\leq \sum_{k_1,k_3\in\mathbb{Z}}\|\widehat{f}_{-k_1,-k_3}(y)\|_{L^2}\|\widehat{g}_{k_1,k_3}(y)\|_{L^{\infty}}\\
				\leq& \sum_{k_1,k_3\in\mathbb{Z}}\|\widehat{f}_{-k_1,-k_3}(y)\|_{L^2}\|\widehat{g}_{k_1,k_3}(y)\|^{\frac{1}{2}}_{L^{2}}\|\partial_y\widehat{g}_{k_1,k_3}(y)\|^{\frac{1}{2}}_{L^{2}}
				\leq C \|f\|_{L^{2}}\|g\|_{L^{2}}^{\frac12}
				\|\partial_{y}g\|_{L^{2}}^{\frac12}.
			\end{aligned}
		\end{equation*}
		
		{\bf Estimate of $\eqref{sob_result_6}_2$.}	
		Using $f=\sum_{k_1\in\mathbb{Z}}\widehat{f}_{k_1}(y,z)e^{ik_1x}$	
		and
		$g=\sum_{k_1\in\mathbb{Z}}\widehat{g}_{k_1}(y,z)e^{ik_1x},$	
		we have 
		$$(fg)_{0}=\sum_{k_1\in\mathbb{Z}}\widehat{f}_{-k_1}(y,z)\widehat{g}_{k_1}(y,z).$$
		Using $(\ref{sob_result_1})_3$ and $(\ref{sob_result_1})_4$, there holds
		\begin{equation}\label{tempfg_1}
			\begin{aligned}
				&\quad \|(fg)_{0}\|_{L^2}
				\leq\sum_{k_1\in\mathbb{Z}}
				\|\widehat{f}_{-k_1}(y,z)\widehat{g}_{k_1}(y,z)\|_{L^2}
				\leq\sum_{k_1\in\mathbb{Z}}
				\|\widehat{f}_{-k_1}(y,z)\|_{L^{\infty}_zL^2_y}
				\|\widehat{g}_{k_1}(y,z)\|_{L^{\infty}_yL^2_z}\\
				&\leq C\sum_{k_1\in\mathbb{Z}}
				\big(\|\widehat{f}_{-k_1}(y,z)\|_{L^2}
				+\|\partial_z\widehat{f}_{-k_1}(y,z)\|^{\alpha}_{L^2}
				\|\widehat{f}_{-k_1}(y,z)\|^{1-\alpha}_{L^2}\big)
				\|\widehat{g}_{k_1}(y,z)\|^{\frac{1}{2}}_{L^{2}}\|\partial_y\widehat{g}_{k_1}(y,z)\|^{\frac{1}{2}}_{L^{2}}\\
				&\leq C(\|f\|_{L^2}+\|\partial_zf\|_{L^2}^{\alpha}\|f\|_{L^2}^{1-\alpha})
				\|g\|^{\frac{1}{2}}_{L^2}
				\|\partial_yg\|_{L^2}^{\frac{1}{2}},
			\end{aligned}
		\end{equation}
		which implies $\eqref{sob_result_6}_2$.
		
		{\bf Estimate of $\eqref{sob_result_6}_3$.}	
		For (\ref{tempfg_1}), if we use 
		$\|(fg)_{0}\|_{L^2}\leq\sum_{k_1\in\mathbb{Z}}
		\|\widehat{f}_{-k_1}(y,z)\|_{L^{\infty}}
		\|\widehat{g}_{k_1}(y,z)\|_{L^2},$ we can prove $\eqref{sob_result_6}_3$ with the help of $(\ref{sob_result_1})_1.$
		
		The proof is complete.
	\end{proof}
	
	The following lemma is used to deal the  interactions between the zero mode and the z-part non-zero mode.
	\begin{lemma}\label{lemma_non_zz}
		For given functions $f=f(x,y,z)$ and $g=g(x,y,z)$, there hold
		\begin{equation}\label{sob_result_zz}
			\begin{aligned}
				&\|f_{(0,\neq)}g_{(0,\neq)}\|_{L^2}^2\leq C \big(\|\partial_yg_{(0,\neq)}\|_{L^2}\|g_{(0,\neq)}\|_{L^2}
				+\|\partial_zg_{(0,\neq)}\|_{L^2}
				\|\partial_z\partial_yg_{(0,\neq)}\|^{2\alpha-1}_{L^2}
				\|\partial_yg_{(0,\neq)}\|^{2-2\alpha}_{L^2}\big)\|f_{(0,\neq)}\|^2_{L^2},\\
				&\|f_{(0,0)}g_{(0,\neq)}\|^2_{L^2}\leq \|\partial_yg_{(0,\neq)}\|_{L^2}\|g_{(0,\neq)}\|_{L^2}
				\|f_{(0,0)}\|^2_{L^2},
			\end{aligned}
		\end{equation}
		where $ \alpha $ is a constant with $\alpha\in(\frac{1}{2},1]$.
	\end{lemma}
	\begin{proof}
		Due to $f_{(0,\neq)}$ is independent of the variable $x$, using $(\ref{sob_result_1})_2$, there holds
		\begin{align*}
			&\quad\|f_{(0,\neq)}g_{(0,\neq)}\|^2_{L^2}\leq \|g_{(0,\neq)}\|^2_{L^{\infty}_{y,z}}
			\|f_{(0,\neq)}\|^2_{L^2}\\
			&\leq C \Big(\|\partial_yg_{(0,\neq)}\|_{L^2}\|g_{(0,\neq)}\|_{L^2}
			+\|\partial_zg_{(0,\neq)}\|_{L^2}
			\|\partial_z\partial_yg_{(0,\neq)}\|^{2\alpha-1}_{L^2}
			\|\partial_yg_{(0,\neq)}\|^{2-2\alpha}_{L^2}\Big)
			\|f_{(0,\neq)}\|^2_{L^2},
		\end{align*}
		which implies $\eqref{sob_result_zz}_1$.
		
		Due to $f_{(0,0)}$ is independent of $x$ and $z$, using $(\ref{sob_result_1})_4$, there holds
		\begin{align*}
			\quad\|f_{(0,0)}g_{(0,\neq)}\|^2_{L^2}&=\int_{\mathbb{R}}
			\|g_{(0,\neq)}\|^2_{L^2_{z}}|f_{(0,0)}|^2dy\leq \|g_{(0,\neq)}\|^2_{L^{\infty}_{y}L^2_{z}}
			\|f_{(0,0)}\|^2_{L^2}\\
			&\leq \|\partial_yg_{(0,\neq)}\|_{L^2}\|g_{(0,\neq)}\|_{L^2}
			\|f_{(0,0)}\|^2_{L^2},
		\end{align*}
		which is $\eqref{sob_result_zz}_2$.
	\end{proof}

	\subsection{Elliptic estimates}
	The following  elliptic estimates are necessary.
	\begin{lemma}\label{lem:ellip_0}
		Let $c_0$ and $n_{0}$ be the zero mode of $c$ and $n$, respectively, satisfying
		$$-\triangle c_0+c_0=n_{0},$$
		then there hold
		\ben\label{eq:elliptic}
		&&\|\triangle c_0(t)\|_{L^2}+\|\nabla c_0(t)\|_{L^2}
		\leq C\|n_{0}(t)\|_{L^2},\nonumber\\
		&&\|\partial_z\triangle c_0(t)\|_{L^2}+\|\partial_z\nabla c_0(t)\|_{L^2}
		\leq C\|\partial_zn_{(0,\neq)}(t)\|_{L^2},\nonumber\\
		&&\|\partial_z^2\triangle c_0(t)\|_{L^2}+\|\partial_z^2\nabla c_0(t)\|_{L^2}
		\leq C\|\partial_z^2n_{(0,\neq)}(t)\|_{L^2},\nonumber\\
		&&\|\nabla c_0(t)\|_{L^4}\leq C\|n_{0}(t)\|_{L^2},
		\een
		for any $t\geq0$.
	\end{lemma}
	\begin{proof}
		The basic energy estimates yield
		\begin{equation}
			\begin{aligned}
				\|\triangle c_0(t)\|^2_{L^2}+2\|\nabla c_0(t)\|^2_{L^2}+\|c_0(t)\|^2_{L^2}
				=\|n_{0}(t)\|^2_{L^2},
				\nonumber
			\end{aligned}
		\end{equation}
		which implies $\eqref{eq:elliptic}_1$.
		Similarly, note that  $\|\partial_z^jn_{0}(t)\|_{L^2}=\|\partial_z^jn_{(0,\neq)}(t)\|_{L^2},$ and one can prove 
		$$\|\partial_z^j\triangle c_0(t)\|_{L^2}
		+\|\partial_z^j\nabla c_0(t)\|_{L^2}\leq C\|\partial_z^jn_{(0,\neq)}(t)\|_{L^2},$$
		where $j=1,2.$ Moreover,
		using the Gagliardo-Nirenberg inequality, we have
		$$\|\nabla c_0(t)\|_{L^4}
		\leq C\|\triangle c_0(t)\|^{\frac{1}{2}}_{L^2}
		\|\nabla c_0(t)\|^{\frac{1}{2}}_{L^2}
		\leq C\|n_{0}(t)\|_{L^2},$$	
		which is $\eqref{eq:elliptic}_4$.
	\end{proof}

	\begin{lemma}\label{lem:ellip_2}
		Let $c_{\neq}$ and $n_{\neq}$ be the non-zero mode of $c$ and $n$,
		respectively, satisfying
		$$-\triangle c_{\neq}+c_{\neq}=n_{\neq},$$
		then there hold
		\begin{align*}
			\|\partial_x^j\triangle c_{\neq}(t&)\|_{L^2}
			+\|\partial_x^j\nabla c_{\neq}(t)\|_{L^2}\leq C\|\partial_x^jn_{\neq}(t)\|_{L^2},\\
			\|\partial_z^j\triangle c_{\neq}(t&)\|_{L^2}
			+\|\partial_z^j\nabla c_{\neq}(t)\|_{L^2}\leq C\|\partial_z^jn_{\neq}(t)\|_{L^2},\\
			\|\partial_x\partial_z\triangle c_{\neq}(t&)\|_{L^2}
			+\|\partial_x\partial_z\nabla c_{\neq}(t)\|_{L^2}\leq C\|\partial_x\partial_zn_{\neq}(t)\|_{L^2},
		\end{align*}
		and
		\begin{equation*}
			\begin{aligned}
				\|\partial_x^j\nabla c_{\neq}(t)\|_{L^4}&\leq C\|\partial_x^jn_{\neq}(t)\|_{L^2},\\
				\|\partial_z^j\nabla c_{\neq}(t)\|_{L^4}&\leq C\|\partial_z^jn_{\neq}(t)\|_{L^2},
			\end{aligned}
		\end{equation*}
		where $j=0,1,2.$
	\end{lemma}
	\begin{proof}
		By integration by parts, note that
		\begin{equation}
			\begin{aligned}
				\|\triangle c_{\neq}(t)\|^2_{L^2}+\|\nabla c_{\neq}(t)\|^2_{L^2}
				+\|c_{\neq}(t)\|^2_{L^2}
				&\leq C\|n_{\neq}(t)\|^2_{L^2}.
				\nonumber
			\end{aligned}
		\end{equation}
		Using the Gagliardo-Nirenberg inequality, we obtain
		$$\|\nabla c_{\neq}(t)\|_{L^4}\leq
		C\| c_{\neq}(t)\|^{\frac{1}{8}}_{L^2}
		\|\triangle c_{\neq}(t)\|^{\frac{7}{8}}_{L^2}
		\leq C\|n_{\neq}(t)\|_{L^2}.$$
		Other results are similar and we omitted it.

	\end{proof}
	
	\begin{lemma}\label{lem:ellip_3}
		Let $c_{(0,0)}$ and $n_{(0,0)}$ be the z-part zero mode of $c_0$ and $n_0$,
		respectively, satisfying
		$$-\partial_{yy} c_{(0,0 )}+c_{(0,0 )}=n_{(0,0)},$$
		then there hold
		\ben\label{eq:c00}
		\|\partial_{yy} c_{(0,0 )}(t)\|^2_{L^2}
		+2\|\partial_{y} c_{(0,0 )}(t)\|^2_{L^2}
		+\|c_{(0,0 )}(t)\|^2_{L^2}=\|n_{(0,0 )}(t)\|^2_{L^2},
		\een
		and
		\ben\label{eq:c00infty}
		\|\partial_{y} c_{(0,0 )}(t)\|^2_{L^{\infty}}
		\leq \frac{1}{2}\|n_{(0,0 )}(t)\|^2_{L^2}.
		\een
	\end{lemma}
	\begin{proof}
		A direct calculation yields that 
		\begin{align*}
			\|\partial_{yy} c_{(0,0 )}(t)\|^2_{L^2}
			+2\|\partial_{y} c_{(0,0 )}(t)\|^2_{L^2}
			+\|c_{(0,0 )}(t)\|^2_{L^2}=\|n_{(0,0 )}(t)\|^2_{L^2},
		\end{align*}
		which implies \eqref{eq:c00}.
		
		Due to 
		$$\|\partial_{y} c_{(0,0 )}(t)\|^2_{L^2}\leq 
		\|\partial_{yy} c_{(0,0 )}(t)\|_{L^2}\| c_{(0,0 )}(t)\|_{L^2}
		\leq \frac{1}{4}\|n_{(0,0 )}(t)\|^2_{L^2},$$
		thus
		$$\|\partial_{y} c_{(0,0 )}(t)\|^2_{L^{\infty}}\leq \|\partial_{y} c_{(0,0 )}(t)\|_{L^{2}}\|\partial_{yy} c_{(0,0 )}(t)\|_{L^{2}}
		\leq \frac{1}{2}\|n_{(0,0 )}(t)\|^2_{L^2},$$
		which implies \eqref{eq:c00infty}.
		
	\end{proof}	
	\begin{lemma}\label{lem:ellip_30}
		Let $c_{(0,0)}$ and $n_{(0,0)}$ be the z-part zero mode of $c_0$ and $n_0$,
		respectively, satisfying
		$$-\partial_{yy} c_{(0,0 )}+c_{(0,0 )}=n_{(0,0)},$$
		then there holds
		\begin{equation*}
			\begin{aligned}
				\|c_{(0,0)}(t)\|^2_{L^{\infty}}\leq \frac{\|n_{(0,0)}(t)\|^2_{L^1}}{4}.    		
			\end{aligned}
		\end{equation*}
	\end{lemma}
	\begin{proof}
		Multiplying $c_{(0,0)}$ on both sides of $-\partial_{yy} c_{(0,0 )}+c_{(0,0 )}=n_{(0,0)},$ the energy estimate shows that 
		\begin{equation*}
			\|\partial_{y} c_{(0,0 )}(t)\|^2_{L^2}
			+\|c_{(0,0 )}(t)\|_{L^2}^2=<n_{(0,0)}(t),c_{(0,0)}(t)>
			\leq \|c_{(0,0)}(t)\|_{L^{\infty}}\|n_{(0,0)}(t)\|_{L^1}.
		\end{equation*}
		By 1-D Gagliardo-Nirenberg inequality \eqref{eq:1DGN-1}, we
		get  
		\begin{equation}\label{c00_0}
			\|c_{(0,0)}(t)\|_{L^{\infty}}\leq 
			\|\partial_yc_{(0,0)}(t)\|_{L^{2}}^{\frac12}
			\|c_{(0,0)}(t)\|_{L^{2}}^{\frac12}.
		\end{equation}
		Therefore, there holds
		\begin{equation}\label{c00_1}
			\|\partial_{y} c_{(0,0 )}(t)\|^2_{L^2}
			+\|c_{(0,0 )}(t)\|_{L^2}^2
			\leq \|\partial_yc_{(0,0)}(t)\|_{L^{2}}^{\frac12}
			\|c_{(0,0)}\|_{L^{2}}^{\frac12}\|n_{(0,0)}(t)\|_{L^1}.
		\end{equation}
		
		Using Young's inequality, we have
		\begin{equation*}
			\begin{aligned}
				\|\partial_yc_{(0,0)}(t)\|_{L^{2}}^{\frac12}
				\|c_{(0,0)}(t)\|_{L^{2}}^{\frac12}\|n_{(0,0)}(t)\|_{L^1}
				&\leq \|\partial_yc_{(0,0)}(t)\|_{L^{2}}
				\|c_{(0,0)}(t)\|_{L^{2}}+
				\frac{\|n_{(0,0)}(t)\|^2_{L^1}}{4}\\
				&\leq \frac{\|\partial_yc_{(0,0)}(t)\|_{L^{2}}^2}{2}
				+\frac{\|c_{(0,0)}(t)\|_{L^{2}}^2}{2}+
				\frac{\|n_{(0,0)}(t)\|^2_{L^1}}{4},
			\end{aligned}
		\end{equation*}
		which along with (\ref{c00_0}) and (\ref{c00_1}) show that 
		\begin{equation*}
			\|c_{(0,0)}(t)\|^2_{L^{\infty}}\leq 
			\frac{	\|\partial_yc_{(0,0)}(t)\|_{L^{2}}^2}{2}+
			\frac{	\|c_{(0,0)}(t)\|_{L^{2}}^2}{2}\leq \frac{\|n_{(0,0)}(t)\|^2_{L^1}}{4}.
		\end{equation*}
		
		%
	\end{proof}
	
	\begin{lemma}\label{lem:ellip_4}
		Let $c_{(0,\neq)}$ and $n_{(0,\neq)}$ be the z-part non-zero mode of $c_0$ and $n_0$,
		respectively, satisfying
		$$-\triangle c_{(0,\neq )}+c_{(0,\neq )}=n_{(0,\neq )},$$
		then there hold
		\begin{align*}
			&\|\triangle c_{(0,\neq )}(t)\|^2_{L^2}
			+2\|\nabla c_{(0,\neq )}(t)\|^2_{L^2}
			+\|c_{(0,\neq )}(t)\|^2_{L^2}=\|n_{(0,\neq )}(t)\|^2_{L^2},\\
			&\|\nabla c_{(0,\neq )}(t)\|_{L^4}\leq C\|n_{(0,\neq )}(t)\|_{L^2},
		\end{align*}
		and
		\begin{equation*}
			\begin{aligned}
				\|\partial_z^j\triangle c_{(0,\neq )}(t)\|^2_{L^2}
				+2\|\partial_z^j\nabla c_{(0,\neq )}(t)\|^2_{L^2}
				+\|\partial_z^j c_{(0,\neq )}(t)\|^2_{L^2}=\|\partial_z^jn_{(0,\neq )}(t)\|^2_{L^2},
			\end{aligned}
		\end{equation*}
		where $j=1,2.$
	\end{lemma}
	
	\begin{proof}
		Using energy estimates and Gagliardo-Nirenberg inequality
		$$\|\nabla c_{(0,\neq )}(t)\|_{L^4}\leq C\|\nabla c_{(0,\neq )}(t)\|_{L^2}^{\frac14}\|\triangle c_{(0,\neq )}(t)\|_{L^2}^{\frac34} 
		\leq C\|n_{(0,\neq )}(t)\|_{L^2},$$
		we  get the first two inequalities.
		
		Direct calculations yield that 
		\begin{align*}
			\|\partial_z^j\triangle c_{(0,\neq )}(t)\|^2_{L^2}
			+2\|\partial_z^j\nabla c_{(0,\neq )}(t)\|^2_{L^2}
			+\|\partial_z^j c_{(0,\neq )}(t)\|^2_{L^2}
			=\|\partial_z^j n_{(0,\neq )}(t)\|^2_{L^2},
		\end{align*}
		which is the last result.
	\end{proof}
	
	\subsection{Velocity estimates}\label{sec_velocty}
	
	\begin{lemma}[]\label{lemma_u}
		Assume that $u_{\neq}\in H^{2}(\mathbb{T}\times\mathbb{R}\times\mathbb{T}),$  there hold
		\begin{equation}\label{eq:velocity embed}
			\begin{aligned}
				&\left\|(\partial_x,
				\partial_z)u_{\neq}\right\|_{L^2}\leq C(\|\omega_{2,\neq}\|_{L^2}
				+\|\nabla u_{2,\neq}\|_{L^2}),\\
				&\left\|(\partial_x,
				\partial_z)\partial_xu_{\neq}\right\|_{L^2}\leq C(\|\partial_x\omega_{2,\neq}\|_{L^2}
				+\|\triangle u_{2,\neq}\|_{L^2}),\\
				&\left\|(\partial_x,
				\partial_z)\partial_yu_{\neq}\right\|_{L^2}\leq C(\|\partial_y\omega_{2,\neq}\|_{L^2}
				+\|\triangle u_{2,\neq}\|_{L^2}),\\
				&\left\|(\partial_x^2,
				\partial_z^2)u_{3,\neq}\right\|_{L^2}
				\leq C(\|\partial_x\omega_{2,\neq}\|_{L^2}
				+\|\triangle u_{2,\neq}\|_{L^2}),\\
				&\left\|(\partial_x,
				\partial_z)\partial_x\nabla u_{\neq}\right\|_{L^2}\leq C(\|\partial_x\nabla\omega_{2,\neq}\|_{L^2}
				+\|\nabla\triangle u_{2,\neq}\|_{L^2}),\\
				&\left\|(\partial_x,
				\partial_z)\partial_y\nabla u_{\neq}\right\|_{L^2}\leq C(\|\partial_y\nabla\omega_{2,\neq}\|_{L^2}
				+\|\nabla\triangle u_{2,\neq}\|_{L^2}).\end{aligned}
		\end{equation}
	\end{lemma}
	
	\begin{proof}
		Recall that
		\begin{equation}\label{velocity1}
			\begin{aligned}
				\omega_{2,\neq}=\partial_zu_{1,\neq}-\partial_xu_{3,\neq},\\
				-\partial_y u_{2,\neq}=\partial_x u_{1,\neq}+\partial_z u_{3,\neq}.
			\end{aligned}
		\end{equation}
		Using the Fourier series, there hold
		\begin{equation*}
			\begin{aligned}
				\omega_{2,\neq}=-\sum_{k_1\neq0,k_3\in\mathbb{Z}}
				\left(ik_1\hat{u}_{3,k_1,k_3}(y)-ik_3\hat{u}_{1,k_1,k_3}(y)\right)
				{\rm e}^{i(k_1x+k_3z)},
			\end{aligned}
		\end{equation*}
		and 
		\begin{equation*}
			\begin{aligned}
				\partial_yu_{2,\neq}=-\sum_{k_1\neq0,k_3\in\mathbb{Z}}
				\left(ik_3\hat{u}_{3,k_1,k_3}(y)+ik_1\hat{u}_{1,k_1,k_3}(y)\right)
				{\rm e}^{i(k_1x+k_3z)}.
			\end{aligned}
		\end{equation*}
		Then
		\begin{equation*}
			\begin{aligned}
				\|\omega_{2,\neq}\|^2_{L^2}+\|\partial_yu_{2,\neq}\|^2_{L^2}&=|\mathbb{T}|^2
				\sum_{k_1\neq0,k_3\in\mathbb{Z}}\left(k_1^2+k_3^2\right)
				\Big(\|\hat{u}_{1,k_1,k_3}(y)\|_{L^2_y}^2+\|\hat{u}_{3,k_1,k_3}(y)\|^2_{L^2_y}\Big)\\
				&=\|\partial_xu_{1,\neq}\|^2_{L^2}+\|\partial_zu_{1,\neq}\|^2_{L^2}
				+\|\partial_xu_{3,\neq}\|^2_{L^2}+\|\partial_zu_{3,\neq}\|^2_{L^2},
			\end{aligned}
		\end{equation*}
		which implies $\eqref{eq:velocity embed}_1$. Moreover, $\eqref{eq:velocity embed}_4$ can be proved by using $$\partial_x\omega_{2,\neq}+\partial_z\partial_yu_{2,\neq}=-(\partial_x^2+\partial_z^2)u_{3,\neq}.$$
		And other inequalities are similar.
		
		The proof is complete.
	\end{proof}

	\begin{lemma}\label{lemma_neq1}
		It holds that 
		\begin{equation}\label{eq:nonlinear-neq0}
			\begin{aligned}
				\|{\rm e}^{2aA^{-\frac{1}{3}}t}|u_{\neq}|^2\|_{L^2L^2}^2&\leq
				CA^{\frac{2}{3}}(\|\partial_x\omega_{2,\neq}\|^4_{X_a}+\|\triangle u_{2,\neq}\|^4_{X_a}),\\
				\|{\rm e}^{2aA^{-\frac{1}{3}}t}u_{\neq}\cdot\nabla u_{\neq}\|_{L^2L^2}^2&\leq
				CA(\|\partial_x\omega_{2,\neq}\|^4_{X_a}+\|\triangle u_{2,\neq}\|^4_{X_a}),\\
				\|{\rm e}^{2aA^{-\frac{1}{3}}t}\partial_x(u_{\neq}\cdot\nabla u_{\neq})\|^2_{L^2L^2}
				&\leq CA(\|\partial_x\omega_{2,\neq}\|^4_{X_a}
				+\|\triangle u_{2,\neq}\|^4_{X_a}),\\		
				\|{\rm e}^{2aA^{-\frac{1}{3}}t}\nabla(u_{\neq}\cdot\nabla u_{2,\neq})\|^2_{L^2L^2}
				&\leq CA\|\triangle u_{2,\neq}\|^2_{X_a}(\|\partial_x\omega_{2,\neq}\|^2_{X_a}
				+\|\triangle u_{2,\neq}\|^2_{X_a}),\\	
				\|{\rm e}^{2aA^{-\frac{1}{3}}t}\partial_z(u_{\neq}\cdot\nabla u_{3,\neq})\|_{L^2L^2}^2	
				&\leq CA(\|\partial_x\omega_{2,\neq}\|^4_{X_a}
				+\|\triangle u_{2,\neq}\|^4_{X_a}), \\		
				\|{\rm e}^{2aA^{-\frac{1}{3}}t}\nabla(u_{\neq}\cdot\nabla u_{1,\neq})\|^2_{L^2L^2}
				&\leq  CA(\|\partial_x\omega_{2,\neq}\|^2_{X_a}
				+\|\triangle u_{2,\neq}\|^2_{X_a})
				(\|\nabla\omega_{2,\neq}\|^2_{X_a}
				+\|\triangle u_{2,\neq}\|^2_{X_a}),\\
				\|{\rm e}^{2aA^{-\frac{1}{3}}t}\nabla(u_{\neq}\cdot\nabla u_{3,\neq})\|^2_{L^2L^2}
				&\leq  CA(\|\partial_x\omega_{2,\neq}\|^2_{X_a}
				+\|\triangle u_{2,\neq}\|^2_{X_a})
				(\|\nabla\omega_{2,\neq}\|^2_{X_a}
				+\|\triangle u_{2,\neq}\|^2_{X_a}).
			\end{aligned}
		\end{equation}
	\end{lemma}
	\begin{proof}
		{\bf Estimate of $\eqref{eq:nonlinear-neq0}_1$.} Recall the relation of \eqref{eq:fourier ine}, and 	
		by using $\eqref{sob_result_2}_4$ in Lemma \ref{sob_inf_2} and $\eqref{eq:velocity embed}_2$ in Lemma \ref{lemma_u} we have
		\begin{align*}
			\|u_{\neq}\|_{L^{\infty}_{x,z}L^2_y}^2
			&\leq C\big(\|\partial_xu_{\neq}\|^{2\alpha}_{L^2}
			\|u_{\neq}\|^{2-2\alpha}_{L^2}
			+\|\partial_x\partial_zu_{\neq}\|^{2\alpha}_{L^2}
			\|u_{\neq}\|^{2-2\alpha}_{L^2}\big)\\
			&\leq C(\|\partial_x\omega_{2,\neq}\|_{L^2}+\|\triangle u_{2,\neq}\|_{L^2})^2.
		\end{align*}
		Therefore, by $\eqref{sob_result_2}_7$ and $\eqref{eq:velocity embed}_3$ we have  
		\begin{equation*}
			\begin{aligned}
				\||u_{\neq}|^2\|_{L^2}^2&\leq
				C\big(\|\partial_xu_{\neq}\|^{2\alpha}_{L^2}
				\|u_{\neq}\|^{2-2\alpha}_{L^2}
				+\|\partial_x\partial_zu_{\neq}\|^{2\alpha}_{L^2}
				\|u_{\neq}\|^{2-2\alpha}_{L^2}\big)\|\partial_yu_{\neq}\|_{L^2}\|u_{\neq}\|_{L^2}\\
				&\leq C(\|\partial_x\omega_{2,\neq}\|^3_{L^2}+\|\triangle u_{2,\neq}\|_{L^2}^3)(\|\nabla\omega_{2,\neq}\|_{L^2}+\|\triangle u_{2,\neq}\|_{L^2}),
			\end{aligned}
		\end{equation*}
		and 
		$$\|{\rm e}^{2aA^{-\frac{1}{3}}t}|u_{\neq}|^2\|_{L^2L^2}^2\leq CA^{\frac{2}{3}}(\|\partial_x\omega_{2,\neq}\|^4_{X_a}+\|\triangle u_{2,\neq}\|^4_{X_a}). $$	
		
		{\bf Estimates of $\eqref{eq:nonlinear-neq0}_2$ and $\eqref{eq:nonlinear-neq0}_3$.} Using Lemma \ref{sob_inf_2} and Lemma \ref{lemma_u}  again, there hold
		\begin{equation}\label{appa_1}
			\begin{aligned}
				&\|u_{2,\neq}\|_{L^{\infty}_{y,z}L^2_x}^2
				\leq C\|\nabla u_{2,\neq}\|_{L^2}\|\triangle u_{2,\neq}\|_{L^2},\\
				&\|u_{\neq}\|_{L^{\infty}_{x,z}L^2_y}^2
				\leq C\big(\|\partial_xu_{\neq}\|^{2}_{L^2}
				+\|(\partial_x,\partial_z)\partial_xu_{\neq}\|_{L^2}^{2}
				\big)\leq C(\|\partial_x\omega_{2,\neq}\|^2_{L^2}+\|\triangle 	u_{2,\neq}\|_{L^2}^2),\\
				&\|(\partial_x,\partial_z)u_{\neq}\|_{L^{\infty}_{y}L^2_{x,z}}^2+
				\|\partial_y u_{\neq}\|_{L^{\infty}_{x}L^2_{y,z}}^2\leq 	C(\|\nabla\omega_{2,\neq}\|^2_{L^2}+\|\triangle u_{2,\neq}\|^2_{L^2}),		
			\end{aligned}
		\end{equation}
		which implies
		\begin{equation*}
			\begin{aligned}
				\|u_{\neq}\cdot\nabla u_{\neq}\|^2_{L^2}
				&\leq \|u_{1,\neq}\partial_x u_{\neq}\|_{L^2}^2+
				\|u_{2,\neq}\partial_y u_{\neq}\|_{L^2}^2+
				\|u_{3,\neq}\partial_z u_{\neq}\|_{L^2}^2\\
				&\leq \|u_{\neq}\|_{L^{\infty}_{x,z}L^2_y}^2\|(\partial_x,\partial_z)u_{\neq}\|_{L^{\infty}_{y}L^2_{x,z}}^2+
				\|u_{2,\neq}\|_{L^{\infty}_{y,z}L^2_x}^2
				\|\partial_y u_{\neq}\|_{L^{\infty}_{x}L^2_{y,z}}^2\\
				&\leq C(\|\partial_x\omega_{2,\neq}\|^2_{L^2}+\|\triangle u_{2,\neq}\|_{L^2}^2)(\|\nabla\omega_{2,\neq}\|^2_{L^2}+\|\triangle u_{2,\neq}\|^2_{L^2}).
			\end{aligned}
		\end{equation*}
		Thence, $\|{\rm e}^{2aA^{-\frac{1}{3}}t}u_{\neq}\cdot\nabla u_{\neq}\|^2_{L^2L^2}\leq CA(\|\partial_x\omega_{2,\neq}\|^4_{X_a}+\|\triangle u_{2,\neq}\|^4_{X_a}),$ that is $\eqref{eq:nonlinear-neq0}_2$.
		Similar to $(\ref{appa_1})$,
		we also have
		\begin{equation}\label{appa_2}
			\begin{aligned}
				&\|\partial_xu_{\neq}\|_{L^{\infty}_{z}L^2_{x,y}}^2+
				\|\partial_xu_{2,\neq}\|_{L^{\infty}_{y}L^2_{x,z}}^2\leq 
				C(\|\partial_x\omega_{2,\neq}\|^2_{L^2}+\|\triangle u_{2,\neq}\|_{L^2}^2),\\
				&\|(\partial_x,\partial_z)u_{\neq}\|_{L^{\infty}_{x,y}L^2_{z}}^2+
				\|\partial_y u_{\neq}\|_{L^{\infty}_{x,z}L^2_{y}}^2\leq 
				C(\|\partial_x\nabla\omega_{2,\neq}\|^2_{L^2}+\|\nabla\triangle u_{2,\neq}\|^2_{L^2}),\\
				&\|u_{\neq}\|_{L^{\infty}}^2\leq C(\|\partial_x\omega_{2,\neq}\|_{L^2}+\|\triangle u_{2,\neq}\|_{L^2})(\|\nabla\omega_{2,\neq}\|_{L^2}+\|\triangle u_{2,\neq}\|_{L^2}),\\
				&\|u_{2,\neq}\|_{L^{\infty}}^2\leq C\|\triangle u_{2,\neq}\|_{L^2}^2,
			\end{aligned}
		\end{equation}
		then
		\begin{equation*}
			\begin{aligned}
				\|\partial_x(u_{\neq}\cdot\nabla u_{\neq})\|^2_{L^2}
				&\leq 
				\|\partial_xu_{1,\neq}\partial_x u_{\neq}\|_{L^2}^2+
				\|\partial_xu_{2,\neq}\partial_y u_{\neq}\|_{L^2}^2+
				\|\partial_xu_{3,\neq}\partial_z u_{\neq}\|_{L^2}^2\\
				&+\|u_{1,\neq}\partial_x^2 u_{\neq}\|_{L^2}^2+
				\|u_{2,\neq}\partial_x\partial_y u_{\neq}\|_{L^2}^2+
				\|u_{3,\neq}\partial_x\partial_z u_{\neq}\|_{L^2}^2\\
				&\leq \|\partial_xu_{\neq}\|_{L^{\infty}_{z}L^2_{x,y}}^2
				\|(\partial_x,\partial_z)u_{\neq}\|_{L^{\infty}_{x,y}L^2_{z}}^2+
				\|\partial_xu_{2,\neq}\|_{L^{\infty}_{y}L^2_{x,z}}^2\|\partial_y u_{\neq}\|_{L^{\infty}_{x,z}L^2_{y}}^2\\
				&+ \|u_{\neq}\|_{L^{\infty}}^2
				\|(\partial_x,\partial_z)\partial_x u_{\neq}\|_{L^2}^2+
				\|u_{2,\neq}\|_{L^{\infty}}^2
				\|\partial_x\partial_y u_{\neq}\|_{L^2}^2\\
				&\leq C(\|\partial_x\omega_{2,\neq}\|^2_{L^2}+\|\triangle u_{2,\neq}\|_{L^2}^2)(\|\partial_x\nabla\omega_{2,\neq}\|^2_{L^2}
				+\|\nabla\triangle u_{2,\neq}\|^2_{L^2}),
			\end{aligned}
		\end{equation*}
		which implies that 
		$$\|{\rm e}^{2aA^{-\frac{1}{3}}t}\partial_x(u_{\neq}\cdot\nabla u_{\neq})\|^2_{L^2L^2}
		\leq CA(\|\partial_x\omega_{2,\neq}\|^4_{X_a}
		+\|\triangle u_{2,\neq}\|^4_{X_a}).$$

		{\bf Estimate of $\eqref{eq:nonlinear-neq0}_4$.}
		Similar to $(\ref{appa_1})_2$, one can prove
		$$\|\nabla u_{\neq}\|_{L^{\infty}_{x,z}L^2_y}^2
		\leq C(\|\partial_x\nabla\omega_{2,\neq}\|^2_{L^2}+\|\nabla\triangle u_{2,\neq}\|_{L^2}^2),$$
		and combining it with $(\ref{appa_2})_3,$ there holds
		\begin{align*}
			\|\nabla(u_{\neq}\cdot\nabla u_{2,\neq})\|^2_{L^2}&\leq 
			\|\triangle u_{2,\neq}\|^2_{L^{2}}\|u_{\neq}\|^2_{L^{\infty}}
			+\|\nabla u_{2,\neq}\|_{L^{\infty}_{y}L^2_{x,z}}^2\|\nabla u_{\neq}\|_{L^{\infty}_{x,z}L^2_{y}}^2\\ 
			&\leq C\|\triangle u_{2,\neq}\|^2_{L^2}(\|\partial_x\nabla\omega_{2,\neq}\|^2_{L^2}
			+\|\nabla\triangle u_{2,\neq}\|^2_{L^2}),
		\end{align*}
		which implies that  $\eqref{eq:nonlinear-neq0}_4$.
		
		{\bf Estimate of $\eqref{eq:nonlinear-neq0}_5$.}
		Using Lemma \ref{sob_inf_2}  and Lemma \ref{lemma_u}  again,
		we get
		\begin{equation}\label{appa_3}
			\begin{aligned}
				&\|\partial_zu_{\neq}\|_{L^{\infty}_{x}L^2_{y,z}}^2+
				\|\partial_zu_{2,\neq}\|_{L^{\infty}_{y}L^2_{x,z}}^2\leq 
				C(\|\partial_x\omega_{2,\neq}\|^2_{L^2}+\|\triangle u_{2,\neq}\|_{L^2}^2).
			\end{aligned}
		\end{equation}
		Combining (\ref{appa_3}), $(\ref{appa_2})_3$ and $(\ref{appa_2})_4$, we have
		\begin{equation*}
			\begin{aligned}
				\|\partial_z(u_{\neq}\cdot\nabla u_{3,\neq})\|^2_{L^2}
				\leq& \|\partial_zu_{\neq}\|_{L^{\infty}_{x}L^2_{y,z}}^2
				\|(\partial_x,\partial_z)u_{3,\neq}\|_{L^{\infty}_{y,z}L^2_{x}}^2+
				\|\partial_zu_{2,\neq}\|_{L^{\infty}_{y}L^2_{x,z}}^2\|\partial_y u_{3,\neq}\|_{L^{\infty}_{x,z}L^2_{y}}^2\\
				&+ \|u_{\neq}\|_{L^{\infty}}^2
				\|(\partial_x,\partial_z)\partial_z u_{3,\neq}\|_{L^2}^2+
				\|u_{2,\neq}\|_{L^{\infty}}^2
				\|\partial_y\partial_z u_{3,\neq}\|_{L^2}^2\\
				\leq& C(\|\partial_x\omega_{2,\neq}\|^2_{L^2}+\|\triangle u_{2,\neq}\|_{L^2}^2)(\|\partial_x\nabla\omega_{2,\neq}\|^2_{L^2}
				+\|\nabla\triangle u_{2,\neq}\|^2_{L^2}),
			\end{aligned}
		\end{equation*}
		which implies that 
		$\|{\rm e}^{2aA^{-\frac{1}{3}}t}\partial_z(u_{\neq}\cdot\nabla u_{3,\neq})\|_{L^2L^2}^2	
		\leq CA(\|\partial_x\omega_{2,\neq}\|^4_{X_a}
		+\|\triangle u_{2,\neq}\|^4_{X_a}).$
		
		{\bf Estimates of $\eqref{eq:nonlinear-neq0}_6$ and $\eqref{eq:nonlinear-neq0}_7$.}
		For $j=1,3$, using Lemma \ref{sob_inf_2} and \ref{lemma_u} again, there are
		\begin{equation}\label{appa_4}
			\begin{aligned}
				&\|(\partial_x,\partial_z) u_{j,\neq}\|^2_{L^{\infty}_xL^2_{y,z}}
				\leq C\|(\partial_x,\partial_z)\partial_x u_{j,\neq}\|^2_{L^2}
				\leq C(\|\partial_x\omega_{2,\neq}\|^2_{L^2}
				+\|\triangle u_{2,\neq}\|^2_{L^2}),\\
				&\|\nabla u_{\neq}\|^2_{L^{\infty}_{y,z}L^2_x}
				+\|\nabla(\partial_x,\partial_z) u_{j,\neq}\|^2_{L^{\infty}_yL^2_{x,z}}
				\leq C
				\|\nabla (\partial_x,\partial_z) u_{\neq}\|_{L^2}
				\|\nabla \partial_y(\partial_x,\partial_z) u_{\neq}\|_{L^2}\\
				&\qquad \leq C(\|\partial_y\omega_{2,\neq}\|_{L^2}
				+\|\triangle u_{2,\neq}\|_{L^2})(\|\partial_y\nabla\omega_{2,\neq}\|_{L^2}
				+\|\nabla\triangle u_{2,\neq}\|_{L^2}),\\
				&\|\nabla u_{2,\neq}\|^2_{L^{\infty}_yL^2_{x,z}}\leq 
				C\|\nabla u_{2,\neq}\|_{L^2}\|\triangle u_{2,\neq}\|_{L^2},\\
				&\|\nabla\partial_y u_{j,\neq}\|^2_{L^2}\leq C
				(\|\partial_y\nabla\omega_{2,\neq}\|^2_{L^2}
				+\|\triangle u_{2,\neq}\|^2_{L^2}).
			\end{aligned}
		\end{equation}
		By $(\ref{appa_1})_2$, $(\ref{appa_2})_4$ and (\ref{appa_4}), we obtain that 
		\begin{equation*}
			\begin{aligned}
				\|\nabla(u_{\neq}\cdot\nabla u_{j,\neq})\|^2_{L^2}
				\leq& \|(\partial_x,\partial_z) u_{j,\neq}\|^2_{L^{\infty}_xL^2_{y,z}}
				\|\nabla u_{\neq}\|^2_{L^{\infty}_{y,z}L^2_x}
				+\|\nabla u_{2,\neq}\|^2_{L^{\infty}_yL^2_{x,z}}
				\|\partial_y u_{j,\neq}\|^2_{L^{\infty}_{x,z}L^2_{y}}\\
				&+\|u_{\neq}\|^2_{L^{\infty}_{x,z}L^2_y}
				\|\nabla(\partial_x,\partial_z) u_{j,\neq}\|^2_{L^{\infty}_yL^2_{x,z}}
				+\|u_{2,\neq}\|^2_{L^{\infty}}
				\|\nabla\partial_y u_{j,\neq}\|^2_{L^2}\\
				\leq& C(\|\partial_x\omega_{2,\neq}\|^2_{L^2}+\|\triangle u_{2,\neq}\|_{L^2}^2)(\|\triangle \omega_{2,\neq}\|^2_{L^2}
				+\|\nabla\triangle u_{2,\neq}\|^2_{L^2}),
			\end{aligned}
		\end{equation*}
		which implies $\eqref{eq:nonlinear-neq0}_6$ and $\eqref{eq:nonlinear-neq0}_7$.
		
		The proof is complete.	
	\end{proof}
	
	\begin{lemma}\label{lemma_neq2}
		For $j=2,3$, there hold
		\begin{equation}\label{lemma_neq2_2}
			\begin{aligned}&\|{\rm e}^{aA^{-\frac{1}{3}}t}
				u_{1,0}\partial_x\nabla u_{2,\neq}\|_{L^2L^2}^2\leq
				CE_{1,3}^2A^{1-2\epsilon}\|\triangle u_{2,\neq}\|_{X_a}\|\partial_x^2u_{2,\neq}\|_{X_{\frac32a}},\\
				&\|{\rm e}^{aA^{-\frac{1}{3}}t}
				u_{1,0}\partial_x(\partial_x,\partial_z) u_{3,\neq}\|_{L^2L^2}^2\leq
				CE_{1,3}^2A^{1-2\epsilon}(\|\partial_x\omega_{2,\neq}\|_{X_a}+\|\triangle u_{2,\neq}\|_{X_a})
				\|\partial_x^2u_{3,\neq}\|_{X_{\frac32a}},\\
				&\|{\rm e}^{aA^{-\frac{1}{3}}t}
				u_{1,0}\partial_x^{2} u_{1,\neq}\|_{L^2L^2}^2\leq
				CE_{1,3}^2A^{1-2\epsilon}(\|\partial_x\omega_{2,\neq}\|_{X_a}
				+\|\triangle u_{2,\neq}\|_{X_a})
				\|\partial_x^2(u_2,u_3)_{\neq}\|_{X_{\frac32a}},\\
				&\|e^{aA^{-\frac13}t}u_{1,0}\partial_{x}\partial_{z}u_{1,\neq}\|_{L^{2}L^{2}}^{2}\leq CE_{1,3}^{2}A^{\frac53-2\epsilon}\left(\|\partial_{x}\omega_{2,\neq}\|_{X_{a}}+\|\triangle u_{2,\neq}\|_{X_{a}} \right)\|\partial_{x}^{2}(u_{2},u_{3})_{\neq}\|_{X_{\frac32 a}},
				\\&\|{\rm e}^{aA^{-\frac{1}{3}}t}
				\nabla u_{1,0}\partial_x u_{j,\neq}\|_{L^2L^2}^2\leq
				CE_{1,3}^2A^{\frac43-2\epsilon}(\|\partial_x\omega_{2,\neq}\|_{X_a}+\|\triangle u_{2,\neq}\|_{X_a})
				\|\partial_x^2u_{j,\neq}\|_{X_{\frac32a}},
				\\&\|e^{aA^{-\frac13}t}\partial_{z}u_{1,0}\partial_{x}u_{1,\neq}\|_{L^{2}L^{2}}^{2}\leq CE_{1,3}^{2}A^{\frac53-2\epsilon}\left(\|\partial_{x}\omega_{2,\neq}\|_{X_{a}}+\|\triangle u_{2,\neq}\|_{X_{a}} \right)\|\partial_{x}^{2}(u_{2},u_{3})_{\neq}\|_{X_{\frac32 a}},
				\\
				&\|{\rm e}^{aA^{-\frac{1}{3}}t}(\partial_y,\partial_z)
				(u_{2,\neq}\nabla u_{1,0} )\|^2_{L^2L^2}\leq
				CE_{1,3}^2A^{1-2\epsilon}\|\triangle u_{2,\neq}\|^{\frac32}_{X_a}
				\|\partial_x^2u_{2,\neq}\|_{X_{\frac32a}}^{\frac12},\\
				&\|{\rm e}^{aA^{-\frac{1}{3}}t}\partial_z
				(u_{3,\neq}\nabla u_{1,0} )\|^2_{L^2L^2}\leq
				CE_{1,3}^2A^{\frac43-2\epsilon}(\|\partial_x\omega_{2,\neq}\|_{X_a}+\|\triangle u_{2,\neq}\|_{X_a})
				\|\partial_x^2u_{3,\neq}\|_{X_{\frac32a}}.
			\end{aligned}
		\end{equation}
	\end{lemma}
	
	\begin{proof}
		{\bf Estimate of $\eqref{lemma_neq2_2}_1$.}
		First of all, by (\ref{eq:E1}) we get
		\ben\label{eq:u10}
				\|{u_{1,0}}\|_{H^2}
				&\leq \|\widehat{u_{1,0}}\|_{H^2}+\|\widetilde{u_{1,0}}\|_{H^2}\leq \int_0^t \|\partial_s\widehat{u_{1,0}}(s)\|_{H^2}ds
				+\|\widetilde{u_{1,0}}\|_{H^2}\nonumber\\
				&\leq CE_{1,3}A^{\frac13-\epsilon}(1+A^{-\frac13}t).
			\een
		Then
		$\big\|{\rm e}^{-\frac{a}{4}A^{-\frac13}t}\|{u_{1,0}}\|_{H^2}\big\|_{L^{\infty}_t} \leq CA^{\frac13-\epsilon}E_{1,3}$, since $\lim_{t\rightarrow\infty}(1+A^{-\frac13}t){\rm e}^{-\frac{a}{4}A^{-\frac13}t}=0.$
		Combining this with $\|u_{1,0}\|_{L^{\infty}}\leq C\|u_{1,0}\|_{H^2}$, by 
		H\"{o}lder's inequality we get
		\begin{equation*}
			\begin{aligned}
				\|{\rm e}^{aA^{-\frac{1}{3}}t}
				u_{1,0}\partial_x\nabla	u_{2,\neq}\|_{L^2L^2}^2
				&\leq C\int_0^t{\rm 	e}^{-\frac{a}{2}A^{-\frac13}s}\|{u_{1,0}}\|^2_{H^2}
				{\rm e}^{\frac{5a}{2}A^{-\frac13}s}
				\|\partial_x^2 u_{2,\neq}\|_{L^2}
				\|\triangle u_{2,\neq}\|_{L^2}ds\\
				&\leq CA^{\frac23-2\epsilon}E_{1,3}^2
				\|{\rm e}^{\frac{3}{2}aA^{-\frac{1}{3}}t}\partial_x^2
				u_{2,\neq}\|_{L^2L^2}
				\|{\rm e}^{aA^{-\frac{1}{3}}t}\triangle 	
				u_{2,\neq}\|_{L^2L^2}\\
				&\leq CA^{1-2\epsilon}E_{1,3}^2
				\|\partial_x^2u_{2,\neq}\|_{X_{\frac32a}}
				\|\triangle u_{2,\neq}\|_{X_{a}}.
			\end{aligned}
		\end{equation*}
		
		{\bf Estimate of $\eqref{lemma_neq2_2}_2$.}
		Similarly, with the help of Lemma \ref{lemma_u}, we have
		\begin{equation*}
			\begin{aligned}
				&\quad\|{\rm e}^{aA^{-\frac{1}{3}}t}
				u_{1,0}\partial_x(\partial_x,\partial_z)u_{3,\neq}\|_{L^2L^2}^2\\
				&\leq CA^{\frac23-2\epsilon}E_{1,3}^2
				\|{\rm e}^{\frac{3}{2}aA^{-\frac{1}{3}}t}\partial_x^2
				u_{3,\neq}\|_{L^2L^2}
				\|{\rm e}^{aA^{-\frac{1}{3}}t}(\partial_x^2,\partial_z^2)
				u_{3,\neq}\|_{L^2L^2}\\
				&\leq CE^2_{1,3}A^{1-2\epsilon}
				\|\partial_x^2u_{3,\neq}\|_{X_{\frac32a}}
				(\|\partial_x \omega_{2,\neq}\|_{X_{a}}+\|\triangle u_{2,\neq}\|_{X_{a}}).
			\end{aligned}
		\end{equation*}
		
		{\bf Estimate of $\eqref{lemma_neq2_2}_3$.}
		It can be proved by using the divergence-free property 
		\begin{equation}\label{a}
			\partial_{x}u_{1,\neq}=-\left(\partial_{y}u_{2,\neq}+\partial_{z}u_{3,\neq} \right),
		\end{equation}
		and 
		\begin{equation*}
			\begin{aligned}
				\|\partial_x^2u_{1,\neq}\|_{L^2}^2
				&\leq C(\|\partial_x\partial_yu_{2,\neq}\|_{L^2}^2
				+\|\partial_x\partial_zu_{3,\neq}\|_{L^2}^2)\\
				&\leq C(\|\partial_x\omega_{2,\neq}\|_{L^2}+\|\triangle u_{2,\neq}\|_{L^2})\|\partial_x^2(u_2,u_3)_{\neq}\|_{L^2}.
			\end{aligned}
		\end{equation*}	
		
		{\bf Estimate of $\eqref{lemma_neq2_2}_4$.} Using (\ref{a}) and Lemma \ref{lemma_u}, one deduces
		\begin{equation*}
			\begin{aligned}
				&\|u_{1,0}\partial_{x}\partial_{z}u_{1,\neq}\|_{L^{2}}^{2}\\\leq&C\|u_{1,0}\|_{H^{2}}^{2}\|\partial_{x}^{2}u_{1,\neq}\|_{L^{2}}\|\partial_{z}^{2}u_{1,\neq}\|_{L^{2}}\\\leq&C\|u_{1,0}\|_{H^{2}}^{2}\left(\|\partial_{x}\partial_{y}u_{2,\neq}\|_{L^{2}}+\|\partial_{x}\partial_{z}u_{3,\neq}\|_{L^{2}} \right)\|\partial_{z}^{2}u_{1,\neq}\|_{L^{2}}\\\leq&C\|u_{1,0}\|_{H^{2}}^{2}\|\partial_{x}\nabla(u_{2}, u_{3})_{\neq}\|_{L^{2}}\left(\|\partial_{x}\nabla\omega_{2,\neq}\|_{L^{2}}+\|\nabla\triangle u_{2,\neq}\|_{L^{2}} \right),
			\end{aligned}
		\end{equation*}
		which gives $(\ref{lemma_neq2_2})_4$.
		
		{\bf Estimate of $\eqref{lemma_neq2_2}_5$.}
		Using Lemma \ref{sob_inf_1}, $(\ref{sob_result_2})_6$ and Lemma \ref{lemma_u}, we get
		\begin{equation*}
			\begin{aligned}
				&\quad\|\nabla u_{1,0}\partial_x u_{j,\neq}\|_{L^2}^2
				\leq \|\nabla u_{1,0}\|_{L^{\infty}_yL^2_z}^2
				\|\partial_x u_{j,\neq}\|^2_{L^{\infty}_zL^2_{x,y}}\\
				&\leq C\| u_{1,0}\|_{H^2}^2
				\|\partial_x(\partial_x,\partial_z) u_{j,\neq}\|_{L^2}^{2}
				\leq C\| u_{1,0}\|_{H^2}^2
				\|\partial_x^2 u_{j,\neq}\|_{L^2}
				\|(\partial_x^2,\partial_z^2) u_{j,\neq}\|_{L^2},
			\end{aligned}
		\end{equation*}
		which implies $\eqref{lemma_neq2_2}_5$.
		
		{\bf Estimate of $\eqref{lemma_neq2_2}_6$.}
		Similar to $\eqref{lemma_neq2_2}_5$, and using (\ref{a}), we obtain
		\begin{equation*}
			\begin{aligned}
				\|\partial_{z}u_{1,0}\partial_{x}u_{1,\neq}\|_{L^{2}}^{2}\leq &C\|u_{1,0}\|_{H^{2}}^{2}\|\partial_{x}^{2}u_{1,\neq}\|_{L^{2}}\|(\partial_{x}^{2},\partial_{z}^{2})u_{1,\neq}\|_{L^{2}}\\\leq&C\|u_{1,0}\|_{H^{2}}^{2}\|\partial_{x}\nabla(u_{2},u_{3})_{\neq}\|_{L^{2}}\left(\|\partial_{x}\nabla\omega_{2,\neq}\|_{L^{2}}+\|\nabla\triangle u_{2,\neq}\|_{L^{2}} \right),
			\end{aligned}
		\end{equation*}
		which gives $\eqref{lemma_neq2_2}_6$.

		{\bf Estimate of $\eqref{lemma_neq2_2}_7$.}
		Using $(\ref{appa_1})_1$, $(\ref{appa_4})_3$ and H\"{o}lder's inequality, there holds 
		\begin{equation*}
			\begin{aligned}
				\|(\partial_y,\partial_z)
				(u_{2,\neq}\nabla u_{1,0} )\|^2_{L^2}
				&\leq C\|u_{1,0}\|^2_{H^2}
				\|\nabla u_{2,\neq}\|_{L^2}
				\|\triangle	u_{2,\neq}\|_{L^2}\\
				&\leq C\|u_{1,0}\|^2_{H^2}
				\| u_{2,\neq}\|_{L^2}^{\frac12}
				\|\triangle	u_{2,\neq}\|_{L^2}^{\frac32},
			\end{aligned}
		\end{equation*}
		which along with $\eqref{lemma_neq2_2}_1$ indicate that 
		\begin{equation*}
			\begin{aligned}
				&\quad\|{\rm e}^{aA^{-\frac{1}{3}}t}(\partial_y,\partial_z)
				(u_{2,\neq}\nabla u_{1,0} )\|^2_{L^2L^2}\\
				&\leq C\big\|{\rm 	e}^{-\frac{a}{8}A^{-\frac13}t}\|{u_{1,0}}\|_{H^2}\big\|_{L^{\infty}_t}^2
				\|{\rm e}^{\frac{3}{2}aA^{-\frac{1}{3}}t}
				u_{2,\neq}\|_{L^2L^2}^{\frac12}
				\|{\rm e}^{aA^{-\frac{1}{3}}t}\triangle u_{2,\neq}\|_{L^2L^2}^{\frac32}\\
				&\leq CE_{1,3}^2A^{1-2\epsilon}\|\triangle u_{2,\neq}\|_{X_a}^{\frac32}
				\|\partial_x^2u_{2,\neq}\|_{X_{\frac32a}}^{\frac12}.
			\end{aligned}
		\end{equation*}
		
		{\bf Estimate of $\eqref{lemma_neq2_2}_8$.}
		Using Lemma \ref{sob_inf_2}, there holds
		\begin{equation*}
			\begin{aligned}
				\|\partial_zu_{3,\neq}\|^2_{L^{\infty}_yL^2_{x,z}}
				&\leq C\|\partial_z\partial_yu_{3,\neq}\|_{L^2}
				\|\partial_zu_{3,\neq}\|_{L^2}\\
				&\leq C
				\|u_{3,\neq}\|^{\frac12}_{L^2}
				\|\partial_z^2u_{3,\neq}\|^{^\frac12}_{L^2}
				\|\partial_yu_{3,\neq}\|^{\frac12}_{L^2}
				\|\partial_y\partial_z^2u_{3,\neq}\|
				^{^\frac12}_{L^2},
			\end{aligned}
		\end{equation*}
		and 
		\begin{equation*}
			\begin{aligned}
				\|u_{3,\neq}\|^2_{L^{\infty}_{y,z}L^2_x}
				&\leq C\|(\partial_x,\partial_z)\partial_yu_{3,\neq}\|_{L^2}
				\|(\partial_x,\partial_z)u_{3,\neq}\|_{L^2}\\
				&\leq C
				\|u_{3,\neq}\|^{\frac12}_{L^2}
				\|(\partial_x^2,\partial_z^2)u_{3,\neq}\|^{^\frac12}_{L^2}
				\|\partial_yu_{3,\neq}\|^{\frac12}_{L^2}
				\|\partial_y(\partial_x^2,\partial_z^2)u_{3,\neq}\|
				^{^\frac12}_{L^2}.
			\end{aligned}
		\end{equation*}
		Using  Lemma \ref{lemma_u} again, we get
		\begin{equation*}
			\begin{aligned}
				\quad\|{\rm e}^{aA^{-\frac{1}{3}}t}\partial_z
				(u_{3,\neq}\nabla u_{1,0} )\|^2_{L^2L^2}
				\leq CE_{1,3}^2A^{\frac43-2\epsilon}
				(\|\triangle u_{2,\neq}\|_{X_a}
				+\|\partial_x \omega_{2,\neq}\|_{X_a})
				\|\partial_x^2u_{3,\neq}\|_{X_{\frac32a}}.
			\end{aligned}
		\end{equation*}
		
		The proof is complete.
	\end{proof}
	
	\begin{lemma}\label{lemma_neq3}
		For $j=2,3$, it holds 
		\begin{equation}\label{eq:zeroneq0}
			\begin{aligned}
				\|{\rm e}^{aA^{-\frac{1}{3}}t}u_{j,0}(\partial_x,\partial_z)
				\nabla u_{\neq}\|_{L^2L^2}^2\leq 
				CA(\|u_{2,0}\|^2_{L^{\infty}H^2}+\|u_{3,0}\|^2_{L^{\infty}H^1})(\|\partial_x\omega_{2,\neq}\|^2_{X_a}
				+\|\triangle u_{2,\neq}\|^2_{X_a}),\\
				\|{\rm e}^{aA^{-\frac{1}{3}}t}\partial_z\nabla u_{j,0}
				\cdot u_{\neq}\|_{L^2L^2}^2\leq CA^{\frac23}(\|u_{2,0}\|^2_{L^{\infty}H^2}+\|u_{3,0}\|^2_{L^{\infty}H^1})(\|\partial_x\omega_{2,\neq}\|^2_{X_a}
				+\|\triangle u_{2,\neq}\|^2_{X_a}),\\
				\|{\rm e}^{aA^{-\frac{1}{3}}t}\partial_z u_{j,0}
				\nabla u_{\neq}\|_{L^2L^2}^2\leq CA(\|u_{2,0}\|^2_{L^{\infty}H^2}+\|u_{3,0}\|^2_{L^{\infty}H^1})(\|\partial_x\omega_{2,\neq}\|^2_{X_a}
				+\|\triangle u_{2,\neq}\|^2_{X_a}),\\
				\|e^{aA^{-\frac13}t}\partial_{y}u_{3,0}(\partial_{x},\partial_{z})u_{2,\neq}\|_{L^{2}L^{2}}^{2}\leq CA^{\frac13}\left(\|u_{2,0}\|_{L^{\infty}H^{2}}^{2}+\|u_{3,0}\|_{L^{\infty}H^{1}}^{2} \right)\|\triangle u_{2,\neq}\|_{X_{a}}^{2}.\quad\quad\quad\quad
			\end{aligned}
		\end{equation}
	\end{lemma}
	\begin{proof}
		By Lemma \ref{sob_inf_1} and $\partial_yu_{2,0}+\partial_zu_{3,0}=0$, we have 
		\begin{equation}\label{u23_infty}
			\begin{aligned}
				\|u_{2,0}\|^2_{L^{\infty}}&\leq C\|u_{2,0}\|^2_{H^2},\quad
				\|u_{3,0}\|^2_{L^{\infty}}\leq C(\|u_{2,0}\|^2_{H^2}+\|u_{3,0}\|^2_{H^1}),
			\end{aligned}
		\end{equation}
		which along with Lemma \ref{lemma_u} imply that
		\begin{equation*}
			\|{\rm e}^{aA^{-\frac{1}{3}}t}u_{j,0}(\partial_x,\partial_z)
			\nabla u_{\neq}\|_{L^2L^2}^2
			\leq CA(\|u_{2,0}\|^2_{L^{\infty}H^2}+\|u_{3,0}\|^2_{L^{\infty}H^1})(\|\partial_x\omega_{2,\neq}\|^2_{X_a}
			+\|\triangle u_{2,\neq}\|^2_{X_a}).
		\end{equation*}
		Moreover, $\eqref{eq:zeroneq0}_2$ and $\eqref{eq:zeroneq0}_3$ can be proved by using Lemma \ref{lemma_u} and 
		\begin{equation*}
			\begin{aligned}
				&\| u_{\neq}\|^2_{L^{\infty}_{y,z}L^2_x}\leq 
				C\| (\partial_x,\partial_z)\partial_yu_{\neq}\|_{L^2}
				\| (\partial_x,\partial_z)u_{\neq}\|_{L^2},\\
				&\| \nabla u_{\neq}\|^2_{L^{\infty}_{z}L^2_{x,y}}\leq 
				C\| (\partial_x,\partial_z)\nabla u_{\neq}\|^2_{L^2}.
			\end{aligned}
		\end{equation*}
		
		Using Lemma \ref{sob_inf_1}, $ (\ref{sob_result_2})_{7} $ in Lemma \ref{sob_inf_2} and $ \partial_{z}u_{3,0}=-\partial_{y}u_{2,0} $, we get
		\begin{equation*}
			\begin{aligned}
				\|\partial_{y}u_{3,0}(\partial_{x},\partial_{z})u_{2,\neq}\|_{L^{2}}^{2}\leq&\|\partial_{y}u_{3,0}\|_{L^{2}_{y}L^{\infty}_{z}}^{2}\|(\partial_{x},\partial_{z})u_{2,\neq}\|_{L^{\infty}_{y}L^{2}_{x,z}}^{2}\\\leq&C\left(\|\partial_{y}u_{3,0}\|_{L^{2}}^{2}+\|\partial_{y}\partial_{z}u_{3,0}\|_{L^{2}}^{2} \right)\|(\partial_{x},\partial_{z})\partial_{y}u_{2,\neq}\|_{L^{2}}^{2}\\\leq&C\left(\|u_{2,0}\|_{H^{2}}^{2}+\|u_{3,0}\|_{H^{1}}^{2} \right)\|\triangle u_{2,\neq}\|_{L^{2}}^{2},
			\end{aligned}
		\end{equation*}
		which follows that
		\begin{equation*}
			\begin{aligned}
				\|e^{aA^{-\frac13}t}\partial_{y}u_{3,0}(\partial_{x},\partial_{z})u_{2,\neq}\|_{L^{2}L^{2}}^{2}\leq CA^{\frac13}\left(\|u_{2,0}\|_{L^{\infty}H^{2}}^{2}+\|u_{3,0}\|_{L^{\infty}H^{1}}^{2} \right)\|\triangle u_{2,\neq}\|_{X_{a}}^{2}.
			\end{aligned}
		\end{equation*}
		
		The proof is complete.
	\end{proof}
	\section{Estimates for the zero mode: Proof of Proposition \ref{pro0}}
	In this section, we are aimed to estimate the zero modes for both $n$ and 
	$u$, where the energy transfer mechanism plays an important role.
	
	\noindent\textbf{Energy transfer mechanism:}
	
	To estimate zero modes, we first need to observe and analyze the energy transfer mechanism.
	\begin{itemize}
		\item Under assumptions (\ref{assumption}), the energy $\|n_{(0,0)}\|_{L^{\infty}L^2}$ will not be affected by any other energy. And we use the equilibrium point analysis method to obtain the precise bound of this energy.
		\item  	The z-part non-zero mode $n_{(0,\neq)}$ is affected by $n_{(0,0)}$, as long as $ n_{(0,0)} $ satisfies certain conditions, we can estimate it under smallness $A^{\epsilon}\|(n_{\rm in})_{(0,\neq)}\|_{L^2(\mathbb{T}\times\mathbb{R}\times\mathbb{T})}\leq C$.
		\item The energy $E_{1,2}$ is affected by $E_{1,1}$.
		After estimating $E_{1,1}$, it is easy to 
		estimates  $E_{1,2}$ under the smallness $A^{\epsilon}\|(u_{\rm in})_{0}\|_{H^2(\mathbb{T}\times\mathbb{R}\times\mathbb{T})}\leq C.$
		\item Due to 3D lift-up effect, $E_{1,3}$ is affected by $E_{1,2}$,
		after obtaining all estimates of $u_{2,0}$ and $u_{3,0},$ the energy $E_{1,3}$ can be obtained directly.
	\end{itemize} 
	In this way, we can estimate all zero modes in terms of
	$$n_{(0,0)}\rightarrow n_{(0,\neq)}\rightarrow \{u_{2,0}, u_{3,0}\}\rightarrow u_{1,0}.$$	
	\subsection{Energy estimates for $E_{1,1}(t)$}		
	It follows from (\ref{ini11}) that the zero mode $ n_{0} $ satisfies
	\begin{equation}\label{eq:n-zero mode}
		\begin{aligned}
			\partial_t n_{0}-\frac{1}{A}\triangle n_{0}=&-\frac{1}{A}\left[\nabla\cdot(n_{\neq}\nabla c_{\neq})_{0}+\partial_{y}(n_{0}\partial_{y}c_{0})+\partial_{z}(n_{0}\partial_{z}c_{0}) \right]\\&-\frac{1}{A}\left[\nabla\cdot(u_{\neq}n_{\neq})_{0}+\partial_{y}(u_{2,0}n_{0})+\partial_{z}(u_{3,0}n_{0}) \right].
		\end{aligned}
	\end{equation}
	Due to $u_{2,(0,0)}=0$, the z-part zero mode $ n_{(0,0)} $ follows:
	\begin{equation}\label{n001}
		\begin{aligned}
			\partial_t n_{(0,0)}-\frac{1}{A}\partial_{yy} n_{(0,0)}=&-\frac{1}{A}\Big(\partial_{y}\left(n_{(0,0)}\partial_{y}c_{(0,0)} \right)+\partial_{y}(n_{\neq}\partial_y c_{\neq})_{(0,0)}+\partial_{y}(u_{2,\neq}n_{\neq})_{(0,0)} \\&+\partial_{y}\big(n_{(0,\neq)}\partial_{y}c_{(0,\neq)} \big)_{(0,0)}+\partial_{y}\big(u_{2,(0,\neq)}n_{(0,\neq)} \big)_{(0,0)} \Big).
		\end{aligned}
	\end{equation}
	To  estimate $n_{(0,0)}$, we first consider a simplified model of (\ref{n001}) in the form of
	\begin{equation}\label{n002}
		\begin{aligned}
			\partial_t n_{(0,0)}-\frac{1}{A}\partial_{yy} n_{(0,0)}=-\frac{1}{A}\partial_{y}\left(n_{(0,0)}
			\partial_{y}c_{(0,0)}\right).
		\end{aligned}
	\end{equation}
	\begin{lemma}\label{c00_lemma1}
		For the simplified equation (\ref{n002}), there holds
		\begin{align*}
			\|n_{(0,0)}\|^2_{L^{\infty}L^2}\leq \max\left\{\|(n_{\rm in})_{(0,0)}\|^2_{L^2},\frac{27M_1^4}{32\pi^2}\right\}.
		\end{align*}
		It should be noted that $M_1=\|n_{(0,0)}\|_{L^1}=
		\frac{1}{|\mathbb{T}|^2}\int_{\mathbb{T}\times\mathbb{R}\times\mathbb{T}}n(t,x,y,z)dxdydz=\frac{M}{|\mathbb{T}|^2}.$
	\end{lemma} 
	
	\begin{proof}
		Our proof mainly relies on equilibrium point analysis method in dynamical systems and a priori estimates in Section \ref{sec_estimate_1}.
		
		By energy estimates, we have 
		\begin{equation*}
			\begin{aligned}
				\frac{1}{2}\partial_t\|n_{(0,0)}\|^2_{L^2}
				+\frac{1}{A}\|\partial_{y} n_{(0,0)}\|^2_{L^2}
				\leq \frac{1}{A}
				\|n_{(0,0)}\partial_{y}c_{(0,0)}\|_{L^2}\|\partial_{y}n_{(0,0)}\|_{L^2}.
			\end{aligned}
		\end{equation*}
		Using the Gagliardo-Nirenberg inequality \eqref{eq:1DGN-1} and elliptic estimates in Lemma \ref{lem:ellip_3}, there holds
		\begin{align*}
			\|\partial_{y}c_{(0,0)}\|_{L^{\infty}}^2\leq 
			\|\partial_{y}c_{(0,0)}\|_{L^{2}}
			\|\partial_{yy}c_{(0,0)}\|_{L^{2}}
			\leq \frac{1}{2}\|n_{(0,0)}\|_{L^{2}}^2,
		\end{align*} 
		which implies that 
		\begin{equation}\label{eq:c1}
			\begin{aligned}
				\frac{1}{2}\partial_t\|n_{(0,0)}\|^2_{L^2}
				+\frac{1}{A}\|\partial_{y} n_{(0,0)}\|^2_{L^2}
				&\leq \frac{1}{\sqrt{2}A}
				\|n_{(0,0)}\|^2_{L^2}\|\partial_{y}n_{(0,0)}\|_{L^2}\\
				&\leq\frac{\|n_{(0,0)}\|^4_{L^2}}{2\sqrt{2}c_1A}+\frac{c_1}{2\sqrt{2}A}\|\partial_{y} n_{(0,0)}\|^2_{L^2},
			\end{aligned}
		\end{equation}
		where $c_1$ is a positive constant with $c_1\in(0,2\sqrt{2}).$
		Thanks to the Nash's inequality (\ref{eq:1DGN-2}) with the sharp constant (see also \cite{LW1,Na1})
		$$\|n_{(0,0)}\|_{L^2}\leq \left(\frac{16\pi^2}{27}\right)^{-\frac{1}{6}}
		\|\partial_yn_{(0,0)}\|_{L^2}^{\frac{1}{3}}
		\|n_{(0,0)}\|_{L^1}^{\frac{2}{3}},$$
		we have 
		\begin{equation}\label{n00_proof_1}
			\begin{aligned}
				\frac{1}{2}\partial_t\|n_{(0,0)}\|^2_{L^2}\leq 
				-\frac{1}{A}\Big(\frac{4\sqrt{2}\pi^2(2\sqrt{2}-c_1)}{27M_1^4}\|n_{(0,0)}\|_{L^2}^6-\frac{1}{2\sqrt{2}c_1}\|n_{(0,0)}\|^4_{L^2}\Big),
			\end{aligned}
		\end{equation}
		where $M_1=\|n_{(0,0)}\|_{L^1}.$

		Introduce the auxiliary function $h(t)=\|n_{(0,0)}(t)\|^2_{L^2}$ satisfying $h\geq0,$
		then (\ref{n00_proof_1}) becomes
		\begin{equation}\label{n00_proof_2}
			\left\{
			\begin{array}{lr}
				\frac{1}{2}\frac{dh}{dt}\leq 
				-\frac{1}{A}\Big(\frac{4\sqrt{2}\pi^2(2\sqrt{2}-c_1)}{27M_1^4}h^3-\frac{1}{2\sqrt{2}c_1}h^2\Big):=H(h), \quad h\geq0,\\
				\\
				h(0)=\|(n_{\rm in})_{(0,0)}\|^2_{L^2}.
			\end{array}
			\right.
		\end{equation}
		Let $\mathcal{H}(\frac{dh}{dt},h)=\frac{1}{2}\frac{dh}{dt}-H(h),$
		and the
		bound of $\|n_{(0,0)}(t)\|_{L^2}$ can be obtained by studying the orbits  on the phase plane $(\frac{dh}{dt},h)\in \mathbb{R}\times\mathbb{R}^+.$
		The phase portraits are obtained by plotting the level curves of the function $\mathcal{H}(\frac{dh}{dt},h)$,
		and the phase portrait in the phase plane $(\frac{dh}{dt},h)\in \mathbb{R}\times\mathbb{R}^+$ is shown in Figure \ref{fig_1},
		where the black dotted line is $\frac{dh}{dt}=0.$
		It should be noted that the point where the dashed line intersects with the dotted line is the equilibrium point.
		\begin{figure}[h]
			\centering
			\includegraphics[width=2.8in,height=2.3in]{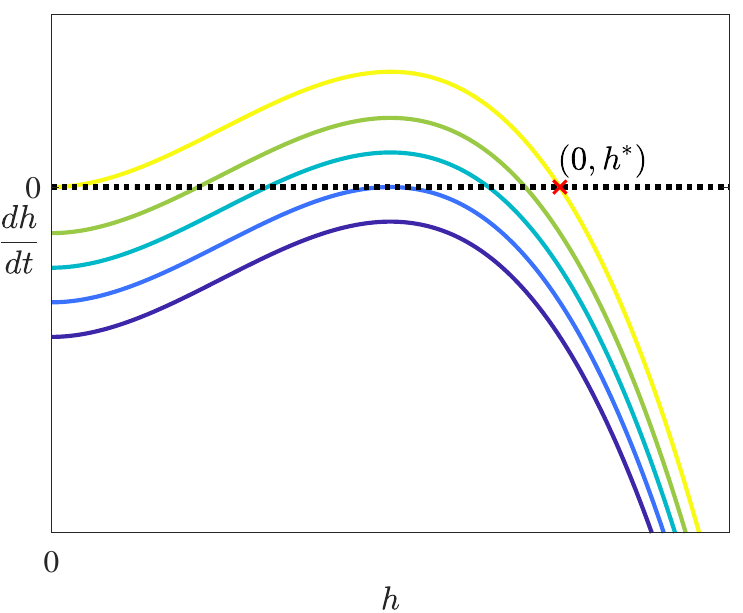}
			\caption{Phase portrait in the phase plane $(\frac{dh}{dt},h)\in \mathbb{R}\times\mathbb{R}^+.$}
			\label{fig_1}
		\end{figure}
		
		In Figure \ref{fig_1}, the yellow solid line is the orbit for $\mathcal{H}(\frac{dh}{dt},h)=0$ corresponding to the case of  
		$$\frac{1}{2}\partial_t\|n_{(0,0)}\|^2_{L^2}= 
		-\frac{1}{A}\Big(\frac{4\sqrt{2}\pi^2(2\sqrt{2}-c_1)}{27M_1^4}\|n_{(0,0)}\|_{L^2}^6-\frac{1}{2\sqrt{2}c_1}\|n_{(0,0)}\|^4_{L^2}\Big),$$
		and other solid lines represent  $\mathcal{H}(\frac{dh}{dt},h)=c$ with
		$c\in(-\infty,0)$, which corresponding to
		\begin{align*}
			\frac{1}{2}\partial_t\|n_{(0,0)}\|^2_{L^2}&= 
			-\frac{1}{A}\Big(\frac{4\sqrt{2}\pi^2(2\sqrt{2}-c_1)}{27M_1^4}\|n_{(0,0)}\|_{L^2}^6-\frac{1}{2\sqrt{2}c_1}\|n_{(0,0)}\|^4_{L^2}\Big)+c\\
			&<	-\frac{1}{A}\Big(\frac{4\sqrt{2}\pi^2(2\sqrt{2}-c_1)}{27M_1^4}\|n_{(0,0)}\|_{L^2}^6-\frac{1}{2\sqrt{2}c_1}\|n_{(0,0)}\|^4_{L^2}\Big).
		\end{align*}
		We only need to consider
		$\mathcal{H}(\frac{dh}{dt},h)=0,$ which corresponding to a dynamic system:
		\begin{equation}\label{n00_proof_3}
			\left\{
			\begin{array}{lr}
				\frac{1}{2}\frac{dh}{dt}=
				-\frac{1}{A}\Big(\frac{4\sqrt{2}\pi^2(2\sqrt{2}-c_1)}{27M_1^4}h^3-\frac{1}{2\sqrt{2}c_1}h^2\Big):=H(h), \quad h\geq0,\\
				\\
				h(0)=\|(n_{\rm in})_{(0,0)}\|^2_{L^2}.
			\end{array}
			\right.
		\end{equation}
		The red cross in Figure \ref{fig_1} is the equilibrium point of system (\ref{n00_proof_3}) satisfying
		$$\frac{dh^*}{dt}=0,\quad h^*=\frac{27M_1^4}{16\pi^2(2\sqrt{2}-c_1)c_1}.$$
		It is easy to get that $\frac{d}{dh}H(h^*)<0,$ thus this equilibrium point is stable.
		Actually, we would like to remind that $h=h^*$ is the largest positive root of the polynomial $H(h).$
		
		A stable equilibrium point means that, as the system evolves over time, the state of the system will gradually approach the stable equilibrium point and eventually stop there.
		On the contrary, the system will not stop at an unstable equilibrium point and cross it.
		For a given initial condition $h(0)=\|(n_{\rm in})_{(0,0)}\|^2_{L^2}$, due to $h=h^*$ is a stable equilibrium point, the system (\ref{n00_proof_3}) will move along the yellow orbit $\mathcal{H}(\frac{dh}{dt},h)=0$ in Figure \ref{fig_1}. When it reaches the equilibrium point (the red cross in Figure \ref{fig_1}), the derivative becomes zero, and the system stops at the equilibrium point without crossing it.
		
		Thus, it's enough to consider two cases:
		\begin{align*}
			h(0)=\|(n_{\rm in})_{(0,0)}\|^2_{L^2}\leq h^*,\quad 
			h(0)=\|(n_{\rm in})_{(0,0)}\|^2_{L^2}>h^*.	
		\end{align*}
		{\bf Case I: $h(0)=\|(n_{\rm in})_{(0,0)}\|^2_{L^2}\leq h^*.$} The initial state of (\ref{n00_proof_3}) falls to the left side of the red cross. It will move to the right along the yellow orbit, and finally stop at the red cross, which implies that
		$$h(t)=\|n_{(0,0)}(t)\|^2_{L^2}\in\left[\|(n_{\rm in})_{(0,0)}\|^2_{L^2}, \frac{27M_1^4}{16\pi^2(2\sqrt{2}-c_1)c_1}\right].$$
		{\bf Case II: $h(0)=\|(n_{\rm in})_{(0,0)}\|^2_{L^2}> h^*.$}  The initial state of (\ref{n00_proof_3}) falls to the right side of the red cross. It will move to the left along the yellow orbit, and finally stop at the red cross, which implies that
		$$h(t)=\|n_{(0,0)}(t)\|^2_{L^2}\in\left[\frac{27M_1^4}{16\pi^2(2\sqrt{2}-c_1)c_1},\|(n_{\rm in})_{(0,0)}\|^2_{L^2}\right].$$
		In this way, we obtain the exact bound for $\|n_{(0,0)}(t)\|^2_{L^2}$ in the system (\ref{n00_proof_3}):
		\begin{align*}
			\|n_{(0,0)}(t)\|^2_{L^{\infty}L^2}
			=\sup_{t\geq0}h(t)&\leq \max\left\{
			\|(n_{\rm in})_{(0,0)}\|^2_{L^2},
			\inf_{c_1\in(0,2\sqrt{2})}\frac{27M_1^4}{16\pi^2(2\sqrt{2}-c_1)c_1}\right\}.
		\end{align*}
		By choosing $c_1=\sqrt{2}$, one can obtain 
		\begin{align*}
			\|n_{(0,0)}(t)\|^2_{L^{\infty}L^2}\leq \max\left\{\|(n_{\rm in})_{(0,0)}\|^2_{L^2},\frac{27M_1^4}{32\pi^2}\right\}.
		\end{align*}
		Through the phase plane analysis and orbit analysis in Figure \ref{fig_1},
		the equal sign strictly holds if and only if the system (\ref{n00_proof_2}) becomes the system (\ref{n00_proof_3}), otherwise, $\|n_{(0,0)}\|^2_{L^{\infty}L^2}$ in the system  (\ref{n00_proof_2}) is less than the system  (\ref{n00_proof_3}).
		
		The proof is complete.
	\end{proof}

	Now, let us consider the full  equation (\ref{n001}).
	\begin{lemma}\label{lemma_n001}
		For the  equation (\ref{n001}), under the assumption (\ref{assumption}), there exists a constant $A_1$ independent of $t$ and $A$, such that if $A>A_1$, there holds
		\begin{align*}
			\|n_{(0,0)}(t)\|^2_{L^{\infty}L^2}\leq \max\left\{\|(n_{\rm in})_{(0,0)}\|^2_{L^2}+\epsilon_1,\frac{27M_1^4}{32\pi^2}(1+\epsilon_1)+\epsilon_1\right\},
		\end{align*}
		where $\epsilon_1$ can be any small positive constant satisfying $\epsilon_1\in(0,1)$.
	\end{lemma}
	\begin{proof}
		Energy estimate shows that 
		\begin{equation}\label{n00_1}
			\begin{aligned}
				&\quad\frac{1}{2}\partial_t\|n_{(0,0)}\|^2_{L^2}
				+\frac{1}{A}\|\partial_{y} n_{(0,0)}\|^2_{L^2}\\
				=&-\frac{1}{A}
				<\partial_{y}\big(n_{(0,0)}\partial_{y}c_{(0,0)}\big)
				+\partial_{y}(n_{\neq}\partial_y c_{\neq})_{(0,0)}
				+\partial_{y}\big(n_{(0,\neq)}\partial_{y}c_{(0,\neq)} \big)_{(0,0)} ,n_{(0,0)}>\\&-\frac{1}{A}<\partial_{y}(u_{2,\neq}n_{\neq})_{(0,0)}
				+\partial_{y}\left(u_{2,(0,\neq)}n_{(0,\neq)} \right)_{(0,0)},n_{(0,0)}>.
			\end{aligned}
		\end{equation}
		Direct calculation gives that 	
		\begin{equation}\label{n00_2}
			\begin{aligned}
				<\partial_{y}(n_{\neq}\partial_y c_{\neq})_{(0,0)},n_{(0,0)}>&
				\leq \|(n_{\neq}\partial_y c_{\neq})_{(0,0)}\|_{L^2}\|\partial_{y} n_{(0,0)}\|_{L^2}\\
				&\leq \frac{4}{\epsilon_1}\|(n_{\neq}\partial_y c_{\neq})_{(0,0)}\|^2_{L^2}+\frac{\epsilon_1}{16}\|\partial_{y} n_{(0,0)}\|^2_{L^2},	
			\end{aligned}
		\end{equation}
		where $\epsilon_1$ is a small positive constant with $\epsilon_1\in(0,1).$
		Similarly, we have 
		\begin{equation}\label{n00_3}
			\begin{aligned}
				<\partial_{y}(u_{2,\neq}n_{\neq})_{(0,0)},n_{(0,0)}>
				&\leq \frac{4}{\epsilon_1}\|(u_{2,\neq}n_{\neq})_{(0,0)}\|^2_{L^2}
				+\frac{\epsilon_1}{16}\|\partial_{y} n_{(0,0)}\|^2_{L^2},\\
				<\partial_{y}\big(n_{(0,\neq)}\partial_{y}c_{(0,\neq)} \big)_{(0,0)},n_{(0,0)}>
				&\leq \frac{4}{\epsilon_1}\|\big(n_{(0,\neq)}\partial_{y}c_{(0,\neq)} \big)_{(0,0)}\|^2_{L^2}+\frac{\epsilon_1}{16}\|\partial_{y} n_{(0,0)}\|^2_{L^2},\\
				<\partial_{y}\left(u_{2,(0,\neq)}n_{(0,\neq)}\right)_{(0,0)} ,n_{(0,0)}>
				&\leq \frac{4}{\epsilon_1}\|(u_{2,(0,\neq)}n_{(0,\neq)})_{(0,0)}\|^2_{L^2}
				+\frac{\epsilon_1}{16}\|\partial_{y} n_{(0,0)}\|^2_{L^2}.
			\end{aligned}
		\end{equation}
		With the help of (\ref{n00_2}) and (\ref{n00_3}), as in (\ref{eq:c1}) we rewrite (\ref{n00_1}) into
		\begin{equation*}
			\begin{aligned}
				\frac{1}{2}\partial_t(\|n_{(0,0)}\|^2_{L^2}-\hat{G}(t))
				+\frac{\|\partial_{y} n_{(0,0)}\|^2_{L^2}}{A}
				&\leq\frac{\|n_{(0,0)}\|^2_{L^2}}{2\sqrt{2}c_1A}+\frac{c_1}{2\sqrt{2}A}\|\partial_{y} n_{(0,0)}\|^2_{L^2}+\frac{\epsilon_1\|\partial_{y} n_{(0,0)}\|^2_{L^2}}{4A},
			\end{aligned}
		\end{equation*}
		where 
		\begin{align*}
			\hat{G}(t)=\frac{8}{A\epsilon_1}\Big(&\|(n_{\neq}\partial_y c_{\neq})_{(0,0)}\|^2_{L^2L^2}
			+\|(u_{2,\neq}n_{\neq})_{(0,0)}\|^2_{L^2L^2}
			+\|\big(n_{(0,\neq)}\partial_{y}c_{(0,\neq)} \big)_{(0,0)}\|^2_{L^2L^2}\\
			+&\|(u_{2,(0,\neq)}n_{(0,\neq)})_{(0,0)}\|^2_{L^2L^2}\Big).
		\end{align*}
		Similar to (\ref{n00_proof_1}) in Lemma \ref{c00_lemma1}, we have
		\begin{equation}\label{n00_5}
			\begin{aligned}
				\frac{1}{2}\partial_t(\|n_{(0,0)}\|^2_{L^2}-\hat{G}(t))
				&\leq
				-\frac{1}{A}\Big(\frac{4\sqrt{2}\pi^2(2\sqrt{2}-c_1-\frac{\sqrt{2}}{2}\epsilon_1)}{27M_1^4}\|n_{(0,0)}\|_{L^2}^6-\frac{1}{2\sqrt{2}c_1}\|n_{(0,0)}\|^4_{L^2}\Big).
			\end{aligned}
		\end{equation}
		Using Lemma \ref{lemma_non_zz0}, Lemma \ref{lem:ellip_2}, Lemma \ref{lem:ellip_4} and Lemma \ref{lemma_u}, under the assumption (\ref{assumption}), if $$A\geq{\rm max}\{1,C\epsilon_1^{-\frac{2}{3\epsilon}}(E_1+E_{2})^{\frac{4}{3\epsilon}}\}:=A_1,$$ there holds
		\begin{align*}
			\hat{G}(t)\leq&\frac{C}{A\epsilon_1}\Big(\|\triangle c_{\neq}\|_{L^{\infty}L^2}
			\|\nabla c_{\neq}\|_{L^{\infty}L^2}
			\|n_{\neq}\|^2_{L^2L^2}+
			\|\partial_y u_{\neq}\|_{L^{2}L^2}
			\|u_{\neq}\|_{L^{2}L^2}
			\|n_{\neq}\|^2_{L^{\infty}L^2}\\
			+&(\|\triangle c_{(0,\neq)}\|_{L^{\infty}L^2}
			\|\nabla c_{(0,\neq)}\|_{L^{\infty}L^2}+\|\partial_y u_{2,(0,\neq)}\|_{L^{\infty}L^2}
			\|u_{2,(0,\neq)}\|_{L^{\infty}L^2})
			\|n_{(0,\neq)}\|^2_{L^2L^2}\Big)\\
			\leq&\frac{C}{A\epsilon_1}\Big(A^{\frac{1}{3}}\|n_{\neq}\|^4_{X_a}
			+A^{\frac{2}{3}}\|n_{\neq}\|^2_{X_a}(\|\triangle u_{2,\neq}\|^2_{X_a}
			+\|\partial_x\omega_{2,\neq}\|^2_{X_a})+A\|n_{(0,\neq)}\|^4_{Y_0}\\
			&+A\|\nabla u_{2,(0,\neq)}\|_{Y_0}\|u_{2,(0,\neq)}\|_{Y_0}
			\|n_{(0,\neq)}\|^2_{Y_0}\Big)
			\leq\frac{C(E_1^2+E_2^2)^2}{A^{3\epsilon}\epsilon_1}\leq {\epsilon_1}.
		\end{align*}
		Introducing the auxiliary function $g(t)=\|n_{(0,0)}(t)\|^2_{L^2}$ satisfying $g\geq0,$
		then there holds
		\begin{equation*}
			\begin{aligned}
				&\frac{1}{2}\frac{d}{dt}(g-\hat{G})\leq 
				-\frac{1}{A}\Big(\frac{4\sqrt{2}\pi^2(2\sqrt{2}-c_1
					-\frac{\sqrt{2}}{2}\epsilon_1)}{27M_1^4}g^3
				-\frac{1}{2\sqrt{2}c_1}g^2\Big):=G(g), \quad g\geq0,\\
				&g(0)=\|(n_{\rm in})_{(0,0)}\|^2_{L^2}.
			\end{aligned}
		\end{equation*}
		The largest positive root of $G(g)$ satisfying $G(g^{*})=0$ is
		$g^*=\frac{27M_1^4}{16\pi^2(2\sqrt{2}-c_1-\frac{\sqrt{2}}{2}\epsilon_1)c_1},$
		and the stagnation point  of $G(g)$ satisfying $\frac{d}{dg}G(g^{**})=0$ is 
		$g^{**}=\frac{9M_1^4}{8\pi^2(2\sqrt{2}-c_1-\frac{\sqrt{2}}{2}\epsilon_1)c_1}.$
		An important fact is that, as long as $g(t)-\hat{G}\geq g^{**},$
		there holds $G(g)\leq G(g(t)-\hat{G}).$ 
		
		Thus, we need to consider two cases:
		\begin{equation*}
			g(0)=\|(n_{\rm in})_{(0,0)}\|^2_{L^2}\leq g^*,\quad
			g(0)=\|(n_{\rm in})_{(0,0)}\|^2_{L^2}\geq g^*.	
		\end{equation*}
		
		{\bf $\bullet$ Case I: $\|(n_{\rm in})_{(0,0)}\|^2_{L^2}\leq g^*.$}
		Denote $t=t_0$ as the first time $g(t_0)-{\epsilon_1}=g^{**},$ then one can construct a new system  with $t\in[t_0,T)$:
		\begin{equation*}
			\begin{aligned}
				&\frac{1}{2}\frac{d}{dt}(g-\hat{G})\leq G(g)\leq 
				-\frac{1}{A}\Big(\frac{4\sqrt{2}\pi^2(2\sqrt{2}-c_1-\frac{\sqrt{2}}{2}\epsilon_1)}
				{27M_1^4}(g-\hat{G})^3-\frac{(g-\hat{G})^2}{2\sqrt{2}c_1}\Big)
				=G(g-\hat{G}),\\
				&g(t)-\hat{G}(t)\geq0~{\rm in}~{t\in[t_0,T)},\quad g(t_0)-\hat{G}(t_0)=g^{**}+{\epsilon_1}-\hat{G}(t_0).
			\end{aligned}
		\end{equation*}
		It should be noted that $t_0$ is a non-negative constant and can be zero or positive infinity.
		
		Applying the equilibrium point analysis method in Lemma \ref{c00_lemma1}, 
		for $t\in[t_0,T)$, we get 
		$$g-\hat{G}\leq g^*=\frac{27M_1^4}{16\pi^2(2\sqrt{2}-c_1-\frac{\sqrt{2}}{2}\epsilon_1)c_1},$$
		and one can extend $T$ to $\infty.$
		Combining it with $g(t)-\hat{G}\leq g^{**}+{\epsilon_1}<g^{*}$ in $t\in[0,t_0],$ we obtain that 
		$$g(t)\leq \frac{27M_1^4}{16\pi^2(2\sqrt{2}-c_1-\frac{\sqrt{2}}{2}\epsilon_1)c_1}+\hat{G},
		~{\rm for}~t\in[0,\infty).$$
		
		{\bf $\bullet$ Case II:  $\|(n_{\rm in})_{(0,0)}\|^2_{L^2}\geq g^*.$} Due to $G(g)\leq G(g(t)-\hat{G}),$
		one can construct  another system:
		\begin{equation*}
			\begin{aligned}
				&\frac{1}{2}\frac{d}{dt}(g-\hat{G})\leq G(g)\leq 
				-\frac{1}{A}\Big(\frac{4\sqrt{2}\pi^2(2\sqrt{2}-c_1-\frac{\sqrt{2}}{2}\epsilon_1)}
				{27M_1^4}(g-\hat{G})^3-\frac{(g-\hat{G})^2}{2\sqrt{2}c_1}\Big)
				=G(g-\hat{G}),\\
				&g(0)-\hat{G}(0)=\|(n_{\rm in})_{(0,0)}\|^2_{L^2}.
			\end{aligned}
		\end{equation*}
		Applying the equilibrium point analysis method again, there holds
		$$g(t)\in \left[\frac{27M_1^4}{16\pi^2(2\sqrt{2}-c_1-\frac{\sqrt{2}}{2}\epsilon_1)c_1}+\hat{G}(t),\|(n_{\rm in})_{(0,0)}\|^2_{L^2}+\hat{G}(t)\right] ,
		~{\rm for}~t\in[0,\infty).$$
		For $\epsilon_1\in(0,1),$ thanks to Taylor series, there holds
		$${2}{\Big(\sqrt{2}-\frac{\sqrt{2}}{4}\epsilon_1\Big)^{-2}}
		=1+\frac{1}{2}\epsilon_1+\frac{3}{16}\epsilon_1^2
		+\frac{1}{16}\epsilon_1^3+O(\epsilon_1^4)\leq 1+\epsilon_1.$$
		Using $\hat{G}(t)\leq {\epsilon_1},$ we conclude that
		\begin{align*}
			\|n_{(0,0)}(t)\|^2_{L^{\infty}L^2}
			=\sup_{t\geq0}g(t)
			&\leq \max\left\{\|(n_{\rm in})_{(0,0)}\|^2_{L^2}+{\epsilon_1},\frac{27M_1^4}{32\pi^2}(1+\epsilon_1)+{\epsilon_1}\right\}.
		\end{align*}
		The proof is complete.
	\end{proof}
	
	Then, we are devoted to estimate $n_{(0,\neq)}$ under the assumptions (\ref{assumption}). 
	By \eqref{eq:n-zero mode}, the z-part non-zero mode $ n_{(0,\neq)} $ satisfies 	
	\begin{equation}\label{n_0_neq}
		\begin{aligned}
			&\partial_t n_{(0,\neq)}-\frac{1}{A}\triangle n_{(0,\neq)}
			=-\frac{1}{A}\left[\nabla\cdot(n_{\neq}\nabla c_{\neq})_{(0,\neq)}
			+\nabla\cdot(u_{\neq}n_{\neq})_{(0,\neq)}\right]\\
			&-\frac{1}{A}\left[\partial_y\big(n_{(0,0)}\partial_{y}c_{(0,\neq)}
			+n_{(0,\neq)}\partial_{y}c_{(0,0)}
			+(n_{(0,\neq)}\partial_{y}c_{(0,\neq)})_{(0,\neq)}\big)\right]\\
			&-\frac{1}{A}\left[\partial_z\big(n_{(0,0)}\partial_{z}c_{(0,\neq)}  
			+n_{(0,\neq)}\partial_{z}c_{(0,\neq)}\big)
			+\partial_y\big(u_{2,(0,\neq)}n_{(0,0)}
			+(u_{2,(0,\neq)}n_{(0,\neq)})_{(0,\neq)}\big)\right]\\
			&-\frac{1}{A}\left[\partial_z\big(u_{3,(0,0)}n_{(0,\neq)}
			+u_{3,(0,\neq)}n_{(0,0)}
			+(u_{3,(0,\neq)}n_{(0,\neq)})_{(0,\neq)}\big)\right].
		\end{aligned}
	\end{equation}
	In fact, we can regard (\ref{n_0_neq}) as 
	\begin{equation}\label{n_0_neq_1}
		\begin{aligned}
			\partial_t n_{(0,\neq)}-\frac{1}{A}\triangle n_{(0,\neq)}
			=&-\frac{1}{A}\Big(\partial_y\big(n_{(0,0)}\partial_{y}c_{(0,\neq)})+\partial_z\big(n_{(0,0)}\partial_{z}c_{(0,\neq)}\big)
			+\partial_y(\partial_{y}c_{(0,0)}n_{(0,\neq)})\\
			&+\partial_y\big(n_{(0,0)}u_{2,(0,\neq)}\big)
			+\partial_z(n_{(0,0)}u_{3,(0,\neq)})\Big)
			+``{\rm good~terms}".
		\end{aligned}
	\end{equation}
	For (\ref{n_0_neq_1}), the basic energy estimate shows that  
	$$A^{2\epsilon}\|n_{(0,\neq)}\|_{L^2}^2
	+A^{2\epsilon-1}\|\nabla n_{(0,\neq)}\|_{L^2}^2
	\leq C+C^{(1)}A^{2\epsilon-1}\|n_{(0,0)}\|_{L^{\infty}L^2}^2\|\nabla n_{(0,\neq)}\|_{L^2}^2+``{\rm good~terms}".$$
	Therefore, if we impose the condition $\|n_{(0,0)}\|_{L^{\infty}L^2}^2<\frac{1}{C^{(1)}},$ 
	this energy estimate can be closed successfully, yielding the result $A^{\epsilon}\|n_{(0,\neq)}\|_{Y_0}\leq C$. 
	
	However, since $\|n_{(0,0)}\|_{L^{\infty}L^2}$ is related to $\|(n_{\rm in})_{(0,0)}\|_{L^{\infty}L^2}$ (see Lemma \ref{lemma_n001}), this result is not satisfactory. 
	A preferable outcome would be to impose a restriction on the total mass $M$ without requiring a restriction on $\|(n_{\rm in})_{(0,0)}\|_{L^2}$.
	
	Our important find is that $\|\partial_zn_{(0,\neq)}\|^2_{L^2L^2}$ is enough
	to control the energies $E_{1,2}$ and $E_{1,3}$ under the assumptions
	(\ref{assumption}).
	A fact is that, for the bad nonlinear terms $(0,0)\cdot(0,\neq)$ of (\ref{n_0_neq_1}), the derivative 
	$\partial_z$ acts only on $(0,\neq)$ part, and does not affect  $(0,0)$ part. 
	That is to say, for any given functions $f$ and $g$, we have 
	$\partial_z^j(f_{(0,0)}g_{(0,\neq)})=f_{(0,0)}\partial_z^jg_{(0,\neq)},$
	where $j\geq 1.$ 
	Thus, for (\ref{n_0_neq_1}), we judge that $\partial_z$ is a good derivative, while $\partial_y$ is a bad derivative.  
	Due to the elliptic condition $n_{(0,\neq)}=-\triangle c_{(0,\neq)}+c_{(0,\neq)},$ the bad derivative $\partial_y$ of $n_{(0,\neq)}$ can 
	be moved out by  $c_{(0,\neq)}.$
	
	Naturally, this leads us to introduce a new idea for energy estimates. 
	Taking $\partial_z$ to (\ref{n_0_neq}) and multiplying  $\partial_zc_{(0,\neq)}$ on both sides, 
	with the help of elliptic condition $n_{(0,\neq)}=-\triangle c_{(0,\neq)}+c_{(0,\neq)}$ and some precise elliptic estimates,
	we can obtain that 
	\begin{equation*}
		\begin{aligned}
			&\quad\frac{\|\partial_z\nabla c_{(0,\neq)}\|^2_{L^{\infty}H^1}
			}{2A^{-2\epsilon}}
			+\frac{\|\partial_z\triangle c_{(0,\neq)}\|^2_{L^2L^2}+
				\|\partial_z\nabla c_{(0,\neq)}\|^2_{L^{2}L^2}}{A^{1-2\epsilon}}\\
			&\leq A^{2\epsilon}\|(n_{\rm in})_{(0,\neq)}\|^2_{L^2}
			+\frac{C^{(2)}M(\|\partial_z\triangle c_{(0,\neq)}\|^2_{L^2L^2}+
				\|\partial_z\nabla c_{(0,\neq)}\|^2_{L^{2}L^2})}{A^{1-2\epsilon}}
			+``{\rm good~terms}".
		\end{aligned}
	\end{equation*}
	In this way, as long as we impose the condition $M<\frac{1}{C^{(2)}},$ 
	we can finish the proof by using $\|\partial_zn_{(0,\neq)}\|^2_{L^2L^2}\leq C\|\partial_z\triangle c_{(0,\neq)}\|^2_{L^2L^2}$.
	The next step is to perform precise calculations and get the constant $C^{(2)}.$

	\begin{lemma}\label{lemma_n002}
		Under the conditions of Theorem \ref{result} and the assumptions (\ref{assumption}), 
		as long as $$M< \frac{24}{5}\pi^2,$$
		there exists a constant $A_2$ independent of $t$ and $A$, such that if $A>A_2$, 
		there holds 
		$$
		A^{2\epsilon}\|\partial_z\nabla c_{(0,\neq)}\|^2_{L^{\infty}L^2}+
		\frac{\|\partial_zn_{(0,\neq)}\|^2_{L^2L^2}}{A^{1-2\epsilon}}\leq C.$$
	\end{lemma}
	\begin{proof}
		
		Taking $\partial_z$ to (\ref{n_0_neq}) and multiplying $\partial_zc_{(0,\neq)},$
		by the elliptic condition $n_{(0,\neq)}=-\triangle c_{(0,\neq)}+c_{(0,\neq)},$
		then the energy estimates show that 
		\begin{equation}\label{n0_neq_1}
			\begin{aligned}
				&\partial_t\frac{\|\partial_z\nabla c_{(0,\neq)}\|^2_{L^2}
					+\|\partial_zc_{(0,\neq)}\|^2_{L^2}}{2A^{-2\epsilon}}
				+\frac{\|\partial_z\triangle c_{(0,\neq)}\|^2_{L^2}+
					\|\partial_z\nabla c_{(0,\neq)}\|^2_{L^2}}{A^{1-2\epsilon}}= \\
				&-A^{2\epsilon-1}\big(<\partial_y\partial_z\big(n_{(0,0)}\partial_{y}c_{(0,\neq)}),\partial_zc_{(0,\neq)}>
				+<\partial_z^2\big(n_{(0,0)}\partial_{z}c_{(0,\neq)}\big),\partial_zc_{(0,\neq)}>\big)\\
				&-A^{2\epsilon-1}\big(<\partial_y\partial_z(\partial_{y}c_{(0,0)}n_{(0,\neq)}),\partial_zc_{(0,\neq)}>
				+<\partial_y\partial_z\big(u_{2,(0,\neq)}n_{(0,0)}\big),\partial_zc_{(0,\neq)}>\big)\\	&-A^{2\epsilon-1}\big(<\partial_z^2(u_{3,(0,\neq)}n_{(0,0)}),
				\partial_zc_{(0,\neq)}>
				+<\nabla\cdot\partial_z(u_{\neq}n_{\neq})_{(0,\neq)},
				\partial_zc_{(0,\neq)}>\big)\\
				&-A^{2\epsilon-1}\big(<\nabla\cdot\partial_z(n_{\neq}\nabla c_{\neq})_{(0,\neq)},\partial_zc_{(0,\neq)}>
				+<\partial_y\partial_z(n_{(0,\neq)}
				\partial_{y}c_{(0,\neq)})_{(0,\neq)},
				\partial_zc_{(0,\neq)}>\big)\\
				&-A^{2\epsilon-1}\big(<\partial_z^2(n_{(0,\neq)}
				\partial_{z}c_{(0,\neq)})_{(0,\neq)},\partial_zc_{(0,\neq)}>
				+<\partial_y\partial_z(u_{2,(0,\neq)}n_{(0,\neq)})_{(0,\neq)},
				\partial_zc_{(0,\neq)}>\big)\\
				&-A^{2\epsilon-1}\big(<\partial_z^2(u_{3,(0,0)}n_{(0,\neq)}),
				\partial_zc_{(0,\neq)}>+
				<\partial_z^2(u_{3,(0,\neq)}n_{(0,\neq)})_{(0,\neq)},
				\partial_zc_{(0,\neq)}>\big).
			\end{aligned}
		\end{equation}
		
		After integrating in time, we infer from (\ref{n0_neq_1}) that  
		\begin{equation}\label{n0_neq_2}
			\begin{aligned}
				&\quad\frac{\|\partial_z\nabla c_{(0,\neq)}\|^2_{L^{\infty}L^2}
					+\|\partial_zc_{(0,\neq)}\|^2_{L^{\infty}L^2}}{2A^{-2\epsilon}}
				+\frac{\|\partial_z\triangle c_{(0,\neq)}\|^2_{L^2L^2}+
					\|\partial_z\nabla c_{(0,\neq)}\|^2_{L^{2}L^2}}{A^{1-2\epsilon}}\\
				&\leq A^{2\epsilon}\|(n_{\rm in})_{(0,\neq)}\|^2_{L^2}+T_{1,1}+\cdots+T_{1,12},
			\end{aligned}
		\end{equation}
		where $T_{1,1}-T_{1,5}$ can be regard to bad terms and 
		$T_{1,6}-T_{1,12}$ can be regard to good terms.
		
		\textbf{Estimates of $T_{1,1}$, $T_{1,2}$ and $T_{1,3}$.}	
		First, we need some deformations
		\begin{align}
			&-<\partial_y\partial_z(n_{(0,0)}\partial_yc_{(0,\neq)}),\partial_zc_{(0,\neq)}>=
			<n_{(0,0)},(\partial_y\partial_zc_{(0,\neq)})^2>,
			\label{n0_temp1}\\
			&-<\partial_z^2(n_{(0,0)}\partial_zc_{(0,\neq)}),\partial_zc_{(0,\neq)}>=<n_{(0,0)},(\partial_z^2c_{(0,\neq)})^2>.\label{n0_temp2}
		\end{align}
		Due to $n_{(0,\neq)}=-\triangle c_{(0,\neq)}+c_{(0,\neq)},$ 
		then 
		\begin{equation*}
			\begin{aligned}
				&-<\partial_y\partial_z(\partial_{y}c_{(0,0)}n_{(0,\neq)}),
				\partial_zc_{(0,\neq)}>=
				<\partial_{y}c_{(0,0)}\partial_zn_{(0,\neq)},
				\partial_y\partial_zc_{(0,\neq)}>\\
				=&-<\partial_{y}c_{(0,0)}\partial_z\triangle c_{(0,\neq)},
				\partial_y\partial_zc_{(0,\neq)}>
				+<\partial_{y}c_{(0,0)}\partial_z c_{(0,\neq)},
				\partial_y\partial_zc_{(0,\neq)}>.			
			\end{aligned}
		\end{equation*}
		Using $-\partial_y^2c_{(0,0)}=n_{(0,0)}-c_{(0,0)}$ and $c_{(0,0)}\geq0$, we have 
		\begin{align*}
			&-<\partial_{y}c_{(0,0)}\partial_y^2\partial_z c_{(0,\neq)},
			\partial_y\partial_zc_{(0,\neq)}>=\frac{<\partial_{y}^2c_{(0,0)},(\partial_y\partial_zc_{(0,\neq)})^2>}{2},\\
			&-<\partial_{y}c_{(0,0)}\partial_z^3 c_{(0,\neq)},
			\partial_y\partial_zc_{(0,\neq)}>=
			-\frac{<\partial_{y}^2c_{(0,0)},(\partial_z^2c_{(0,\neq)})^2>}{2}
			\leq \frac{<n_{(0,0)},(\partial_z^2c_{(0,\neq)})^2>}{2}
			,\\
			&<\partial_{y}c_{(0,0)}\partial_z c_{(0,\neq)},
			\partial_y\partial_zc_{(0,\neq)}>=
			-\frac{<\partial_{y}^2c_{(0,0)},(\partial_zc_{(0,\neq)})^2>}{2}
			\leq \frac{<n_{(0,0)},(\partial_zc_{(0,\neq)})^2>}{2},
		\end{align*}	
		which imply that 	
		\begin{equation}\label{n0_temp3}
			\begin{aligned}
				&-<\partial_y\partial_z(\partial_{y}c_{(0,0)}n_{(0,\neq)}),
				\partial_zc_{(0,\neq)}>
				=\frac{<\partial_{y}^2c_{(0,0)},(\partial_y\partial_zc_{(0,\neq)})^2-(\partial_z^2c_{(0,\neq)})^2-(\partial_zc_{(0,\neq)})^2>}{2}
				\\
				&\leq \frac{<\partial_{y}^2c_{(0,0)}
					,(\partial_y\partial_zc_{(0,\neq)})^2>}{2}
				+\frac{<n_{(0,0)},(\partial_z^2c_{(0,\neq)})^2+
					(\partial_zc_{(0,\neq)})^2>}{2}.
			\end{aligned}
		\end{equation}
		
		By (\ref{n0_temp1}), (\ref{n0_temp2}) and (\ref{n0_temp3}), we get
		\begin{equation}\label{cn_temp00}
			\begin{aligned}
				&T_{1,1}+T_{1,2}+T_{1,3}\\
				=&\frac{\int_{0}^t<n_{(0,0)},(\partial_z^2c_{(0,\neq)})^2>dt}{A^{1-2\epsilon}}
				+\frac{\int_{0}^t<\partial_{y}c_{(0,0)},\partial_z c_{(0,\neq)}
					\partial_y\partial_zc_{(0,\neq)}+\partial_z^2 c_{(0,\neq)}
					\partial_y\partial_z^2c_{(0,\neq)}>dt}{A^{1-2\epsilon}}\\
				&+\frac{\int_{0}^t<c_{(0,0)},(\partial_y\partial_zc_{(0,\neq)})^2>dt}{2A^{1-2\epsilon}}
				+\frac{\int_{0}^t<n_{(0,0)},(\partial_y\partial_zc_{(0,\neq)})^2>dt}{2A^{1-2\epsilon}}
				\\\leq&\frac{\int_{0}^t<n_{(0,0)},3(\partial_z^2c_{(0,\neq)})^2
					+(\partial_zc_{(0,\neq)})^2>dt}
				{2A^{1-2\epsilon}}
				+\frac{\int_{0}^t<n_{(0,0)},(\partial_y\partial_zc_{(0,\neq)})^2>dt}{2A^{1-2\epsilon}}\\
				&+\frac{\int_{0}^t<c_{(0,0)},(\partial_y\partial_zc_{(0,\neq)})^2>dt}{2A^{1-2\epsilon}}.
			\end{aligned}
		\end{equation}
		By Lemma \ref{sob_inf_3} and Young's inequality, there hold
		\begin{equation}\label{cn_temp01}
			\begin{aligned}
				<n_{(0,0)},(\partial_z^2c_{(0,\neq)})^2>&\leq \|n_{(0,0)}\|_{L^1}
				\|\partial_z^2c_{(0,\neq)}\|^2_{L^{\infty}_yL^2_z}
				\leq \|n_{(0,0)}\|_{L^1}
				\|\partial_y\partial_z^2c_{(0,\neq)}\|_{L^2}
				\|\partial_z^2c_{(0,\neq)}\|_{L^2}\\
				&\leq \|n_{(0,0)}\|_{L^1}
				\Big(\frac{\|\partial_y\partial_z^2c_{(0,\neq)}\|_{L^2}^2
					+\|\partial_z^2c_{(0,\neq)}\|^2_{L^2}}{2}\Big),\\
				<n_{(0,0)},(\partial_zc_{(0,\neq)})^2>&\leq \|n_{(0,0)}\|_{L^1}
				\|\partial_zc_{(0,\neq)}\|^2_{L^{\infty}_yL^2_z}
				\leq \|n_{(0,0)}\|_{L^1}
				\|\partial_y\partial_zc_{(0,\neq)}\|_{L^2}
				\|\partial_zc_{(0,\neq)}\|_{L^2}\\
				&\leq \|n_{(0,0)}\|_{L^1}
				\Big(\frac{\|\partial_y\partial_zc_{(0,\neq)}\|_{L^2}^2
					+\|\partial_zc_{(0,\neq)}\|^2_{L^2}}{2}\Big),\\
				<n_{(0,0)},(\partial_y\partial_zc_{(0,\neq)})^2>
				&\leq \|n_{(0,0)}\|_{L^1}
				\Big(\frac{\|\partial_y^2\partial_zc_{(0,\neq)}\|_{L^2}^2
					+\|\partial_y\partial_zc_{(0,\neq)}\|^2_{L^2}}{2}\Big).
			\end{aligned}
		\end{equation}
		Using Lemma \ref{lem:ellip_30}, we have 
		\begin{equation}\label{cn_temp02}
			<c_{(0,0)},(\partial_y\partial_zc_{(0,\neq)})^2>
			\leq 	\|c_{(0,0)}\|_{L^{\infty}}
			\|\partial_y\partial_zc_{(0,\neq)}\|_{L^2}^2
			\leq \frac{\|n_{(0,0)}\|_{L^1}
				\|\partial_y\partial_zc_{(0,\neq)}\|_{L^2}^2}{2}.
		\end{equation}
		By $\|\partial_z\triangle c_{(0,\neq)}\|_{L^2L^2}^2=
		\|\partial_y^2\partial_z c_{(0,\neq)}\|_{L^2L^2}^2
		+2\|\partial_y\partial_z^2 c_{(0,\neq)}\|_{L^2L^2}^2
		+\|\partial_z^3 c_{(0,\neq)}\|_{L^2L^2}^2,$
		using (\ref{cn_temp01}), (\ref{cn_temp02}) and  
		$\|\partial_z^2c_{(0,\neq)}\|_{L^2L^2}^2
		\leq \|\partial_z^3c_{(0,\neq)}\|_{L^2L^2}^2,$
		we infer from \eqref{cn_temp00} that 
		\begin{equation*}
			\begin{aligned}
				&T_{1,1}+T_{1,2}+T_{1,3}\leq \frac{\|n_{(0,0)}\|_{L^\infty 	L^1}}{A^{1-2\epsilon}}
				\Big(\frac{\|\partial_z\triangle c_{(0,\neq)}\|_{L^2L^2}^2
					+\|\partial_z\nabla c_{(0,\neq)}\|_{L^2L^2}^2
				}{2}-\frac{\|\partial_y^2\partial_z c_{(0,\neq)}
					\|_{L^2L^2}^2}{4}\Big).
			\end{aligned}
		\end{equation*}

		\textbf{Estimates of $T_{1,4}$ and $T_{1,5}$.}	
		By Lemma \ref{sob_inf_3}, there hold
		\begin{equation}\label{t13_temp1}
			\begin{aligned}
				&\quad<\partial_y\partial_z\big(u_{2,(0,\neq)}n_{(0,0)}\big),\partial_zc_{(0,\neq)}>\leq \|n_{(0,0)}\|_{L^1}
				\|\partial_y\partial_zc_{(0,\neq)}\|_{L^{\infty}_yL^2_z}
				\|\partial_zu_{2,(0,\neq)}\|_{L^{\infty}_yL^2_z}\\
				&\leq \|n_{(0,0)}\|_{L^1}
				\|\partial_y^2\partial_zc_{(0,\neq)}\|^{\frac12}_{L^2}
				\|\partial_y\partial_zc_{(0,\neq)}\|^{\frac12}_{L^2}
				\|\partial_y\partial_zu_{2,(0,\neq)}\|^{\frac12}_{L^2}
				\|\partial_zu_{2,(0,\neq)}\|^{\frac12}_{L^2},
			\end{aligned}
		\end{equation}
		and 
		\begin{equation}\label{t13_temp2}
			\begin{aligned}
				&\quad<\partial_z^2\big(u_{3,(0,\neq)}n_{(0,0)}\big),\partial_zc_{(0,\neq)}>\leq \|n_{(0,0)}\|_{L^1}
				\|\partial_z^2c_{(0,\neq)}\|_{L^{\infty}_yL^2_z}
				\|\partial_zu_{3,(0,\neq)}\|_{L^{\infty}_yL^2_z}\\
				&\leq \|n_{(0,0)}\|_{L^1}
				\|\partial_y\partial_z^2c_{(0,\neq)}\|^{\frac12}_{L^2}
				\|\partial_z^2c_{(0,\neq)}\|^{\frac12}_{L^2}
				\|\partial_y\partial_zu_{3,(0,\neq)}\|^{\frac12}_{L^2}
				\|\partial_zu_{3,(0,\neq)}\|^{\frac12}_{L^2}.
			\end{aligned}
		\end{equation}
		
		It follows from \eqref{ini11} that 
		\begin{equation}\label{eq:u2u30}
			\left\{
			\begin{array}{lr}
				\partial_t\partial_zu_{2,0}-\frac{1}{A}\triangle \partial_zu_{2,0}
				+\frac{1}{A}\partial_z(u\cdot\nabla u_{2})_0
				+\frac{1}{A}\partial_y\partial_zP^{N_1}_0
				+\frac{1}{A}\partial_y\partial_z P^{N_2}_0+\frac{1}{A}\partial_y\partial_z P^{N_3}_0=\frac{\partial_zn_0}{A}, \\
				\partial_t\partial_zu_{3,0}-\frac{1}{A}\triangle\partial_z u_{3,0}
				+\frac{1}{A}\partial_z(u\cdot\nabla u_3)_0
				+\frac{1}{A}\partial_z^2P^{N_1}_0
				+\frac{1}{A}\partial_z^2 P^{N_2}_0
				+\frac{1}{A}\partial_z^2 P^{N_3}_0=0.
			\end{array}
			\right.
		\end{equation}
		Using Lemma \ref{sob_12} and $\partial_zu_{3,0}=-\partial_yu_{2,0}$, 
		for $j=2,3,$
		there holds
		\begin{equation}\label{u23_zero}
			\begin{aligned}	
				\|\nabla(u_{0}\cdot\nabla u_{j,0})\|_{L^2L^2}
				\leq C(\|u_{2,0}\|_{L^{\infty}H^2}+\|u_{3,0}\|_{L^{\infty}H^1})
				(\|\nabla
				u_{2,0}\|_{L^{2}H^2}+\|\nabla u_{3,0}\|_{L^{2}H^1}).
			\end{aligned}
		\end{equation}
		Combining (\ref{u23_zero}) with Lemma \ref{lemma_neq1}, we obtain
		\begin{equation}\label{u23_zero_1}
			\begin{aligned}
				\|\partial_z(u\cdot\nabla u_j)_0\|_{L^2L^2}\leq
				\|\partial_z(u_{\neq}\cdot\nabla u_{j,\neq})\|_{L^2L^2}
				+\|\partial_z(u_{0}\cdot\nabla u_{j,0})\|_{L^2L^2}
				\leq CA^{\frac12-\frac32\epsilon}(E_1^2+E_2^2).
			\end{aligned}
		\end{equation}
		Using  above results, the 
		energy estimates of \eqref{eq:u2u30} show that 
		\begin{equation}\label{u0_ans1}
			\begin{aligned}
				&\frac{A^{2\epsilon}}{2}(\|\partial_z u_{2,0}\|_{L^{\infty}L^2}^2+\|\partial_z u_{3,0}\|_{L^{\infty}L^2}^2)
				+A^{2\epsilon-1}\Big(
				\|\partial_z\nabla u_{2,0}\|_{L^2L^2}^2
				+\|\partial_z\nabla u_{3,0}\|_{L^2L^2}^2\Big)\\
				&\leq C\left(1+\frac{E_1^3+E_2^3}{A^{\frac{\epsilon}{2}}}\right)
				+A^{2\epsilon-1}
				\int_0^t<\partial_zn_{0},\partial_zu_{2,(0,\neq)}>dt.
			\end{aligned}
		\end{equation}
		Due to $\partial_zn_{0}=\partial_zn_{(0,\neq)}=-\partial_z\triangle c_{(0,\neq)}+\partial_zc_{(0,\neq)},$ then 
		\begin{equation}\label{u0_ans2}
			\begin{aligned}
				&\quad<\partial_zn_{0}, \partial_zu_{2,(0,\neq)}>
				=<-\partial_z\triangle c_{(0,\neq)}+\partial_zc_{(0,\neq)}, 	\partial_zu_{2,(0,\neq)}>\\
				&\leq \|\partial_z\partial_y u_{2,(0,\neq)}\|_{L^2}
				\|\partial_y\partial_z c_{(0,\neq)}\|_{L^2}
				+\|\partial_z^2u_{2,(0,\neq)}\|_{L^2}
				(\|\partial_z^2c_{(0,\neq)}\|_{L^2}
				+\|c_{(0,\neq)}\|_{L^2})
				\\
				&\leq \frac{\|\partial_z \nabla u_{2,(0,\neq)}\|_{L^2}^2
				}{2}+\frac{\|\partial_y\partial_z c_{(0,\neq)}\|_{L^2}^2}{2}
				+\|\partial_z^2 c_{(0,\neq)}\|_{L^2}^2
				+\| c_{(0,\neq)}\|_{L^2}^2.
			\end{aligned}
		\end{equation}
		Using (\ref{u0_ans1}) and (\ref{u0_ans2}), we have 
		\begin{equation}\label{u0_ans3}
			\begin{aligned}
				&\quad 
				A^{2\epsilon-1}\Big(\frac{\|\partial_z\nabla u_{2,0}\|_{L^2L^2}^2}{2}
				+\|\partial_z\nabla u_{3,0}\|_{L^2L^2}^2\Big)\\
				&\leq C\Big(1+\frac{E_1^3+E_2^3}{A^{\frac{\epsilon}{2}}}\Big)
				+A^{2\epsilon-1}\Big(\frac{\|\partial_y\partial_z c_{(0,\neq)}\|_{L^2}^2}{2}
				+\|\partial_z^2 c_{(0,\neq)}\|_{L^2}^2
				+\| c_{(0,\neq)}\|_{L^2}^2\Big).
			\end{aligned}
		\end{equation}
		
		For given positive functions $f_j~(j=1,2,3,4),$  the following Young's inequality holds
		\begin{equation}\label{young_1}
			\begin{aligned}
				f_1^{\frac{1}{2}}f_2^{\frac{1}{2}}
				f_3^{\frac{1}{2}}f_4^{\frac{1}{2}}
				\leq \frac{f_1^2}{4c_1c_2}+\frac{c_2f_2^2}{4c_1}
				+\frac{c_1f_3^2}{4}+\frac{c_1f_4^2}{4}, 
			\end{aligned}
		\end{equation}
		where $c_1$ and $c_2$ are positive constants.
		
		For (\ref{t13_temp1}), we use Young's inequality with $\{c_1,c_2\}=\{\frac{1}{2},1\},$ and 
		for (\ref{t13_temp2}), we use Young's inequality with $\{c_1,c_2\}=\{1,\frac12\}.$
		With the help of (\ref{u0_ans3}), 
		there holds
		\begin{equation*}
			\begin{aligned}
				T_{1,4}+T_{1,5}\leq& 
				\frac{\|n_{(0,0)}\|_{L^{\infty}L^1}}{A^{1-2\epsilon}}
				\Big(\frac{\|\partial_y^2\partial_zc_{(0,\neq)}\|^{2}_{L^2L^2}}{2}
				+\frac{\|\partial_y\partial_zc_{(0,\neq)}\|^{2}_{L^2L^2}}{2}
				+\frac{\|\partial_z\nabla u_{2,(0,\neq)}\|^{2}_{L^2L^2}}{8}\\&
				+\frac{\|\partial_y\partial_z^2c_{(0,\neq)}\|^{2}_{L^2L^2}}{2}
				+\frac{\|\partial_z^2c_{(0,\neq)}\|^{2}_{L^2L^2}}{8}
				+\frac{\|\partial_z\nabla u_{3,(0,\neq)}\|^{2}_{L^2L^2}}{4}\Big)\\
				\leq& 
				\frac{\|n_{(0,0)}\|_{L^{\infty}L^1}}{A^{1-2\epsilon}}
				\Big(\frac{\|\partial_y^2\partial_zc_{(0,\neq)}\|^{2}_{L^2L^2}}
				{2}
				+\frac{\|\partial_y\partial_z^2c_{(0,\neq)}\|^{2}_{L^2L^2}}{2}
				+\frac{5\|\partial_y\partial_zc_{(0,\neq)}\|^{2}_{L^2L^2}}{8}
				\\&
				+\frac{3 \|\partial_z^2c_{(0,\neq)}\|^{2}_{L^2L^2}}{8}
				+\frac{\|c_{(0,\neq)}\|^{2}_{L^2L^2}}{4}\Big)+C\Big(1+\frac{E_1^3+E_2^3}{A^{\frac{\epsilon}{2}}}\Big).
			\end{aligned}
		\end{equation*}
		Due to 
		$$\|\partial_y\partial_z c_{(0,\neq)}\|_{L^2}^2
		\leq \|\partial_y\partial_z^2 c_{(0,\neq)}\|_{L^2}^2,~~~
		\|\partial_z c_{(0,\neq)}\|_{L^2}^2
		\leq \|\partial_z^3 c_{(0,\neq)}\|_{L^2}^2,$$
		when $A>A_1,$ we have
		\begin{equation}\label{t14_1}
			\begin{aligned}
				T_{1,4}+T_{1,5}
				&\leq C+
				\frac{\|n_{(0,0)}\|_{L^{\infty}L^1}}{A^{1-2\epsilon}}
				\Big(\frac{\|\partial_z(\partial_y^2,\partial_z^2) c_{(0,\neq)}\|^{2}_{L^2L^2}}{4} 
				+\frac{3\|\partial_z\nabla c_{(0,\neq)}\|^{2}_{L^2L^2}}{8}\\
				&\quad +\frac{3\|\partial_y\partial_z^2 c_{(0,\neq)}\|^{2}_{L^2L^2}}{4}
				+\frac{\|\partial_y^2\partial_z c_{(0,\neq)}
					\|_{L^2L^2}^2}{4}\Big)\\
				&=C+
				\frac{\|n_{(0,0)}\|_{L^{\infty}L^1}}{A^{1-2\epsilon}}
				\Big(
				\frac{3(\|\partial_z\triangle c_{(0,\neq)}\|^{2}_{L^2L^2}	
					+\|\partial_z\nabla c_{(0,\neq)}\|^{2}_{L^2L^2})}{8}\\
				&\quad-\frac{\|\partial_z(\partial_y^2,\partial_z^2) c_{(0,\neq)}\|^{2}_{L^2L^2}}{8}
				+\frac{\|\partial_y^2\partial_z c_{(0,\neq)}\|_{L^2L^2}^2}{4} 
				\Big).
			\end{aligned}
		\end{equation}
		
		{\bf Estimates of $T_{1,6}$ and $T_{1,7}$.}
		Using Lemma \ref{lemma_u} and 
		\begin{equation}\label{n_infty}
			\|n_{\neq}\|_{L^{\infty}L^{\infty}}\leq 
			\|n\|_{L^{\infty}L^{\infty}}+\|n_0\|_{L^{\infty}L^{\infty}}
			\leq
			C\|n\|_{L^{\infty}L^{\infty}}\leq CE_3,
		\end{equation}
		there holds
		\begin{equation}\label{n0_sub1}
			\begin{aligned}
				&\quad A^{\epsilon-\frac{1}{2}}\|(u_{\neq} n_{\neq})_{(0,\neq)}\|_{L^2L^2}
				\leq CA^{\epsilon-\frac{1}{2}}\|u_{\neq} n_{\neq}\|_{L^2L^2}
				\leq CA^{\epsilon-\frac{1}{2}}
				E_3
				\|u_{\neq}\|_{L^2L^2}\\
				&\leq CA^{\epsilon-\frac{1}{3}}E_3(\| \triangle u_{2,\neq}\|_{X_a}
				+\| \partial_x \omega_{2,\neq}\|_{X_a})
				\leq 
				\frac{CE_2E_3}{ A^{\frac13-\frac{\epsilon}{4}}}.
			\end{aligned}
		\end{equation}
		Due to $n_{\neq}=-\triangle c_{\neq}+c_{\neq},$ then for $j=1,2,3,$ there holds
		\begin{equation*}
			\partial_x\partial_j\triangle^{-1}n_{\neq}
			=-\partial_x\partial_j c_{\neq}
			+\partial_x\partial_j\triangle^{-1}c_{\neq}.
		\end{equation*}
		The energy estimates show that 
		\begin{equation*}
			\begin{aligned}
				\|\partial_x\partial_j\triangle^{-1}n_{\neq}\|_{L^2}^2
				&=\|\partial_x\partial_jc_{\neq}\|_{L^2}^2
				-2<\partial_x\partial_j c_{\neq}
				,\partial_x\partial_j\triangle^{-1}c_{\neq}>
				+\|\partial_x\partial_j\triangle^{-1}c_{\neq}\|_{L^2}^2\\
				&=\|\partial_x\partial_jc_{\neq}\|_{L^2}^2
				+2<\partial_x c_{\neq}
				,\partial_x\partial_j^2\triangle^{-1}c_{\neq}>
				+\|\partial_x\partial_j\triangle^{-1}c_{\neq}\|_{L^2}^2,		
			\end{aligned}
		\end{equation*}
		which along with $\partial_x(\partial_x^2+\partial_y^2+\partial_z^2)
		\triangle^{-1}c_{\neq}=\partial_x c_{\neq}$ imply that 
		\begin{equation}\label{n_neq_11}
			\begin{aligned}
				\|\partial_x\nabla\triangle^{-1}n_{\neq}\|_{L^2}^2
				&=\|\partial_x\nabla c_{\neq}\|_{L^2}^2
				+2<\partial_x c_{\neq}
				,\partial_x(\partial_x^2+\partial_y^2+\partial_z^2)
				\triangle^{-1}c_{\neq}>
				+\|\partial_x\nabla\triangle^{-1}c_{\neq}\|_{L^2}^2\\
				&=\|\partial_x\nabla c_{\neq}\|_{L^2}^2
				+2\|\partial_x c_{\neq}\|_{L^2}^2
				+\|\partial_x\nabla\triangle^{-1}c_{\neq}\|_{L^2}^2.		
			\end{aligned}
		\end{equation}
		By  (\ref{n_infty}) and (\ref{n_neq_11}), we get 
		\begin{equation}\label{n0_sub2}
			\begin{aligned}
				&\quad A^{\epsilon-\frac{1}{2}}\|(n_{\neq}\nabla c_{\neq})_{(0,\neq)}\|_{L^2L^2}
				\leq CA^{\epsilon-\frac{1}{2}}E_3
				\|\nabla c_{\neq}\|_{L^2L^2}\leq 
				\frac{CE_3\| n_{\neq}\|_{X_a}}
				{ A^{\frac12-\epsilon}}\leq\frac{CE_2E_3}{ A^{\frac12-\epsilon}},
			\end{aligned}
		\end{equation}
		where we use $\|\nabla c_{\neq}\|_{L^2L^2}\leq 
		\|\partial_x\nabla\triangle^{-1} n_{\neq}\|_{L^2L^2}
		\leq \|n_{\neq}\|_{X_a}.$
		
		Therefore, according to \eqref{n0_sub1} and \eqref{n0_sub1}, we obtain 
		\begin{equation*}
			T_{1,6}
			=\frac{\int_{0}^t<\nabla\cdot\partial_z(u_{\neq}n_{\neq})_{(0,\neq)},
				\partial_zc_{(0,\neq)}>dt}{A^{1-2\epsilon}}
			\leq \frac{\|(u_{\neq} n_{\neq})_{(0,\neq)}\|_{L^2L^2}
				\|\partial_z^2\nabla c_{(0,\neq)}\|_{L^2L^2}}{A^{1-2\epsilon}}\leq 
			\frac{CE_1E_2E_3}{ A^{\frac13-\frac{\epsilon}{4}}},
		\end{equation*}
		and 
		\begin{equation*}
			T_{1,7}
			=\frac{\int_{0}^t<\nabla\cdot\partial_z(n_{\neq}\nabla c_{\neq})_{(0,\neq)},
				\partial_zc_{(0,\neq)}>dt}{A^{1-2\epsilon}}
			\leq \frac{\|(n_{\neq}\nabla c_{\neq} )_{(0,\neq)}\|_{L^2L^2}
				\|\partial_z^2\nabla c_{(0,\neq)}\|_{L^2L^2}}{A^{1-2\epsilon}}\leq 
			\frac{CE_1E_2E_3}{ A^{\frac12-\epsilon}}.
		\end{equation*}

		{\bf  Estimates of $T_{1,8}$ and $T_{1,9}$.}
		Using Lemma \ref{sob_inf_3} and Lemma \ref{lem:ellip_4},
		we have
		\begin{equation*}
			\begin{aligned}
				&\|\nabla c_{(0,\neq)}\|_{L^{\infty}_yL^2_z}
				\leq \|\partial_y\nabla c_{(0,\neq)}\|_{L^2}^{\frac12}
				\|\nabla c_{(0,\neq)}\|_{L^2}^{\frac12}
				\leq C\|n_{(0,\neq)}\|_{L^2}^{\frac12}
				\|\nabla c_{(0,\neq)}\|_{L^2}^{\frac12},\\
				&\|n_{(0,\neq)}\|_{L^{\infty}_zL^2_y}\leq 
				\|\partial_zn_{(0,\neq)}\|_{L^2},
			\end{aligned}
		\end{equation*}
		which imply that
		\begin{align*}
			&\quad <\partial_z\nabla\cdot(n_{(0,\neq)}\nabla c_{(0,\neq)})_{(0,\neq)},\partial_z
			c_{(0,\neq)}>\leq C
			\|n_{(0,\neq)}\nabla c_{(0,\neq)}\|_{L^2}
			\|\partial_z^2\nabla c_{(0,\neq)}\|_{L^2}\\
			&\leq C\|n_{(0,\neq)}\|_{L^2}^{\frac12}
			\|\nabla c_{(0,\neq)}\|_{L^2}^{\frac12}
			\|\partial_zn_{(0,\neq)}\|_{L^2}
			\|\partial_z^2\nabla c_{(0,\neq)}\|_{L^2}.
		\end{align*}
		Therefore
		\begin{equation*}
			T_{1,8}+T_{1,9}\leq 
			\frac{C\|n_{(0,\neq)}\|_{L^{\infty}L^2}^{\frac12}
				\|\nabla c_{(0,\neq)}\|_{L^{\infty}L^2}^{\frac12}
				\|\partial_zn_{(0,\neq)}\|_{L^2L^2}
				\|\partial_z^2\nabla c_{(0,\neq)}\|_{L^2L^2}}{A^{1-2\epsilon}}
			\leq \frac{CE_1^3}{ A^{\frac{1}{2}\epsilon}}.
		\end{equation*}
		
		{\bf Estimates of $T_{1,10}$, $T_{1,11}$ and $T_{1,12}$.}
		We just need to estimate  $<\partial_z\partial_j (u_{j,0}n_{(0,\neq)}),
		\partial_zc_{(0,\neq)}>,$ where $j=2,3.$
		By (\ref{u23_infty}), we get
		\begin{equation*}
			\begin{aligned}
				<\partial_z\partial_j (u_{j,0}n_{(0,\neq)}),
				\partial_zc_{(0,\neq)}>
				&\leq
				C\|u_{j,0}n_{(0,\neq)}\|_{L^2}
				\|\partial_z\triangle c_{(0,\neq)}\|_{L^2}\\
				&\leq C(\|u_{2,0}\|_{H^2}+\|u_{3,0}\|_{H^1})
				\|n_{(0,\neq)}\|_{L^2}
				\|\partial_z\triangle c_{(0,\neq)}\|_{L^2},	
			\end{aligned}
		\end{equation*}
		which along with $\|f_{(0,0)}\|_{L^2}+
		\|f_{(0,\neq)}\|_{L^2}\leq C\|f_{0}\|_{L^2}$ imply that 
		\begin{equation*}
			\begin{aligned}
				&\quad T_{1,10}+T_{1,11}+T_{1,12}\\
				&\leq
				\frac{C(\|u_{2,0}\|_{L^{\infty}H^2}+\|u_{3,0}\|_{L^{\infty}H^1})
					\|n_{(0,\neq)}\|_{L^2L^2}
					\|\partial_z\triangle c_{(0,\neq)}\|_{L^2L^2}}{A^{1-2\epsilon}}
				\leq \frac{CE_1^3}{ A^{\epsilon}}.		
			\end{aligned}
		\end{equation*}

		{\bf  Close the energy estimates.}
		By above calculations, we infer from (\ref{n0_neq_1}) and (\ref{n0_neq_2}) that 
		\begin{equation}\label{n0_neq_3}
			\begin{aligned}
				&\frac{\|\partial_z\nabla c_{(0,\neq)}\|^2_{L^{\infty}H^1}}{2A^{-2\epsilon}}
				+\frac{\|\partial_z\triangle c_{(0,\neq)}\|^2_{L^2L^2}+
					\|\partial_z\nabla c_{(0,\neq)}\|^2_{L^{2}L^2}}{A^{1-2\epsilon}}
				\\
				\leq& A^{2\epsilon}\|(n_{\rm in})_{(0,\neq)}\|^2_{L^2}+C\left(\frac{E_1^3}{ A^{\frac{1}{2}\epsilon}}
				+\frac{E_1^2+E_2^3+E_3^3}{ A^{\frac13-\frac{\epsilon}{4}}}
				+\frac{E_1^3+E_2^3+E_3^3}{ A^{\frac12-\epsilon}}\right)\\&+\|n_{(0,0)}\|_{L^\infty L^1}
				\left[
				\frac{7\left(\|\partial_z\triangle c_{(0,\neq)}\|_{L^2L^2}^2
					+\|\partial_z\nabla c_{(0,\neq)}\|_{L^2L^2}^2\right)
				}{8A^{1-2\epsilon}}-\frac{\|\partial_{z}(\partial_{y}^{2},\partial_{z}^{2})c_{(0,\neq)}\|_{L^{2}L^{2}}^{2}}{8A^{1-2\epsilon}}\right]
				\\\leq&A^{2\epsilon}\|(n_{\rm in})_{(0,\neq)}\|^2_{L^2}+C\left(\frac{E_1^3}{ A^{\frac{1}{2}\epsilon}}
				+\frac{E_1^2+E_2^3+E_3^3}{ A^{\frac13-\frac{\epsilon}{4}}}
				+\frac{E_1^3+E_2^3+E_3^3}{ A^{\frac12-\epsilon}}\right)\\&+\|n_{(0,0)}\|_{L^{\infty}L^{1}}\frac{5\left(\|\partial_z\triangle c_{(0,\neq)}\|_{L^2L^2}^2
					+\|\partial_z\nabla c_{(0,\neq)}\|_{L^2L^2}^2\right)}{6A^{1-2\epsilon}},
			\end{aligned}
		\end{equation}
		where we used the fact of $ \|\partial_{z}\triangle c_{(0,\neq)}\|_{L^{2}L^{2}}^{2}+\|\partial_{z}\nabla c_{(0,\neq)}\|_{L^{2}L^{2}}^{2}\leq 3\|\partial_{z}(\partial_{y}^{2}, \partial_{z}^{2})c_{(0,\neq)}\|_{L^{2}L^{2}}^{2}. $
		The energy estimates can be closed by imposing the condition
		$\|n_{(0,0)}\|_{L^\infty L^1}<\frac65,$
		which along with $\|n_{(0,0)}\|_{L^{\infty}L^1}=\frac{M}{4\pi^2},$
		imply that 
		\begin{equation*}
			M<\frac{24\pi^2}{5}.
		\end{equation*}
		
		In conclusion,
		when $$A\geq\max\{A_1,(E_1^3+E_2^3+E_3^1)^{\frac{\epsilon}{2}},
		(E_1^3+E_2^3+E_3^1)^{\frac{12}{4-3\epsilon}},(E_1^3+E_2^3+E_3^1)^{\frac{2}{1-2\epsilon}}\}:=A_2,$$ as long as $M< \frac{24\pi^2}{5},$
		there holds 
		\begin{equation*}
			\frac{A^{2\epsilon}\|\partial_z\nabla c_{(0,\neq)}\|^2_{L^{\infty}L^2}}{2}+\left(1-\frac56\|n_{(0,0)}\|_{L^{\infty}L^{1}} \right)
			\frac{\|\partial_z\triangle c_{(0,\neq)}\|_{L^2L^2}^2
				+\|\partial_z\nabla c_{(0,\neq)}\|_{L^2L^2}^2}{A^{1-2\epsilon}}
			\leq C.
		\end{equation*}
		We finish the proof with the help of 
		$\|\partial_zn_{(0,\neq)}\|_{L^{2}}^{2}\leq C\|\partial_z\triangle c_{(0,\neq)}\|_{L^{2}}^{2}$.
	\end{proof}
	
	\begin{lemma}\label{lemma_n003}
		Under conditions of Theorem \ref{result}, Lemma \ref{lemma_n002} and the assumptions (\ref{assumption}), 	there exists a constant $A_3$ independent of $t$ and $A$, such that if $A>A_3$, then
		$$\|\partial_z^2n_{(0,\neq)}\|^2_{L^{\infty}L^2}+
		\frac{1}{A}\|\partial_z^2\nabla n_{(0,\neq)}\|^2_{L^2L^2}\leq C(\|(\partial_z^2n_{\rm in})_{(0,\neq)}\|_{L^2}^2
		+1).$$
	\end{lemma}
	\begin{proof}
		Taking $\partial_z^2$ for (\ref{n_0_neq}), we have
		\begin{equation}\label{n_0_neq11}
			\begin{aligned}
				&\partial_t \partial_z^2n_{(0,\neq)}
				-\frac{1}{A}\partial_z^2\triangle n_{(0,\neq)}
				=-\frac{1}{A}\left[\nabla\cdot\partial_z^2(n_{\neq}\nabla c_{\neq})_{(0,\neq)}
				+\nabla\cdot\partial_z^2(u_{\neq}n_{\neq})_{(0,\neq)}\right]\\
				&-\frac{1}{A}\left[\partial_y\partial_z^2\big(n_{(0,0)}\partial_{y}c_{(0,\neq)}
				+n_{(0,\neq)}\partial_{y}c_{(0,0)}
				+(n_{(0,\neq)}\partial_{y}c_{(0,\neq)})_{(0,\neq)}\big)\right]\\
				&-\frac{1}{A}\left[\partial_z^3\big(n_{(0,0)}\partial_{z}c_{(0,\neq)}  
				+n_{(0,\neq)}\partial_{z}c_{(0,\neq)}\big)
				+\partial_y\partial_z^2\big(u_{2,(0,\neq)}n_{(0,0)}
				+(u_{2,(0,\neq)}n_{(0,\neq)})_{(0,\neq)}\big)\right]\\
				&-\frac{1}{A}\left[\partial_z^3\big(u_{3,(0,0)}n_{(0,\neq)}
				+u_{3,(0,\neq)}n_{(0,0)}
				+(u_{3,(0,\neq)}n_{(0,\neq)})_{(0,\neq)}\big)\right].
			\end{aligned}
		\end{equation}
		After integrating in time, we infer from (\ref{n_0_neq11}) that 
		\begin{equation}\label{n_0_neq12}
			\|\partial_z^2n_{(0,\neq)}\|^2_{L^{\infty}L^2}+
			\frac{1}{A}\|\partial_z^2\nabla n_{(0,\neq)}\|^2_{L^2L^2}= \frac{\|(\partial_z^2n_{\rm in})_{(0,\neq)}\|_{L^2}^2}{2}
			+T_{2,1}+T_{2,2}+...+T_{2,12}.
		\end{equation}
		
		{\bf Estimate of $T_{2,1}$.} 
		Using Lemma \ref{sob_inf_2}, Lemma \ref{lemma_non_zz0} 
		and Lemma \ref{lem:ellip_2}, we have
		\begin{equation*}
			\begin{aligned}
				&\|\partial_z^2(n_{\neq}\nabla c_{\neq})_{(0,\neq)}\|_{L^2L^2}\\
				\leq 
				&\|(\partial_z^2n_{\neq}\nabla c_{\neq})_{0}\|_{L^2L^2}+
				2\|(\partial_zn_{\neq}\nabla \partial_zc_{\neq})_{0}\|_{L^2L^2}+
				\|(n_{\neq}\nabla \partial_z^2c_{\neq})_{0}\|_{L^2L^2}\\
				\leq&\|\partial_z^2n_{\neq}\|_{L^{\infty}L^2}
				\|\nabla c_{\neq}\|_{L^2L^{\infty}}
				+2\|\partial_zn_{\neq}\|_{L^{\infty}_{t,z}L^2_{x,y}}
				\|\nabla\partial_zc_{\neq}\|_{L^2_{t,x,z}L^{\infty}_{y}}
				+\|n_{\neq}\|_{L^{\infty}_{t,z}L^2_{x,y}}
				\|\nabla\partial_z^2c_{\neq}\|_{L^2_{t,x,z}L^{\infty}_{y}}\\
				\leq &CA^{\frac16}(\|\partial_x^2n_{\neq}\|_{X_a}^2
				+\|\partial_z^2n_{\neq}\|_{X_a}^2)\leq CA^{\frac16}E_2^2.
			\end{aligned}
		\end{equation*}
		Thus
		\begin{equation*}
			T_{2,1}=\frac{\int_0^t<\nabla\cdot\partial_z^2(n_{\neq}\nabla c_{\neq})_{(0,\neq)},\partial_z^2 n_{(0,\neq)}>dt}{A}\leq C\frac{E_1^3+E_2^3}{A^{\frac13}}.
		\end{equation*}
		
		{\bf Estimate of $T_{2,2}$.} 
		For $j=2~{\rm or}~3$, we  have 
		\begin{equation*}
			\begin{aligned}
				&\|\partial_z^2(u_{j,\neq} n_{\neq})_{(0,\neq)}\|_{L^2L^2}\\
				\leq 
				&\|(\partial_z^2u_{j,\neq} n_{\neq})_{0}\|_{L^2L^2}+
				2\|(\partial_zu_{j,\neq} \partial_zn_{\neq})_{0}\|_{L^2L^2}+	\|(u_{j,\neq} \partial_z^2n_{\neq})_{0}\|_{L^2L^2}\\
				\leq
				&\|\partial_z^2u_{j,\neq}\|_{L^{\infty}L^2}
				\|n_{\neq}\|_{L^2L^{\infty}}+
				2\|\partial_zu_{j,\neq}\|_{L^{\infty}L^2}
				\|\partial_zn_{\neq}\|_{L^2_{t,x}L^{\infty}_{y,z}}
				+\|\partial_z^2n_{\neq}\|_{L^{\infty}L^2}
				\|u_{j,\neq}\|_{L^2L^{\infty}}\\
				\leq &CA^{\frac{5}{12}}(\|\partial_x^2n_{\neq}\|_{X_a}^2
				+\|\partial_z^2n_{\neq}\|_{X_a}^2+\|\triangle u_{2,\neq}\|^2_{X_a}
				+\| \partial_x\omega_{2,\neq}\|^2_{X_a})
				\leq CA^{\frac{5}{12}}E_2^2,
			\end{aligned}
		\end{equation*}
		which implies that 
		\begin{equation*}
			T_{2,2}
			=\frac{\int_0^t<\nabla\cdot
				\partial_z^2(u_{\neq}n_{\neq})_{(0,\neq)},
				\partial_z^2 n_{(0,\neq)}>dt}{A}\leq C\frac{E_1^3+E_2^3}{A^{\frac{1}{12}}}.
		\end{equation*}
		
		{\bf Estimates of $T_{2,3}$ and $T_{2,6}$.} 
		Using Lemma \ref{lemma_non_zz}, Lemma \ref{lem:ellip_4}
		and Lemma \ref{lemma_n001}, 
		we have 
		\begin{equation*}
			\begin{aligned}
				&\quad <\nabla\cdot\partial_z^2(n_{(0,0)}\nabla c_{(0,\neq)})
				,\partial_z^2 n_{(0,\neq)}>
				\leq C\|n_{(0,0)}\partial_z^2\nabla c_{(0,\neq)}\|_{L^2}
				\|\partial_z^2 \nabla n_{(0,\neq)}\|_{L^2}\\
				&\leq C
				\|\partial_z^2\partial_y\nabla c_{(0,\neq)}\|_{L^2}^{\frac12}
				\|\partial_z^2\nabla c_{(0,\neq)}\|_{L^2}^{\frac12}
				\|\partial_z^2 \nabla n_{(0,\neq)}\|_{L^2}\leq C
				\|\partial_z n_{(0,\neq)}\|_{L^2}^{\frac12}
				\|\partial_z^2 \nabla n_{(0,\neq)}\|_{L^2}^{\frac32},
			\end{aligned}
		\end{equation*}
		which implies that 
		\begin{equation*}
			T_{2,3}+T_{2,6}
			=\frac{\int_0^t<\nabla\cdot\partial_z^2(n_{(0,0)}
				\nabla c_{(0,\neq)}),\partial_z^2 n_{(0,\neq)}>dt}{A}
			\leq \frac{CE_1^2}{A^{\frac{\epsilon}{2}}}.
		\end{equation*}
		
		{\bf Estimate of $T_{2,4}$.} 
		Using Lemma \ref{lem:ellip_3}, Lemma \ref{lemma_n001}
		and \begin{equation*}
			\|\partial_z^2n_{(0,\neq)}\|_{L^2}^2
			\leq \|\partial_zn_{(0,\neq)}\|_{L^2}
			\|\partial_z^3n_{(0,\neq)}\|_{L^2},
		\end{equation*}
		we have
		\begin{equation*}
			\begin{aligned}
				&\quad 	<\partial_y\partial_z^2(n_{(0,\neq)}\partial_{y}c_{(0,0)}),
				\partial_z^2n_{(0,\neq)}>
				\leq 
				\|\partial_{y}c_{(0,0)}\partial_z^2n_{(0,\neq)}\|_{L^2}
				\|\partial_z^2\nabla n_{(0,\neq)}\|_{L^2}\\
				&\leq C\|n_{(0,0)}\|_{L^2}\|\partial_z^2n_{(0,\neq)}\|_{L^2}
				\|\partial_z^2\nabla n_{(0,\neq)}\|_{L^2}
				\leq C\|\partial_zn_{(0,\neq)}\|_{L^2}^{\frac12}
				\|\partial_z^2\nabla n_{(0,\neq)}\|_{L^2}^{\frac32},
			\end{aligned}
		\end{equation*}
		which implies that 
		\begin{equation*}
			T_{2,4}=
			\frac{\int_0^t<\partial_y\partial_z^2(n_{(0,\neq)}
				\partial_{y}c_{(0,0)}),\partial_z^2n_{(0,\neq)}>dt}{A}
			\leq \frac{CE_1^2}{A^{\frac{\epsilon}{2}}}.
		\end{equation*}
		
		{\bf Estimates of $T_{2,5}$ and $T_{2,7}$.} 
		Using Lemma \ref{sob_inf_1}, Lemma \ref{lemma_non_zz0}, Lemma \ref{lem:ellip_4}
		and Lemma \ref{lemma_n002}, there holds
		\begin{equation*}
			\begin{aligned}
				&\|\partial_z^2(n_{(0,\neq)}\nabla c_{(0,\neq)})\|_{L^2L^2}\\
				\leq&\|\partial_z^2n_{(0,\neq)}\nabla c_{(0,\neq)}\|_{L^2L^2}
				+2\|\partial_zn_{(0,\neq)}\nabla 	\partial_zc_{(0,\neq)}\|_{L^2L^2}
				+\|n_{(0,\neq)}\nabla \partial_z^2c_{(0,\neq)}\|_{L^2L^2}\\
				\leq&\|\partial_z^2n_{(0,\neq)}\|_{L^{\infty}L^2}\|\nabla 	c_{(0,\neq)}\|_{L^2L^{\infty}}
				+2\|\partial_zn_{(0,\neq)}\|_{L^{\infty}_{t,z}L^2_{y}}
				\|\nabla \partial_zc_{(0,\neq)}\|_{L^2_{t,z}L^{\infty}_y}\\
				&+\|\nabla \partial_z^2c_{(0,\neq)}\|_{L^{\infty}_{t,y}L^2_z}
				\|n_{(0,\neq)}\|_{L^{2}_{t,y}L^{\infty}_z}\\
				\leq &C\|\partial_z^2n_{(0,\neq)}\|_{L^{\infty}L^2}
				\|\partial_z n_{(0,\neq)}\|_{L^2L^{2}},
			\end{aligned}
		\end{equation*}
		which implies that 
		\begin{equation*}
			T_{2,5}+T_{2,7}
			=\frac{\int_0^t<\nabla\cdot\partial_z^2(n_{(0,\neq)}
				\nabla c_{(0,\neq)})_{(0,\neq)},\partial_z^2 n_{(0,\neq)}>dt}{A}
			\leq \frac{CE_1^3}{A^{\epsilon}}.
		\end{equation*}
		
		{\bf Estimates of $T_{2,8}$ and $T_{2,11}$.} 
		It follows from (\ref{eq:u2u30}) that 
		\begin{equation*}
			\left\{
			\begin{array}{lr}
				\partial_t\partial_z^2u_{2,0}-\frac{1}{A}\triangle \partial_z^2 u_{2,0}
				+\frac{1}{A}\partial_z^2(u\cdot\nabla u_{2})_0
				+\frac{1}{A}\partial_y\partial_z^2P^{N_1}_0
				+\frac{1}{A}\partial_y\partial_z^2 P^{N_2}_0
				+\frac{1}{A}\partial_y\partial_z^2 P^{N_3}_0=\frac{\partial_z^2n_0}{A}, \\
				\partial_t\partial_z^2u_{3,0}-\frac{1}{A}\triangle \partial_z^2u_{3,0}
				+\frac{1}{A}\partial_z^2(u\cdot\nabla u_3)_0
				+\frac{1}{A}\partial_z^3P^{N_1}_0
				+\frac{1}{A}\partial_z^3P^{N_2}_0
				+\frac{1}{A}\partial_z^3P^{N_3}_0=0.
			\end{array}
			\right.
		\end{equation*}	
		Using (\ref{u23_zero}), (\ref{u23_zero_1}) and Lemma \ref{lemma_n002}, the energy estimates show that
		\begin{equation}\label{nzz_temp1}
			\begin{aligned}
				&\quad
				\frac{\|\partial_z^2\nabla (u_{2,0},u_{3,0})\|_{L^2L^2}^2}{A}
				\leq C\Big(\frac{1+E_1^3+E_2^3}{A^{3\epsilon}}
				+\frac{\|\partial_z n_{(0,\neq)}\|_{L^2L^2}^2}{2A}\Big)
				\leq C\frac{E_1^3+E_2^3+1}{A^{3\epsilon}}.
			\end{aligned}
		\end{equation}
		For $j=2,3,$ using Lemma \ref{sob_inf_3},
		we have 
		\begin{equation*}
			\begin{aligned}
				<\partial_j\partial_z^2(u_{j,(0,\neq)}n_{(0,0)}),
				\partial_z^2n_{(0,\neq)}>
				\leq C\|n_{(0,0)}\|_{L^2}
				\|\partial_z^2\nabla u_{j,(0,\neq)}\|_{L^2}
				\|\partial_z^2\nabla n_{(0,\neq)}\|_{L^2},
			\end{aligned}
		\end{equation*}
		which along with (\ref{nzz_temp1}) give that 
		\begin{equation*}
			\begin{aligned}
				T_{2,8}+T_{2,11}=\frac{\int_{0}^t\sum_{j=2}^3<\partial_j\partial_z^2(u_{j,(0,\neq)}n_{(0,0)}),
					\partial_z^2n_{(0,\neq)}>dt}{A}\leq C\frac{E_1^3+E_2^3+1}{A^{\epsilon}}.
			\end{aligned}
		\end{equation*}
		
		{\bf Estimate of $T_{2,10}$.} 
		A direct calculation shows that 
		\begin{equation*}
			T_{2,10}=
			\frac{\int_0^t<u_{3,(0,0)},\partial_z^2n_{(0,\neq)}
				\partial_z^3n_{(0,\neq)}>dt}{A}=0.
		\end{equation*}
		
		{\bf Estimates of $T_{2,9}$ and $T_{2,12}$.} 
		Thanks to ${\rm div}~u_{(0,\neq)}=0,$ for $j=2,3$, we have 
		\begin{equation*}
			\begin{aligned}
				\|\partial_z^2(u_{j,(0,\neq)}n_{(0,\neq)})\|_{L^2L^2}&\leq
				C(\|u_{2,(0,\neq)}\|_{L^{\infty}H^2}+\|\partial_zu_{3,(0,\neq)}\|_{L^{\infty}H^1})
				\|\partial_z\nabla n_{(0,\neq)}\|_{L^2L^{2}}\\
				&\leq CA^{-\epsilon}E_1\|\nabla \partial_z n_{(0,\neq)}\|_{L^2L^{2}}
				\leq CA^{-\epsilon}E_1\|\nabla \partial_z^2 n_{(0,\neq)}\|_{L^2L^{2}},
			\end{aligned}
		\end{equation*}
		which implies   that 
		\begin{equation*}
			\begin{aligned}	
				T_{2,8}+T_{2,11}=&\frac{\int_{0}^t\sum_{j=2}^3
					<\partial_j\partial_z^2(u_{j,(0,\neq)}n_{(0,\neq)})_{\neq},
					\partial_z^2n_{(0,\neq)}>dt}{A}\\
				&\leq \frac{CE_1\|\nabla \partial_z^2 n_{(0,\neq)}\|_{L^2L^{2}}^2}{A^{1+\epsilon}}
				\leq \frac{CE_1^3}{A^{\epsilon}}.
			\end{aligned}
		\end{equation*}
		
		When
		$A\geq{\rm max}\{A_2,
		C(1+E_1^3+E_2^3)^{\frac{1}{\epsilon}}\}:=A_3,$ 
		we infer from (\ref{n_0_neq12}) that
		\begin{equation*}
			\begin{aligned}
				\|\partial_{z}^2n_{(0,\neq)}\|^2_{L^{\infty}L^2}+
				\frac{1}{A}\|\nabla \partial_{z}^2n_{(0,\neq)}\|_{L^2L^2}^2\leq C(\|(\partial_z^2n_{\rm in})_{(0,\neq)}\|_{L^2}^2
				+1).
			\end{aligned}
		\end{equation*}
	\end{proof}
	
	\subsection{Energy estimates for $E_{1,2}(t)$ and $E_{1,3}(t)$}
	We recall that 
	\begin{equation}\label{u_zero_1}
		\left\{
		\begin{array}{lr}
			\partial_tu_{2,0}-\frac{1}{A}\triangle u_{2,0}
			+\frac{1}{A}(u\cdot\nabla u_{2})_0
			+\frac{1}{A}\partial_yP^{N_1}_0
			+\frac{1}{A}\partial_y P^{N_2}_0+\frac{1}{A}\partial_y P^{N_3}_0=\frac{n_0}{A}, \\
			\partial_tu_{3,0}-\frac{1}{A}\triangle u_{3,0}
			+\frac{1}{A}(u\cdot\nabla u_3)_0
			+\frac{1}{A}\partial_zP^{N_1}_0
			+\frac{1}{A}\partial_z P^{N_2}_0
			+\frac{1}{A}\partial_z P^{N_3}_0=0, \\
			\partial_yu_{2,0}+\partial_zu_{3,0}=0.
		\end{array}
		\right.
	\end{equation}	
	\begin{lemma}\label{lemma_u23_1}
		Under the conditions of Theorem \ref{result} and the assumptions (\ref{assumption}),
		there exists a constant $A_4$ independent of $t$ and $A$, such that if $A>A_4$, 
		there holds 
		\begin{equation}\label{eq:u20u30}
			\begin{aligned}
				&A^{2\epsilon}(\|u_{2,0}\|^2_{Y_0}+\|u_{3,0}\|^2_{Y_0})\leq C,\\
				&A^{2\epsilon}(\|\nabla u_{2,0}\|^2_{Y_0}+\|\nabla u_{3,0}\|^2_{Y_0})\leq C,\\
				&A^{2\epsilon}\|\triangle u_{2,0}\|^2_{Y_0}\leq C,\\
				&A^{2\epsilon}\|\min\{h_A(t)^{\frac{1}{2}},1\}\triangle u_{3,0}\|^2_{Y_0}\leq C,
			\end{aligned}
		\end{equation}
		where $h_A(t)=A^{-\frac{2}{3}}+A^{-1}t.$
	\end{lemma}
	\begin{proof}
		Recall that 
		\begin{equation}\label{eq:u2u300}
			\left\{
			\begin{array}{lr}
				\partial_tu_{2,0}-\frac{1}{A}\triangle u_{2,0}
				+\frac{1}{A}(u\cdot\nabla u_{2})_0
				+\frac{1}{A}\partial_yP^{N_1}_0
				+\frac{1}{A}\partial_y P^{N_2}_0+\frac{1}{A}\partial_y P^{N_3}_0=\frac{n_0}{A}, \\
				\partial_tu_{3,0}-\frac{1}{A}\triangle u_{3,0}
				+\frac{1}{A}(u\cdot\nabla u_3)_0
				+\frac{1}{A}\partial_zP^{N_1}_0
				+\frac{1}{A}\partial_z P^{N_2}_0
				+\frac{1}{A}\partial_z P^{N_3}_0=0.
			\end{array}
			\right.
		\end{equation}		
		Due to 
		$u_{2,(0,0)}=0,$
		we have 
		$$<n_0,u_{2,0}>=<n_{(0,\neq)}+n_{(0,0)},u_{2,(0,\neq)}>=<n_{(0,\neq)},u_{2,(0,\neq)}>\leq
		\|n_{(0,\neq)}\|_{L^2}\|\partial_z u_{2,(0,\neq)}\|_{L^2},$$
		and combining it with $<u_0\cdot\nabla u_{2,0},u_{2,0}>+<u_0\cdot\nabla u_{3,0},u_{3,0}>=0$ and Lemma \ref{lemma_neq1}, 
		energy estimates of \eqref{eq:u2u300} show that 
		\begin{equation}\label{u0_ans11}
			\begin{aligned}
				&\frac{A^{2\epsilon}}{2}(\|u_{2,0}\|_{L^{\infty}L^2}^2+\|u_{3,0}\|_{L^{\infty}L^2}^2)
				+A^{2\epsilon-1}(\|\partial_y u_{2,0}\|_{L^2L^2}^2
				+\frac{1}{2}\|\partial_z u_{2,0}\|_{L^2L^2}^2
				+\|\nabla u_{3,0}\|_{L^2L^2}^2)\\
				&\leq C+A^{2\epsilon-1}\||u_{\neq}|^2\|_{L^2L^2}
				(\|\nabla u_{2,0}\|_{L^2L^2}+\|\nabla u_{3,0}\|_{L^2L^2})
				+\frac{A^{2\epsilon-1}}{2}\|\partial_z n_{(0,\neq)}\|_{L^2L^2}^2\\
				&\leq C+\frac{CE_1E_2^2}{A^{\frac{1}{6}+\frac{\epsilon}{2}}}
				+\frac{A^{2\epsilon-1}}{2}\|\partial_zn_{(0,\neq)}\|_{L^2L^2}^2,
			\end{aligned}
		\end{equation}
		
		It follows from (\ref{u0_ans1}) and Lemma \ref{lemma_n002} that 
		$$A^{2\epsilon}(\|u_{2,0}\|^2_{Y_0}+\|u_{3,0}\|^2_{Y_0})\leq C(1+A^{2\epsilon-1}\|n_{(0,\neq)}\|^2_{L^2L^2})\leq C.$$
		%
		
		{\bf Estimate of $\eqref{eq:u20u30}_2$.}
		Multiplying $\triangle u_{2,0}$ on $(\ref{u_zero_1})_1$ and 
		$\triangle u_{3,0}$ on $(\ref{u_zero_1})_2,$ energy estimates give that 
		\begin{equation}\label{u23_result00}
			\begin{aligned}
				&\quad A^{2\epsilon}(\|\nabla u_{2,0}\|^2_{L^{\infty}L^2}+\|\nabla u_{3,0}\|^2_{L^{\infty}L^2})+A^{2\epsilon-1}
				(\|\triangle u_{2,0}\|^2_{L^{2}L^2}+\|\triangle u_{3,0}\|^2_{L^{2}L^2})\\ 
				&\leq C+C A^{2\epsilon-1}
				\Big(\|u_0\cdot\nabla u_{2,0}\|^2_{L^2L^2}
				+\|(u_{\neq}\cdot\nabla u_{2,\neq})_0\|^2_{L^2L^2}
				+\|\partial_z n_{(0,\neq)}\|^2_{L^2L^2}\\
				&\quad+\|u_0\cdot\nabla u_{3,0}\|^2_{L^2L^2}
				+\|(u_{\neq}\cdot\nabla u_{3,\neq})_0\|^2_{L^2L^2}\Big),	
			\end{aligned}
		\end{equation}
		where we use $A^{\epsilon}\|(u_{\rm in})_{0}\|_{H^2}\leq C.$
		By (\ref{u23_infty}), we get
		\begin{equation}\label{temp_u23_1}
			\begin{aligned}
				A^{2\epsilon-1}\|u_0\cdot\nabla u_{3,0}\|^2_{L^2L^2}
				\leq CA^{2\epsilon-1}(\|u_{2,0}\|^2_{L^{\infty}H^2}
				+\|u_{3,0}\|^2_{L^{\infty}H^1})\|\nabla u_{3,0}\|^2_{L^2L^2}
				\leq CA^{-2\epsilon}E_1^4,\\
				A^{2\epsilon-1}\|u_0\cdot\nabla u_{2,0}\|^2_{L^2L^2}
				\leq CA^{2\epsilon-1}(\|u_{2,0}\|^2_{L^{\infty}H^2}
				+\|u_{3,0}\|^2_{L^{\infty}H^1})\|\nabla u_{2,0}\|^2_{L^2L^2}
				\leq CA^{-2\epsilon}E_1^4.
			\end{aligned}
		\end{equation}
		Using (\ref{temp_u23_1}), Lemma \ref{lemma_neq1} and Lemma \ref{lemma_n002},
		if $A\geq \max\{A_3, C(E_1^4+E_2^4)^{\frac{1}{\epsilon}}\}:=A_4$, we 
		obtain that 
		\begin{equation}\label{u23_result1}
			\begin{aligned}
				A^{2\epsilon}(\|\nabla u_{2,0}\|^2_{Y_0}+\|\nabla u_{3,0}\|^2_{Y_0})
				\leq& C\big(A^{-\epsilon}(E_1^4+E_2^4)+1\big)\leq C.
			\end{aligned}
		\end{equation}
		
		{\bf Estimate of $\eqref{eq:u20u30}_3$.}
		The $H^2$ energy estimate shows that 
		\begin{equation}\label{u23_result3}
			\begin{aligned}
				&\quad A^{2\epsilon}(\|\triangle u_{2,0}\|^2_{L^{\infty}L^2}+A^{-1}
				\|\nabla \triangle u_{2,0}\|^2_{L^{2}L^2})\\ 
				&\leq C+C A^{2\epsilon-1}\Big(\|\nabla(u_0\cdot\nabla u_{2,0})\|^2_{L^2L^2}+\|\nabla(u_{\neq}\cdot\nabla u_{2,\neq})\|^2_{L^2L^2}+\|\partial_z n_{(0,\neq)}\|^2_{L^2L^2}+\|\triangle P^{N_3}_0\|^2_{L^2L^2}\Big).	
			\end{aligned}
		\end{equation}
		Due to ${\rm div}~(u\cdot\nabla u)=\partial_x(u\cdot\nabla u_{1})
		+\partial_y(u\cdot\nabla u_{2})+\partial_z(u\cdot\nabla u_{3}),$
		thus 
		\begin{equation}\label{u23_result4}
			\begin{aligned}
				&\|{\rm div}~(u\cdot\nabla u)_0\|_{L^2}^2\\\leq& 
				\|\partial_y(u\cdot\nabla u_{2})_0\|_{L^2}^2+\|\partial_z(u\cdot\nabla u_{3})_0\|_{L^2}^2\\
				\leq& \|\partial_y(u_0\cdot\nabla u_{2,0})\|_{L^2}^2
				+\|\partial_z(u_0\cdot\nabla u_{3,0})\|_{L^2}^2
				+\|\partial_y(u_{\neq}\cdot\nabla u_{2,\neq})\|_{L^2}^2
				+\|\partial_z(u_{\neq}\cdot\nabla u_{3,\neq})\|_{L^2}^2.
			\end{aligned}
		\end{equation}
		Using (\ref{u23_zero}), (\ref{u23_result4}), Lemma \ref{lemma_neq1} and Lemma \ref{lemma_n002},
		there are  
		\begin{equation}\label{u23_result41}
			A^{2\epsilon-1}\|\triangle P^{N_3}_0\|^2_{L^2L^2}=A^{2\epsilon-1}\|{\rm div}~(u\cdot\nabla u)_0\|_{L^2L^2}^2\leq  CA^{-\epsilon}(E_1^4+E_2^4)\leq C,
		\end{equation}
		and 
		\begin{equation*}\label{u23_result5}
			\begin{aligned}
				A^{2\epsilon}\|\triangle u_{2,0}\|^2_{Y_0}
				\leq&C+C\big(A^{-\epsilon}(E_1^4+E_2^4)+1\big)\leq C.
			\end{aligned}
		\end{equation*}
		
		{\bf Estimate of $\eqref{eq:u20u30}_4$.}
		Taking $H^2$ energy estimate on $(\ref{u_zero_1})_2$, then we obtain 
		\begin{equation*}
			\begin{aligned}
				&\quad \partial_t\|\triangle u_{3,0}\|^2_{L^2}+A^{-1}
				\|\nabla \triangle u_{3,0}\|^2_{L^2}\\ 
				&\leq \frac{C}{A}\Big(\|\nabla(u_0\cdot\nabla u_{3,0})\|^2_{L^2}+\|\nabla(u_{\neq}\cdot\nabla u_{3,\neq})\|^2_{L^2}+\|\partial_z n_{(0,\neq)}\|^2_{L^2}+\|\triangle P^{N_3}_0\|^2_{L^2}\Big), 
			\end{aligned}
		\end{equation*}
		which implies that 
		\begin{equation}\label{u23_result6}
			\begin{aligned}
				&\quad \partial_t(\min\{h_A(t),1\}\|\triangle u_{3,0}\|^2_{L^2})
				+\frac{\min\{h_A(t),1\}}{A}
				\|\nabla \triangle u_{3,0}\|^2_{L^2}\\ 
				&\leq \frac{C}{A}\Big(\|\nabla(u_0\cdot\nabla u_{3,0})\|^2_{L^2}+h_A(t)\|\nabla(u_{\neq}\cdot\nabla u_{3,\neq})\|^2_{L^2}+\|\partial_z n_{(0,\neq)}\|^2_{L^2}+\|\triangle P^{N_3}_0\|^2_{L^2}+\|\triangle u_{3,0}\|^2_{L^2}\Big),
			\end{aligned}
		\end{equation}
		where $h_A(t)=A^{-\frac{2}{3}}+A^{-1}t.$
		
		By (\ref{u23_zero}), (\ref{u23_result1}), (\ref{u23_result41}), Lemma \ref{lemma_neq1}, Lemma \ref{lemma_n002} and $\|h_A(t){\rm e}^{-aA^{-\frac13}t}\|_{L^{\infty}_t}\leq CA^{-\frac{2}{3}},$ we get  
		\begin{equation}\label{u23_result7}
			\begin{aligned}
				&A^{2\epsilon}\|\min\{h_A(t)^{\frac{1}{2}},1\}\triangle u_{3,0}\|^2_{Y_0}
				\leq C+\frac{C}{A^{1-2\epsilon}}\int_{0}^t h_A(s)\|\nabla(u_{\neq}\cdot\nabla u_{3,\neq})\|^2_{L^2}ds\\
				&+\frac{C}{A^{1-2\epsilon}}\Big(\|\nabla(u_0\cdot\nabla u_{3,0})\|^2_{L^2L^2}+\|\partial_z n_{(0,\neq)}\|^2_{L^2L^2}+\|\triangle P^{N_3}_0\|^2_{L^2L^2}+\|\triangle u_{3,0}\|^2_{L^2L^2}\Big)\\
				&\leq C+\frac{C}{A^{1-2\epsilon}}\|h_A(t){\rm e}^{-aA^{-\frac13}t}\|_{L^{\infty}_t}
				\|{\rm e}^{2aA^{-\frac{1}{3}}t}\nabla(u_{\neq}\cdot\nabla u_{3,\neq})\|^2_{L^2L^2}\leq C(1+\frac{E_2^4}{A^{\epsilon}})\leq C.
			\end{aligned}
		\end{equation}
		
		The proof is complete.
	\end{proof}

	Recall that 
	\begin{equation}\label{u_zero_2}
		\left\{
		\begin{array}{lr}
			\partial_t\widehat{u_{1,0}}-\frac{1}{A}\triangle\widehat{u_{1,0}}
			=-\frac{1}{A}(u_{2,0}\partial_y\widehat{u_{1,0}}+u_{3,0}\partial_z\widehat{u_{1,0}})-u_{2,0},\\
			\partial_t\widetilde{u_{1,0}}-\frac{1}{A}\triangle\widetilde{u_{1,0}}
			=-\frac{1}{A}(u_{2,0}\partial_y\widetilde{u_{1,0}}+u_{3,0}\partial_z\widetilde{u_{1,0}})-\frac{1}{A}(u_{\neq}\cdot \nabla u_{1,\neq})_0,
		\end{array}
		\right.
	\end{equation}
	and 
	\begin{equation*}
		\begin{aligned}
			\widehat{u_{1,0}}|_{t=0}=0,~~~ 
			\widetilde{u_{1,0}}|_{t=0}=(u_{1,\rm in})_0.
		\end{aligned}
	\end{equation*}

	\begin{lemma}\label{u1_hat1} 
		Under the conditions of Theorem \ref{result} and the assumptions (\ref{assumption}), if $A>A_4$, 
		there holds 
		$$A^{2\epsilon}(\|\widetilde{u_{1,0}}\|^2_{L^{\infty}H^2}+\frac{1}{A}\|\nabla\widetilde{u_{1,0}}\|^2_{L^{2}H^2})\leq C(1+A^{\frac{2}{3}}).$$
	\end{lemma}
	\begin{proof}
		
		Due to ${\rm div}~u_0=0$, we have 
		$<u_{2,0}\partial_y\widetilde{u_{1,0}}
		+u_{3,0}\partial_z\widetilde{u_{1,0}}, \widetilde{u_{1,0}}>=0.$
		Taking $L^2$ product with $(\ref{u_zero_2})_2$, then we get 
		\begin{equation*}
			\frac{1}{2}\partial_t\|\widetilde{u_{1,0}}\|_{L^2}^2+\frac{\|\nabla \widetilde{u_{1,0}}\|^2_{L^2}}{A}=
			-\frac{<u_{2,0}\partial_y\widetilde{u_{1,0}}
				+u_{3,0}\partial_z\widetilde{u_{1,0}}, \widetilde{u_{1,0}}>}{A}
			-\frac{<(u_{\neq}\cdot \nabla u_{1,\neq})_0, \widetilde{u_{1,0}}>}{A}.
		\end{equation*}
		Integrating about $t$ and using Lemma \ref{lemma_neq1}, thanks to ${\rm div}~u_{\neq}=0$, if $A>A_4,$ we have
		\begin{equation*}
			\begin{aligned}
				A^{2\epsilon}\|\widetilde{u_{1,0}}\|_{Y_0}^2
				\leq C\Big(1+\frac{\||u_{\neq}|^2\|_{L^2L^2}^2}{A}\Big)\leq C,
			\end{aligned}
		\end{equation*}
		where we use $A^{2\epsilon}\|\widetilde{(u_{1,\rm in})_0}\|_{L^2}^2
		=A^{2\epsilon}\|(u_{1,\rm in})_0\|_{L^2}^2\leq C.$
		
		The $H^2$ energy estimate for $(\ref{u_zero_2})_2$ yields that 
		\begin{equation}\label{u1_h2_1}
			A^{2\epsilon}\|\triangle\widetilde{u_{1,0}}\|_{Y_0}^2
			\leq C\Big(1+\frac{\|\nabla(u_{2,0}\partial_y\widetilde{u_{1,0}}
				+u_{3,0}\partial_z\widetilde{u_{1,0}})\|_{L^2L^2}^2}{A^{1-2\epsilon}}
			+\frac{\|\nabla(u_{\neq}\cdot \nabla u_{1,\neq})
				\|_{L^2L^2}^2}{A^{1-2\epsilon}}\Big).
		\end{equation}
		Using Lemma \ref{sob_12} and $\partial_zu_{3,0}=-\partial_yu_{2,0}$, 
		if $A>A_4,$
		there holds 
		\begin{equation*}
			\begin{aligned}
				A^{2\epsilon-1}\|\nabla(u_{2,0}\partial_y\widetilde{u_{1,0}}
				+u_{3,0}\partial_z\widetilde{u_{1,0}})\|_{L^2L^2}^2
				&\leq 
				CA^{2\epsilon-1}(\|u_{3,0}\|_{L^{\infty}H^1}^2
				+\|u_{2,0}\|^2_{L^{\infty}H^2})
				\|\nabla\widetilde{u_{1,0}}\|^2_{L^2H^{1}}\\
				&\leq CA^{\frac{2}{3}-2\epsilon}E_1^4 \leq CA^{\frac{2}{3}}.
			\end{aligned}
		\end{equation*}
		By Lemma \ref{lemma_neq1}, we have 
		\begin{equation*}
			\begin{aligned}
				A^{2\epsilon-1}\|\nabla(u_{\neq}\cdot \nabla u_{1,\neq})
				\|_{L^2L^2}^2
				&\leq 
				CA^{2\epsilon}(\|\partial_x\omega_{2,\neq}\|^2_{X_a}
				+\|\triangle u_{2,\neq}\|^2_{X_a})
				(\|\nabla\omega_{2,\neq}\|^2_{X_a}
				+\|\triangle u_{2,\neq}\|^2_{X_a})\\
				&\leq CA^{\frac{2}{3}-\epsilon}E_2^4\leq CA^{\frac{2}{3}}.
			\end{aligned}
		\end{equation*}
		Therefore, it follows from (\ref{u1_h2_1}) that
		\begin{equation*}
			\begin{aligned}
				A^{2\epsilon}\|\triangle\widetilde{u_{1,0}}\|_{Y_0}^2\leq 	C(1+A^{\frac{2}{3}}).
			\end{aligned}
		\end{equation*}
		Due to $\|\nabla \widetilde{u_{1,0}}\|_{L^2}^2\leq 
		\| \widetilde{u_{1,0}}\|_{L^2}
		\|\triangle \widetilde{u_{1,0}}\|_{L^2}\leq 
		\frac{1}{2}(\| \widetilde{u_{1,0}}\|^2_{L^2}+
		\|\triangle \widetilde{u_{1,0}}\|^2_{L^2}),$ we complete this proof.
	\end{proof}

	\begin{lemma}\label{u1_hat2}
		Under the conditions of Theorem \ref{result} and the assumptions (\ref{assumption}), if $A>A_4$, 
		there holds 
		\begin{equation}
			\begin{aligned}
				A^{2\epsilon}
				\Big(\frac{\|\widehat{u_{1,0}}\|^2_{L^{\infty}H^4}}{A^{2}}
				+\frac{\|\nabla\widehat{u_{1,0}}\|^2_{L^{2}H^4}}{A^{3}}
				+\|\partial_t\widehat{u_{1,0}}\|^2_{L^{\infty}H^2}\Big)
				\leq C.  
			\end{aligned}
		\end{equation}
	\end{lemma}
	\begin{proof}\  
		
		\noindent{\bf Estimates of $\|\widehat{u_{1,0}}\|^2_{L^{\infty}H^4}$ and $\|\nabla\widehat{u_{1,0}}\|^2_{L^{2}H^4}$.}
		Taking $L^2$ product with $(\ref{u_zero_2})_1$, we get
		\begin{equation*}
			\frac{1}{2}\partial_t\|\widehat{u_{1,0}}\|_{L^2}^2+\frac{1}{A}\|\nabla \widehat{u_{1,0}}\|^2_{L^2}=-
			\frac{1}{A}<u_{2,0}\partial_y\widehat{u_{1,0}}+u_{3,0}\partial_z\widehat{u_{1,0}}, \widehat{u_{1,0}}>-<u_{2,0}, \widehat{u_{1,0}}>.
		\end{equation*}
		Thanks to $u_{2,(0,0)}=0,$ there holds
		\begin{equation*}
			\begin{aligned}
				&<u_{2,0}, \widehat{u_{1,0}}>
				=<u_{2,(0,\neq)}, \widehat{u_{1,(0,\neq)}}>,
			\end{aligned}
		\end{equation*}
		where $\widehat{u_{1,(0,\neq)}}$ denotes the non-zero mode of $z$ direction of $\widehat{u_{1,0}}$,
		and using H\"{o}lder's inequality, we have
		\begin{equation*}
			<u_{2,0}, \widehat{u_{1,0}}>
			\leq \|u_{2,(0,\neq)}\|_{L^2}\|\widehat{u_{1,(0,\neq)}}\|_{L^2}
			\leq \|\partial_zu_{2,(0,\neq)}\|_{L^2}\|\partial_z\widehat{u_{1,0}}\|_{L^2}.
		\end{equation*}
		Furthermore, ${\rm div}~u_0=0$ implies that 
		$<u_{2,0}\partial_y\widehat{u_{1,0}}+u_{3,0}\partial_z\widehat{u_{1,0}}, \widehat{u_{1,0}}>=0.$
		Thus, by Lemma \ref{lemma_u23_1}, we get 
		\begin{equation}\label{u_1_result0}
			\frac{\|\widehat{u_{1,0}}\|_{L^{\infty}L^2}^2}{A^{2-2\epsilon}}
			+\frac{\|\nabla \widehat{u_{1,0}}\|^2_{L^2L^2}}{A^{3-2\epsilon}}
			\leq\frac{\|(\widehat{u_{1,\rm in}})_0\|_{L^2}^2 }{A^{2-2\epsilon}}
			+\frac{\|\nabla u_{2,0}\|_{L^2L^2}^2}{A^{1-2\epsilon}}\leq C.
		\end{equation}
		The $H^4$ energy estimate shows that 
		\begin{equation}\label{u_1_temp}
			\frac{\|\triangle^2\widehat{u_{1,0}}\|^2_{L^{\infty}L^2}}{A^{2-2\epsilon}}+\frac{\|\nabla\triangle^2\widehat{u_{1,0}}\|^2_{L^{2}L^2}}{A^{3-2\epsilon}}
			\leq 
			\frac{C\|\nabla\triangle(u_{2,0}\partial_y\widehat{u_{1,0}}
				+u_{3,0}\partial_z\widehat{u_{1,0}})\|^2_{L^{2}L^2}}{A^{3-2\epsilon}}
			+\frac{C\|\nabla\triangle u_{2,0}\|^2_{L^{2}L^2}}{A^{1-2\epsilon}}.	
		\end{equation}
		Thanks to Lemma \ref{sob_12}, if $A>A_4,$ we have 
		\begin{equation}
			\begin{aligned}
				\frac{\|\nabla\triangle(u_{2,0}\partial_y\widehat{u_{1,0}})\|^2_{L^{2}L^2}}{A^{3-2\epsilon}}
				\leq C
				\frac{\|\nabla u_{2,0}\|^2_{L^{2}H^2}\|\widehat{u_{1,0}}\|^2_{L^{\infty}H^4}+\|u_{2,0}\|^2_{L^{\infty}H^2}\|\partial_y\widehat{u_{1,0}}\|^2_{L^{2}H^3}}{A^{3-2\epsilon}}
				\leq \frac{CE_1^4}{A^{2\epsilon}}\leq C,
			\end{aligned}
		\end{equation}
		\begin{equation}
			\begin{aligned}
				\frac{\|u_{3,0}\nabla\triangle\partial_z\widehat{u_{1,0}}+\nabla u_{3,0}\triangle\partial_z\widehat{u_{1,0}}\|^2_{L^{2}L^2}}{A^{3-2\epsilon}}
				\leq C
				\frac{\|u_{3,0}\|^2_{L^{\infty}H^1}
					\|\partial_z\widehat{u_{1,0}}\|^2_{L^{2}H^4}}{A^{3-2\epsilon}}
				\leq \frac{CE_1^4}{A^{2\epsilon}}\leq C,
			\end{aligned}
		\end{equation}
		and
		\begin{equation}
			\begin{aligned}
				\frac{\|\triangle u_{3,0}
					\nabla\partial_z\widehat{u_{1,0}}\|^2_{L^{2}L^2}}{A^{3-2\epsilon}}
				\leq C
				\frac{\|\triangle u_{3,0}\|^2_{L^{2}L^2}
					\|\widehat{u_{1,0}}\|^2_{L^{\infty}H^4}}{A^{3-2\epsilon}}
				\leq \frac{CE_1^4}{A^{2\epsilon}}\leq C.
			\end{aligned}
		\end{equation}
		By $\|\widehat{u_{1,0}}\|^2_{H^3}\leq \|\widehat{u_{1,0}}\|_{H^2}
		\|\widehat{u_{1,0}}\|_{H^4},$ there holds
		$$\frac{\|\widehat{u_{1,0}}\|^2_{H^3}}{A^{-1}t}\leq \frac{\|\widehat{u_{1,0}}\|_{H^2}\|\widehat{u_{1,0}}\|_{H^4}}{A^{-1}t}
		\leq A\|\partial_t\widehat{u_{1,0}}\|_{L^{\infty}H^2}
		\|\widehat{u_{1,0}}\|_{H^4},$$
		where we use
		$$\|\widehat{u_{1,0}}\|_{H^2}\leq \int_0^t \|\partial_s\widehat{u_{1,0}}(s)\|_{H^2}ds\leq t\|\partial_t\widehat{u_{1,0}}\|_{L^{\infty}H^2}.$$
		Combining it with  $\|\nabla\triangle u_{3,0}\partial_z\widehat{u_{1,0}}\|^2_{L^2}\leq 
		C\|\nabla\triangle u_{3,0}\|^2_{L^2}\|\partial_z\widehat{u_{1,0}}\|^2_{H^2},$
		we get
		\begin{equation*}
			\begin{aligned}
				\|\nabla\triangle u_{3,0}\partial_z\widehat{u_{1,0}}\|^2_{L^2}
				&\leq C\|\min\{h_A(t)^{\frac{1}{2}},1\}\nabla\triangle u_{3,0}\|^2_{L^2}
				\Big(\frac{\|\widehat{u_{1,0}}\|^2_{H^3}}{A^{-1}t}
				+\|\widehat{u_{1,0}}\|^2_{H^4}\Big)\\
				&\leq C\|\min\{h_A(t)^{\frac{1}{2}},1\}\triangle u_{3,0}\|^2_{L^2}
				\Big(A^2\|\partial_t\widehat{u_{1,0}}\|^2_{L^{\infty}H^2}
				+\|\widehat{u_{1,0}}\|^2_{H^4}\Big),
			\end{aligned}
		\end{equation*}
		which implies that 
		\begin{equation}\label{u_0_temp0}
			\begin{aligned}			
				\frac{\|\nabla\triangle u_{3,0}\partial_z\widehat{u_{1,0}}\|^2_{L^{2}L^2}}{A^{3-2\epsilon}}
				&\leq CA^{2\epsilon}\frac{\|\min\{h_A(t)^{\frac{1}{2}},1\}
					\nabla\triangle u_{3,0}\|^2_{L^{2}L^2}}{A}
				\Big(\|\partial_t\widehat{u_{1,0}}\|^2_{L^{\infty}H^2}
				+\frac{\|\widehat{u_{1,0}}\|^2_{L^{\infty}H^4}}{A^2}\Big)\\
				&\leq C\frac{E_1^4}{A^{2\epsilon}}\leq C.
			\end{aligned}
		\end{equation}
		Using (\ref{u_1_temp})-(\ref{u_0_temp0}) and Lemma \ref{lemma_u23_1}, we have
		$$\frac{\|\triangle^2\widehat{u_{1,0}}\|^2_{L^{\infty}L^2}}{A^{2-2\epsilon}}+\frac{\|\nabla\triangle^2\widehat{u_{1,0}}\|^2_{L^{2}L^2}}{A^{3-2\epsilon}}
		\leq C.$$
		By $\|\widehat{u_{1,0}}\|^2_{H^4}
		\leq C(\|\widehat{u_{1,0}}\|^2_{L^2}+\|\triangle^2\widehat{u_{1,0}}\|^2_{L^2}),$
		we obtain that 
		\begin{equation}\label{u_1_result2}
			\frac{\|\widehat{u_{1,0}}\|^2_{L^{\infty}H^4}}{A^{2-2\epsilon}}
			+\frac{\|\nabla\widehat{u_{1,0}}\|^2_{L^{2}H^4}}{A^{3-2\epsilon}}
			\leq C.
		\end{equation}
		
		\noindent{\bf Estimate of $\|\partial_t\widehat{u_{1,0}}\|^2_{L^{\infty}H^2}$.}
		Taking $\triangle$ for $(\ref{u_zero_2})_1$, there holds
		$$\partial_t\triangle\widehat{u_{1,0}}-\frac{1}{A}\triangle^2\widehat{u_{1,0}}
		=\frac{1}{A}\triangle(u_{2,0}\partial_y\widehat{u_{1,0}}
		+u_{3,0}\partial_z\widehat{u_{1,0}})+\triangle u_{2,0},$$
		which implies that 
		\begin{equation}\label{u_1_temp0}
			A^{2\epsilon}\|\partial_t\triangle\widehat{u_{1,0}}\|^2_{L^{\infty}L^2}
			\leq\frac{\|\triangle^2\widehat{u_{1,0}}\|^2_{L^{\infty}L^2}}{A^{2-2\epsilon}}
			+\frac{\|\triangle(u_{2,0}\partial_y\widehat{u_{1,0}}
				+u_{3,0}\partial_z\widehat{u_{1,0}})\|^2_{L^{\infty}L^2}}{A^{2-2\epsilon}}
			+A^{2\epsilon}\|\triangle u_{2,0}\|^2_{L^{\infty}L^2}.	
		\end{equation}
		Thanks to Lemma \ref{sob_12}, if $A>A_4,$ we have 
		\begin{equation}\label{u_1_temp1}
			\begin{aligned}
				\frac{\|\triangle(u_{2,0}\partial_y\widehat{u_{1,0}})\|^2_{L^{\infty}L^2}}{A^{2-2\epsilon}}
				\leq C\|u_{2,0}\|^2_{L^{\infty}H^2}
				\frac{\|\widehat{u_{1,0}}\|^2_{L^{\infty}H^4}}{A^{2-2\epsilon}}
				\leq C\frac{E_1^4}{A^{2\epsilon}}\leq C,
			\end{aligned}
		\end{equation}
		and 
		\begin{equation}\label{u_1_temp2}
			\begin{aligned}
				\frac{\|u_{3,0}\triangle\partial_z\widehat{u_{1,0}}+\nabla u_{3,0}\cdot \nabla\partial_z\widehat{u_{1,0}}\|^2_{L^{\infty}L^2}}{A^{2-2\epsilon}}
				\leq C\|u_{3,0}\|^2_{L^{\infty}H^1}
				\frac{\|\widehat{u_{1,0}}\|^2_{L^{\infty}H^4}}{A^{2-2\epsilon}}
				\leq \frac{CE_1^4}{A^{2\epsilon}}\leq C.
			\end{aligned}
		\end{equation}
		Combining 	$\frac{\|\widehat{u_{1,0}}\|^2_{H^3}}{A^{-1}t}
		\leq A\|\partial_t\widehat{u_{1,0}}\|_{L^{\infty}H^2}
		\|\widehat{u_{1,0}}\|_{H^4},$ with  $\|\triangle u_{3,0}\partial_z\widehat{u_{1,0}}\|^2_{L^2}\leq 
		C\|\triangle u_{3,0}\|^2_{L^2}\|\widehat{u_{1,0}}\|^2_{H^3},$
		we get
		\begin{equation*}
			\begin{aligned}
				\|\triangle u_{3,0}\partial_z\widehat{u_{1,0}}\|^2_{L^2}
				&\leq C\|\min\{h_A(t)^{\frac{1}{2}},1\}\triangle u_{3,0}\|^2_{L^2}
				\Big(\frac{\|\widehat{u_{1,0}}\|^2_{H^3}}{A^{-1}t}
				+\|\triangle\partial_z\widehat{u_{1,0}}\|^2_{L^2}\Big)\\
				&\leq C\|\min\{h_A(t)^{\frac{1}{2}},1\}\triangle u_{3,0}\|^2_{L^2}
				\big(A^2\|\partial_t\widehat{u_{1,0}}\|^2_{L^{\infty}H^2}
				+\|\widehat{u_{1,0}}\|^2_{H^4}\big),
			\end{aligned}
		\end{equation*}
		which implies that 
		\begin{equation}\label{u_1_temp3}
			\begin{aligned}			
				\frac{\|\triangle u_{3,0}\partial_z\widehat{u_{1,0}}\|^2_{L^{\infty}L^2}}{A^{2-2\epsilon}}
				&\leq CA^{2\epsilon}\|\min\{h_A(t)^{\frac{1}{2}},1\}
				\triangle u_{3,0}\|^2_{L^{\infty}L^2}
				\Big(\|\partial_t\widehat{u_{1,0}}\|^2_{L^{\infty}H^2}
				+\frac{\|\widehat{u_{1,0}}\|^2_{L^{\infty}H^4}}{A^2}\Big)
				\leq C.
			\end{aligned}
		\end{equation}
		Using (\ref{u_1_result2}), (\ref{u_1_temp0})-(\ref{u_1_temp3}) and Lemma \ref{lemma_u23_1}, one obtains
		\begin{equation}\label{u_1_result1}
			A^{2\epsilon}\|\partial_t\triangle\widehat{u_{1,0}}\|^2_{L^{\infty}L^2}
			\leq \frac{\|\triangle^2\widehat{u_{1,0}}\|^2_{L^{\infty}L^2}}
			{A^{2-2\epsilon}}
			+ C\leq C.	
		\end{equation}
		Due to $\partial_t\widehat{u_{1,0}}-\frac{1}{A}\triangle\widehat{u_{1,0}}
		=-\frac{1}{A}(u_{2,0}\partial_y\widehat{u_{1,0}}+u_{3,0}\partial_z\widehat{u_{1,0}})-u_{2,0},$
		using (\ref{u_1_result2}) and 
		\begin{equation*}
			\frac{\|u_{2,0}\partial_y\widehat{u_{1,0}}
				+u_{3,0}\partial_z\widehat{u_{1,0}}\|_{L^{\infty}L^2}}{A^{1-\epsilon}}
			\leq (\|u_{2,0}\|_{L^{\infty}L^{\infty}}
			+\|u_{3,0}\|_{L^{\infty}L^{\infty}})
			\frac{\|\nabla\widehat{u_{1,0}}\|_{L^{\infty}L^2}}{A^{1-\epsilon}}
			\leq CA^{-\epsilon}E_1^2,
		\end{equation*} 
		we have 
		\begin{equation*}
			\begin{aligned}
				A^{\epsilon}\|\partial_t\widehat{u_{1,0}}\|_{L^{\infty}L^2}
				\leq& A^{\epsilon}
				\Big(\frac{\|\triangle\widehat{u_{1,0}}\|_{L^{\infty}L^2}}{A}
				+\frac{\|u_{2,0}\partial_y\widehat{u_{1,0}}
					+u_{3,0}\partial_z\widehat{u_{1,0}}\|_{L^{\infty}L^2}}{A}
				+\|u_{2,0}\|_{L^{\infty}L^2}\Big)
				\leq C,
			\end{aligned}
		\end{equation*}
		which along with (\ref{u_1_result1}) imply that 
		\begin{equation*}
			\begin{aligned}
				A^{\epsilon}\|\partial_t\widehat{u_{1,0}}\|_{L^{\infty}H^2}
				\leq C.
			\end{aligned}
		\end{equation*}
		The proof is complete.
	\end{proof}
	
	\begin{corollary}
		Under the conditions of Theorem \ref{result} and the assumptions (\ref{assumption}),
		according to Lemma \ref{lemma_n001}, Lemma \ref{lemma_n002}, Lemma \ref{lemma_u23_1}, Lemma \ref{u1_hat1},  Lemma \ref{u1_hat2}, when $A\geq \max\{A_1,A_2,A_3,A_4\}:=C_{(1)},$
		there holds
		\begin{equation}
			E_1(t)\leq C(\|(n_{\rm in})_{(0,0)}\|^2_{L^2}+\|(\partial_z^2n_{\rm in})_{(0,\neq)}\|_{L^2}^2+1):=E_1.
		\end{equation}
	\end{corollary}

	\section{Estimates for the non-zero mode: Proof of Proposition \ref{pro1}}\label{sec_pro}
	\subsection{Energy estimates for $E_{2,1}(t)$}
	\begin{lemma}\label{result_0_1}
		Under the conditions of Theorem \ref{result} and the assumptions (\ref{assumption}),
		there exists a constant $A_5$ independent of $t$ and $A$, such that
		if $A\geq A_5,$ there holds 
		$$E_{2,1}(t)\leq C\big(\|(\partial_x^2n_{\rm in})_{\neq}\|_{L^2}
		+\|(\partial_z^2n_{\rm in})_{\neq}\|_{L^2}+1\big).$$
	\end{lemma}
	\begin{proof} 
		By \eqref{eq:fourier ine}, it is sufficient  to estimate 
		$\|\partial_x^2n_{\neq}\|^2_{X_a}$ and $\|\partial_z^2n_{\neq}\|^2_{X_a}.$
		
		{\bf Step I: Estimate of $\|\partial_x^2n_{\neq}\|^2_{X_a}$.} According to  
		$(\ref{ini2})_1$, the non-zero mode $\partial_x^2n_{\neq}$ satisfies 
		\begin{equation}\label{c_temp_11}
			\begin{aligned}
				\partial_t\partial_x^2n_{\neq}+(y+\frac{\widehat{u_{1,0}}}{A})\partial_x^3 n_{\neq}-\frac{1}{A}\triangle \partial_x^2n_{\neq}=-\frac{1}{A}\widetilde{u_{1,0}}\partial_x^3 n_{\neq}
				-\frac{1}{A}(u_{2,0}\partial_y\partial_x^2n_{\neq}+u_{3,0}\partial_z\partial_x^2n_{\neq})\\
				-\frac{1}{A}\nabla\cdot(\partial_x^2u_{\neq}n_0)
				-\frac{1}{A}\nabla\cdot\partial_x^2(u_{\neq}n_{\neq})_{\neq}
				-\frac{1}{A}\nabla\cdot\partial_x^2(n\nabla c)_{\neq}.
			\end{aligned}
		\end{equation}
		Applying Proposition \ref{timespace2} to (\ref{c_temp_11}), we have
		\begin{equation}\label{c_temp_12}
			\begin{aligned}
				&\quad\|\partial_x^2n_{\neq}\|^2_{X_b}\leq C\Big( \|(\partial_x^2n_{\rm in})_{\neq}\|^2_{L^2}
				+\frac{1}{A^{\frac{5}{3}}}
				\|{\rm e}^{bA^{-\frac{1}{3}}t}\widetilde{u_{1,0}}\partial_x^3 n_{\neq}\|_{L^2L^2}^2+
				\frac{1}{A^{\frac{5}{3}}}\|{\rm e}^{bA^{-\frac{1}{3}}t}u_{2,0}\partial_x^2\partial_yn_{\neq}\|_{L^2L^2}^2\\
				&+\frac{1}{A^{\frac{5}{3}}}\|{\rm e}^{bA^{-\frac{1}{3}}t}u_{3,0}\partial_x^2\partial_zn_{\neq}\|_{L^2L^2}^2
				+\frac{1}{A^{\frac{5}{3}}}\|{\rm e}^{bA^{-\frac{1}{3}}t}n_{0}\partial_x^3u_{1,\neq}\|_{L^2L^2}^2
				+\frac{1}{A}\|{\rm e}^{bA^{-\frac{1}{3}}t}n_{0}
				\partial_x^2(u_{2},u_{3})_{\neq}\|_{L^2L^2}^2\\
				&+\frac{1}{A}\|{\rm e}^{bA^{-\frac{1}{3}}t}
				\partial_x^2(u_{\neq}n_{\neq})_{\neq}\|^2_{L^2L^2}
				+\frac{1}{A}\|{\rm e}^{bA^{-\frac{1}{3}}t}
				\partial_x^2(n\nabla c)_{\neq}\|^2_{L^2L^2}\Big)\\
				&:= C\left( \|(\partial_x^2n_{\rm in})_{\neq}\|^2_{L^2} +T_{3,1}+\cdots+T_{3,7}\right).
			\end{aligned}
		\end{equation}
		
		\noindent{\bf Estimate of $T_{3,1}$.} According to Lemma \ref{u1_hat1}, there holds $A^{\epsilon}\|\widetilde{u_{1,0}}\|_{L^{\infty}H^2}\leq CA^{\frac13},$ as long as $A$ is big enough, then
		\begin{equation*}
			\begin{aligned}
				\|{\rm e}^{bA^{-\frac{1}{3}}t}\widetilde{u_{1,0}}\partial_x^3 n_{\neq}\|_{L^2L^2}^2
				\leq C\|\widetilde{u_{1,0}}\|_{L^{\infty}H^2}^2
				\|{\rm e}^{bA^{-\frac{1}{3}}t}\partial_x^3 n_{\neq}\|_{L^2L^2}^2\leq C{A^{\frac53-2\epsilon}}\|\partial_x^2 n_{\neq}\|_{X_b}^2 ,
			\end{aligned}
		\end{equation*}
		which implies that $\widetilde{u_{1,0}}$ can be regard as a perturbation.
		
		\noindent{\bf Estimates of $T_{3,2}$ and $T_{3,3}$.} Thanks to (\ref{u23_infty}) and Lemma \ref{lemma_u23_1}, for $j=2,3,$ we have
		\begin{equation*}
			\begin{aligned}
				{\|{\rm e}^{bA^{-\frac{1}{3}}t}u_{j,0}\partial_x^2\nabla n_{\neq}\|_{L^2L^2}^2}
				\leq 	{\|u_{j,0}\|_{L^{\infty}L^{\infty}}\|{\rm e}^{bA^{-\frac{1}{3}}t}\partial_x^2\nabla n_{\neq}\|_{L^2L^2}^2}
				\leq CA\|\partial_x^2n_{\neq}\|^2_{X_b}.
			\end{aligned}
		\end{equation*}
		
		\noindent{\bf Estimates of $T_{3,4}$ and $T_{3,5}$.} Using $\|n_0\|_{L^{\infty}L^{\infty}}\leq \|n\|_{L^{\infty}L^{\infty}}\leq 2E_3$ and $\partial_xu_{1,\neq}+\partial_yu_{2,\neq}+\partial_zu_{3,\neq}=0$, there are
		\begin{equation*}
			\begin{aligned}
				{\|{\rm e}^{bA^{-\frac{1}{3}}t}n_{0}\partial_x^3u_{1,\neq}\|_{L^2L^2}^2}
				\leq {CE_3^2\|{\rm e}^{bA^{-\frac{1}{3}}t}\partial_x^3u_{1,\neq}\|_{L^2L^2}^2}
				\leq {CAE_3^2\|\partial_x^2(u_{2,\neq},u_{3,\neq})\|^2_{X_b}},
			\end{aligned}
		\end{equation*}
		and	\begin{equation*}
			\begin{aligned}
				\|{\rm e}^{bA^{-\frac{1}{3}}t}n_{0}
				\partial_x^2(u_{2},u_{3})_{\neq}\|_{L^2L^2}^2
				\leq CE_3^2\|{\rm e}^{bA^{-\frac{1}{3}}t}
				\partial_x^2(u_{2},u_{3})_{\neq}\|_{L^2L^2}^2
				\leq CA^{\frac13}E_3^2\|\partial_x^2(u_2,u_3)_{\neq}\|^2_{X_b}.
			\end{aligned}
		\end{equation*}

		\noindent{\bf Estimate of $T_{3,6}$.} By Lemma \ref{sob_inf_2},  there holds
		\begin{equation*}
			\begin{aligned}
				&\quad\|\partial_x^2(u_{\neq}n_{\neq})\|^2_{L^2}
				\leq C(\|\partial_x^2u_{\neq}\|^2_{L^2}
				\|n_{\neq}\|^2_{L^{\infty}}
				+\|\partial_x^2n_{\neq}\|^2_{L^2}\|u_{\neq}\|^2_{L^{\infty}}
				+\|\partial_xu_{\neq}\|^2_{L^{\infty}_xL^2_{y,z}}
				\|\partial_xn_{\neq}\|^2_{L^{\infty}_{y,z}L^2_{x}})\\
				&\leq C\big(\|\partial_x^2u_{\neq}\|^2_{L^2}
				\|(\partial_z,1)\partial_xn_{\neq}\|_{L^2}
				\|(\partial_z,1)\partial_xn_{\neq}\|_{H^1}
				+\|\partial_x^2n_{\neq}\|^2_{L^2}
				\|(\partial_z,1)\partial_xu_{\neq}\|_{L^2}
				\|(\partial_z,1)\partial_xu_{\neq}\|_{H^1}\big),
			\end{aligned}
		\end{equation*}
		which along with velocity estimates in Lemma \ref{lemma_u} imply that 
		\begin{equation}\label{un_1}
			\begin{aligned}
				{\|{\rm e}^{bA^{-\frac{1}{3}}t}
					\partial_x^2(u_{\neq}n_{\neq})_{\neq}\|^2_{L^2L^2}}
				\leq {CA^{\frac23}(\|\triangle u_{2,\neq}\|^2_{X_a}+\|\partial_x \omega_{2,\neq}\|^2_{X_a})\|(\partial_x^2,\partial_z^2)n_{\neq}\|^2_{X_a}}\leq {CA^{\frac23-\frac{3}{2}\epsilon}E_2^4}.
			\end{aligned}
		\end{equation}
		
		\noindent{\bf Estimate of $T_{3,7}$.} According to nonlinear interaction, we have
		\begin{equation}\label{nc_1}
			\begin{aligned}
				&\quad\|{\rm e}^{bA^{-\frac{1}{3}}t}\partial_x^2(n\nabla c)_{\neq}\|^2_{L^2L^2} \\
				&\leq C\Big(
				\|{\rm e}^{bA^{-\frac{1}{3}}t}n_0\partial_x^2\nabla c_{\neq}\|^2_{L^2L^2}
				+\|{\rm e}^{bA^{-\frac{1}{3}}t}\partial_x^2n_{\neq}\nabla c_{0}\|^2_{L^2L^2}
				+\|{\rm e}^{bA^{-\frac{1}{3}}t}\partial_x^2(n_{\neq} \nabla c_{\neq})_{\neq}\|^2_{L^2L^2}\Big).
			\end{aligned}
		\end{equation}
		Using Lemma \ref{lem:ellip_0} and Lemma \ref{lemma_n002}, it follows that
		\begin{equation}\label{nc_2}
			\begin{aligned}	
				&\quad\|{\rm e}^{bA^{-\frac{1}{3}}t}\partial_x^2n_{\neq}\nabla c_{0}\|^2_{L^2L^2}
				\leq
				\|\nabla c_{0}\|^2_{L^{\infty}L^{4}}
				\|{\rm e}^{bA^{-\frac{1}{3}}t}\partial_x^2n_{\neq}\|^2_{L^2L^4} \\
				&\leq C
				\|\nabla c_{0}\|_{L^{\infty}L^{4}}^2
				\|{\rm e}^{bA^{-\frac{1}{3}}t}\partial_x^2n_{\neq}\|^{\frac12}_{L^2L^2}
				\|{\rm e}^{bA^{-\frac{1}{3}}t}\partial_x^2\nabla n_{\neq}\|^{\frac32}_{L^2L^2}
				\leq CA^{\frac{5}{12}}
				\|\partial_x^2n_{\neq}\|^2_{X_b}.
			\end{aligned}
		\end{equation}
		Using Lemma \ref{lem:ellip_2} and $\|n_0\|_{L^{\infty}L^{\infty}}\leq \|n\|_{L^{\infty}L^{\infty}}\leq 2E_3$, there holds 
		\begin{equation}\label{nc_3}
			\begin{aligned}
				\|{\rm e}^{bA^{-\frac{1}{3}}t}n_0\partial_x^2\nabla c_{\neq}\|_{L^2L^2}^2
				&\leq ||n_0||_{L^{\infty}L^{\infty}}^2
				\|{\rm e}^{bA^{-\frac{1}{3}}t}\partial_x^2\nabla c_{\neq}\|^2_{L^2L^2}
				\leq CA^{\frac{1}{3}}E_3^2||\partial_x^2n_{\neq}||^2_{X_b}.
			\end{aligned}
		\end{equation}
		Similar to the proof for (\ref{un_1}), using Lemma \ref{lem:ellip_2}, we can prove that 
		\begin{equation}\label{nc_4}
			\begin{aligned}
				\|{\rm e}^{bA^{-\frac{1}{3}}t}\partial_x^2(n_{\neq} \nabla c_{\neq})_{\neq}\|_{L^2L^2}^2
				&\leq CA^{\frac{2}{3}}\|(\partial_x^2,\partial_z^2)n_{\neq}\|^4_{X_a}
				\leq CA^{\frac{2}{3}}E_2^4.
			\end{aligned}
		\end{equation}
		Combining (\ref{nc_1})-(\ref{nc_4}) give
		\begin{equation}\label{nc_end}
			\begin{aligned}
				\|{\rm e}^{aA^{-\frac{1}{3}}t}\partial_x^2(n\nabla c)_{\neq}\|^2_{L^2L^2}
				\leq CA^{\frac{2}{3}}(||\partial_x^2n_{\neq}||^2_{X_b}
				+E_2^4
				+E_3^2||\partial_x^2n_{\neq}||^2_{X_b}).
			\end{aligned}
		\end{equation}
		Using above, we infer from (\ref{c_temp_12}) that 
		\begin{equation}\label{n_xx_result}
			\begin{aligned}
				&\quad\|\partial_x^2n_{\neq}\|^2_{X_b}\leq C\Big( \|(\partial_x^2n_{\rm in})_{\neq}\|^2_{L^2}
				+\frac{E_2^4+(1+E_3^2)\big(\|\partial_x^2n_{\neq}\|^2_{X_b}
					+\|\partial_x^2u_{2,\neq}\|^2_{X_b}
					+\|\partial_x^2u_{3,\neq}\|^2_{X_b}\big)}{A^{\frac{1}{3}}}\Big).
			\end{aligned}
		\end{equation}

		{\bf Step II: Estimate of $\|\partial_z^2n_{\neq}\|^2_{X_a}$.}
		The non-zero mode $\partial_z^2n_{\neq}$ satisfies 
		\begin{equation}\label{c_temp_21}
			\begin{aligned}
				\partial_t\partial_z^2n_{\neq}+\Big(y+\frac{\widehat{u_{1,0}}}{A}
				\Big)\partial_x\partial_z^2 n_{\neq}-\frac{\triangle \partial_z^2n_{\neq}}{A}=-\frac{\partial_z^2(\widetilde{u_{1,0}}\partial_x n_{\neq})}{A}
				-\frac{\partial_z^2\widehat{u_{1,0}}\partial_x n_{\neq}}{A}
				-\frac{2\partial_z\widehat{u_{1,0}}\partial_x\partial_z n_{\neq}}{A}
				\\-\frac{\partial_z^2\partial_y(u_{2,0}n_{\neq})}{A}
				-\frac{\partial_z^3(
					u_{3,0}n_{\neq})}{A}
				-\frac{\nabla\cdot\partial_z^2(u_{\neq}n_0)}{A}
				-\frac{\nabla\cdot\partial_z^2(u_{\neq}n_{\neq})_{\neq}}{A}
				-\frac{\nabla\cdot\partial_z^2(n\nabla c)_{\neq}}{A}.
			\end{aligned}
		\end{equation}
		Applying Proposition \ref{timespace2} to (\ref{c_temp_21}), we get
		\begin{equation}\label{c_temp_22}
			\begin{aligned}
				&\quad\|\partial_z^2n_{\neq}\|^2_{X_a}\leq C
				\Big( \|(\partial_z^2n_{\rm in})_{\neq}\|^2_{L^2}
				+\frac{	\|{\rm e}^{aA^{-\frac{1}{3}}t}\partial_z^2(\widetilde{u_{1,0}}\partial_x n_{\neq})\|_{L^2L^2}^2}{A^{\frac{5}{3}}}
				+\frac{\|{\rm e}^{aA^{-\frac{1}{3}}t}
					\partial_z\widehat{u_{1,0}}\partial_x\partial_z n_{\neq}\|_{L^2L^2}^2}{A^{\frac{5}{3}}}\\
				&+\frac{\|{\rm e}^{aA^{-\frac{1}{3}}t}
					\partial_z^2\widehat{u_{1,0}}\partial_x n_{\neq}\|_{L^2L^2}^2}{A^{\frac{5}{3}}}		
				+\frac{\|{\rm e}^{aA^{-\frac{1}{3}}t}\partial_z^2(u_{2,0}
					n_{\neq})\|_{L^2L^2}^2}{A}
				+\frac{\|{\rm e}^{aA^{-\frac{1}{3}}t}\partial_z^2(u_{3,0}
					n_{\neq})\|_{L^2L^2}^2}{A}\\
				&+\frac{\|{\rm e}^{aA^{-\frac{1}{3}}t}
					\partial_x\partial_z(u_{1,\neq}n_{0})\|_{L^2L^2}^2}{A}
				+\frac{\|{\rm e}^{aA^{-\frac{1}{3}}t}
					\partial_z^2(u_{2,\neq}n_{0})\|_{L^2L^2}^2}{A}
				+\frac{\|{\rm e}^{aA^{-\frac{1}{3}}t}
					\partial_z^2(u_{3,\neq}n_0)\|_{L^2L^2}^2}{A}\\
				&+\frac{\|{\rm e}^{aA^{-\frac{1}{3}}t}
					\partial_x\partial_z(u_{1,\neq}n_{\neq})_{\neq}\|^2_{L^2L^2}}{A}
				+\frac{\|{\rm e}^{aA^{-\frac{1}{3}}t}
					\partial_z^2(u_{2,\neq}n_{\neq})_{\neq}\|^2_{L^2L^2}}{A}
				+\frac{\|{\rm e}^{aA^{-\frac{1}{3}}t}
					\partial_z^2(u_{3,\neq}n_{\neq})_{\neq}\|^2_{L^2L^2}}{A}\\
				&+\frac{\|{\rm e}^{aA^{-\frac{1}{3}}t}
					\partial_z^2(n\nabla c)_{\neq}\|^2_{L^2L^2}}{A}\Big).
			\end{aligned}
		\end{equation}
		For convenience, we rewrite (\ref{c_temp_22}) as 
		$$\quad\|\partial_z^2n_{\neq}\|^2_{X_a}\leq C(\|(\partial_z^2n_{\rm in})_{\neq}\|^2_{L^2}+T_{4,1}+\cdots+T_{4,12}),$$
		where $T_{4,1}$ can be regarded as a perturbation, $T_{4,2}$ and $T_{4,3}$ are difficult terms. 
		
		\noindent\textbf{Estimate of $T_{4,1}$:}
		Due to $A^{\epsilon}\|\widetilde{u_{1,0}}\|_{L^{\infty}H^2}\leq CA^{\frac13}$ in Lemma \ref{u1_hat1},  using Lemma \ref{sob_inf_1} and Lemma \ref{sob_inf_2}, we get 
		\begin{equation*}
			\begin{aligned}
				&\|\partial_z^2\widetilde{u_{1,0}}\partial_x n_{\neq}\|_{L^2}^2\leq
				\|\partial_z^2\widetilde{u_{1,0}}\|_{L^2}^2
				\|\partial_x n_{\neq}\|_{L^{\infty}}^2
				\leq CA^{\frac23-2\epsilon}
				\|(\partial_x^2,\partial_z^2) n_{\neq}\|_{L^2}
				\|(\partial_x^2,\partial_z^2)\nabla n_{\neq}\|_{L^2},\\
				&\|\partial_z\widetilde{u_{1,0}}\partial_x\partial_z n_{\neq}\|_{L^2}^2\leq
				\|\partial_z\widetilde{u_{1,0}}\|_{L^{\infty}_zL^2_y}^2
				\|\partial_x\partial_z n_{\neq}\|_{L^{\infty}_yL^2_{x,z}}^2
				\leq CA^{\frac23-2\epsilon}
				\|(\partial_x^2,\partial_z^2) n_{\neq}\|_{L^2}
				\|(\partial_x^2,\partial_z^2)\nabla n_{\neq}\|_{L^2},\\	
				&\|\widetilde{u_{1,0}}\partial_x\partial_z^2 n_{\neq}\|_{L^2}^2
				\leq\|\widetilde{u_{1,0}}\|_{L^{\infty}}^2
				\|\partial_x\partial_z^2 n_{\neq}\|_{L^2}^2
				\leq CA^{\frac23-2\epsilon}
				\|\partial_z^2\nabla n_{\neq}\|^2_{L^2},	
			\end{aligned}
		\end{equation*}
		which imply that 
		\begin{equation}\label{temp1_1}
			\begin{aligned}
				\frac{\|{\rm e}^{aA^{-\frac{1}{3}}t}
					\partial_z^2(\widetilde{u_{1,0}}\partial_x n_{\neq})\|_{L^2L^2}^2}{A^{\frac{5}{3}}}
				\leq C\Big(\frac{\|\partial_z^2 n_{\neq}\|^2_{X_a}}{A^{2\epsilon}}
				+\frac{\|\partial_x^2 n_{\neq}\|^2_{X_a}
					+\|\partial_z^2 n_{\neq}\|^2_{X_a}}{A^{\frac13+2\epsilon}}\Big).
			\end{aligned}
		\end{equation}
		
		\noindent\textbf{Estimate of $T_{4,2}$:}
		Using Lemma \ref{u1_hat2},  it holds
		\begin{equation*}
			\begin{aligned}
				\|\widehat{u_{1,0}}\|_{H^2}\leq \int_0^t \|\partial_s\widehat{u_{1,0}}(s)\|_{H^2}ds
				\leq CA^{-\epsilon}t.
			\end{aligned}
		\end{equation*}
		A direct calculation shows that $\|A^{-\frac13}t{\rm e}^{-\frac{a}{4}A^{-\frac13}t}\|_{L^{\infty}_t}\leq C,$
		thus
		$\big\|{\rm e}^{-\frac{a}{4}A^{-\frac13}t}\|\widehat{u_{1,0}}\|_{H^2}\big\|_{L^{\infty}_t} \leq CA^{\frac13-\epsilon}.$
		Combining Lemma \ref{sob_inf_1} and Lemma \ref{sob_inf_2}, we get
		\begin{equation}\label{temp1_2}
			\begin{aligned}
				&\quad\|{\rm e}^{aA^{-\frac{1}{3}}t}
				\partial_z\widehat{u_{1,0}}\partial_x\partial_z 	n_{\neq}\|_{L^2L^2}^2
				\leq C\int_0^t{\rm 	e}^{-\frac{a}{2}A^{-\frac13}s}\|\widehat{u_{1,0}}\|^2_{H^2}
				{\rm e}^{\frac{5a}{2}A^{-\frac13}s}
				\|\partial_x\partial_z n_{\neq}\|_{L^2}
				\|\partial_x\partial_z \nabla n_{\neq}\|_{L^2}ds\\
				&\leq CA^{\frac23-2\epsilon}
				\|{\rm 	e}^{\frac{3}{2}aA^{-\frac{1}{3}}t}
				\partial_x^2n_{\neq}\|_{L^2L^2}^{\frac12}
				\|{\rm 	e}^{aA^{-\frac{1}{3}}t}
				\partial_z^2n_{\neq}\|_{L^2L^2}^{\frac12}
				\|{\rm e}^{\frac{3}{2}aA^{-\frac{1}{3}}t}\partial_x^2\nabla n_{\neq}\|^{\frac12}_{L^2L^2}
				\|{\rm e}^{aA^{-\frac{1}{3}}t}\partial_z^2\nabla 	n_{\neq}\|^{\frac12}_{L^2L^2}\\
				&\leq CA^{\frac{4}{3}-2\epsilon}
				\|\partial_x^2n_{\neq}\|_{X_{\frac32a}}
				\|\partial_z^2n_{\neq}\|_{X_{a}}.
			\end{aligned}
		\end{equation}

		\noindent\textbf{Estimate of $T_{4,3}$:}
		By Lemma \ref{sob_inf_2} and setting $\alpha=1$, we have 
		\begin{equation*}
			\begin{aligned}
				&\quad\|\partial_z^2\widehat{u_{1,0}}\partial_x n_{\neq}\|_{L^2}^2
				\leq \|\partial_z^2\widehat{u_{1,0}}\|_{L^2}^2
				\|\partial_x n_{\neq}\|_{L^{\infty}_{y,z}L^2_x}^2\\
				&\leq C\|\widehat{u_{1,0}}\|_{H^2}^2
				(\|\partial_x\partial_zn_{\neq}\|_{L^2}
				\|\partial_x\partial_z\partial_yn_{\neq}\|_{L^2}
				+\|\partial_x^2n_{\neq}\|_{L^2}
				\|\partial_x^2\partial_yn_{\neq}\|_{L^2}),
			\end{aligned}
		\end{equation*}
		which gives that 
		\begin{equation}\label{temp1_3}
			\begin{aligned}
				&\quad\|{\rm e}^{aA^{-\frac{1}{3}}t}
				\partial_z^2\widehat{u_{1,0}}\partial_x n_{\neq}\|_{L^2L^2}^2\\
				&\leq C\int_0^t{\rm e}^{-\frac{a}{2}A^{-\frac13}s}\|\widehat{u_{1,0}}\|^2_{H^2}
				{\rm e}^{\frac{5a}{2}A^{-\frac13}s}
				\|\partial_x\partial_z n_{\neq}\|_{L^2}
				\|\partial_x\partial_z \nabla n_{\neq}\|_{L^2}ds
				+C\|\widehat{u_{1,0}}\|^2_{L^{\infty}H^2}A^{\frac23}
				\|\partial_x^2n_{\neq}\|^2_{X_{\frac32a}}\\
				&\leq CA^{\frac{4}{3}-2\epsilon}
				\big(\|\partial_x^2n_{\neq}\|_{X_{\frac32a}}
				\|\partial_z^2n_{\neq}\|_{X_{a}}
				+\|\partial_x^2n_{\neq}\|^2_{X_{\frac32a}}\big).
			\end{aligned}
		\end{equation}
		
		\noindent\textbf{Estimates of $T_{4,4}$ and $T_{4,5}$:}
		For $j=2,3$, there holds $$\|
		\partial_z^2(u_{j,0}n_{\neq})\|_{L^2}^2
		\leq C(
		\|u_{j,0}\partial_z^2n_{\neq}\|^2_{L^2}
		+\|\partial_zu_{j,0}\partial_zn_{\neq}\|^2_{L^2}
		+\|\partial_z^2u_{j,0}n_{\neq}\|^2_{L^2}).$$
		By (\ref{u23_infty}), we obtain 
		\begin{equation*}
			\begin{aligned}
				\|{\rm e}^{aA^{-\frac{1}{3}}t}u_{j,0}\partial_z^2n_{\neq}\|^2_{L^2L^2}
				\leq \|u_{j,0}\|^2_{L^{\infty}L^{\infty}}\|{\rm e}^{aA^{-\frac{1}{3}}t}\partial_z^2n_{\neq}\|^2_{L^2L^2}
				\leq CA^{\frac13-2\epsilon}E_1^2
				\|\partial_z^2n_{\neq}\|^2_{X_a}.
			\end{aligned}
		\end{equation*}
		Using Lemma \ref{sob_inf_1} and Lemma \ref{sob_inf_2}, there hold
		\begin{equation*}
			\begin{aligned}						&\|\partial_zu_{j,0}\partial_zn_{\neq}\|^2_{L^2}\leq 
				\|\partial_zu_{j,0}\|^2_{L^{\infty}_yL^2_z}
				\|\partial_zn_{\neq}\|^2_{L^{\infty}_{z}L^2_{x,y}}
				\leq C\|\partial_zu_{j,0}\|^2_{H^1}
				\|\partial_z^2n_{\neq}\|_{L^2}^2,\\
				&\|\partial_z^2u_{j,0}n_{\neq}\|^2_{L^2}\leq 
				\|\partial_z^2u_{j,0}\|^2_{L^2}
				\|n_{\neq}\|^2_{L^{\infty}_{y,z}L^2_{x}}
				\leq C\|\partial_z^2u_{j,0}\|^2_{L^2}
				\|(\partial_x^2,\partial_z^2)n_{\neq}\|_{L^2}
				\|(\partial_x^2,\partial_z^2)\partial_yn_{\neq}\|_{L^2},
			\end{aligned}
		\end{equation*}
		which along with  (\ref{u23_infty}) imply that 
		\begin{equation*}
			\begin{aligned}
				&\|{\rm e}^{aA^{-\frac{1}{3}}t}\partial_zu_{j,0}
				\partial_zn_{\neq}\|^2_{L^2L^2}
				\leq CA^{\frac13}\|\partial_zu_{j,0}\|^2_{L^{\infty}H^1}
				\|\partial_z^2n_{\neq}\|^2_{X_a}\leq CA^{\frac13-2\epsilon}E_1^2
				\|\partial_z^2n_{\neq}\|^2_{X_a},\\
				&\|{\rm e}^{aA^{-\frac{1}{3}}t}\partial_z^2u_{j,0}
				n_{\neq}\|^2_{L^2L^2}
				\leq CA^{\frac23}\|\partial_z^2u_{j,0}\|^2_{L^{\infty}L^2}
				\|(\partial_x^2,\partial_z^2)n_{\neq}\|^2_{X_a}\leq CA^{\frac23-2\epsilon}E_1^2
				\|(\partial_x^2,\partial_z^2)n_{\neq}\|^2_{X_a}.
			\end{aligned}
		\end{equation*}
		Thus, we get 
		\begin{equation}\label{temp1_4}
			\begin{aligned}
				\|{\rm e}^{aA^{-\frac{1}{3}}t}\partial_z^2(u_{j,0}
				n_{\neq})\|_{L^2L^2}^2\leq 
				CA^{\frac23-2\epsilon}E_1^2\|(\partial_x^2,\partial_z^2)n_{\neq}\|^2_{X_a}.
			\end{aligned}
		\end{equation}
		
		\noindent\textbf{Estimate of $T_{4,6}$:}
		By Lemma \ref{sob_inf_1}, Lemma \ref{sob_inf_2} and Lemma \ref{lemma_n003}, 
		we get 
		\begin{equation}\label{temp1_5}
			\begin{aligned}
				&\quad \|{\rm e}^{aA^{-\frac{1}{3}}t}
				\partial_x\partial_z(u_{1,\neq}n_{0})\|_{L^2L^2}^2
				\leq C\big(\|n_0\|^2_{L^{\infty} L^{\infty}}\|{\rm 	e}^{aA^{-\frac{1}{3}}t}
				\partial_x\partial_zu_{1,\neq}\|_{L^2L^2}^2\\
				&+\|\partial_z^2n_{(0,\neq)}\|^2_{L^{\infty} L^{2}}
				\|{\rm e}^{aA^{-\frac{1}{3}}t}
				\partial_xu_{1,\neq}\|_{L^2L^2}
				\|{\rm e}^{aA^{-\frac{1}{3}}t}
				\partial_x\partial_yu_{1,\neq}\|_{L^2L^2}\big)\\
				\leq&CA^{\frac23}
				\big(\|(\partial_z^2n_{\rm in})_{(0,\neq)}\|_{L^2}^2
				+E_3^2+1\big)
				\big(\|\triangle u_{2,\neq}\|^2_{X_a}
				+\|\partial_x\omega_{2,\neq}\|^2_{X_a}\big),
			\end{aligned}
		\end{equation}
		where we use 	
		$\|\partial_xu_{1,\neq}\partial_zn_{0}\|_{L^2}\leq 
		\|\partial_zn_{0}\|_{L^{\infty}_zL^2_y}
		\|\partial_xu_{1,\neq}\|_{L^{\infty}_yL^2_{x,z}}
		\leq\|\partial_z^2n_{(0,\neq)}\|_{L^2}
		\|\partial_xu_{1,\neq}\|_{L^{\infty}_yL^2_{x,z}}.$
		
		\noindent\textbf{Estimates of $T_{4,7}$ and $T_{4,8}$:}
		For $j=2,3$, there holds $$\|
		\partial_z^2(u_{j,\neq}n_0)\|_{L^2}^2
		\leq C(\|\partial_z^2u_{j,\neq}n_0\|^2_{L^2}
		+\|\partial_zu_{j,\neq}\partial_zn_0\|^2_{L^2}
		+\|u_{j,\neq}\partial_z^2n_0\|^2_{L^2}).$$
		Due to $\|n_0\|_{L^{\infty}L^{\infty}}\leq 2E_3,$ thanks to  Lemma \ref{lemma_u}, thus 
		\begin{equation*}
			\begin{aligned}
				\|{\rm e}^{aA^{-\frac{1}{3}}t}\partial_z^2u_{j,\neq}n_0\|^2_{L^2L^2}
				\leq CE_3^2\|{\rm e}^{aA^{-\frac{1}{3}}t}\partial_z^2u_{j,\neq}\|^2_{L^2L^2}
				\leq CA^{\frac13}E_3^2(\|\triangle u_{2,\neq}\|^2_{X_a}
				+\|\partial_x\omega_{2,\neq}\|^2_{X_a}).
			\end{aligned}
		\end{equation*}
		Using Lemma \ref{sob_inf_1}, Lemma \ref{sob_inf_2}
		and $\|\partial_zn_{0}\|_{L^2}=\|\partial_zn_{(0,\neq)}\|_{L^2}\leq \|\partial_z^2n_{(0,\neq)}\|_{L^2}$, we have
		\begin{equation*}
			\begin{aligned}						\|\partial_zu_{j,\neq}\partial_zn_{0}\|^2_{L^2}\leq 
				\|\partial_zn_0\|^2_{L^{\infty}_zL^2_y}
				\|\partial_zu_{j,\neq}\|^2_{L^{\infty}_{y}L^2_{x,z}}
				\leq C\|\partial_z^2n_{(0,\neq)}\|^2_{L^2}
				\|\partial_zu_{j,\neq}\|_{L^2}
				\|\partial_z\partial_yu_{j,\neq}\|_{L^2},\\
				\|u_{j,\neq}\partial_z^2n_0\|^2_{L^2}\leq 
				\|\partial_z^2n_0\|^2_{L^2}\|u_{j,\neq}\|^2_{L^{\infty}_{y,z}L^2_x}
				\leq C\|\partial_z^2n_{(0,\neq)}\|^2_{L^2}
				\|(\partial_x,\partial_z)u_{j,\neq}\|_{L^2}
				\|(\partial_x,\partial_z)\partial_yu_{j,\neq}\|_{L^2},
			\end{aligned}
		\end{equation*}
		which along with Lemma \ref{lemma_u} and Lemma \ref{lemma_n003} give that 
		\begin{equation*}
			\begin{aligned}
				&\quad\|{\rm e}^{aA^{-\frac{1}{3}}t}\partial_zu_{j,\neq}
				\partial_zn_0\|^2_{L^2L^2}
				+\|{\rm e}^{aA^{-\frac{1}{3}}t}u_{j,\neq}
				\partial_z^2n_0\|^2_{L^2L^2}\\
				&\leq CA^{\frac23}\big(\|(\partial_z^2n_{\rm in})_{(0,\neq)}\|^2_{L^2}+1\big)\big(\|\triangle u_{2,\neq}\|^2_{X_a}
				+\|\partial_x\omega_{2,\neq}\|^2_{X_a}\big).
			\end{aligned}
		\end{equation*}
		Therefore, for $j=2,3$, we obtain that 
		\begin{equation}\label{temp1_6}
			\begin{aligned}
				{\|{\rm e}^{aA^{-\frac{1}{3}}t}
					\partial_z^2(u_{j,\neq}n_0)\|_{L^2L^2}^2}
				\leq{CA^{\frac23-\frac32\epsilon}
					\big(\|(\partial_z^2n_{\rm in})_{(0,\neq)}
					\|_{L^2}^2+E_3^2+1\big)E_2^2}.
			\end{aligned}
		\end{equation}
		\noindent\textbf{Estimate of $T_{4,9}$:}
		Using $(\ref{appa_1})_2$ and Lemma \ref{sob_inf_2},
		there holds 
		\begin{equation*}
			\begin{aligned}
				\|u_{1,\neq}\partial_x\partial_zn_{\neq}\|^2_{L^2}
				\leq C(\|\partial_x\omega_{2,\neq}\|^2_{L^2}
				+\|\triangle u_{2,\neq}\|_{L^2}^2)
				\|\partial_x\partial_zn_{\neq}\|_{L^2}
				\|\partial_x\partial_z\partial_yn_{\neq}\|_{L^2},
			\end{aligned}
		\end{equation*}
		where we use  $\|u_{1,\neq}\partial_x\partial_zn_{\neq}\|^2_{L^2}
		\leq \|u_{1,\neq}\|^2_{L^{\infty}_{x,z}L^2_y}
		\|\partial_x\partial_zn_{\neq}\|^2_{L^2_{x,z}L^{\infty}_y}.$
		By $(\ref{appa_4})_1$ and Lemma \ref{sob_inf_2}, there holds 
		\begin{equation*}
			\begin{aligned}
				&\quad\|\partial_xu_{1,\neq}\partial_zn_{\neq}\|^2_{L^2}
				+\|\partial_zu_{1,\neq}\partial_xn_{\neq}\|^2_{L^2}\\
				&\leq C(\|\partial_x\omega_{2,\neq}\|^2_{L^2}+\|\triangle 		u_{2,\neq}\|_{L^2}^2)\|(\partial_x^2,\partial_z^2)n_{\neq}\|_{L^2}
				\|(\partial_x^2,\partial_z^2)\partial_yn_{\neq}\|_{L^2}.
			\end{aligned}
		\end{equation*}
		Lemma \ref{sob_inf_2} show that $\|n_{\neq}\|^2_{L^{\infty}}\leq C \|(\partial_x^2,\partial_z^2)n_{\neq}\|_{L^{2}}\|(\partial_x^2,\partial_z^2)\partial_yn_{\neq}\|_{L^{2}},$ thus
		\begin{equation*}
			\begin{aligned}
				\|\partial_x\partial_zu_{1,\neq}n_{\neq}\|^2_{L^2}
				\leq C(\|\partial_x\omega_{2,\neq}\|^2_{L^2}
				+\|\triangle u_{2,\neq}\|_{L^2}^2)
				\|(\partial_x^2,\partial_z^2)n_{\neq}\|_{L^{2}}
				\|(\partial_x^2,\partial_z^2)\partial_yn_{\neq}\|_{L^{2}}.
			\end{aligned}
		\end{equation*}
		Therefore, we obtain that 
		\begin{equation}\label{temp1_7}
			\begin{aligned}
				&\quad\|{\rm e}^{aA^{-\frac{1}{3}}t}
				\partial_x\partial_z(u_{1,\neq}n_{\neq})_{\neq}\|^2_{L^2L^2}
				\leq \|{\rm e}^{2aA^{-\frac{1}{3}}t}
				\partial_x\partial_z(u_{1,\neq}n_{\neq})\|^2_{L^2L^2}\\
				&\leq CA^{\frac23}(\|\partial_x\omega_{2,\neq}\|^2_{X_a}
				+\|\triangle u_{2,\neq}\|_{X_a}^2)
				\|(\partial_x^2,\partial_z^2)n_{\neq}\|^2_{X_a}\leq CA^{\frac23-\frac32\epsilon}E_2^4.
			\end{aligned}
		\end{equation}
		
		\noindent\textbf{Estimates of $T_{4,10}$ and $T_{4,11}$:}
		By $(\ref{appa_1})_2$, Lemma \ref{sob_inf_2} and $$\|u_{j,\neq}\partial_z^2n_{\neq}\|^2_{L^2}
		\leq \|u_{j,\neq}\|^2_{L^{\infty}_{x,z}L^2_y}
		\|\partial_z^2n_{\neq}\|^2_{L^2_{x,z}L^{\infty}_y},$$
		there is 
		\begin{equation*}
			\begin{aligned}
				\|u_{j,\neq}\partial_z^2n_{\neq}\|^2_{L^2}
				\leq C(\|\partial_x\omega_{2,\neq}\|^2_{L^2}
				+\|\triangle u_{2,\neq}\|_{L^2}^2)
				\|\partial_z^2n_{\neq}\|_{L^2}
				\|\partial_z^2\partial_yn_{\neq}\|_{L^2},
			\end{aligned}
		\end{equation*} 
		where $j=2,3.$
		Combining Lemma \ref{lemma_u} with $\|n_{\neq}\|^2_{L^{\infty}}\leq C \|(\partial_x^2,\partial_z^2)n_{\neq}\|_{L^{2}}\|(\partial_x^2,\partial_z^2)\partial_yn_{\neq}\|_{L^{2}},$ then
		\begin{equation*}
			\begin{aligned}
				\|\partial_z^2u_{j,\neq}n_{\neq}\|^2_{L^2}
				\leq C(\|\partial_x\omega_{2,\neq}\|^2_{L^2}
				+\|\triangle u_{2,\neq}\|_{L^2}^2)
				\|(\partial_x^2,\partial_z^2)n_{\neq}\|_{L^{2}}
				\|(\partial_x^2,\partial_z^2)\partial_yn_{\neq}\|_{L^{2}}.
			\end{aligned}
		\end{equation*}	
		Using (\ref{appa_3}), Lemma \ref{sob_inf_2}
		and $\|\partial_zu_{j,\neq}\partial_zn_{\neq}\|_{L^2}\leq 
		\|\partial_zu_{j,\neq}\|_{L^{\infty}_xL^2_{y,z}}
		\|\partial_zn_{\neq}\|_{L^{\infty}_{y,z}L^2_{x}}$, we have 
		\begin{equation*}
			\begin{aligned}
				\|\partial_zu_{j,\neq}\partial_zn_{\neq}\|^2_{L^2}
				\leq C(\|\partial_x\omega_{2,\neq}\|^2_{L^2}
				+\|\triangle u_{2,\neq}\|_{L^2}^2)
				\|\partial_z^2n_{\neq}\|_{L^{2}}
				\|\partial_z^2\partial_yn_{\neq}\|_{L^{2}}.
			\end{aligned}
		\end{equation*}	
		
		Thus, for $j=2,3,$ we get that 
		\begin{equation}\label{temp1_8}
			\begin{aligned}
				&\quad\|{\rm e}^{aA^{-\frac{1}{3}}t}
				\partial_z^2(u_{j,\neq}n_{\neq})_{\neq}\|^2_{L^2L^2}
				\leq 
				\|{\rm e}^{2aA^{-\frac{1}{3}}t}
				\partial_z^2(u_{j,\neq}n_{\neq})\|^2_{L^2L^2}\\
				&\leq CA^{\frac23}\big(\|\partial_x\omega_{2,\neq}\|^2_{X_a}
				+\|\triangle u_{2,\neq}\|_{X_a}^2\big)
				\|(\partial_x^2,\partial_z^2)n_{\neq}\|^2_{X_a}\leq CA^{\frac23-\frac32\epsilon}E_2^4.
			\end{aligned}
		\end{equation}
		
		\noindent\textbf{Estimate of $T_{4,12}$:}
		According to nonlinear interaction, there holds 
		\begin{equation*}
			\|\partial_z^2(n\nabla c)_{\neq}\|^2_{L^2}\leq C
			\big(\|\partial_z^2(n_0\nabla c_{\neq})\|^2_{L^2}+
			\|\partial_z^2(n_{\neq}\nabla c_{0})\|^2_{L^2}+
			\|\partial_z^2(n_{\neq}\nabla c_{\neq})\|^2_{L^2}\big).
		\end{equation*}
		First, for the term of $\|\partial_z^2(n_0\nabla c_{\neq})\|^2_{L^2}$,
		using  Lemma \ref{sob_inf_2} and elliptic estimates in Lemma \ref{lem:ellip_2}, we have
		\begin{equation}\label{c_temp1}
			\begin{aligned}
				&\|\nabla c_{\neq}\|^2_{L^{\infty}}
				\leq C\|(\partial_x^2,\partial_z^2)\triangle c_{\neq}\|^2_{L^{2}}
				\leq C\|(\partial_x^2,\partial_z^2)n_{\neq}\|^2_{L^{2}},\\
				&\|\partial_z\nabla c_{\neq}\|^2_{L^{\infty}_{x,y}L^2_z}
				\leq C\|\partial_x\partial_z\triangle c_{\neq}\|^2_{L^{2}}
				\leq C\|(\partial_x^2,\partial_z^2)n_{\neq}\|^2_{L^{2}},\\
				&\|\partial_z^2\nabla c_{\neq}\|^2_{L^{\infty}_{y}L^2_{x,z}}
				\leq C\|\partial_z^2\triangle c_{\neq}\|^2_{L^{2}}
				\leq C\|\partial_z^2n_{\neq}\|^2_{L^{2}},
			\end{aligned}
		\end{equation}
		which imply that  
		\begin{equation*}
			\begin{aligned}
				&\quad\|\partial_z^2(n_0\nabla c_{\neq})\|^2_{L^2}\\
				&\leq C\big(\|\partial_zn_0\|^2_{L^{\infty}_{z}L^2_{y}}
				\|\partial_z\nabla c_{\neq}\|^2_{L^{\infty}_{x,y}L^2_{z}}
				+\|\partial_z^2n_{0}\|^2_{L^2}
				\|\nabla c_{\neq}\|^2_{L^{\infty}}
				+\|n_0\|^2_{L^{\infty}_{z}L^2_{y}}
				\|\partial_z^2\nabla c_{\neq}\|^2_{L^{\infty}_{y}L^2_{x,z}}
				\big)\\
				&\leq C(\|\partial_z^2n_{(0,\neq)}\|^2_{L^2}+
				\|n_{0}\|^2_{L^2})
				\|(\partial_x^2,\partial_z^2)n_{\neq}\|^2_{L^{2}},
			\end{aligned}
		\end{equation*}
		where we use $\|n_0\|^2_{L^{\infty}_{z}L^2_{y}}+\|\partial_zn_0\|^2_{L^{\infty}_{z}L^2_{y}}
		\leq C(\|\partial_z^2n_{(0,\neq)}\|^2_{L^2}+
		\|n_{0}\|^2_{L^2}).$
		Combining it with Lemma \ref{lemma_n001}, Lemma \ref{lemma_n002} and Lemma \ref{lemma_n003},  we get
		\begin{equation*}
			\begin{aligned}
				\|{\rm e}^{aA^{-\frac{1}{3}}t}
				\partial_z^2(n_0\nabla c_{\neq})\|^2_{L^2L^2}\leq 
				CA^{\frac13}\big(\|(\partial_z^2n_{\rm in})_{(0,\neq)}\|_{L^2}^2
				+1\big)
				\|(\partial_x^2,\partial_z^2)n_{\neq}\|^2_{X_a}.
			\end{aligned}
		\end{equation*}
		
		Second, for the term of $\|\partial_z^2(n_{\neq}\nabla c_{0})\|^2_{L^2}$,  by Lemma \ref{sob_inf_1} and Lemma \ref{lem:ellip_0}, it holds 
		\begin{equation*}
			\begin{aligned}
				&\|\nabla c_{0}\|^2_{L^{\infty}}
				\leq C(\|\partial_z\triangle c_{0}\|^2_{L^{2}}+\|\nabla c_{0}\|^2_{H^{1}})
				\leq C(\|\partial_zn_{(0,\neq)}\|^2_{L^{2}}
				+\|n_{0}\|^2_{L^{2}}),\\
				&\|\partial_z\nabla c_{0}\|^2_{L^{\infty}_{y}L^2_z}
				\leq C(\|\partial_z\triangle c_{0}\|^2_{L^{2}}+
				\|\partial_z\nabla c_{0}\|^2_{L^{2}})
				\leq C\|\partial_zn_{(0,\neq)}\|^2_{L^{2}},\\
				&\|\partial_z^2\nabla c_{0}\|^2_{L^{\infty}_yL^2_z}
				\leq C\|\partial_z^2\triangle c_{0}\|^2_{L^{2}}
				\leq C\|\partial_z^2 n_{(0,\neq)}\|^2_{L^{2}}.
			\end{aligned}
		\end{equation*}
		Using Lemma \ref{sob_inf_2}, we get 
		\begin{equation*}
			\begin{aligned}
				&\|\partial_zn_{\neq}\|^2_{L^{\infty}_{z}L^2_{x,y}}\leq C\|\partial_z^2n_{\neq}\|^2_{L^2},~~
				&\|n_{\neq}\|^2_{L^{\infty}_{z}L^2_{x,y}}\leq C\|(\partial_x,\partial_z)n_{\neq}\|^2_{L^{2}}.
			\end{aligned}
		\end{equation*}
		Combining above, there is   
		\begin{equation*}
			\begin{aligned}
				&\quad\|\partial_z^2(n_{\neq}\nabla c_{0})\|^2_{L^2}\\
				&\leq C\big(\|\partial_zn_{\neq}\|^2_{L^{\infty}_{z}L^2_{x,y}}
				\|\partial_z\nabla c_{0}\|^2_{L^{\infty}_{y}L^2_{z}}
				+\|\partial_z^2n_{\neq}\|^2_{L^2}
				\|\nabla c_{0}\|^2_{L^{\infty}}
				+\|n_{\neq}\|^2_{L^{\infty}_{z}L^2_{x,y}}
				\|\partial_z^2\nabla c_{0}\|^2_{L^{\infty}_{y}L^2_{z}}
				\big)\\
				&\leq C(\|\partial_z^2n_{(0,\neq)}\|^2_{L^2}+
				\|n_{0}\|^2_{L^2})
				\|(\partial_x^2,\partial_z^2)n_{\neq}\|^2_{L^{2}},
			\end{aligned}
		\end{equation*}
		which implies that 
		\begin{equation*}
			\begin{aligned}
				\|{\rm e}^{aA^{-\frac{1}{3}}t}
				\partial_z^2(n_{\neq}\nabla c_{0})\|^2_{L^2L^2}\leq 
				CA^{\frac13}\big(\|(\partial_z^2n_{\rm in})_{(0,\neq)}\|_{L^2}^2
				+1\big)
				\|(\partial_x^2,\partial_z^2)n_{\neq}\|^2_{X_a}.
			\end{aligned}
		\end{equation*}
		
		Last, for the term of $
		\|\partial_z^2(n_{\neq}\nabla c_{\neq})\|^2_{L^2}$,  using (\ref{c_temp1}) and 
		\begin{equation*}
			\begin{aligned}
				\|\partial_zn_{\neq}\|^2_{L^{\infty}_{z}L^2_{x,y}}\leq 
				C\|\partial_z^2n_{\neq}\|^2_{L^2},~~
				\|n_{\neq}\|^2_{L^{\infty}_{x,z}L^2_{y}}\leq 
				C\|(\partial_x^2,\partial_z^2)n_{\neq}\|^2_{L^{2}},\\
			\end{aligned}
		\end{equation*}		
		we get 
		\begin{equation*}
			\begin{aligned}
				&\quad\|\partial_z^2(n_{\neq}\nabla c_{\neq})\|^2_{L^2}\\
				&\leq C\big(\|\partial_zn_{\neq}\|^2_{L^{\infty}_{z}L^2_{x,y}}
				\|\partial_z\nabla c_{\neq}\|^2_{L^{\infty}_{x,y}L^2_{z}}
				+\|\partial_z^2n_{\neq}\|^2_{L^2}
				\|\nabla c_{\neq}\|^2_{L^{\infty}}
				+\|n_{\neq}\|^2_{L^{\infty}_{x,z}L^2_{y}}
				\|\partial_z^2\nabla c_{\neq}\|^2_{L^{\infty}_{y}L^2_{x,z}}
				\big)\\
				&\leq C
				\|(\partial_x^2,\partial_z^2)n_{\neq}\|^4_{L^{2}},
			\end{aligned}
		\end{equation*}
		which indicates 
		\begin{equation}\label{nc_neq1}
			\begin{aligned}
				\|{\rm e}^{aA^{-\frac{1}{3}}t}
				\partial_z^2(n_{\neq}\nabla c_{\neq})\|^2_{L^2L^2}\leq 
				\|{\rm e}^{2aA^{-\frac{1}{3}}t}
				\partial_z^2(n_{\neq}\nabla c_{\neq})\|^2_{L^2L^2}
				\leq CA^{\frac13}
				\|(\partial_x^2,\partial_z^2)n_{\neq}\|^4_{X_a}.
			\end{aligned}
		\end{equation}
		Thus, we conclude  
		\begin{equation}\label{temp1_9}
			\begin{aligned}
				\|{\rm e}^{aA^{-\frac{1}{3}}t}
				\partial_z^2(n\nabla c)_{\neq}\|^2_{L^2L^2}\leq 
				CA^{\frac13}\big(\|(\partial_z^2n_{\rm in})_{(0,\neq)}\|^2_{L^2}
				+E_2^2+1\big)E_2^2.
			\end{aligned}
		\end{equation}
		Using (\ref{temp1_1}),(\ref{temp1_2})-(\ref{temp1_8}),(\ref{temp1_9}) and Proposition 2.1,
		we  get 
		\begin{equation}\label{n_zz_result}
			\begin{aligned}
				\|\partial_z^2n_{\neq}\|^2_{X_a}\leq 
				C\Big(\|(\partial_z^2n_{\rm in})_{\neq}\|^2_{L^2}+\frac{(1
					+\|n_{\rm in}\|_{H^2}^2
					+E_2^2+E_3^2+E_4^2)E_2^2}{A^{\frac13}}\Big).
			\end{aligned}
		\end{equation}
		
		Combining (\ref{n_xx_result}) and (\ref{n_zz_result}) and setting $b=a$, 
		when $$A>\max\{A_4, C(1+\|n_{\rm in}\|_{H^2}^2
		+E_2^2+E_3^2+E_4^2)^3E_2^6\}:=A_5,$$
		there is 
		$E_{2,1}(t)\leq C(\|\partial_x^2n_{\neq}\|_{X_a}+\|\partial_z^2n_{\neq}\|_{X_a})
		\leq C(\|(\partial_x^2n_{\rm in})_{\neq}\|_{L^2}
		+\|(\partial_z^2n_{\rm in})_{\neq}\|_{L^2}+1).$
		
		The proof is complete.
	\end{proof}
	
	\subsection{Energy estimates for $E_{2,2}(t)$}
	\begin{lemma}\label{result_0_2}
		Under the conditions of Theorem \ref{result} and the assumptions (\ref{assumption}),
		there exists a constant $A_6$ independent of  $t$ and $A$, such that
		if $A\geq A_6,$ there holds
		$$E_{2,2}(t)\leq C.$$
	\end{lemma}
	
	\begin{proof}
		First, taking $\partial_x$ to $(\ref{ini2})_3,$ 
		one obtains a coupled system 
		\begin{equation}\label{ini_3}
			\left\{
			\begin{array}{lr}
				\partial_t\partial_x\omega_2+y\partial_x^2\omega_2
				-\frac{1}{A}\triangle\partial_x\omega_2+\partial_x\partial_zu_2=
				-\frac{1}{A}\partial_x\partial_z(u\cdot\nabla u_1)+\frac{1}{A}\partial_x^2(u\cdot\nabla u_3), \\
				\partial_t\triangle u_2+y\partial_x \triangle u_2
				-\frac{1}{A}\triangle(\triangle u_2) =\frac{1}{A}\partial_x^2n+\frac{1}{A}\partial_z^2n 
				-\frac{1}{A}(\partial_x^2+\partial_z^2)(u\cdot\nabla u_2)\\
				\qquad\qquad
				+\frac{1}{A}\partial_y[\partial_x(u\cdot\nabla u_1)+\partial_z(u\cdot\nabla u_3)].
			\end{array}
			\right.
		\end{equation}
		Applying Proposition \ref{timespace1} to (\ref{ini_3}) and setting $b=a$, we get 
		\begin{equation*}
			\begin{aligned}
				&\|\partial_x\omega_{2,\neq}\|_{X_a}^2+
				\|\triangle u_{2,\neq}\|_{X_a}^2
				\leq C\Big(	\|(\partial_x\omega_{2,{\rm in}})_{\neq}\|_{L^2}^2+
				\|(\triangle u_{2,{\rm in}})_{\neq}\|_{L^2}^2+\frac{\|{\rm e}^{aA^{-\frac{1}{3}}t}
					(\partial_x,\partial_z)n_{\neq}\|_{L^2L^2}^2}{A}\\
				&+\frac{
					\|{\rm e}^{aA^{-\frac{1}{3}}t}\partial_x(u\cdot\nabla u_1)_{\neq}\|_{L^2L^2}^2}{A}
				+\frac{
					\|{\rm e}^{aA^{-\frac{1}{3}}t}(\partial_x,\partial_z)
					(u\cdot\nabla u_2)_{\neq}\|_{L^2L^2}^2}{A}
				+\frac{
					\|{\rm e}^{aA^{-\frac{1}{3}}t}(\partial_x,\partial_z)
					(u\cdot\nabla u_3)_{\neq}\|_{L^2L^2}^2}{A}\Big).
			\end{aligned}
		\end{equation*}
		According to the nonlinear interaction, for $j=1,2,3,$ there holds 
		\begin{equation*}
			(u\cdot\nabla u_j)_{\neq}=u_0\cdot\nabla u_{j,\neq}+u_{\neq}\cdot\nabla u_{j,0}+(u_{\neq}\cdot\nabla u_{j,\neq})_{\neq},
		\end{equation*}
		
		Under the  assumptions (\ref{assumption}), using   Proposition \ref{pro0}, Lemma \ref{lemma_neq1}, Lemma \ref{lemma_neq2} and Lemma \ref{lemma_neq3}, we get  
		\begin{equation}\label{u_neq_1}
			\begin{aligned}
				\|\partial_x\omega_{2,\neq}\|_{X_a}^2+
				\|\triangle u_{2,\neq}\|_{X_a}^2
				\leq& C\Big(	\|(u_{\rm in})_{\neq}\|_{H^2}^2+\frac{\|
					(\partial_x^2,\partial_z^2)n_{\neq}\|_{X_a}^2}{A^{\frac23}}
				+\frac{E_1^4+E_2^4+E_3^4+E_5^4}{A^{3\epsilon}}\Big).
			\end{aligned}
		\end{equation}
		Taking $\partial_y$ and $\partial_z$ to $(\ref{ini2})_3,$ there are
		\begin{equation}\label{ini_4}
			\begin{aligned}
				&\partial_t\partial_y\omega_2+y\partial_x\partial_y\omega_2
				-\frac{1}{A}\triangle\partial_y\omega_2+\partial_z\partial_yu_2
				+\partial_x\omega_2=
				-\frac{1}{A}\partial_y\partial_z(u\cdot\nabla 	u_1)+\frac{1}{A}\partial_x\partial_y(u\cdot\nabla u_3),\\
				&\partial_t\partial_z\omega_2+y\partial_x\partial_z\omega_2
				-\frac{1}{A}\triangle\partial_z\omega_2+\partial_z^2u_2=
				-\frac{1}{A}\partial_z^2(u\cdot\nabla 	u_1)+\frac{1}{A}\partial_x\partial_z(u\cdot\nabla u_3).
			\end{aligned}
		\end{equation}
		Applying Proposition \ref{timespace0} to $(\ref{ini_4})$, we have
		\begin{equation}\label{u_neq_2}
			\begin{aligned}
				\frac{\|\partial_y\omega_{2,\neq}\|_{X_a}^2}{A^{\frac23}}
				\leq &C\Big(\frac{\|(\partial_y\omega_{2,{\rm in}})_{\neq}\|_{L^2}^2}{A^{\frac23}}
				+\|\partial_x\omega_{2,\neq}\|_{X_a}^2
				+\|\triangle u_{2,\neq}\|_{X_a}^2
				+\frac{
					\|{\rm e}^{aA^{-\frac{1}{3}}t}\partial_z(u\cdot\nabla u_1)_{\neq}\|_{L^2L^2}^2}{A^{\frac53}}\\
				&+\frac{
					\|{\rm e}^{aA^{-\frac{1}{3}}t}\partial_x(u\cdot\nabla u_3)_{\neq}\|_{L^2L^2}^2}{A^{\frac53}}\Big),
			\end{aligned}
		\end{equation}
		and 
		\begin{equation}\label{u_neq_3}
			\begin{aligned}
				\frac{\|\partial_z\omega_{2,\neq}\|_{X_a}^2}{A^{\frac23}}
				\leq &C\Big(\frac{\|(\partial_z\omega_{2,{\rm in}})_{\neq}\|_{L^2}^2}{A^{\frac23}}
				+\|\triangle u_{2,\neq}\|_{X_a}^2
				+\frac{
					\|{\rm e}^{aA^{-\frac{1}{3}}t}\partial_z(u\cdot\nabla u_1)_{\neq}\|_{L^2L^2}^2}{A^{\frac53}}\\
				&+\frac{
					\|{\rm e}^{aA^{-\frac{1}{3}}t}\partial_x(u\cdot\nabla u_3)_{\neq}\|_{L^2L^2}^2}{A^{\frac53}}\Big).
			\end{aligned}
		\end{equation}	
		Due to $\partial_xu_{1,\neq}+\partial_yu_{2,\neq}
		+\partial_zu_{3,\neq}=0,$ Lemma \ref{lemma_neq2} is enough to estimate 
		$\|{\rm e}^{aA^{-\frac{1}{3}}t}\partial_z(u_{1,0}\partial_x u_{1,\neq})\|_{L^2L^2}^2$.
		Using Lemma \ref{lemma_neq1}, Lemma \ref{lemma_neq2}, Lemma \ref{lemma_neq3} and Proposition \ref{pro0}, we
		have
		\begin{equation*}
			\begin{aligned}
				\|{\rm e}^{aA^{-\frac{1}{3}}t}\partial_z(u\cdot\nabla u_1)_{\neq}\|_{L^2L^2}^2
				\leq CA^{\frac43-3\epsilon}(1+E_2^4+E_3^4+E_5^4),\\
				\|{\rm e}^{aA^{-\frac{1}{3}}t}\partial_x(u\cdot\nabla u_3)_{\neq}\|_{L^2L^2}^2\leq CA^{1-3\epsilon}(1+E_2^4+E_3^4+E_5^4).
			\end{aligned}
		\end{equation*}
		Using (\ref{u_neq_1}), (\ref{u_neq_2}), (\ref{u_neq_3}) and Lemma \ref{result_0_1}, when 
		$$A>\max\{A_5, C(1+E_2^2+E_3^2+E_5^2)^{\frac{4}{3\epsilon}}\}:=A_6,$$
		one conclude that 
		\begin{equation}\label{result_u1}
			E_{2,2}\leq C\left(A^{\frac34\epsilon}\|(u_{\rm in})_{\neq}\|_{H^2}+1+\frac{\|(n_{\rm in})_{\neq}\|_{H^2}+1}{A^{\frac13-\frac34\epsilon}}+\frac{1+E_2^2+E_3^2+E_5^2}
			{A^{\frac34\epsilon}}\right)\leq C.
		\end{equation} 
		
		Finally, as long as $\epsilon\leq\frac49$ is impoesd, we have 
		$\frac{1}{3}-\frac{3}{4}\epsilon>0,$ then energy $E_{2,2}$ can be closed
		according to (\ref{result_u1}).
		The proof is complete.	
	\end{proof}
	
	\begin{corollary}
		Under the conditions of Theorem \ref{result} and the assumptions (\ref{assumption}),
		according to Lemma \ref{result_0_1} and Lemma \ref{result_0_2}, when $A\geq \max\{A_5,A_6\}:=C_{(2)},$
		there holds
		\begin{equation}
			E_2(t)=E_{2,1}(t)+E_{2,2}(t)\leq 
			C(\|(\partial_x^2n_{\rm in})_{\neq}\|_{L^2}
			+\|(\partial_z^2n_{\rm in})_{\neq}\|_{L^2}+1):=E_2.
		\end{equation}
	\end{corollary}
	
	\section{
		The density estimates with higher weight: Proof of Proposition \ref{pro3}}
	\begin{lemma}\label{result_0_3}
		Under the conditions of Theorem \ref{result} and the assumptions (\ref{assumption}),
		there exists a constant $A_7$ independent of $t$ and $A,$ such that 
		if $A\geq A_7,$ then
		$$E_{4}(t)\leq C\big(\|(\partial_x^2n_{\rm in})_{\neq}\|_{L^2}
		+1\big).$$
	\end{lemma}
	\begin{proof}

		{\bf Estimate of $\|\partial_x^2n_{\neq}\|_{X_{\frac32a}}.$	}
		When $A\geq \max\{A_6, C(1+E_2^4+E_3^4+E_4^4+E_5^4)^6\}:=A_7,$ setting $b=\frac32a,$
		it follows from $(\ref{n_xx_result})$ that  
		\begin{equation}\label{n_xx_result1}
			\begin{aligned}
				\|\partial_x^2n_{\neq}\|^2_{X_{\frac32a}}\leq C\Big( \|(\partial_x^2n_{\rm in})_{\neq}\|^2_{L^2}
				+\frac{E_2^4+(1+E_3^2)(E_4^2+E_5^2)}{A^{\frac{1}{3}}}\Big)
				\leq C( \|(\partial_x^2n_{\rm in})_{\neq}\|^2_{L^2}+1).
			\end{aligned}
		\end{equation}
		
		{\bf Estimate of $\partial_x\partial_zn_{\neq}$.}
		The non-zero mode $\partial_x\partial_zn_{\neq}$ satisfies 
		\begin{equation}\label{n_neq_temp_1}
			\begin{aligned}
				\partial_t\partial_x\partial_zn_{\neq}+\Big(y+\frac{\widehat{u_{1,0}}}{A}
				\Big)\partial_x^2\partial_z n_{\neq}-\frac{\triangle \partial_x\partial_zn_{\neq}}{A}=
				-\frac{\partial_z(\widetilde{u_{1,0}}\partial_x^2 n_{\neq})}{A}
				-\frac{\partial_z\widehat{u_{1,0}}\partial_x^2 n_{\neq}}{A}
				-\frac{\partial_x\partial_y\partial_z(u_{2,0}n_{\neq})}{A}\\
				-\frac{\partial_x\partial_z^2(
					u_{3,0}n_{\neq})}{A}
				-\frac{\nabla\cdot\partial_x\partial_z(u_{\neq}n_0)}{A}
				-\frac{\nabla\cdot\partial_x\partial_z(u_{\neq}n_{\neq})_{\neq}}{A}
				-\frac{\nabla\cdot\partial_x\partial_z(n\nabla c)_{\neq}}{A}.
			\end{aligned}
		\end{equation}
		Applying Proposition \ref{timespace2} to (\ref{n_neq_temp_1}), we get
		\begin{equation}\label{n_neq_temp_2}
			\begin{aligned}
				&\quad\|\partial_x\partial_zn_{\neq}\|^2_{X_{\frac32a}}\leq C\Big( 	\|(\partial_x\partial_zn_{\rm in})_{\neq}\|^2_{L^2}		
				+\frac{\|{\rm e}^{\frac32aA^{-\frac{1}{3}}t}
					\partial_z\widehat{u_{1,0}}\partial_x^2 	n_{\neq}\|_{L^2L^2}^2}{A^{\frac{5}{3}}}
				+\frac{	\|{\rm e}^{\frac32aA^{-\frac{1}{3}}t}
					\partial_z(\widetilde{u_{1,0}}\partial_x^2 	n_{\neq})\|_{L^2L^2}^2}{A^{\frac{5}{3}}}\\
				&+\frac{\|{\rm e}^{\frac32aA^{-\frac{1}{3}}t}\partial_z(u_{2,0}
					\partial_xn_{\neq})\|_{L^2L^2}^2}{A}
				+\frac{\|{\rm e}^{\frac32aA^{-\frac{1}{3}}t}\partial_z(u_{3,0}
					\partial_xn_{\neq})\|_{L^2L^2}^2}{A}
				+\frac{\|{\rm e}^{\frac32aA^{-\frac{1}{3}}t}
					n_{0}\partial_x^2u_{1,\neq}\|_{L^2L^2}^2}{A}\\
				&+\frac{\|{\rm e}^{\frac32aA^{-\frac{1}{3}}t}
					\partial_z(n_{0}\partial_xu_{2,\neq})\|_{L^2L^2}^2}{A}
				+\frac{\|{\rm e}^{\frac32aA^{-\frac{1}{3}}t}
					\partial_z(n_0\partial_xu_{3,\neq})\|_{L^2L^2}^2}{A}
				+\frac{\|{\rm e}^{\frac32aA^{-\frac{1}{3}}t}
					\partial_x\partial_z(u_{1,\neq}n_{\neq})_{\neq}\|^2_{L^2L^2}}{A}\\
				&+\frac{\|{\rm e}^{\frac32aA^{-\frac{1}{3}}t}
					\partial_x\partial_z(u_{2,\neq}n_{\neq})_{\neq}\|^2_{L^2L^2}}{A}
				+\frac{\|{\rm e}^{\frac32aA^{-\frac{1}{3}}t}
					\partial_x\partial_z(u_{3,\neq}n_{\neq})_{\neq}\|^2_{L^2L^2}}{A}
				+\frac{\|{\rm e}^{\frac32aA^{-\frac{1}{3}}t}
					\partial_x\partial_z(n\nabla c)_{\neq}\|^2_{L^2L^2}}{A}\Big).
			\end{aligned}
		\end{equation}
		For convenience, we mark (\ref{n_neq_temp_2}) as 
		$$\|\partial_x\partial_zn_{\neq}\|^2_{X_{\frac32a}}\leq C(\|(\partial_x\partial_zn_{\rm in})_{\neq}\|^2_{L^2}+T_{5,1}+\cdots+T_{5,11}).$$
		
		\noindent\textbf{Estimate of $T_{5,1}$:}
		Using Lemma \ref{u1_hat2}, there holds 
		\begin{equation*}
			\|{\rm e}^{\frac32aA^{-\frac{1}{3}}t}
			\partial_z\widehat{u_{1,0}}\partial_x^2 n_{\neq}\|_{L^2L^2}^2	
			\leq CA^{\frac13}
			\|\widehat{u_{1,0}}\|_{L^{\infty}H^4}^2	
			\|\partial_x^2 n_{\neq}\|_{X_{\frac32a}}^2
			\leq CA^{\frac73-2\epsilon}\|\partial_x^2 n_{\neq}\|_{X_{\frac32a}}^2.
		\end{equation*}
		
		\noindent\textbf{Estimate of $T_{5,2}$:}
		Using Lemma \ref{u1_hat1}, there holds 
		\begin{equation*}
			\begin{aligned}
				\|{\rm e}^{\frac32aA^{-\frac{1}{3}}t}
				\partial_z(\widetilde{u_{1,0}}\partial_x^2 	n_{\neq})\|_{L^2L^2}^2
				&\leq C\|\widetilde{u_{1,0}}\|_{L^{\infty}H^2}^2
				\|{\rm e}^{\frac32aA^{-\frac{1}{3}}t}
				\partial_x^2\nabla 	n_{\neq}\|_{L^2L^2}^2\\
				&\leq CA\|\widetilde{u_{1,0}}\|_{L^{\infty}H^2}^2
				\|\partial_x^2 n_{\neq}\|_{X_{\frac32a}}^2
				\leq CA^{\frac53-2\epsilon}
				\|\partial_x^2 n_{\neq}\|_{X_{\frac32a}}^2.			
			\end{aligned}
		\end{equation*}
		
		\noindent\textbf{Estimates of $T_{5,3}$ and $T_{5,4}$:}
		Using (\ref{u23_infty}) and Lemma \ref{lemma_u23_1}, we obtain 
		\begin{equation*}
			\begin{aligned}
				\|{\rm e}^{\frac32aA^{-\frac{1}{3}}t}u_{j,0}\partial_x\partial_zn_{\neq}\|^2_{L^2L^2}
				\leq \|u_{j,0}\|^2_{L^{\infty}L^{\infty}}\|{\rm e}^{\frac32aA^{-\frac{1}{3}}t}\partial_x\partial_zn_{\neq}\|^2_{L^2L^2}
				\leq CA^{\frac13-2\epsilon}
				\|\partial_x\partial_zn_{\neq}\|^2_{X_{\frac32a}}.
			\end{aligned}
		\end{equation*}
		By Lemma \ref{sob_inf_1} and Lemma \ref{sob_inf_2}, there holds
		\begin{equation*}
			\begin{aligned}						\|\partial_zu_{j,0}\partial_xn_{\neq}\|^2_{L^2}\leq 
				\|\partial_zu_{j,0}\|^2_{L^{\infty}_yL^2_z}
				\|\partial_xn_{\neq}\|^2_{L^{\infty}_{z}L^2_{x,y}}
				\leq C\|\partial_zu_{j,0}\|^2_{H^1}
				\|(\partial_x,\partial_z)\partial_xn_{\neq}\|_{L^2}^2,
			\end{aligned}
		\end{equation*}
		which along with Lemma \ref{lemma_u23_1} imply that 
		\begin{equation*}
			\begin{aligned}
				&\|{\rm e}^{\frac32aA^{-\frac{1}{3}}t}\partial_zu_{j,0}
				\partial_zn_{\neq}\|^2_{L^2L^2}
				\leq CA^{\frac13}\|\partial_zu_{j,0}\|^2_{L^{\infty}H^1}
				\|(\partial_x,\partial_z)\partial_xn_{\neq}\|^2_{X_{\frac32a}}\leq CA^{\frac13-2\epsilon}
				\|(\partial_x,\partial_z)\partial_xn_{\neq}\|^2_{X_{\frac32a}}.
			\end{aligned}
		\end{equation*}
		Thus, we get  
		\begin{equation*}
			\begin{aligned}
				\|{\rm e}^{\frac32aA^{-\frac{1}{3}}t}\partial_z(u_{j,0}
				\partial_xn_{\neq})\|_{L^2L^2}^2\leq 
				CA^{\frac13-2\epsilon}
				\|(\partial_x,\partial_z)\partial_xn_{\neq}\|^2_{X_{\frac32a}}.
			\end{aligned}
		\end{equation*}

		\noindent\textbf{Estimate of $T_{5,5}$:}
		Due to ${\rm div}~u_{\neq}=0,$ thus
		\begin{equation*}
			\begin{aligned}
				\|{\rm e}^{\frac32aA^{-\frac{1}{3}}t}
				n_{0}\partial_x^2u_{1,\neq}\|_{L^2L^2}^2
				&\leq CE_3^2(  \|{\rm e}^{\frac32aA^{-\frac{1}{3}}t}
				\partial_x\partial_yu_{2,\neq}\|_{L^2L^2}^2+  		
				\|{\rm e}^{\frac32aA^{-\frac{1}{3}}t}
				\partial_x\partial_zu_{3,\neq}\|_{L^2L^2}^2)
				\\ &\leq CAE_3^2(\|\partial_x^2 u_{2,\neq}\|_{X_{\frac32a}}^2
				+\|\partial_x^2 u_{3,\neq}\|_{X_{\frac32a}}^2)
				\leq CA^{1-\frac32\epsilon}E_2^2E_3^2.
			\end{aligned}
		\end{equation*}
		
		\noindent\textbf{Estimates of $T_{5,6}$ and $T_{5,7}$:}
		For $j=2,3$, there holds $$\|
		\partial_z(n_0\partial_xu_{j,\neq})\|_{L^2}^2
		\leq C(\|n_0\partial_x\partial_zu_{j,\neq}\|^2_{L^2}
		+\|\partial_zn_0\partial_xu_{j,\neq}\|^2_{L^2}).$$
		Due to $\|n_0\|_{L^{\infty}L^{\infty}}\leq 2E_3,$ thanks to  Lemma \ref{lemma_u}, thus 
		\begin{equation*}
			\begin{aligned}
				\|{\rm e}^{\frac32aA^{-\frac{1}{3}}t}
				n_0\partial_x\partial_zu_{j,\neq}\|^2_{L^2L^2}
				\leq CE_3^2\|{\rm e}^{aA^{-\frac{1}{3}}t}\partial_x\partial_zu_{j,\neq}\|^2_{L^2L^2}
				\leq CAE_3^2\|\partial_x^2u_{j,\neq}\|^2_{X_{\frac32a}}.
			\end{aligned}
		\end{equation*}
		Using Lemma \ref{sob_inf_1}, Lemma \ref{sob_inf_2}
		and $\|\partial_zn_{0}\|_{L^2}=\|\partial_zn_{(0,\neq)}\|_{L^2}\leq \|\partial_z^2n_{(0,\neq)}\|_{L^2}$, we have
		\begin{equation*}
			\begin{aligned}						\|\partial_xu_{j,\neq}\partial_zn_{0}\|^2_{L^2}\leq 
				\|\partial_zn_0\|^2_{L^{\infty}_zL^2_y}
				\|\partial_xu_{j,\neq}\|^2_{L^{\infty}_{y}L^2_{x,z}}
				\leq C\|\partial_z^2n_{(0,\neq)}\|^2_{L^2}
				\|\partial_x\nabla u_{j,\neq}\|_{L^2}^2.
			\end{aligned}
		\end{equation*}
		Combining above with Lemma \ref{lemma_u} and Lemma \ref{lemma_n003},  we have
		\begin{equation*}
			\begin{aligned}
				{\|{\rm e}^{\frac32aA^{-\frac{1}{3}}t}
					\partial_z(n_0\partial_xu_{j,\neq})\|_{L^2L^2}^2}
				&\leq CA
				\big(\|(\partial_z^2n_{\rm in})_{(0,\neq)}
				\|_{L^2}^2+E_3^2+1\big)
				\|\partial_x^2u_{j,\neq}\|^2_{X_{\frac32a}}\\
				&\leq CA^{1-\frac32\epsilon}
				\big(\|(\partial_z^2n_{\rm in})_{(0,\neq)}
				\|_{L^2}^2+E_3^2+1\big)
				E_2^2.
			\end{aligned}
		\end{equation*}
		
		\noindent\textbf{Estimates of $T_{5,8}-T_{5,10}$:}
		By $(\ref{appa_1})_2$, Lemma \ref{sob_inf_2} and $$\|u_{\neq}\partial_x\partial_zn_{\neq}\|^2_{L^2}
		\leq \|u_{\neq}\|^2_{L^{\infty}_{x,z}L^2_y}
		\|\partial_x\partial_zn_{\neq}\|^2_{L^2_{x,z}L^{\infty}_y},$$
		there holds
		\begin{equation*}
			\begin{aligned}
				\|u_{\neq}\partial_x\partial_zn_{\neq}\|^2_{L^2}
				\leq C(\|\partial_x\omega_{2,\neq}\|^2_{L^2}
				+\|\triangle u_{2,\neq}\|_{L^2}^2)
				\|\partial_x\partial_zn_{\neq}\|_{L^2}
				\|\partial_x\partial_z\partial_yn_{\neq}\|_{L^2}.
			\end{aligned}
		\end{equation*} 
		Combining Lemma \ref{lemma_u} with $\|n_{\neq}\|^2_{L^{\infty}}\leq C \|(\partial_x^2,\partial_z^2)n_{\neq}\|_{L^{2}}\|(\partial_x^2,\partial_z^2)\partial_yn_{\neq}\|_{L^{2}},$ then
		\begin{equation*}
			\begin{aligned}
				\|\partial_x\partial_zu_{\neq}n_{\neq}\|^2_{L^2}
				\leq C(\|\partial_x\omega_{2,\neq}\|^2_{L^2}
				+\|\triangle u_{2,\neq}\|_{L^2}^2)
				\|(\partial_x^2,\partial_z^2)n_{\neq}\|_{L^{2}}
				\|(\partial_x^2,\partial_z^2)\partial_yn_{\neq}\|_{L^{2}}.
			\end{aligned}
		\end{equation*}	
		Using $(\ref{appa_2})_1$, (\ref{appa_3}), Lemma \ref{sob_inf_2}
		and 
		\begin{equation*}
			\begin{aligned}
				\|\partial_zu_{\neq}\partial_xn_{\neq}\|_{L^2}\leq 
				\|\partial_zu_{\neq}\|_{L^{\infty}_xL^2_{y,z}}
				\|\partial_xn_{\neq}\|_{L^{\infty}_{y,z}L^2_{x}},\\
				\|\partial_xu_{\neq}\partial_zn_{\neq}\|_{L^2}\leq 
				\|\partial_xu_{\neq}\|_{L^{\infty}_zL^2_{y,z}}
				\|\partial_zn_{\neq}\|_{L^{\infty}_{x,y}L^2_{z}},
			\end{aligned}
		\end{equation*}
		we have 
		\begin{equation*}
			\begin{aligned}
				&\quad\|\partial_zu_{\neq}\partial_xn_{\neq}\|^2_{L^2}
				+\|\partial_xu_{\neq}\partial_zn_{\neq}\|^2_{L^2}\\
				&\leq C(\|\partial_x\omega_{2,\neq}\|^2_{L^2}
				+\|\triangle u_{2,\neq}\|_{L^2}^2)
				\|(\partial_x^2,\partial_z^2)n_{\neq}\|_{L^{2}}
				\|(\partial_x^2,\partial_z^2)\partial_yn_{\neq}\|_{L^{2}}.
			\end{aligned}
		\end{equation*}	
		Thus,  we get that 
		\begin{equation*}
			\begin{aligned}
				&\quad\|{\rm e}^{\frac32aA^{-\frac{1}{3}}t}
				\partial_x\partial_z(u_{\neq}n_{\neq})_{\neq}\|^2_{L^2L^2}
				\leq 
				\|{\rm e}^{2aA^{-\frac{1}{3}}t}
				\partial_x\partial_z(u_{\neq}n_{\neq})\|^2_{L^2L^2}\\
				&\leq CA^{\frac23}\big(\|\partial_x\omega_{2,\neq}\|^2_{X_a}
				+\|\triangle u_{2,\neq}\|_{X_a}^2\big)
				\|(\partial_x^2,\partial_z^2)n_{\neq}\|^2_{X_a}\leq CA^{\frac23-\frac32\epsilon}E_2^4.
			\end{aligned}
		\end{equation*} 
		
		\noindent\textbf{Estimate of $T_{5,11}$:}
		According to nonlinear interactions, we get
		\begin{equation*}
			\|\partial_x\partial_z(n\nabla c)_{\neq}\|^2_{L^2}\leq C
			\big(\|\partial_z(n_0\partial_x\nabla c_{\neq})\|^2_{L^2}+
			\|\partial_z(\partial_xn_{\neq}\nabla c_{0})\|^2_{L^2}+
			\|\partial_x\partial_z(n_{\neq}\nabla c_{\neq})\|^2_{L^2}\big).
		\end{equation*}
		
		First, 
		using  Lemma \ref{sob_inf_2} and Lemma \ref{lem:ellip_2}, there hold
		\begin{equation*}
			\begin{aligned}
				&\|\partial_x\nabla c_{\neq}\|^2_{L^{\infty}_{x,y}L^2_z}
				\leq C\|\partial_x^2\triangle c_{\neq}\|^2_{L^{2}}
				\leq C\|\partial_x^2n_{\neq}\|^2_{L^{2}},\\
				&\|\partial_x\partial_z\nabla c_{\neq}\|^2_{L^2}
				\leq C\|\partial_x\triangle c_{\neq}\|^2_{L^{2}}
				\leq C\|\partial_xn_{\neq}\|^2_{L^{2}},
			\end{aligned}
		\end{equation*}
		which imply that  
		\begin{equation*}
			\begin{aligned}
				\quad\|\partial_z(n_0\partial_x\nabla c_{\neq})\|^2_{L^2}
				&\leq C\big(\|\partial_zn_0\|^2_{L^{\infty}_{z}L^2_{y}}
				\|\partial_x\nabla c_{\neq}\|^2_{L^{\infty}_{x,y}L^2_{z}}
				+\|n_0\|^2_{L^{\infty}L^{\infty}}
				\|\partial_z^2\nabla c_{\neq}\|^2_{L^{\infty}_{y}L^2_{x,z}}
				\big)\\
				&\leq C(\|\partial_z^2n_{(0,\neq)}\|^2_{L^2}+E_3^2+1)
				\left(\|\partial_x^2n_{\neq}\|^2_{L^{2}}+\|\partial_{x}\partial_{z}n_{\neq}\|^2_{L^{2}}\right),
			\end{aligned}
		\end{equation*}
		where we use $\|\partial_zn_0\|^2_{L^{\infty}_{z}L^2_{y}}
		\leq C\|\partial_z^2n_{(0,\neq)}\|^2_{L^2}.$
		Using Lemma \ref{lemma_n003} and $\|n_0\|_{L^{\infty}L^{\infty}}\leq C E_3$,  we get
		\begin{equation*}
			\begin{aligned}
				\|{\rm e}^{\frac32aA^{-\frac{1}{3}}t}\partial_x
				\partial_z(n_0\nabla c_{\neq})\|^2_{L^2L^2}\leq 
				CA^{\frac13}\big(\|(\partial_z^2n_{\rm in})_{(0,\neq)}\|_{L^2}^2
				+1\big)
				\|(\partial_x,\partial_z)\partial_xn_{\neq}\|^2_{X_{\frac32a}}.
			\end{aligned}
		\end{equation*}
		
		Second, by Lemma \ref{sob_inf_1} and Lemma \ref{lem:ellip_0}, it holds 
		\begin{equation*}
			\begin{aligned}
				&\|\nabla c_{0}\|^2_{L^{\infty}}
				\leq C(\|\partial_z\triangle c_{0}\|^2_{L^{2}}+\|\nabla c_{0}\|^2_{H^{1}})
				\leq C(\|\partial_zn_{(0,\neq)}\|^2_{L^{2}}
				+\|n_{0}\|^2_{L^{2}}),\\
				&\|\partial_z\nabla c_{0}\|^2_{L^{\infty}_{y,z}}
				\leq C(\|\partial_z^2\triangle c_{0}\|^2_{L^{2}}+
				\|\partial_z\nabla c_{0}\|^2_{L^{2}})
				\leq C\|\partial_z^2n_{(0,\neq)}\|^2_{L^{2}}.
			\end{aligned}
		\end{equation*}
		Thence,  
		\begin{equation*}
			\begin{aligned}
				\quad\|\partial_z(\nabla c_{0}\partial_xn_{\neq})\|^2_{L^2}
				&\leq C\big(
				\|\partial_x\partial_zn_{\neq}\|^2_{L^2}
				\|\nabla c_{0}\|^2_{L^{\infty}}
				+\|\partial_xn_{\neq}\|^2_{L^2}
				\|\partial_z\nabla c_{0}\|^2_{L^{\infty}_{y,z}}
				\big)\\
				&\leq C(\|\partial_z^2n_{(0,\neq)}\|^2_{L^2}+
				\|n_{0}\|^2_{L^2})
				\|(\partial_x,\partial_z)\partial_xn_{\neq}\|^2_{L^{2}},
			\end{aligned}
		\end{equation*}
		which along with Lemma \ref{lemma_n001},  Lemma \ref{lemma_n002},  Lemma \ref{lemma_n003} imply that 
		\begin{equation*}
			\begin{aligned}
				\|{\rm e}^{\frac32aA^{-\frac{1}{3}}t}
				\partial_x\partial_z(n_{\neq}\nabla c_{0})\|^2_{L^2L^2}\leq 
				CA^{\frac13}\big(\|(\partial_z^2n_{\rm in})_{(0,\neq)}\|_{L^2}^2
				+1\big)
				\|(\partial_x,\partial_z)\partial_xn_{\neq}\|^2_{X_{\frac32a}}.
			\end{aligned}
		\end{equation*}
		Similar to (\ref{nc_neq1}),
		one can prove 
		\begin{equation*}
			\begin{aligned}
				\|{\rm e}^{\frac32aA^{-\frac{1}{3}}t}
				\partial_x\partial_z(n_{\neq}\nabla c_{\neq})\|^2_{L^2L^2}\leq 
				\|{\rm e}^{2aA^{-\frac{1}{3}}t}
				\partial_x\partial_z(n_{\neq}\nabla c_{\neq})\|^2_{L^2L^2}
				\leq CA^{\frac13}
				\|(\partial_x^2,\partial_z^2)n_{\neq}\|^4_{X_a}.
			\end{aligned}
		\end{equation*}
		
		Thus, one obtains  that 
		\begin{equation*}
			\begin{aligned}
				\|{\rm e}^{\frac32aA^{-\frac{1}{3}}t}
				\partial_x\partial_z(n\nabla c)_{\neq}\|^2_{L^2L^2}\leq 
				CA^{\frac13}\left(\|(\partial_z^2n_{\rm in})_{(0,\neq)}\|^4_{L^2}
				+1+E_2^4+E_4^4\right).
			\end{aligned}
		\end{equation*}
		
		\noindent\textbf{Close the energy estimate:}
		In this way, we have completed the estimate of $E_4$, when $$A>\max\{A_6, C(1+\|n_{\rm in}\|_{H^2}^4+E_2^4+E_4^2+E_4^4+E_5^4)^{3}\}:=A_7,$$  
		due to $\epsilon>\frac13,$
		we conclude  that 
		\begin{equation*}
			\begin{aligned}
				&\|\partial_x^2n_{\neq}\|^2_{X_{\frac32a}}
				\leq C( \|(\partial_x^2n_{\rm in})_{\neq}\|^2_{L^2}+1),\\
				&\|\partial_x\partial_zn_{\neq}\|^2_{X_{\frac32a}}
				\leq 
				C\Big( \|(\partial_x\partial_zn_{\rm in})_{\neq}\|^2_{L^2}+
				\frac{\|\partial_x^2 n_{\neq}\|_{X_{\frac32a}}^2}
				{A^{2\epsilon-\frac{2}{3}}}+\frac{1
					+\|n_{\rm in}\|_{H^2}^4
					+E_2^4+E_3^4+E_4^4+E_5^4}{A^{\frac13}}\Big).
			\end{aligned}
		\end{equation*}
		and
		\begin{equation*}
			\begin{aligned}
				E_4(t)=\|\partial_x^2n_{\neq}\|_{X_{\frac32a}}+
				\|\partial_x\partial_zn_{\neq}\|_{X_{\frac32a}}
				\leq C\big(\|(\partial_x^2n_{\rm in})_{\neq}\|^2_{L^2}
				+\|(\partial_x\partial_zn_{\rm in})_{\neq}\|^2_{L^2}+1\big),
			\end{aligned}
		\end{equation*}
		The proof is complete.
	\end{proof}
	
	\begin{corollary}
		Under the conditions of Theorem \ref{result} and the assumptions (\ref{assumption}),
		according to Lemma \ref{result_0_3}, when $A\geq A_7:=C_{(3)},$
		there holds
		\begin{equation}
			E_4(t)\leq C\big(\|(\partial_x^2n_{\rm in})_{\neq}\|^2_{L^2}
			+\|(\partial_x\partial_zn_{\rm in})_{\neq}\|^2_{L^2}+1\big):=E_4.
		\end{equation}
	\end{corollary}
	
	\section{The velocity estimates with higher weight: Proof of Proposition \ref{pro4}}
	To estimate $\|\partial^2_{x}u_{2,\neq}\|_{X_{\frac{3}{2}a}}$ and 
	$\|\partial_{x}^2u_{3,\neq}\|_{X_{\frac{3}{2}a}},$ it is important to 
	introduce the new quantity $W$ defined by 
	$$W=u_{2,\neq}+\kappa u_{3,\neq},$$
	where $$V=y+\frac{\widehat{u_{1,0}}}{A},\quad{\rm and}~~ \kappa=\frac{\partial_zV}{\partial_yV}.$$
	\textcolor[rgb]{0,0,0}{
		The similar quality was first proposed by Wei-Zhang in \cite{wei2} and further applied in \cite{Chen1}. 
		Here, we make a new observation: by introducing a new quasi-linear decomposition of $W,$  $\|\partial_x\nabla W\|_{X_{\frac{3}{2}a}}$ and $\|\partial_x^2 u_{3,\neq}\|_{X_{\frac{3}{2}a}}$ are enough to close the  estimates for $E_{5},$  without relying on any additional terms.}

	For convenience, we denote $\mathcal{L}_V=\partial_t+V\partial_x -\frac{1}{A}\triangle.$
	For $j=2,3,$ there holds
	$(u\cdot\nabla u_{j})_{\neq}=u_0\cdot\nabla u_{j,\neq}
	+u_{\neq}\cdot\nabla u_{j,0}+(u_{\neq}\cdot\nabla u_{j,\neq})_{\neq},$
	and we infer from (\ref{ini1}) that   
	\begin{equation}\label{ini_5}
		\left\{
		\begin{array}{lr}
			\mathcal{L}_Vu_{2,\neq}
			+\frac{\widetilde{u_{1,0}}\partial_xu_{2,\neq}}{A}
			+\frac{g_{2,1}+g_{2,2}+{G_{2,3}}}{A}
			+\frac{\partial_y(P^{N_1}_{\neq}+P^{N_3}_{\neq})}{A}=\frac{n_{\neq}-\partial_y P^{N_2}_{\neq}}{A},\\
			\mathcal{L}_Vu_{3,\neq}
			+\frac{\widetilde{u_{1,0}}\partial_xu_{3,\neq}}{A}
			+\frac{g_{3,1}+g_{3,2}+{G_{3,3}}}{A}
			+\frac{\partial_z(P^{N_1}_{\neq}+P^{N_3}_{\neq})}{A}=
			\frac{-\partial_z P^{N_2}_{\neq}}{A},
		\end{array}
		\right.
	\end{equation}
	where 
	\begin{equation*}
		\begin{aligned}
			g_{j,1}={u_{2,0}\partial_yu_{j,\neq}+u_{3,0}\partial_zu_{j,\neq}},
			~~g_{j,2}={u_{\neq}\cdot\nabla u_{j,0}},
			~~G_{j,3}={(u_{\neq}\cdot\nabla u_{j,\neq})_{\neq}}.
		\end{aligned}
	\end{equation*}
	Due to ${\rm div}~u=0,$ we have
	\begin{equation*}
		\begin{aligned}
			{\rm div}~(u\cdot\nabla u)_{\neq}
			&	=\partial_x(u\cdot\nabla u_1)_{\neq}
			+\partial_y(u\cdot\nabla u_2)_{\neq}
			+\partial_z(u\cdot\nabla u_3)_{\neq}\\
			&={\rm div}~(u_{\neq}\cdot\nabla u_{\neq})_{\neq}+
			2(\partial_yu_{1,0}\partial_xu_{2,\neq}+
			\partial_zu_{1,0}\partial_xu_{3,\neq})+2\partial_yg_{2,2}
			+2\partial_zg_{3,2},
		\end{aligned}
	\end{equation*}
	which along with $\partial_yV=1+\frac{\partial_y\widehat{ u_{1,0}}}{A}$ imply that 
	\begin{equation*}
		\begin{aligned}
			\frac{P^{N_1}_{\neq}+P^{N_3}_{\neq}}{A}
			=&-2\triangle^{-1}\left(\partial_xu_{2,\neq}+
			\frac{{\rm div}~(u\cdot\nabla u)_{\neq}}{2A}\right)\\
			=&-2\triangle^{-1}\Big((1+\frac{\partial_y\widehat{ u_{1,0}}}{A})
			\partial_xu_{2,\neq}
			+\frac{\partial_z\widehat{ u_{1,0}}}{A}
			\partial_xu_{3,\neq}\\
			&+\frac{{\rm div}~(u_{\neq}\cdot\nabla u_{\neq})_{\neq}}{2A}
			+\frac{\partial_y\widetilde{ u_{1,0}}\partial_xu_{2,\neq}+
				\partial_z\widetilde{u_{1,0}}\partial_xu_{3,\neq}}{A}
			+\frac{\partial_yg_{2,2}+\partial_zg_{3,2}}{A}\Big)\\
			=&-2\triangle^{-1}\left(\partial_yV\partial_xW
			+\frac{\partial_yg_{2,2}+\partial_zg_{3,2}}{A}
			+\frac{P_{1,1}+P_{1,2}}{A}\right),
		\end{aligned}
	\end{equation*}
	where $$P_{1,1}=\frac{{\rm div}~(u_{\neq}\cdot\nabla u_{\neq})_{\neq}}{2},
	\quad P_{1,2}=\partial_y\widetilde{ u_{1,0}}\partial_xu_{2,\neq}+
	\partial_z\widetilde{u_{1,0}}\partial_xu_{3,\neq}.$$
	Therefore, we rewrite (\ref{ini_5}) into 
	\begin{equation}\label{ini_6}
		\left\{
		\begin{array}{lr}
			&\mathcal{L}_Vu_{2,\neq}
			-2\partial_y\triangle^{-1}(\partial_yV\partial_xW)
			+\frac{\widetilde{u_{1,0}}\partial_xu_{2,\neq}}{A}
			+\frac{g_{2,1}+g_{2,2}+{G_{2,3}}}{A}
			=\frac{2\partial_y\triangle^{-1}
				(\partial_yg_{2,2}+\partial_zg_{3,2})}{A}
			\\
			&\qquad+\frac{2\partial_y\triangle^{-1}(P_{1,1}+P_{1,2})}{A}
			+\frac{n_{\neq}-\partial_y P^{N_2}_{\neq}}{A},\\
			&\mathcal{L}_Vu_{3,\neq}-2\partial_z\triangle^{-1}(\partial_yV\partial_xW)
			+\frac{\widetilde{u_{1,0}}\partial_xu_{3,\neq}}{A}
			+\frac{g_{3,1}+g_{3,2}+{G_{3,3}}}{A}
			=\frac{2\partial_z\triangle^{-1}
				(\partial_yg_{2,2}+\partial_zg_{3,2})}{A}\\
			&\qquad+\frac{2\partial_z\triangle^{-1}(P_{1,1}+P_{1,2})}{A}
			-\frac{\partial_z P^{N_2}_{\neq}}{A}.
		\end{array}
		\right.
	\end{equation}
	Then $W=u_{2,\neq}+\kappa u_{3,\neq}$ satisfies 
	\begin{equation*}
		\begin{aligned}
			&\quad \mathcal{L}_VW
			-2(\partial_y+\kappa\partial_z)\triangle^{-1}(\partial_yV\partial_xW)+\frac{\widetilde{u_{1,0}}\partial_x W}{A}
			+\frac{G^{(1)}+G^{(2)}}{A}\\
			&=(\partial_t\kappa-\frac{\triangle\kappa}{A})u_{3,\neq}-\frac{2\nabla\kappa\cdot\nabla u_{3,\neq}}{A}+\frac{n_{\neq}-\partial_yP^{N_2}_{\neq}}{A}-\frac{\kappa\partial_zP^{N_2}_{\neq}}{A},
		\end{aligned}
	\end{equation*}
	where
	\begin{equation}\label{G11}
		\begin{aligned}
			&G^{(1)}=G_{2,3}+\kappa G_{3,3}-2(\partial_y+\kappa\partial_z)
			\triangle^{-1}(P_{1,1}+P_{1,2}),\\
			&G^{(2)}=g_{2,1}+g_{2,2}+\kappa(g_{3,1}+g_{3,2})
			-2(\partial_y+\kappa\partial_z)
			\triangle^{-1}(\partial_yg_{2,2}+\partial_zg_{3,2}).
		\end{aligned}
	\end{equation} 
	
	We introduce the following decomposition $W=W^{(1)}+W^{(2)},$
	satisfying 
	\begin{equation}\label{w_equ1}
		\left\{
		\begin{array}{lr}
			\mathcal{L}_VW^{(1)}
			-2(\partial_y+\kappa\partial_z)
			\triangle^{-1}(\partial_yV\partial_xW^{(1)})
			+\frac{\widetilde{u_{1,0}}\partial_x W}{A}
			=2(\partial_y+\kappa\partial_z)\triangle^{-1}
			(\partial_yV\partial_xW^{(2)})
			\\
			+\frac{n_{\neq}-\partial_yP^{N_2}_{\neq}}{A}-\frac{\kappa\partial_zP^{N_2}_{\neq}}{A}-\frac{G^{(1)}}{A},\\
			\mathcal{L}_VW^{(2)}
			=(\partial_t\kappa-\frac{\triangle\kappa}{A})u_{3,\neq}
			-\frac{2}{A}\nabla\kappa\cdot\nabla u_{3,\neq}-\frac{G^{(2)}}{A},\\
			W^{(1)}_{\rm in}=W_{\rm in},~~~ W^{(2)}_{\rm in}=0.
		\end{array}
		\right.
	\end{equation}	
	By Lemma \ref{u1_hat2}, when $A>A_4,$ there holds
	$$A^{2\epsilon}
	\Big(\frac{\|\widehat{u_{1,0}}\|^2_{L^{\infty}H^4}}{A^{2}}
	+\frac{\|\nabla\widehat{u_{1,0}}\|^2_{L^{2}H^4}}{A^{3}}
	+\|\partial_t\widehat{u_{1,0}}\|^2_{L^{\infty}H^2}\Big)
	\leq C,$$
	which implies that 
	\begin{equation}\label{kappa_estimate}
		\begin{aligned}
			&\|\partial_t\kappa\|^2_{L^{\infty}H^1}\leq CA^{-2}\|\partial_t\widehat{u_{1,0}}\|^2_{L^{\infty}H^2}
			\leq CA^{-2(1+\epsilon)},\\
			&\|\kappa\|^2_{L^{\infty}H^3}\leq CA^{-2}
			\|\widehat{u_{1,0}}\|^2_{L^{\infty}H^4}\leq CA^{-2\epsilon}.
		\end{aligned}
	\end{equation}
	
	\begin{lemma}\label{result_51}
		Under the conditions of Theorem \ref{result} and the assumptions (\ref{assumption}),
		if $A\geq A_7,$ then
		\begin{equation}\label{eq:g1g2}
			\begin{aligned}
				&\|{\rm e}^{\frac32aA^{-\frac{1}{3}}t}
				\nabla G^{(1)}\|^2_{L^2L^2}
				\leq CA^{\frac53-\frac72\epsilon}(E_2^4+E_5^2),\\
				&\|{\rm e}^{\frac32aA^{-\frac{1}{3}}t}
				\partial_x^2 G^{(2)}\|^2_{L^2L^2}
				\leq CA^{1-\frac{7}{2}\epsilon}E_5^2,
			\end{aligned}		
		\end{equation}
		where $ \epsilon\in(\frac13, \frac49] $ is a positive constant.
		
	\end{lemma}
	\begin{proof}
		By Lemma \ref{lemma_neq1} and  the assumptions (\ref{assumption}), 
		there holds
		\begin{equation*}
			\|{\rm e}^{\frac32aA^{-\frac{1}{3}}t}
			\nabla G_{2,3}\|^2_{L^2L^2}\leq 
			\|{\rm e}^{\frac32aA^{-\frac{1}{3}}t}
			\nabla (u_{\neq}\cdot\nabla u_{2,\neq})_{\neq}\|^2_{L^2L^2}
			\leq 
			CA^{1-3\epsilon}E_2^4.
		\end{equation*}
		Using Lemma \ref{lemma_neq1} and \eqref{kappa_estimate}, we have
		\begin{equation*}
			\|{\rm e}^{\frac32aA^{-\frac{1}{3}}t}
			\nabla(\kappa G_{2,3})\|^2_{L^2L^2}
			\leq C\|\kappa \|_{L^{\infty}H^3}^2
			\|{\rm e}^{\frac32aA^{-\frac{1}{3}}t}
			\nabla (u_{\neq}\cdot\nabla u_{3,\neq})_{\neq}\|^2_{L^2L^2}
			\leq 
			CA^{\frac53-5\epsilon}E_2^4.
		\end{equation*}
		Using Lemma \ref{lemma_neq1} and \eqref{kappa_estimate} again,  one obtains that
		\begin{equation}\label{p11}
			\begin{aligned}
				&\quad\|{\rm e}^{\frac32aA^{-\frac{1}{3}}t}
				\nabla \big((\partial_y+\kappa\partial_z)
				\triangle^{-1}P_{1,1}\big) \|^2_{L^2L^2}\\
				&\leq C\|{\rm e}^{\frac32aA^{-\frac{1}{3}}t}P_{1,1} \|^2_{L^2L^2}
				\leq C\big(\|{\rm e}^{2aA^{-\frac{1}{3}}t}\partial_x(u_{\neq}\cdot\nabla u_{1,\neq})\|^2_{L^2L^2}\\
				&+\|{\rm e}^{2aA^{-\frac{1}{3}}t}\partial_y(u_{\neq}\cdot\nabla u_{2,\neq})\|^2_{L^2L^2}
				+\|{\rm e}^{2aA^{-\frac{1}{3}}t}\partial_z(u_{\neq}\cdot\nabla u_{3,\neq})\|^2_{L^2L^2}\big)\leq CA^{1-3\epsilon}E_2^4.
			\end{aligned}
		\end{equation}
		Similarly, by Lemma \ref{sob_inf_1}, Lemma \ref{sob_inf_2} and Lemma \ref{u1_hat1}, we get
		\begin{equation}\label{p12}
			\begin{aligned}
				&\quad\|{\rm e}^{\frac32aA^{-\frac{1}{3}}t}
				\nabla \big((\partial_y+\kappa\partial_z)
				\triangle^{-1}P_{1,2}\big) \|^2_{L^2L^2}\\
				&\leq C\|{\rm e}^{\frac32aA^{-\frac{1}{3}}t}P_{1,2} \|^2_{L^2L^2}
				\leq C\|{\rm e}^{\frac32aA^{-\frac{1}{3}}t} \nabla \widetilde{u_{1,0}}\partial_x(u_{2,\neq},u_{3,\neq}) \|^2_{L^2L^2}
				\\
				&\leq 
				C\|\widetilde{u_{1,0}}\|^2_{L^{\infty}H^2}
				\|{\rm e}^{\frac32aA^{-\frac{1}{3}}t}  \partial_x\nabla(u_{2,\neq},u_{3,\neq}) \|^2_{L^2L^2}
				\leq CA^{\frac53-\frac72\epsilon}E_5^2,
			\end{aligned}
		\end{equation}
		where we use 
		$$\|\nabla \widetilde{u_{1,0}}\partial_xu_{j,\neq} \|^2_{L^2}
		\leq 
		\|\nabla \widetilde{u_{1,0}}\|_{L^{\infty}_zL^2_y}
		\|\partial_xu_{j,\neq} \|^2_{L^{\infty}_yL^2_{x,z}}
		\leq 
		\| \widetilde{u_{1,0}}\|_{H^2}
		\|\partial_x\nabla u_{j,\neq} \|^2_{L^2}.$$
		Thus the proof of $\eqref{eq:g1g2}_1$
		is complete.
		
		For $j=2,3,$  by (\ref{u23_infty}) and Lemma \ref{lemma_u23_1}, there holds
		\begin{equation}\label{g31_1}
			\begin{aligned}
				\|{\rm e}^{\frac32aA^{-\frac{1}{3}}t}
				\partial_x^2 g_{j,1}\|^2_{L^2L^2}
				&\leq C(\|u_{2,0}\|^2_{L^{\infty}H^2}
				+\|u_{3,0}\|^2_{L^{\infty}H^1})
				\|{\rm e}^{\frac32aA^{-\frac{1}{3}}t}  \partial_x^2\nabla(u_{2,\neq},u_{3,\neq}) \|^2_{L^2L^2}\\
				&\leq CA^{1-\frac{7}{2}\epsilon}E_5^2.
			\end{aligned}
		\end{equation}
		Due to $\partial_yu_{2,0}+\partial_zu_{3,0}=0,$
		we have 
		$$\|\nabla u_{2,0}\|^2_{L^{\infty}}
		+\|\nabla u_{3,0}\|^2_{L^{\infty}}
		\leq C(\|\nabla\triangle u_{2,0}\|^2_{L^{2}}
		+\|\triangle u_{3,0}\|^2_{L^{2}}),$$
		which along with Lemma \ref{lemma_u23_1} show that 
		\begin{equation}\label{g22_1}
			\begin{aligned}
				\|{\rm e}^{\frac32aA^{-\frac{1}{3}}t}
				\partial_x^2 g_{j,2}\|^2_{L^2L^2}
				&\leq C(\|\nabla u_{2,0}\|^2_{L^{2}H^2}
				+\|\nabla u_{3,0}\|^2_{L^{2}H^1})
				\|{\rm e}^{\frac32aA^{-\frac{1}{3}}t}  \partial_x^2(u_{2,\neq},u_{3,\neq}) \|^2_{L^{\infty}L^2}\\
				&\leq CA^{1-\frac{7}{2}\epsilon}E_5^2.
			\end{aligned}
		\end{equation}
		One can finish the proof of the second result by using (\ref{g22_1}) and 
		$$\|{\rm e}^{\frac32aA^{-\frac{1}{3}}t}
		(\partial_y+\kappa\partial_z)\triangle^{-1}(\nabla f)\|^2_{L^2L^2}
		\leq C\|{\rm e}^{\frac32aA^{-\frac{1}{3}}t}f \|^2_{L^2L^2},$$
		where $f$ is selected form $g_{2,2}$ and $g_{3,2}.$
	\end{proof}

	\begin{lemma}\label{lemmaw21}
		Under the assumptions of Theorem \ref{result} and the assumptions (\ref{assumption}),
		there hold 
		\begin{equation}\label{eq:derivative of W}
			A^{\frac32\epsilon}\|\partial_x^2W^{(2)}\|^2_{X_{\frac32a}}
			\leq  \frac{CE_5^2}{A^{\frac23+2\epsilon}},
			\quad 		A^{\frac32\epsilon}\|\partial_x\nabla W^{(2)}\|^2_{X_{\frac32a}}
			\leq  \frac{CE_5^2}{A^{2\epsilon}}.
		\end{equation}
	\end{lemma}
	\begin{proof}
		Taking $\partial_x^2$ to $(\ref{w_equ1})_2,$ we have 
		\begin{equation*}
			\mathcal{L}_V\partial_x^2W^{(2)}
			=\partial_x^2\Big((\partial_t\kappa-\frac{\triangle\kappa}{A})u_{3,\neq}\Big)
			-\frac{2}{A}\partial_x^2\Big(\nabla\kappa\cdot\nabla u_{3,\neq}\Big)-\frac{\partial_x^2G^{(2)}}{A}.
		\end{equation*}
		Applying Proposition \ref{timespace2}, there is 
		\begin{equation}\label{w22_ans_1}
			\begin{aligned}
				\|\partial_x^2W^{(2)}\|^2_{X_{\frac32a}}
				\leq &CA^{\frac13}\Big(\|{\rm e}^{\frac32aA^{-\frac{1}{3}}t}
				\partial_x^2(\partial_t\kappa u_{3,\neq})\|^2_{L^2L^2}
				+\|{\rm e}^{\frac32aA^{-\frac{1}{3}}t}
				\partial_x^2\big(\frac{{\triangle\kappa}}{A^2}u_{3,\neq}\big)\|^2_{L^2L^2}\\
				&+\|{\rm e}^{\frac32aA^{-\frac{1}{3}}t}\partial_x^2
				\big(\frac{\nabla\kappa}{A^2}\cdot\nabla u_{3,\neq}\big)\|^2_{L^2L^2}\Big)
				+\frac{C\|{\rm e}^{\frac32aA^{-\frac{1}{3}}t}
					\partial_x^2 G^{(2)}\|^2_{L^2L^2}}{A^{\frac53}}.
			\end{aligned}
		\end{equation}
		Using Lemma \ref{sob_inf_1} and Lemma \ref{sob_inf_2}, we get
		\begin{equation*}
			\begin{aligned}
				&\|\partial_x^2(\partial_t\kappa u_{3,\neq})\|^2_{L^2}
				\leq \|\partial_t\kappa\|^2_{L^{\infty}_zL^2_y}
				\|\partial_x^2u_{3,\neq}\|^2_{L^{\infty}_yL^2_{x,z}}
				\leq  C\|\partial_t\kappa\|^2_{H^1}
				\|\partial_x^2u_{3,\neq}\|_{L^2}
				\|\partial_x^2\partial_yu_{3,\neq}\|_{L^2},\\
				&\|\partial_x^2({\triangle\kappa}
				u_{3,\neq})\|^2_{L^2}
				\leq \|\triangle \kappa\|^2_{L^{\infty}_zL^2_y}
				\|\partial_x^2u_{3,\neq}\|^2_{L^{\infty}_yL^2_{x,z}}
				\leq  C\|\kappa\|^2_{H^3}
				\|\partial_x^2u_{3,\neq}\|_{L^2}
				\|\partial_x^2\partial_yu_{3,\neq}\|_{L^2},\\
				&\|\partial_x^2(\nabla\kappa\cdot\nabla
				u_{3,\neq})\|^2_{L^2}
				\leq \|\nabla \kappa\|^2_{L^{\infty}}
				\|\partial_x^2\nabla u_{3,\neq}\|^2_{L^2}
				\leq  C\|\kappa\|^2_{H^3}
				\|\partial_x^2\nabla u_{3,\neq}\|^2_{L^2},
			\end{aligned}
		\end{equation*}
		which along with (\ref{kappa_estimate}) imply that 
		\begin{equation}\label{u33_1}
			\begin{aligned}
				&\quad \|{\rm e}^{\frac32aA^{-\frac{1}{3}}t}
				\partial_x^2(\partial_t\kappa u_{3,\neq})\|^2_{L^2L^2}
				+\|{\rm e}^{\frac32aA^{-\frac{1}{3}}t}\partial_x^2
				\big(\frac{{\triangle\kappa}}{A^2}u_{3,\neq}\big)\|^2_{L^2L^2}
				\\
				&\leq C\big(\|\partial_t\kappa\|^2_{L^{\infty}H^1}+
				A^{-2}\|\kappa\|^2_{L^{\infty}H^3}\big)
				\|{\rm e}^{\frac32aA^{-\frac{1}{3}}t}
				\partial_x^2 u_{3,\neq}\|_{L^2L^2}
				\|{\rm e}^{\frac32aA^{-\frac{1}{3}}t}
				\partial_x^2\partial_y u_{3,\neq}\|_{L^2L^2}\\
				&\leq CA^{-\frac43-2\epsilon}\|\partial_x^2u_{3,\neq}\|_{X_{\frac32a}}^2,
			\end{aligned}
		\end{equation}
		and 
		\begin{equation}\label{u33_2}
			\begin{aligned}
				&\quad\|{\rm e}^{\frac32aA^{-\frac{1}{3}}t}\partial_x^2
				\big(\frac{\nabla\kappa}{A^2}\cdot\nabla u_{3,\neq}\big)\|^2_{L^2L^2}
				\leq C\frac{\|\kappa\|^2_{L^{\infty}H^3}\|{\rm e}^{\frac32aA^{-\frac{1}{3}}t}
					\partial_x^2\nabla u_{3,\neq}\|_{L^2L^2}^2}{A^2}
				\leq\frac{C\|\partial_x^2u_{3,\neq}\|_{X_{\frac32a}}^2}
				{A^{1+2\epsilon}}.
			\end{aligned}
		\end{equation}
		Using (\ref{u33_1}), (\ref{u33_2}) and Lemma \ref{result_51}, we infer from  (\ref{w22_ans_1}) that
		\begin{equation}\label{w22_ans_2}
			\begin{aligned}
				A^{\frac32\epsilon}\|\partial_x^2W^{(2)}\|^2_{X_{\frac32a}}\leq \frac{CA^{\frac32\epsilon}\|\partial_x^2u_{3,\neq}\|_{X_{\frac32a}}^2}{A^{\frac23+2\epsilon}}+\frac{CA^{\frac32\epsilon}\|e^{\frac32aA^{-\frac13}t}\partial_{x}^{2}G^{(2)}\|_{L^{2}L^{2}}^{2}}{A^{\frac53}}
				\leq  \frac{CE_5^2}{A^{\frac23+2\epsilon}}.
			\end{aligned}
		\end{equation}
		For $j=2,3,$ there holds 
		$$\mathcal{L}_V\partial_x\partial_jW^{(2)}
		=\partial_x\partial_j\Big((\partial_t\kappa-\frac{\triangle\kappa}{A})u_{3,\neq}\Big)
		-\frac{2\partial_x\partial_j
			(\nabla\kappa\cdot\partial_x^2\nabla u_{3,\neq})}{A}
		-\partial_jV\partial_x^2W^{(2)}-\frac{\partial_x\partial_jG^{(2)}}{A}.$$
		Applying Proposition \ref{timespace2} and setting $b=\frac32a$, we have 
		\begin{equation*}\label{w22_ans_3}
			\begin{aligned}
				&\quad\|\partial_x\partial_jW^{(2)}\|^2_{X_{\frac32a}}\\
				&\leq CA\Big(\|{\rm e}^{\frac32aA^{-\frac{1}{3}}t}
				\partial_x(\partial_t\kappa u_{3,\neq})\|^2_{L^2L^2}
				+\|{\rm e}^{\frac32aA^{-\frac{1}{3}}t}
				\partial_x\big(\frac{{\triangle\kappa}}{A^2}u_{3,\neq}\big)\|^2_{L^2L^2}
				+\|{\rm e}^{\frac32aA^{-\frac{1}{3}}t}\partial_x
				\big(\frac{\nabla\kappa}{A^2}\cdot\nabla u_{3,\neq}\big)\|^2_{L^2L^2}\Big)\\
				&+C\Big(A^{\frac{1}{3}}\|{\rm e}^{\frac32aA^{-\frac{1}{3}}t}
				\partial_jV\partial_x^2W^{(2)}\|^2_{L^2L^2}
				+\frac{\|{\rm e}^{\frac32aA^{-\frac{1}{3}}t}
					\partial_x^2 G^{(2)}\|^2_{L^2L^2}}{A}\Big).
			\end{aligned}
		\end{equation*}
		Due to 
		\begin{equation}\label{v1}
			\|\nabla V\|_{L^{\infty}L^{\infty}}\leq C(1+\frac{\|\nabla\widehat{u_{1,0}}\|_{L^{\infty}L^{\infty}}}{A})
			\leq C(1+\frac{\|\widehat{u_{1,0}}\|_{L^{\infty}H^3}}{A})\leq C,
		\end{equation}
		using (\ref{u33_1}), (\ref{u33_2}), (\ref{w22_ans_2}) and Lemma \ref{result_51}, the proof of \eqref{eq:derivative of W} is complete.
		
	\end{proof}
	\begin{lemma}\label{lemmaw22}
		Under the conditions of Theorem \ref{result} and the assumptions \eqref{assumption}, there exists a constant $A_8$ independent of $t$ and $A,$ such that 
		if $A\geq A_8,$ then 
		\begin{equation*}
			A^{\frac32\epsilon}\|\triangle W^{(1)}\|^2_{X_{\frac32a}}\leq C\big(\|(\partial_x^2n_{\rm in})_{\neq}\|^2_{L^2}
			+\|(\partial_z^2n_{\rm in})_{\neq}\|^2_{L^2}+1\big).
		\end{equation*}
	\end{lemma}
	\begin{proof}
		Applying Proposition \ref{timespace4} to $(\ref{w_equ1})_1$, we get 
		\begin{equation}\label{w1_equ1}
			\begin{aligned}
				\|\triangle W^{(1)}\|^2_{X_{\frac32a}}
				\leq &C\Big(\| \triangle W_{\rm in}\|^2_{L^2}
				+\frac{\|{\rm e}^{\frac32aA^{-\frac{1}{3}}t} \nabla(n_{\neq}-\partial_yP^{N_2}_{\neq})\|^2_{L^2L^2}}{A}
				+\frac{\|{\rm e}^{\frac32aA^{-\frac{1}{3}}t} \nabla(\kappa\partial_zP^{N_2}_{\neq})\|^2_{L^2L^2}}{A}\\
				&+A^{\frac13}\|{\rm e}^{\frac32aA^{-\frac{1}{3}}t} 
				\triangle\big((\partial_y+\kappa\partial_z)\triangle^{-1}
				(\partial_yV\partial_xW^{(2)})\big)\|^2_{L^2L^2}
				\\&+\frac{\|{\rm e}^{\frac32aA^{-\frac{1}{3}}t} \nabla G^{(1)}\|^2_{L^2L^2}}{A}+\frac{\|e^{\frac32aA^{-\frac13}t}\nabla(\widetilde{ u_{1,0}}\partial_{x}W)\|_{L^{2}L^{2}}^{2}}{A}
				\Big).
			\end{aligned}
		\end{equation}
		Due to $\triangle P^{N_2}=\partial_yn,$ 
		then  
		$\triangle(n_{\neq}-\partial_yP^{N_2}_{\neq})
		=(\partial_x^2+\partial_z^2)n_{\neq},$
		which implies that 
		\begin{equation*}
			\|\nabla(n_{\neq}-\partial_yP^{N_2}_{\neq})\|_{L^2}
			\leq C\|(\partial_x,\partial_z)n_{\neq}\|_{L^2}
			\leq C\|(\partial_x^2,\partial_x\partial_z)n_{\neq}\|_{L^2},
		\end{equation*}
		and 
		\begin{equation}\label{w1_equ2}
			\|{\rm e}^{\frac32aA^{-\frac{1}{3}}t} \nabla(n_{\neq}-\partial_yP^{N_2}_{\neq})\|^2_{L^2L^2}\leq C A^{\frac13}\left(\|\partial_x^2n_{\neq}\|^2_{X_{\frac32a}}
			+\|\partial_x\partial_zn_{\neq}\|^2_{X_{\frac32a}}\right).
		\end{equation}
		Due to (\ref{kappa_estimate}) and 
		$\|\partial_z\nabla P^{N_2}_{\neq}\|^2_{L^2}\leq \|\partial_zn_{\neq}\|^2_{L^2}\leq \|\partial_x\partial_zn_{\neq}\|^2_{L^2},$
		therefore
		\begin{equation*}
			\begin{aligned}
				\|\nabla(\kappa\partial_zP^{N_2}_{\neq})\|^2_{L^2}\leq 
				C\|\kappa\|^2_{H^3}\|\partial_z\nabla P^{N_2}_{\neq}\|^2_{L^2}
				\leq CA^{-2\epsilon}\|\partial_z\nabla P^{N_2}_{\neq}\|^2_{L^2}
				\leq CA^{-2\epsilon}\|\partial_x\partial_zn_{\neq}\|^2_{L^2},
			\end{aligned}
		\end{equation*}
		which implies that 
		\begin{equation}\label{w1_equ3}
			\|{\rm e}^{\frac32aA^{-\frac{1}{3}}t} \nabla(\kappa\partial_zP^{N_2}_{\neq})\|^2_{L^2L^2}\leq 
			CA^{\frac13-2\epsilon}\|\partial_x\partial_zn_{\neq}\|^2_{X_{\frac32a}} .
		\end{equation}	
		By \eqref{kappa_estimate}, \eqref{v1} and 
		\begin{equation}\label{v3}
			\begin{aligned}
				\|(\partial_y,\partial_z)\nabla V\|_{L^{\infty}L^{\infty}}+
				\|\triangle V\|_{L^{\infty}L^{\infty}}
				\leq CA^{-1}\|\widehat{u_{1,0}}\|_{L^{\infty}H^4}
				\leq CA^{-\epsilon},
			\end{aligned}
		\end{equation}
		there holds
		\begin{equation}\label{result_temp1}
			A^{\frac13}\|{\rm e}^{\frac32aA^{-\frac{1}{3}}t} 
			\triangle\big((\partial_y+\kappa\partial_z)\triangle^{-1}
			(\partial_yV\partial_xW^{(2)})\big)\|^2_{L^2L^2}
			\leq CA^{\frac23}\|\partial_x\nabla W^{(2)}\|_{X_{\frac32a}}^2.
		\end{equation}
		Recalling $ W=u_{2,\neq}+\kappa u_{3,\neq}, $ then using Lemma \ref{sob_inf_1}, Lemma \ref{sob_inf_2}, Lemma \ref{u1_hat1} and (\ref{kappa_estimate}), we get
		\begin{equation}\label{widetilde u10}
			\begin{aligned}
				&\|e^{\frac32aA^{-\frac13}t}\nabla(\widetilde{ u_{1,0}}\partial_{x}W)\|_{L^{2}L^{2}}^{2}\\
				\leq&\|\nabla\widetilde{ u_{1,0}}\|_{L^{\infty}_{t,z}L^{2}_{y}}^{2}\|e^{\frac32aA^{-\frac13}t}\partial_{x}(u_{2,\neq}+\kappa u_{3,\neq})\|_{L^{\infty}_{y}L^{2}_{t,x,z}}^{2}+\|\widetilde{ u_{1,0}}\|_{L^{\infty}L^{\infty}}^{2}\|e^{\frac32aA^{-\frac13}t}\nabla\partial_{x}(u_{2,\neq}+\kappa u_{3,\neq})\|_{L^{2}L^{2}}^{2}\\\leq&C\|\widetilde{ u_{1,0}}\|_{L^{\infty}H^{2}}^{2}\|e^{\frac32aA^{-\frac13}t}\nabla\partial_{x}(u_{2,\neq}, u_{3,\neq})\|_{L^{2}L^{2}}^{2}\\\leq&CA^{\frac53-\frac72\epsilon}E_{5}^{2}.
			\end{aligned}
		\end{equation}
		
		By (\ref{w1_equ1}), (\ref{w1_equ2}), (\ref{w1_equ3}), (\ref{result_temp1}), (\ref{widetilde u10}), Lemma \ref{result_0_3}, Lemma \ref{result_51} and Lemma \ref{lemmaw21}, as long as 
		$$A>\max\{A_7,E_5^{\frac{3}{3\epsilon-1}}\}:=A_8,$$
		we conclude that 
		\begin{equation*}
			\begin{aligned}
				A^{\frac32\epsilon}\|\triangle W^{(1)}\|^2_{X_{\frac32a}}
				&\leq C\Big(A^{\frac32\epsilon}\| u_{\rm in}\|^2_{H^2}
				+\frac{\|\partial_x^2n_{\neq}\|^2_{X_{\frac32a}}
					+\|\partial_x\partial_zn_{\neq}\|^2_{X_{\frac32a}} }{A^{\frac23-\frac32\epsilon}} 
				+\frac{E_2^4+E_5^2}{A^{\frac32\epsilon}}
				+\frac{E_5^2}{A^{2\epsilon-\frac23}}\Big)\\
				&\leq C\left(A^{\frac32\epsilon}\| u_{\rm in}\|^2_{H^2}+\|(\partial_x^2n_{\rm in})_{\neq}\|^2_{L^2}
				+
				\|(\partial_z^2n_{\rm in})_{\neq}\|^2_{L^2}
				+1\right)\\
				&\leq C\big(\|(\partial_x^2n_{\rm in})_{\neq}\|^2_{L^2}
				+\|(\partial_z^2n_{\rm in})_{\neq}\|^2_{L^2}+1\big).
			\end{aligned}
		\end{equation*}
		Here, to close the energy estimate successfully, $\epsilon$ must satisfy 
		$\epsilon\in(\frac13,\frac49]$.
	\end{proof}

	\begin{corollary}\label{c1}
		According to Lemma \ref{lemmaw21} and Lemma \ref{lemmaw22}, 
		when $A>A_8,$
		there holds 
		\begin{equation*}
			\begin{aligned}
				A^{\frac32\epsilon}\|\partial_x\nabla W\|^2_{X_{\frac32a}}
				&\leq CA^{\frac32\epsilon}\|\partial_x\nabla (W^{(1)},W^{(2)})\|^2_{X_{\frac32a}}
				\leq C\big(\|(n_{\rm in})_{\neq}\|^2_{H^2}+1\big).		
			\end{aligned}
		\end{equation*} 
	\end{corollary}
	
	\begin{lemma}\label{c2}
		Under the conditions of Theorem \ref{result} and the assumptions \eqref{assumption}, 
		if $A\geq A_8,$ then 
		\begin{equation*}
			A^{\frac32\epsilon}\|\partial_x^2u_{3,\neq}\|^2_{X_{\frac32a}}\leq C\big(\|(\partial_x^2n_{\rm in})_{\neq}\|^2_{L^2}
			+\|(\partial_z^2n_{\rm in})_{\neq}\|^2_{L^2}+1\big).
		\end{equation*}
	\end{lemma}
	\begin{proof}
		For $\eqref{ini_6}_2$, taking $\partial_x^2$ and applying Proposition
		\ref{timespace2}, by \eqref{v1}, we get 
		\begin{equation*}
			\begin{aligned}
				A^{\frac32\epsilon}\|\partial_x^2u_{3,\neq}\|^2_{X_{\frac32a}}
				\leq& CA^{\frac32\epsilon}\Bigg(\|(u_{3,\rm in})_{\neq}\|_{H^2}^2+\|{\rm e}^{\frac32aA^{-\frac{1}{3}}t}
				\partial_x^2W\|^2_{L^2L^2}
				+\frac{\|{\rm e}^{\frac32aA^{-\frac{1}{3}}t}
					\partial_x^2\partial_zP^{N_2}_{\neq}\|^2_{L^2L^2}}
				{A^{\frac{5}{3}}}\\
				&+\frac{\|{\rm e}^{\frac32aA^{-\frac{1}{3}}t}
					\partial_x^2(g_{2,2},g_{3,1},g_{3,2})\|^2_{L^2L^2}}
				{A^{\frac53}}
				+\frac{\|{\rm e}^{\frac32aA^{-\frac{1}{3}}t}
					(P_{1,1},P_{1,2},\partial_{x}G_{3,3})\|^2_{L^2L^2}}{A}\\&+\frac{\|e^{\frac32aA^{-\frac13}t}\widetilde{ u_{1,0}}\partial_{x}^{3}u_{3,\neq}\|_{L^{2}L^{2}}^{2}}{A^{\frac53}}\Bigg).
			\end{aligned}	
		\end{equation*}
		Using Corollary \ref{c1}, one obtains
		\begin{equation*}
			A^{\frac32\epsilon}\|{\rm e}^{\frac32aA^{-\frac{1}{3}}t}
			\partial_x^2W\|^2_{L^2L^2}\leq 
			A^{\frac32\epsilon}\|\partial_x\nabla W\|^2_{X_{\frac32a}}
			\leq C\big(\|(n_{\rm in})_{\neq}\|^2_{H^2}+1\big).
		\end{equation*}
		By  (\ref{p11}), (\ref{p12}), \eqref{g31_1}, \eqref{g22_1}
		and Lemma \ref{lemma_neq1}, when $A>A_8,$
		there holds
		\begin{equation*}
			\frac{\|{\rm e}^{\frac32aA^{-\frac{1}{3}}t}
				\partial_x^2(g_{2,2},g_{3,1},g_{3,2})\|^2_{L^2L^2}}
			{A^{\frac53-\frac32\epsilon}}\\
			+\frac{\|{\rm e}^{\frac32aA^{-\frac{1}{3}}t}
				(P_{1,1},P_{1,2},\partial_xG_{3,3})\|^2_{L^2L^2}}
			{A^{1-\frac32\epsilon}}
			\leq C\frac{E_5^2+E_2^4}{A^{\frac32\epsilon}}\leq C.
		\end{equation*}
		Due to  $\triangle P^{N_2}=\partial_yn,$  then 
		\begin{equation*}
			\frac{\|{\rm e}^{\frac32aA^{-\frac{1}{3}}t}
				\partial_x^2\partial_zP^{N_2}_{\neq}\|^2_{L^2L^2}}
			{A^{\frac{5}{3}-\frac32\epsilon}}\leq 
			\frac{\|{\rm e}^{\frac32aA^{-\frac{1}{3}}t}
				\partial_x^2 n_{\neq}\|^2_{L^2L^2}}
			{A^{\frac{5}{3}-\frac32\epsilon}}\leq 
			\frac{CE_4^2}
			{A^{\frac{4}{3}-\frac32\epsilon}}\leq C.
		\end{equation*}
		Using Lemma \ref{u1_hat1}, when $ A>A_{8}, $ we have
		\begin{equation*}
			\frac{\|e^{\frac32aA^{-\frac13}t}\widetilde{ u_{1,0}}\partial_{x}^{3}u_{3,\neq}\|_{L^{2}L^{2}}^{2}}{A^{\frac53-\frac32\epsilon}}\leq\frac{CA\|\widetilde{ u_{1,0}}\|_{L^{\infty}H^{2}}^{2}\|\partial_{x}^{2}u_{3,\neq}\|_{X_{\frac32a}}^{2}}{A^{\frac53-\frac32\epsilon}}\leq\frac{CE_{5}^{2}}{A^{2\epsilon}}\leq C.
		\end{equation*}

		Combining above, we conclude that 
		$$A^{\frac32\epsilon}\|\partial_x^2u_{3,\neq}\|^2_{X_{\frac32a}}\leq C\big(\|(\partial_x^2n_{\rm in})_{\neq}\|^2_{L^2}
		+\|(\partial_z^2n_{\rm in})_{\neq}\|^2_{L^2}+1\big).$$
	\end{proof}
	So far, due to $W=u_{2,\neq}+\kappa u_{3,\neq}$, we can complete the energy estimation of $E_5$ with 
	\begin{equation*}
		E_{5}(t)=A^{\frac{3}{4}\epsilon}\left(\|\partial^2_{x}u_{2,\neq}\|_{X_{\frac{3}{2}a}}+
		\|\partial_{x}^2u_{3,\neq}\|_{X_{\frac{3}{2}a}}\right)\leq C
		A^{\frac{3}{4}\epsilon}
		\left(\|\partial_x\nabla W\|_{X_{\frac{3}{2}a}}+
		\|\partial_{x}^2u_{3,\neq}\|_{X_{\frac{3}{2}a}}\right).
	\end{equation*}
	
	\begin{corollary}
		Under the conditions of Theorem \ref{result} and the assumptions (\ref{assumption}),
		according to Corollary \ref{c1} and Lemma \ref{c2}, when $A\geq A_8:=C_{(4)},$
		there holds
		\begin{equation}
			E_5(t)\leq  C\big(\|(\partial_x^2n_{\rm in})_{\neq}\|^2_{L^2}
			+\|(\partial_z^2n_{\rm in})_{\neq}\|^2_{L^2}+1\big):=E_5.
		\end{equation}
	\end{corollary}
	\appendix

	\section{Space-time estimates and Gagliardo-Nirenberg inequalities}
	
	First, the following one-dimensional Gagliardo-Nirenberg interpolation inequalities are frequently used, which can be found in \cite{Na1} and \cite{LW1}.
	\begin{lemma} For $\tau\in \mathbb{R}$, there hold
		\ben\label{eq:1DGN-1}
		\|h(\tau)\|_{L^{\infty}}\leq \|h(\tau)\|_{L^{2}}^{\frac12}\|h'(\tau)\|_{L^{2}}^{\frac12},
		\een
		\ben\label{eq:1DGN-2}
		\|h(\tau)\|_{L^{2}}\leq  \left(\frac{16\pi^2}{27}\right)^{-\frac{1}{6}}\|h(\tau)\|_{L^{1}}^{\frac23}\|h'(\tau)\|_{L^{2}}^{\frac13}.
		\een
	\end{lemma}

	Next, we recall the space-time estimate of the following equation (see Proposition 4.1 in \cite{wei2}).
	Let
	\begin{equation}\label{time_11}
		\begin{aligned}
			\partial_t f-\frac{1}{A}\triangle f+y\partial_xf=\partial_xf_{1}+f_{2}
			+\nabla\cdot f_{3},\quad t\in[0,T]
		\end{aligned}
	\end{equation}
	where  $f_{1}$, $f_{2}$ and $f_{3}$ are given functions.
	\begin{proposition}\label{timespace0}
		Let $f$	be a solution of (\ref{time_11}) and $ P_{0}f=P_{0}f_{1}=P_{0}f_{2}=P_{0}f_{3}=0, $ then there holds
		\begin{equation*}
			\begin{aligned}
				&\|{\rm e}^{aA^{-\frac{1}{3}}t}f_{\neq}\|^2_{L^{\infty}L^{2}}
				+\frac{1}{A^{\frac{1}{3}}}\|{\rm e}^{aA^{-\frac{1}{3}}t} f_{\neq}\|^2_{L^{2}L^{2}}
				+\frac{1}{A}\|{\rm e}^{aA^{-\frac{1}{3}}t}\nabla f_{\neq}\|^2_{L^{2}L^{2}}
				+\|{\rm e}^{aA^{-\frac13}t}\nabla\triangle^{-1}\partial_{x}f_{\neq}\|_{L^{2}L^{2}}^{2} \\
				&\quad \leq C\left(\|f_{\rm in,\neq}\|_{L^2}^2
				+\|{\rm e}^{aA^{-\frac{1}{3}}t}\nabla f_{1,\neq}\|_{L^2L^2}^2
				+A^{\frac{1}{3}}\|{\rm e}^{aA^{-\frac{1}{3}}t}f_{2,\neq}\|_{L^2L^2}^2 +A\|{\rm e}^{aA^{-\frac{1}{3}}t}f_{3,\neq}\|_{L^2L^2}^2\right),
			\end{aligned}
		\end{equation*}
		where ``a'' is a given non-negative constant.
	\end{proposition}

	To estimate the coupled terms $(\triangle u_2,\partial_x\omega_2),$ we have 
	the following proposition, which will be proved later.
	\begin{proposition}\label{timespace1}
		Assume that $(h_1,h_2)$ satisfy
		\begin{equation*}\label{h1}
			\left\{
			\begin{array}{lr}
				\partial_th_1-\frac{1}{A}\triangle h_1+y\partial_x  h_1=
				\nabla\cdot g_1 , \\
				
				\\
				\partial_th_2-\frac{1}{A}\triangle h_2+y\partial_x  h_2+\partial_x\partial_z\triangle^{-1}h_1
				=\nabla\cdot g_2,
			\end{array}
			\right.
		\end{equation*}
		for $ t\in[0,T], $ where $ h_{1},h_{2}, g_{1} $ and $ g_{2} $ are given functions and $ P_{0}h_{1}=P_{0}h_{2}=P_{0}g_{1}=P_{0}g_{2}=0.$ Then for $ b\geq 0, $ it holds that 
		\begin{equation*}
			\|h_{1,\neq}\|^2_{X_b}+	\|h_{2,\neq}\|^2_{X_b}\leq C\left(\|(h_{1,\rm in})_{\neq}\|^2_{L^2}+
			\|(h_{2,\rm in})_{\neq}\|^2_{L^2}
			+A\|{\rm e}^{bA^{-\frac{1}{3}}t}g_{1,\neq}\|_{L^2L^2}^2
			+A\|{\rm e}^{bA^{-\frac{1}{3}}t}g_{2,\neq}\|_{L^2L^2}^2
			\right).
		\end{equation*}
	\end{proposition}

	We also need the following proposition (see Proposition 4.4 in \cite{wei2}).
	\begin{proposition}\label{timespace2}
		Assume that $f$ satisfies 
		$$\partial_t f-\frac{1}{A}\triangle f+\left(y+\frac{\widehat{u_{1,0}}}{A}\right)\partial_xf=\partial_xf_{1}+f_{2}
		+\nabla\cdot f_{3},$$
		for $ t\in[0,T], $ where $f$, $f_{1}$, $f_{2}$ and $f_{3}$ are given functions and $ P_{0}f=P_{0}f_{1}=P_{0}f_{2}=P_{0}f_{3}=0 $.
		As long as $$\frac{\|\widehat{u_{1,0}}\|_{H^4}}{A}+\|\partial_t\widehat{u_{1,0}}\|_{H^2}<c,$$
		for some small $c$ independent of $A$ and $T,$ then for $ b\geq0 $ there holds  
		\begin{equation*}
			\begin{aligned}
				\|f_{\neq}\|^2_{X_b}
				\leq C\left(\|(f_{\rm in})_{\neq}\|_{L^2}^2
				+\|{\rm e}^{bA^{-\frac{1}{3}}t}\nabla f_{1,\neq}\|_{L^2L^2}^2
				+A^{\frac{1}{3}}\|{\rm e}^{bA^{-\frac{1}{3}}t}f_{2,\neq}\|_{L^2L^2}^2 +A\|{\rm e}^{bA^{-\frac{1}{3}}t}f_{3,\neq}\|_{L^2L^2}^2\right).
			\end{aligned}
		\end{equation*}
	\end{proposition}
	
	The following proposition can be derived from Proposition 4.9 in \cite{wei2} by the same way, and we omitted it.
	\begin{proposition}\label{timespace4}
		Assume that $f$ satisfies 
		$$\partial_t f-\frac{1}{A}\triangle f+\left(y+\frac{\widehat{u_{1,0}}}{A}\right)\partial_xf-2(\partial_y+\kappa\partial_z)
		\triangle^{-1}(\partial_yV\partial_xf)=f_{1}+f_{2},$$
		for $ t\in[0,T], $ where $ f, f_{1} $ and $ f_{2} $ are given functions and $ P_{0}f=P_{0}f_{1}=P_{0}f_{2}=0. $
		As long as $$\frac{\|\widehat{u_{1,0}}\|_{H^4}}{A}+\|\partial_t\widehat{u_{1,0}}\|_{H^2}<c,$$
		for some small $c$ independent of $A$ and $T,$ then for $ b\geq 0, $ there holds  
		\begin{equation*}
			\begin{aligned}
				\|\triangle f_{\neq}\|^2_{X_b}
				\leq C\left(\|f_{\rm in,\neq}\|_{H^2}^2
				+A\|{\rm e}^{bA^{-\frac{1}{3}}t}\nabla f_{1,\neq}\|_{L^2L^2}^2+A^{\frac13}\|e^{bA^{-\frac13}t}\triangle f_{2,\neq}\|_{L^{2}L^{2}}^{2}\right).
			\end{aligned}
		\end{equation*}
	\end{proposition}

	\begin{proof}[Proof of Proposition \ref{timespace1}] We follow the same route as in \cite{wei2}.
		By using new variables $(\tilde{x},\tilde{y},\tilde{z})=(x-ty,y,z)$,  let $\tilde{h}_{k,\neq}(\tilde{x},\tilde{y},\tilde{z})
		=h_{k,\neq}(x,y,z)$
		and $\tilde{g}_k(\tilde{x},\tilde{y},\tilde{z})=g_k(x,y,z),$ where $k=1,2$.
		After dropping the superscript, we get
		\begin{align*}
			&\partial_th_1-\frac{1}{A}\triangle_{L} 
			h_1=\nabla_{L}\cdot g_1,\\
			&\partial_th_2-\frac{1}{A}\triangle_{L} 
			h_2=-\partial_x\partial_z\triangle_{L}^{-1}h_1+\nabla_{L}\cdot g_2,
		\end{align*}
		where $\nabla_{L}=(\partial_x,\partial_y-t\partial_x,\partial_z).$
		
		Doing the Fourier transform, we get 
		\begin{equation}\label{f_time_1}
			\begin{array}{lr}
				\partial_t\widehat{h_1}+\frac{r(t)}{A}\widehat{h_1}=
				i(k_1,k_2-tk_1,k_3)\cdot \widehat{g}_1,\\
				\partial_t\widehat{h_2}+\frac{r(t)}{A}
				\widehat{h_2}=-k_1k_3r(t)^{-1}\widehat{h_1}
				+i(k_1,k_2-tk_1,k_3)\cdot \widehat{g_2},
			\end{array}
		\end{equation}
		where $r(t)={k_1^2}+\left(k_2-tk_1\right)^2+{k_3^2}.$
		
		First, we study the following equation:
		\begin{equation}\label{f_time_2}
			\begin{array}{lr}
				\partial_t\widehat{f}+\frac{1}{A} r(t) \widehat{f}={ik_1}\widehat{f}^1+\widehat{f}^2.
			\end{array}
		\end{equation}	
		The solution of (\ref{f_time_2}) is given by
		\begin{equation*}
			\begin{aligned}
				\widehat{f}(t)={\rm e}^{-\frac{1}{A} r_1(t)}\widehat{f}(0)+\int_0^t {\rm e}^{-\frac{1}{A}(r_1(t)-r_1(s))}\left(ik_1\widehat{f}^1(s)
				+\widehat{f}^2(s)\right)ds:=F_{(0)}+F_{(1)}+F_{(2)},
			\end{aligned}
		\end{equation*}
		where 
		$r_1(t)=\int_0^t\left(k_1^2+\left(k_2-sk_1\right)^2+{k_3^2}\right)ds.$
		For $t_1>t_2$, there holds
		\begin{equation*}
			\begin{aligned}
				{r_1(t_1)-r_1(t_2)}&=\int_{t_2}^{t_1}
				\left(k_1^2+\left(k_2-sk_1\right)^2+k_3^2\right)ds
				\geq \frac{(t_1-t_2)^3k_1^2}{12}.
			\end{aligned}
		\end{equation*}
		Thus, for $k_1\in\mathbb{Z}$ and $k_1\neq 0$, we have
		\begin{equation}\label{nu_s}
			{A}^{-1}\left(r_1(t)-r_1(s)\right)\geq {A}^{-1}\frac{(t-s)^3 k_1^2}{12}
			\geq (b+1)A^{-\frac{1}{3}}(t-s)-C,	
		\end{equation}
		where $b$ is a positive constant.
		Thanks to (\ref{nu_s}), thus   
		$$|F_{(0)}|=|{\rm e}^{-{A}^{-1} r_1(t)}\widehat{f}(0)|\leq C{\rm e}^{-(b+1)A^{-\frac{1}{3}}t}|\widehat{f}(0)|,$$
		which implies that 
		\begin{equation}\label{f0_result}
			\begin{aligned}
				&\|{\rm e}^{bA^{-\frac{1}{3}}t}F_{(0)}\|_{L^{\infty}(0,T)}\leq C|\widehat{f}(0)|,~~~
				&\|{\rm e}^{bA^{-\frac{1}{3}}t}F_{(0)}\|^2_{L^{2}(0,T)}\leq CA^{\frac{1}{3}}|\widehat{f}(0)|^2.
			\end{aligned}
		\end{equation}
		Using H$\rm \ddot{o}$lder's inequality, we obtain that 
		\begin{equation}\label{F1_equ1}
			\begin{aligned}
				|F_{(1)}|&\leq C\int_0^t {\rm e}^{(b+1)A^{-\frac{1}{3}}(s-t)}
				|k_1\widehat{f}^1(s)|ds
				\\&\leq C\|k_1r(s)^{-\frac{1}{2}}\|_{L^2(0,t)}
				\|{\rm e}^{(b+1)A^{-\frac{1}{3}}(s-t)}r(s)^{\frac{1}{2}}
				\widehat{f}^1(s)\|_{L^2(0,t)}\\
				&\leq C\|{\rm e}^{(b+1)A^{-\frac{1}{3}}(s-t)}r(s)^{\frac{1}{2}}
				\widehat{f}^1(s)\|_{L^2(0,t)}
				|k_1|^\frac{1}{2}{(k_1^2+k_3^2)^{-\frac{1}{4}}},	
			\end{aligned}
		\end{equation}
		where we use 
		$\|k_1r(s)^{-\frac{1}{2}}\|^2_{L^2(0,T)}=
		\int_0^T \frac{k_1^2}{k_1^2+(k_2-sk_1)^2+k_3^2}ds
		\leq \frac{|k_1|\pi}{(k_1^2+k_3^2)^{\frac{1}{2}}}.$
		Using (\ref{F1_equ1}), 
		we have 
		\begin{equation}\label{f1_result}
			\begin{aligned}
				&\|{\rm e}^{bA^{-\frac{1}{3}}t}F_{(1)}\|^2_{L^{\infty}(0,T)}\leq C{|k_1|(k_1^2+k_3^2)^{-\frac{1}{2}}}
				\|{\rm e}^{bA^{-\frac{1}{3}}t}r(t)^{\frac{1}{2}}\widehat{f}^1(t)\|_{L^2(0,T)}^{2},\\
				&\|{\rm e}^{bA^{-\frac{1}{3}}t}F_{(1)}\|^2_{L^2(0,T)} 
				\leq C{|k_1|(k_1^2+k_3^2)^{-\frac{1}{2}}}
				\int_{0}^{T}\|{\rm e}^{bA^{-\frac{1}{3}}s
					+bA^{-\frac{1}{3}}(s-t)}r(s)^{\frac{1}{2}}\widehat{f}^1(s)\|^2_{L^2(0,t)}dt\\
				&\leq C{|k_1|(k_1^2+k_3^2)^{-\frac{1}{2}}}
				\int_{0}^{T}\int_{s}^{T}{\rm e}^{2bA^{-\frac{1}{3}}s+2bA^{-\frac{1}{3}}(s-t)}r(s)|\widehat{f}^1(s)|^2dtds\\
				&\leq 
				CA^{\frac{1}{3}}\|{\rm e}^{bA^{\frac{1}{3}}t}r(t)^{\frac{1}{2}}\widehat{f}^1(t)\|_{L^2(0,T)}^2
				{|k_1|(k_1^2+k_3^2)^{-\frac{1}{2}}}.
			\end{aligned}
		\end{equation}
		Using integration by parts and (\ref{nu_s}), we get
		\begin{align*}
			&\qquad\|{\rm e}^{-{A}^{-1}\left(r_1(t)-r_1(s)\right)}
			{\rm e}^{(b+\frac12)A^{-\frac13 }(t-s)}r(s)^\frac{1}{2}\|_{L^2(0,t)}^2\\
			&=\int_0^t {\rm e}^{-2{A}^{-1}\left(r_1(t)-r_1(s)\right)}
			{\rm e}^{(2b+1)A^{-\frac13 }(t-s)}r(s)ds
			=\frac{A}{2}\int_0^t 
			{\rm e}^{(2b+1)A^{-\frac13 }(t-s)}d{\rm e}^{-(\frac{2}{A}\left(r_1(t)-r_1(s)\right)}\\
			&=\frac{A}{2}{\rm e}^{(2b+1)A^{-\frac13 }(t-s)-2A^{-1}\left(r_1(t)-r_1(s)\right)}|_{s=0}^{s=t}
			-\frac{A}{2}\int_0^t {\rm e}^{-2A^{-1}\left(r_1(t)-r_1(s)\right)}
			d{\rm e}^{(2b+1)A^{-\frac13 }(t-s)}\\
			&\leq \frac{A}{2}+\frac{CA}{2}(2b+1)A^{-\frac13}\int_{0}^{t}e^{-2(b+1)A^{-\frac13}(t-s)}e^{(2b+1)A^{-\frac13}(t-s)}ds\leq CA,
		\end{align*}
		then we have 
		\begin{align*}
			|F_{(2)}|&\leq 
			\int_0^t {\rm e}^{-A^{-1}\left(r_1(t)-r_1(s)\right)}|\widehat{f}^2(s)|ds
			=\int_0^t {\rm e}^{-A^{-1}\left(r_1(t)-r_1(s)\right)}
			{\rm e}^{(b+\frac12)A^{-\frac13 }(t-s)}
			{\rm e}^{(b+\frac12)A^{-\frac13 }(s-t)}|\widehat{f}^2(s)|ds\\
			&\leq \|{\rm e}^{-A^{-1}\left(r_1(t)-r_1(s)\right)}
			{\rm e}^{(b+\frac12)A^{-\frac13 }(t-s)}r(s)^\frac{1}{2}\|_{L^2(0,t)}
			\|{\rm e}^{(b+\frac12)A^{-\frac13 }(s-t)}
			r(s)^{-\frac{1}{2}}\widehat{f}^2(s)\|_{L^2(0,t)}\\
			&\leq CA^{\frac{1}{2}}	\|{\rm e}^{(b+\frac12)A^{-\frac13 }(s-t)}
			r(s)^{-\frac{1}{2}}\widehat{f}^2(s)\|_{L^2(0,t)}
			\leq CA^{\frac{1}{2}}	\|{\rm e}^{bA^{-\frac13 }(s-t)}
			r(s)^{-\frac{1}{2}}\widehat{f}^2(s)\|_{L^2(0,t)},
		\end{align*}
		which implies that 
		\begin{equation}\label{f3_result}
			\begin{aligned}
				&\|{\rm e}^{bA^{-\frac13 }t}F_{(2)}\|_{L^{\infty}(0,T)}\leq 
				CA^{\frac{1}{2}}
				\|{\rm e}^{bA^{-\frac13 }t}
				r(t)^{-\frac{1}{2}}\widehat{f}^2(t)\|_{L^2(0,T)},\\
				&\|{\rm e}^{bA^{-\frac13 }t}F_{(2)}\|^2_{L^{2}(0,T)}
				\leq CA^{\frac{4}{3}}\|{\rm e}^{bA^{-\frac13 }t}
				r(t)^{-\frac{1}{2}}\widehat{f}^2(t)\|^2_{L^2(0,T)}.
			\end{aligned}	
		\end{equation}
		Combining (\ref{f0_result}), (\ref{f1_result}), (\ref{f3_result}) 
		and 
		\begin{align}\label{eq:k1estimate}
			\|k_1r(t)^{-\frac{1}{2}}{\rm e}^{bA^{-\frac13 }t}\widehat{f}\|_{L^2(0,T)}^2
			&\leq\|k_1r(t)^{-\frac{1}{2}}\|_{L^2(0,T)}^2
			\|{\rm e}^{bA^{-\frac13 }t}\widehat{f}\|_{L^{\infty}(0,T)}^2\nonumber\\
			&\leq C|k_1|(k_1^2+k_3^2)^{-\frac{1}{2}}
			\|{\rm e}^{bA^{-\frac13 }t}\widehat{f}\|_{L^{\infty}(0,T)}^2,
		\end{align}
		we get
		\begin{equation}\label{f_result_1}
			\begin{aligned}
				&\|{\rm e}^{bA^{-\frac13 }t}\widehat{f}\|_{L^{\infty}(0,T)}^2+
				A^{-\frac{1}{3}}\|{\rm e}^{bA^{-\frac13 }t}\widehat{f}\|_{L^{2}(0,T)}^2+
				\|k_1r(t)^{-\frac{1}{2}}{\rm e}^{bA^{-\frac13 }t}\widehat{f}\|_{L^2(0,T)}^2
				\\ \leq& C\left(|\widehat{f}(0)|^2
				+\frac{|k_{1}|}{(k_{1}^{2}+k_{3}^{2})^{\frac12}}\|{\rm e}^{bA^{-\frac{1}{3}}t}r(t)^{\frac{1}{2}}\widehat{f}^1(t)\|^2_{L^2(0,T)}
				+A\|{\rm e}^{bA^{-\frac13 }t}
				r(t)^{-\frac{1}{2}}\widehat{f}^2(t)\|^2_{L^2(0,T)}\right).
			\end{aligned}
		\end{equation}
		Due to 
		\begin{equation*}
			\partial_t\left({\rm e}^{bA^{-\frac13 }t}\widehat{f}\right)+A^{-1} r(t){\rm e}^{bA^{-\frac13 }t}\widehat{f}={ik_1}{\rm e}^{bA^{-\frac13 }t}\widehat{f}^1+{\rm e}^{bA^{-\frac13 }t}\widehat{f}^2+bA^{-\frac{1}{3}}{\rm e}^{bA^{-\frac13 }t}\widehat{f},
		\end{equation*}
		multiplying the complex conjugate of ${\rm e}^{bA^{-\frac13 }t}\widehat{f}$ and using $\rm H\ddot{o}$lder's inequality,  one obtains 
		\begin{equation*}
			\begin{aligned}
				&\quad \quad A^{-1} \|r(t)^{\frac12}{\rm e}^{bA^{-\frac13 }t}\widehat{f}\|^2_{L^2(0,T)}\\
				&\leq 
				C\Big(|\widehat{f}(0)|^2+A^{-\frac{1}{3}}\|{\rm e}^{bA^{-\frac13 }t}\widehat{f}\|^2_{L^2(0,T)}+\|r(t)^{\frac12}{\rm e}^{bA^{-\frac13 }t}\widehat{f}^1\|_{L^2(0,T)}
				\|{k_1}r(t)^{-\frac12}{\rm e}^{bA^{-\frac13 }t}\widehat{f}\|_{L^2(0,T)}\\
				&\quad 
				+\|r(t)^{-\frac12}{\rm e}^{bA^{-\frac13 }t}\widehat{f}^2\|_{L^2(0,T)}
				\|e^{bA^{-\frac13}t}r(t)^{\frac12}{\rm e}^{bA^{-\frac13 }t}\widehat{f}\|_{L^2(0,T)}
				\Big),
			\end{aligned}
		\end{equation*}
		which along with \eqref{eq:k1estimate} and (\ref{f_result_1}) give that 
		\begin{equation}\label{f_result_2}
			\begin{aligned}
				\|\widehat{f}\|^2_{Z_b}\leq C\Big(|\widehat{f}(0)|^2
				+\|{\rm e}^{bA^{-\frac{1}{3}}t}r(t)^{\frac{1}{2}}\widehat{f}^1(t)\|^2_{L^2(0,T)}{|k_1|(k_1^2+k_3^2)^{-\frac{1}{2}}}
				+A\|{\rm e}^{bA^{-\frac13 }t}
				r(t)^{-\frac{1}{2}}\widehat{f}^2(t)\|^2_{L^2(0,T)}\Big),
			\end{aligned}
		\end{equation}
		where  $$\|\widehat{f}\|^2_{Z_b}=
		\|{\rm e}^{bA^{-\frac13 }t}\widehat{f}\|_{L^{\infty}(0,T)}^2+
		\frac{\|{\rm e}^{bA^{-\frac13 }t}\widehat{f}\|_{L^{2}(0,T)}^2}{A^{\frac13}}+
		\|k_1r(t)^{-\frac{1}{2}}{\rm e}^{bA^{-\frac13 }t}\widehat{f}\|_{L^2(0,T)}^2
		+\frac{\|r(t)^{\frac12}{\rm e}^{bA^{-\frac13 }t}\widehat{f}\|^2_{L^2(0,T)}}{A}.$$

		Then, for (\ref{f_time_1}), we have 
		\begin{equation}\label{h12_time}
			\begin{aligned}
				&\|\widehat{h_1}\|^2_{Z_b}
				\leq C(|\widehat{h_1}(0)|^2
				+A\|{\rm e}^{bA^{-\frac13 }t}
				\widehat{g_1}(t)\|^2_{L^2(0,T)}),\\
				&\|\widehat{h_2}\|^2_{Z_b}
				\leq C(|\widehat{h_2}(0)|^2
				+A\|{\rm e}^{bA^{-\frac13 }t}
				\widehat{g_2}(t)\|^2_{L^2(0,T)}
				+\frac{|k_1|\|k_3r(t)^{-\frac{1}{2}}{\rm e}^{bA^{-\frac13 }t}\widehat{h}_1\|_{L^{2}(0,T)}^2}
				{(k_1^2+k_3^2)^{\frac{1}{2}}}).
			\end{aligned}
		\end{equation}
		Due to $ \int_0^T \frac{|k_1|}{k_1^2+(k_2-sk_1)^2+k_3^2}ds
		\leq \frac{\pi}{(k_1^2+k_3^2)^{\frac{1}{2}}}, $ there holds
		\begin{align*}
			&\qquad|k_1|(k_1^2+k_3^2)^{-\frac{1}{2}}
			\|k_3r(t)^{-\frac{1}{2}}{\rm e}^{bA^{-\frac13 }t}\widehat{h}_1\|_{L^{2}(0,T)}^2
			\\
			&\leq\|{\rm e}^{bA^{-\frac13 }t}\widehat{h}_1\|_{L^{\infty}(0,T)}^2
			|k_1|(k_1^2+k_3^2)^{-\frac{1}{2}}
			\|k_3r(t)^{-\frac{1}{2}}\|_{L^{2}(0,T)}^2\\
			&\leq\|{\rm e}^{bA^{-\frac13 }t}\widehat{h}_1\|_{L^{\infty}(0,T)}^2
			\frac{k_3^2}{(k_1^2+k_3^2)^{\frac{1}{2}}}
			\int_0^T \frac{|k_1|}{k_1^2+(k_2-sk_1)^2+k_3^2}ds\\
			&\leq C\|{\rm e}^{bA^{-\frac13 }t}\widehat{h}_1\|_{L^{\infty}(0,T)}^2
			\frac{k_3^2}{k_1^2+k_3^2}
			\leq C\|{\rm e}^{bA^{-\frac13 }t}\widehat{h}_1\|_{L^{\infty}(0,T)}^2,
		\end{align*}	
		which along with (\ref{h12_time}) imply that 
		\begin{equation}\label{h12_1}
			\|\widehat{h_1}\|^2_{Z_b}+\|\widehat{h_2}\|^2_{Z_b}
			\leq C\Big(|\widehat{h_1}(0)|^2+|\widehat{h_2}(0)|^2
			+A\|{\rm e}^{bA^{-\frac13 }t}\widehat{g_1}(t)\|^2_{L^2(0,T)}
			+A\|{\rm e}^{bA^{-\frac13 }t}\widehat{g_2}(t)\|^2_{L^2(0,T)}\Big).
		\end{equation}
		
		For \eqref{h12_1}, with the help of Plancherel's theorem, integrating $k_2$ in $\mathbb{R}$ and summing up in 
		$k_1,k_3\in\mathbb{Z}$ with $k_1\neq0,$ we complete the proof.
	\end{proof}

	\section{Local well-posedness and blow-up criterion}
	 In this section, we aim to establish the local well-posedness results of strong solutions to the system (\ref{ini11}), which state as follows.
	 \begin{theorem}\label{thm:local existence}
	Assume that the non-negative initial data $ n_{\rm in}(x,y,z)\in H^2\cap L^1(\mathbb{T}\times\mathbb{R}\times\mathbb{T})$
	and $ u_{\rm in}(x,y,z)\in H^2
	(\mathbb{T}\times\mathbb{R}\times\mathbb{T})$.
	Then there exist a positive constant $T^{*}=T^{*}(\|n_{\rm in}\|_{H^2\cap L^1}, \|u_{\rm in}\|_{H^2} )$  and a unique strong solution $(n,u)\in C([0,T^*); H^2(\mathbb{T}\times\mathbb{R}\times\mathbb{T}) )$  for the system of {\rm (\ref{ini11})}, with the initial data $(n_{\rm in},u_{\rm in})$. Moreover, if the strong solution of {\rm (\ref{ini11})} can be uniquely continued to a maximal existence
time $T^*$, then 
\ben\label{eq:blow-up condi}
\lim_{t\nearrow T^*}A^{-\frac{1}{12}}\|(\nabla u)(t,\cdot)\|_{L^{2}}+
	\|(\partial_{x}^{2},\partial_{z}^{2})n(t,\cdot)\|_{L^{2}}+\|n(t,\cdot)\|_{L^{\infty}}=+\infty.
\een
\end{theorem}

\begin{proof} 
We sketch the proof and it is divided into four parts.

{\bf Step I. A priori estimates.} By Cauchy-Lipschitz theorem for the standard approximation system (for example, see Sec 3.5 in \cite{Vicol2022}), it is enough to close the energy estimates for the terms of the higher order derivatives. For $A>1$, claim that
\ben\label{eq:higher order}
&&\partial_{t}\left(\|\nabla^{2}n\|_{L^{2}}^{2}+\|\nabla^{2}u\|_{L^{2}}^{2} \right)+\frac{1}{A}\left(\|\nabla^{3}n\|_{L^{2}}^{2}+\|\nabla^{3}u\|_{L^{2}}^{2} \right)\nonumber\\&\leq&\frac{C}{A}\|\nabla n\|_{L^{2}}^{4}\|\nabla^{2}u\|_{L^{2}}^{2}+\frac{C}{A}\|\nabla n\|_{L^{2}}^{2}\|\nabla u\|_{L^{2}}^{8}+\frac{C}{A}\|\nabla^{2}n\|_{L^{2}}^{2}\|n\|_{L^{2}}^{\frac12}\|\nabla n\|_{L^{2}}^{\frac32}\nonumber\\&&+\frac{C}{A}\|\nabla^{2}n\|_{L^{2}}^{2}\|n\|_{L^{2}}\|\nabla n\|_{L^{2}}+\frac{C}{A}\|n\|_{L^{2}}^{\frac12}\|\nabla^{2}n\|_{L^{2}}^{\frac32}\|\nabla n\|_{L^{2}}^{2}\nonumber\\&&+\frac{C}{A}\|\nabla u\|_{L^{2}}^{10}+C\|\nabla^{2}n\|_{L^{2}}^{2}+C\|\nabla^{2}u\|_{L^{2}}^{2}.
\een

It follows from the first equation of (\ref{ini11}) that
		\begin{equation*}
			\partial_{t}\partial_{ij}n+\partial_{ij}(y\partial_{x}n)+\frac{1}{A}\partial_{ij}(u\cdot\nabla n)-\frac{1}{A}\partial_{ij}\triangle n=-\frac{1}{A}\partial_{ij}\nabla\cdot(n\nabla c),
		\end{equation*}
		with $ i,j=1,2,3. $ Multiplying it by $ \partial_{ij}n $ and integrating  over $ \mathbb{T}\times\mathbb{R}\times\mathbb{T} $, we obtain
		\begin{equation}\label{eq B}
			\begin{aligned}
			&\frac12\partial_{t}\|\nabla^{2}n\|_{L^{2}}^{2}+\frac{1}{A}\|\nabla^{3}n\|_{L^{2}}^{2}\\ \leq &-\frac{1}{A}\int_{\mathbb{T}\times\mathbb{R}\times\mathbb{T}}\partial_{ij}(u\cdot\nabla n)\partial_{ij}n-\frac{1}{A}\int_{\mathbb{T}\times\mathbb{R}\times\mathbb{T}}\partial_{ij}\nabla\cdot(n\nabla c)\partial_{ij}n+ C\|\nabla^{2}n\|_{L^{2}}^{2}\\
:=&B_{1}+B_{2}+C\|\nabla^{2}n\|_{L^{2}}^{2}.
			\end{aligned}
		\end{equation}
		By integration by parts, Gagliardo-Nirenberg inequality and $ \nabla\cdot u=0, $ we have
		\begin{equation*}
			\begin{aligned}
			B_{1}=&-\frac{1}{A}\int_{\mathbb{T}\times\mathbb{R}\times\mathbb{T}}\partial_{ij}u\cdot\nabla n\partial_{ij}n-\frac{1}{A}\int_{\mathbb{T}\times\mathbb{R}\times\mathbb{T}}u\cdot\nabla\partial_{ij}n\partial_{ij}n-\frac{2}{A}\int_{\mathbb{T}\times\mathbb{R}\times\mathbb{T}}\partial_{j}u\cdot\nabla\partial_{i}n\partial_{ij}n
\\=&\frac{1}{A}\int_{\mathbb{T}\times\mathbb{R}\times\mathbb{T}}n\partial_{ij}u\cdot\nabla\partial_{ij}n+\frac{2}{A}\int_{\mathbb{T}\times\mathbb{R}\times\mathbb{T}}\partial_{i}n\partial_{j}u\cdot\nabla\partial_{ij}n
\\\leq&\frac{1}{4A}\|\nabla^{3}n\|_{L^{2}}^{2}+\frac{C}{A}\|n\|_{L^{6}}^{2}\|\nabla^{2}u\|_{L^{3}}^{2}+\frac{C}{A}\|\nabla n\|_{L^{\infty}}^{2}\|\nabla u\|_{L^{2}}^{2}\\\leq&\frac{1}{4A}\|\nabla^{3}n\|_{L^{2}}^{2}+\frac{C}{A}\|\nabla n\|_{L^{2}}^{2}\|\nabla^{2}u\|_{L^{2}}\|\nabla^{3}u\|_{L^{2}}+\frac{C}{A}\|\nabla n\|_{L^{2}}^{\frac12}\|\nabla^{3}n\|_{L^{2}}^{\frac32}\|\nabla u\|_{L^{2}}^{2}\\\leq&\frac{1}{2A}\|\nabla^{3}n\|_{L^{2}}^{2}+\frac{1}{2A}\|\nabla^{3}u\|_{L^{2}}^{2}+\frac{C}{A}\|\nabla n\|_{L^{2}}^{4}\|\nabla^{2}u\|_{L^{2}}^{2}+\frac{C}{A}\|\nabla n\|_{L^{2}}^{2}\|\nabla u\|_{L^{2}}^{8}.
			\end{aligned}
		\end{equation*}
For $ B_{2}, $	note that  similar estimates as in  Lemma \ref{lem:ellip_2} imply that 
\ben\label{eq:elliptic-c}
\sum_{j=0}^2\|\nabla^{j+k}c(t)\|_{L^2}^2
			\leq C\|\nabla^k n(t)\|_{L^2}^2,\quad k=0,1,2.
\een
By \eqref{eq:elliptic-c} and Gagliardo-Nirenberg inequality we get
		\begin{equation*}
		\begin{aligned}
		B_{2}=&\frac{1}{A}\int_{\mathbb{T}\times\mathbb{R}\times\mathbb{T}}\partial_{ij}(n\nabla c)\cdot\nabla\partial_{ij}n\\=&\frac{1}{A}\int_{\mathbb{T}\times\mathbb{R}\times\mathbb{T}}\partial_{ij}n\nabla c\cdot\nabla\partial_{ij}n+\frac{2}{A}\int_{\mathbb{T}\times\mathbb{R}\times\mathbb{T}}\partial_{i}n\partial_{j}\nabla c\cdot\nabla\partial_{ij}n+\frac{1}{A}\int_{\mathbb{T}\times\mathbb{R}\times\mathbb{T}}n\partial_{ij}\nabla c\cdot\nabla\partial_{ij}n\\\leq&\frac{1}{2A}\|\nabla^{3}n\|_{L^{2}}^{2}+\frac{C}{A}\|\nabla^{2}n\|_{L^{2}}^{2}\|\nabla c\|_{L^{\infty}}^{2}+\frac{C}{A}\|\nabla n\|_{L^{6}}^{2}\|\triangle c\|_{L^{3}}^{2}+\frac{C}{A}\|n\|_{L^{\infty}}^{2}\|\partial_{ij}\nabla c\|_{L^{2}}^{2}\\\leq&\frac{1}{2A}\|\nabla^{3}n\|_{L^{2}}^{2}+\frac{C}{A}\|\nabla^{2}n\|_{L^{2}}^{2}\|n\|_{L^{2}}^{\frac12}\|\nabla n\|_{L^{2}}^{\frac32}+\frac{C}{A}\|\nabla^{2}n\|_{L^{2}}^{2}\|n\|_{L^{2}}\|\nabla n\|_{L^{2}}+\frac{C}{A}\|n\|_{L^{2}}^{\frac12}\|\nabla^{2}n\|_{L^{2}}^{\frac32}\|\nabla n\|_{L^{2}}^{2}.
		\end{aligned}
		\end{equation*}
		Collecting the estimates of $ B_{1}-B_{2}, $ we get by $ (\ref{eq B}) $ that
		\begin{equation}\label{n''}
			\begin{aligned}
			&\partial_{t}\|\nabla^{2}n\|_{L^{2}}^{2}+\frac{1}{A}\|\nabla^{3}n\|_{L^{2}}^{2}\\\leq&\frac{1}{2A}\|\nabla^{3}u\|_{L^{2}}^{2}+\frac{C}{A}\|\nabla n\|_{L^{2}}^{4}\|\nabla^{2}u\|_{L^{2}}^{2}+\frac{C}{A}\|\nabla n\|_{L^{2}}^{2}\|\nabla u\|_{L^{2}}^{8}+C\|\nabla^{2}n\|_{L^{2}}^{2}\\&+\frac{C}{A}\|\nabla^{2}n\|_{L^{2}}^{2}\|n\|_{L^{2}}^{\frac12}\|\nabla n\|_{L^{2}}^{\frac32}+\frac{C}{A}\|\nabla^{2}n\|_{L^{2}}^{2}\|n\|_{L^{2}}\|\nabla n\|_{L^{2}}+\frac{C}{A}\|n\|_{L^{2}}^{\frac12}\|\nabla^{2}n\|_{L^{2}}^{\frac32}\|\nabla n\|_{L^{2}}^{2}.
			\end{aligned}
		\end{equation}
		
	Rewrite the third equation of (\ref{ini11}) as follows	
	\begin{equation*}
		\begin{aligned}
		\partial_{t}\partial_{ij}u+\partial_{ij}(y\partial_{x}u)+\left(
		\begin{array}{c}
		\partial_{ij}u_{2} \\
		0 \\
		0 \\
		\end{array}
		\right)-\frac{1}{A}\partial_{ij}\triangle u+\frac{1}{A}\partial_{ij}(u\cdot\nabla u)+\frac{1}{A}\partial_{ij}\nabla P=\left(
		\begin{array}{c}
		0 \\
		\frac{\partial_{ij}n}{A} \\
		0 \\
		\end{array}
		\right),
		\end{aligned}
	\end{equation*}
	with $ i,j=1,2,3. $ Multiplying it by $ \partial_{ij}u $ and integrating 
	over $ \mathbb{T}\times\mathbb{R}\times\mathbb{T} $, we obtain
	\begin{equation}\label{eq C}
		\begin{aligned}
		&\frac12\partial_{t}\|\nabla^{2}u\|_{L^{2}}^{2}+\frac{1}{A}\|\nabla^{3}u\|_{L^{2}}^{2}\\\leq &-\int_{\mathbb{T}\times\mathbb{R}\times\mathbb{T}}\partial_{ij}u_{2}\partial_{ij}u_{1}-\frac{1}{A}\int_{\mathbb{T}\times\mathbb{R}\times\mathbb{T}}\partial_{ij}(u\cdot\nabla u)\partial_{ij }u\\&-\frac{1}{A}\int_{\mathbb{T}\times\mathbb{R}\times\mathbb{T}}\partial_{ij}n\partial_{ij}u_{2}+C\|\nabla^{2}u\|_{L^{2}}^{2}:=C_{1}+C_{2}+C_{3}+C\|\nabla^{2}u\|_{L^{2}}^{2}.
		\end{aligned}
	\end{equation}
	For $ C_{1} $ and $ C_{3}, $  direct calculations indicate that
	\begin{equation*}
		C_{1}\leq \|\nabla^{2}u\|_{L^{2}}^{2},\quad C_{3}\leq\frac{1}{2A}\|\nabla^{2}n\|_{L^{2}}^{2}+\frac{1}{2A}\|\nabla^{2}u\|_{L^{2}}^{2}.
	\end{equation*}
	The estimate of $ C_{2} $ is similar as $ B_{1}, $ and we arrive
	\begin{equation*}
		C_{2}\leq \frac{1}{2A}\|\nabla^{3}u\|_{L^{2}}^{2}+\frac{C}{A}\|\nabla u\|_{L^{2}}^{10}.
	\end{equation*}	
		Collecting $ C_{1}-C_{3}, $ then (\ref{eq C}) yields that
		\begin{equation}\label{u''}
			\begin{aligned}
			\partial_{t}\|\nabla^{2}u\|_{L^{2}}^{2}+\frac{3}{2A}\|\nabla^{3}u\|_{L^{2}}^{2}\leq&\frac{C}{A}\|\nabla u\|_{L^{2}}^{10}+\frac{C}{A}\|\nabla^{2}n\|_{L^{2}}^{2}+C\|\nabla^{2}u\|_{L^{2}}^{2}.
			\end{aligned}
		\end{equation}
		
		Combining (\ref{n''}) with (\ref{u''}), the proof of \eqref{eq:higher order} is complete.

	 {\bf{Step II. Uniqueness.}}
	Let $ \tilde{n}=n^{1}-n^{2}, \tilde{c}=c^{1}-c^{2} $ and $ \tilde{u}=u^{1}-u^{2}, $ then direct calculation shows that $ (\tilde{n},\tilde{c},\tilde{u}) $ satisfies
	\begin{equation}\label{eq:tilde}
	\left\{
	\begin{array}{lr}
	\partial_{t}\tilde{n}+y\partial_{x}\tilde{n}+\frac{1}{A}\tilde{u}\cdot\nabla n^{1}+\frac{1}{A}u^{2}\cdot\nabla\tilde{n}-\frac{1}{A}\triangle\tilde{n}=-\frac{1}{A}\nabla\cdot(\tilde{n}\nabla c^{1})-\frac{1}{A}\nabla\cdot(n^{2}\nabla\tilde{c}), \\
	\triangle\tilde{c}+\tilde{n}-\tilde{c}=0, \\
	\partial_t\tilde{u}+y\partial_x \tilde{u}+\left(
	\begin{array}{c}
	\tilde{u}_{2} \\
	0 \\
	0 \\
	\end{array}
	\right)
	-\frac{1}{A}\triangle\tilde{u}+\frac{1}{A}\tilde{u}\cdot\nabla u^{1}+\frac{1}{A}u^{2}\cdot\nabla\tilde{u}+\frac{1}{A}\nabla \tilde{P}=\left(
	\begin{array}{c}
	0 \\
	\frac{\tilde{n}}{A} \\
	0 \\
	\end{array}
	\right).
	\end{array}
	\right.
	\end{equation}
	As $ \nabla\cdot u^{2}=\nabla\cdot\tilde{u}=0, $ the basic energy estimates give
	\begin{equation}\label{tilde n}
		\begin{aligned}
		\frac12\partial_{t}\|\tilde{n}\|_{L^{2}}^{2}+\frac{1}{A}\|\nabla\tilde{n}\|_{L^{2}}^{2}=&\frac{1}{A}\int_{\mathbb{T}\times\mathbb{R}\times\mathbb{T}}n^{1}\tilde{u}\cdot\nabla\tilde{n}\\&+\frac{1}{A}\int_{\mathbb{T}\times\mathbb{R}\times\mathbb{T}}\tilde{n}\nabla c^{1}\cdot\nabla\tilde{n}+\frac{1}{A}\int_{\mathbb{T}\times\mathbb{R}\times\mathbb{T}}n^{2}\nabla\tilde{c}\cdot\nabla\tilde{n}\\:=&I_{1}+I_{2}+I_{3}.
		\end{aligned}
	\end{equation}
	For $ I_{1}, $ using Sobolev inequality, we get
	\begin{equation*}
		I_{1}\leq\frac{1}{A}\|n^{1}\|_{L^{\infty}}\|\tilde{u}\|_{L^{2}}\|\nabla\tilde{n}\|_{L^{2}}\leq\frac{1}{4A}\|\nabla\tilde{n}\|_{L^{2}}^{2}+\frac{C}{A}\|n^{1}\|_{H^{2}}^{2}\|\tilde{u}\|_{L^{2}}^{2}.
	\end{equation*} 
	  For $ I_{2}, $ using \eqref{eq:elliptic-c} we have
	\begin{equation}\label{c1 n1}
		\|\nabla c^{1}\|_{L^{6}}\leq C\|\triangle c^{1}\|_{L^{2}}\leq C\|n^{1}\|_{L^{2}},
	\end{equation}
	and by Gagliardo-Nirenberg inequality, we obtain
	\begin{equation*}
		\begin{aligned}
		I_{2}\leq&\frac{1}{A}\|\nabla\tilde{n}\|_{L^{2}}\|\tilde{n}\|_{L^{3}}\|\nabla c^{1}\|_{L^{6}}\leq \frac{C}{A}\|\nabla\tilde{n}\|_{L^{2}}\|\tilde{n}\|_{L^{2}}^{\frac12}\|\nabla\tilde{n}\|_{L^{2}}^{\frac12}\|n^{1}\|_{L^{2}}\leq\frac{1}{4A}\|\nabla\tilde{n}\|_{L^{2}}^{2}
+\frac{C}{A}\|n^{1}\|_{L^{2}}^{4}\|\tilde{n}\|_{L^{2}}^{2}.
		\end{aligned}
	\end{equation*}
	Similarly,  $ I_{3} $ is bounded by
	\begin{equation*}
		\begin{aligned}
		I_{3}\leq&\frac{1}{A}\|n^{2}\|_{L^{\infty}}\|\nabla\tilde{c}\|_{L^{2}}\|\nabla\tilde{n}\|_{L^{2}}\leq\frac{C}{A}\|n^{2}\|_{L^{\infty}}\|\tilde{n}\|_{L^{2}}
\|\nabla\tilde{n}\|_{L^{2}}\\\leq&\frac{1}{4A}\|\nabla\tilde{n}\|_{L^{2}}^{2}+\frac{C}{A}\|n^{2}\|_{H^{2}}^{2}\|\tilde{n}\|_{L^{2}}^{2}.
		\end{aligned}
	\end{equation*}
	Collecting $ I_{1}-I_{3}, $ then (\ref{tilde n}) yields that
	\begin{equation}\label{end n}
		\begin{aligned} \frac12\partial_{t}\|\tilde{n}\|_{L^{2}}^{2}+\frac{1}{4A}\|\nabla\tilde{n}\|_{L^{2}}^{2}\leq\frac{C}{A}\left(\|n^{1}\|_{H^{2}}^{2}+\|n^{1}\|_{L^{2}}^{4}+\|n^{2}\|_{H^{2}}^{2} \right)\left(\|\tilde{n}\|_{L^{2}}^{2}+\|\tilde{u}\|_{L^{2}}^{2} \right).
		\end{aligned}
	\end{equation}
	
	Using $ \nabla\cdot u^{2}=\nabla\cdot\tilde{u}=0 $ again, then multiplying $ (\ref{eq:tilde})_{3} $ by $ \tilde{u} $, the energy estimates give
	\begin{equation}\label{tilde u}
		\begin{aligned}
		\frac12\partial_{t}\|\tilde{u}\|_{L^{2}}^{2}+\frac{1}{A}\|\nabla\tilde{u}\|_{L^{2}}^{2}=&-\int_{\mathbb{T}\times\mathbb{R}\times\mathbb{T}}\tilde{u}_{2}\tilde{u}_1
+\frac{1}{A}\int_{\mathbb{T}\times\mathbb{R}\times\mathbb{T}}u^{1}\tilde{u}\cdot\nabla\tilde{u}\\&+\frac{1}{A}\int_{\mathbb{T}\times\mathbb{R}\times\mathbb{T}}\tilde{n}\tilde{u}_2:=J_{1}+J_{2}+J_{3}.
		\end{aligned}
	\end{equation}
	For $ J_{1} $ and $ J_{3}, $ it holds that
	\begin{equation*}
		J_{1}\leq\|\tilde{u}\|_{L^{2}}^{2},\quad 	J_{3}\leq \frac{1}{2A}\|\tilde{n}\|_{L^{2}}^{2}+\frac{1}{2A}\|\tilde{u}\|_{L^{2}}^{2}.
	\end{equation*}
	For $ J_{2}, $ using Sobolev inequality, we have
	\begin{equation*}
		J_{2}\leq\frac{1}{A}\|\nabla\tilde{u}\|_{L^{2}}\|u^{1}\|_{L^{\infty}}\|\tilde{u}\|_{L^{2}}\leq\frac{1}{4A}\|\nabla\tilde{u}\|_{L^{2}}^{2}
+\frac{C}{A}\|u^{1}\|_{H^{2}}^{2}\|\tilde{u}\|_{L^{2}}^{2}.
	\end{equation*}
From the estimates of $ J_{1}-J_{3}, $ we get by $ (\ref{tilde u}) $ that
\begin{equation}\label{end u}
	\begin{aligned}
	\frac{1}{2}\partial_{t}\|\tilde{u}\|_{L^{2}}^{2}+\frac{3}{4A}\|\nabla\tilde{u}\|_{L^{2}}^{2}\leq C\left(\frac{1}{A}\|u^{1}\|_{H^{2}}^{2}+1 \right)\left(\|\tilde{n}\|_{L^{2}}^{2}+\|\tilde{u}\|_{L^{2}}^{2} \right).
	\end{aligned}
\end{equation}
Combining (\ref{end n}) with (\ref{end u}), we get
\begin{equation*}
	\begin{aligned}
	\partial_{t}\left(\|\tilde{n}\|_{L^{2}}^{2}+\|\tilde{u}\|_{L^{2}}^{2} \right)\leq& C\left(\frac{1}{A}\|n^{1}\|_{H^{2}}^{2}+\frac{1}{A}\|n^{1}\|_{L^{2}}^{4}+\frac{1}{A}\|n^{2}\|_{H^{2}}^{2}+\frac{1}{A}\|u^{1}\|_{H^{2}}^{2}+1 \right)\left(\|\tilde{n}\|_{L^{2}}^{2}+\|\tilde{u}\|_{L^{2}}^{2}\right)\\\leq&C\eta(t)\left(\|\tilde{n}\|_{L^{2}}^{2}+\|\tilde{u}\|_{L^{2}}^{2} \right).
	\end{aligned}
\end{equation*}
Notice that $ \eta\in L^{1}([0,T^{*})) $ due to $ (n^{1},u^{1}) $ and $ (n^{2}, u^{2}) $ are strong solutions. Then Gronwall's inequality implies $ (\tilde{n}_{\rm in}, \tilde{u}_{\rm in})=(0,0), $ and $ \tilde{c}=0 $ due to \eqref{eq:elliptic-c} . 
Thus, the uniqueness result is established.

   {\bf{ Step III. The positivity of $ n $.}}
	    Setting $ n_{+}=\max\{0, n\}, -n_{-}=\min\{0, n\}, $ then $ n=n_{+}-n_{-}. $ Multiplying the first equation of (\ref{ini11}) by $ n_{-}, $ we arrive
	  \begin{equation*}
	  	\begin{aligned}
	  	\partial_{t}n_{-}n_{-}+y\partial_{x}n_{-}n_{-}+\frac{1}{A}u\cdot\nabla n_{-}n_{-}-\frac{1}{A}\triangle n_{-}n_{-}=-\frac{1}{A}\nabla\cdot(n\nabla c)n_{-}.
	  	\end{aligned}
	  \end{equation*}
	  Then the energy estimates give
	  \beno
	  	\frac12\partial_{t}\|n_{-}\|_{L^{2}}^{2}+\frac{1}{A}\|\nabla n_{-}\|_{L^{2}}^{2}&=&-\frac{1}{2A}\int_{\mathbb{T}\times\mathbb{R}\times\mathbb{T}}n_{-}^2\triangle c\\
&\leq & \frac{C}{A}\|n_{-}\|_{L^{2}}^{2}(\|c\|_{H^2}+ \|n\|_{H^2} ),
	  \eeno
	  which follows that
	  \begin{equation*}
	  	\|n_{-}(t)\|_{L^{2}}^{2}\leq \|n_{-}(0,\cdot)\|_{L^{2}}^{2}e^{\int_0^t\frac{C}{A}(\|c\|_{H^2}+ \|n\|_{H^2} )ds}.
	  \end{equation*}
	  Note that $ n_{-}(0,\cdot)=0 $ due to $ n(0,\cdot)\geq 0, $ which means $ n_{-}=0. $ Thus $ n=n_{+}\geq 0. $

 {\bf Step IV. Blow-up criterion.}

Assume that (\ref{eq:blow-up condi}) fails. That is 
\ben\label{eq:blow-up condi'}
\sup_{0<t< T^*}A^{-\frac{1}{12}}\|(\nabla u)(t,\cdot)\|_{L^{2}}+
	\|(\partial_{x}^{2},\partial_{z}^{2})n(t,\cdot)\|_{L^{2}}+\|n(t,\cdot)\|_{L^{\infty}}\leq C.
\een
For $0<t<T^*$, claim that
	\begin{equation}\label{partial y n}
	\partial_{t}\|\partial_{y}n\|_{L^{2}}^{2}+\frac{1}{A}\|\partial_{y}\nabla n\|_{L^{2}}^{2}\leq C\left(\|\partial_{y}n\|_{L^{2}}^{2}+1 \right), 
	\end{equation}
	and
	\begin{equation}\label{n'' u''}
		\partial_{t}\left(\|\nabla^{2}n\|_{L^{2}}^{2}+\|\nabla^{2}u\|_{L^{2}}^{2} \right)+\frac{1}{A}\left(\|\nabla^{3}n\|_{L^{2}}^{2}+\|\nabla^{3}u\|_{L^{2}}^{2} \right)\leq C\left(\|\nabla^{2}n\|_{L^{2}}^{2}+\|\nabla^{2}u\|_{L^{2}}^{2}+1 \right),
	\end{equation}
which implies that the strong solution  can be extended to time $T^*$.

Note that $ \partial_{y}n $ satisfies
	\begin{equation*}
		\partial_{t}\partial_{y}n+\partial_{x}n+y\partial_{y}\partial_{x}n+\frac{1}{A}\partial_{y}(u\cdot\nabla n)-\frac{1}{A}\partial_{y}\triangle n=-\frac{1}{A}\partial_{y}\nabla\cdot(n\nabla c),
	\end{equation*}
 then the energy estimates give
	\begin{equation}\label{n'}
		\begin{aligned}
		\frac12\partial_{t}\|\partial_{y}n\|_{L^{2}}^{2}+\frac{1}{A}\|\partial_{y}\nabla n\|_{L^{2}}^{2}=&-\int_{\mathbb{T}\times\mathbb{R}\times\mathbb{T}}\partial_{x}n\partial_{y}n-\frac{1}{A}\int_{\mathbb{T}\times\mathbb{R}\times\mathbb{T}}\partial_{y}(u\cdot\nabla n)\partial_{y}n\\&-\frac{1}{A}\int_{\mathbb{T}\times\mathbb{R}\times\mathbb{T}}\partial_{y}\nabla\cdot(n\nabla c)\partial_{y}n:=K_{1}+K_{2}+K_{3}.
		\end{aligned}
	\end{equation}
	For $ K_{1}, $ using (\ref{eq:blow-up condi'}), we have
	\begin{equation*}
		K_{1}\leq C\|\partial_{x}^{2}n\|_{L^{2}}\|\partial_{y}n\|_{L^{2}}\leq C\left(\|\partial_{y}n\|_{L^{2}}^{2}+1 \right).
	\end{equation*}
	Using $ \nabla\cdot u=0 $, integration by parts and (\ref{eq:blow-up condi'}), we can estimate $ K_{2} $ by
	\begin{equation*}
		\begin{aligned}
		K_{2}=\frac{1}{A}\int_{\mathbb{T}\times\mathbb{R}\times\mathbb{T}}n\partial_{y}u\cdot\partial_{y}\nabla n\leq \frac{1}{2A}\|\partial_{y}\nabla n\|_{L^{2}}^{2}+\frac{C}{A}\|n\|_{L^{\infty}}^{2}\|\nabla u\|_{L^{2}}^{2}\leq\frac{1}{2A}\|\partial_{y}\nabla n\|_{L^{2}}^{2}+C.
		\end{aligned}
	\end{equation*}
	Recall that $ \|n\|_{L^{1}}=\|n_{\rm in}\|_{L^{1}}\leq C. $ Then
	due to elliptic estimates \eqref{eq:elliptic-c}, Gagliardo-Nirenberg inequality and (\ref{eq:blow-up condi'}), there holds
	\begin{equation*}
		\begin{aligned}
		K_{3}\leq&\frac{1}{4A}\|\partial_{y}\nabla n\|_{L^{2}}^{2}+\frac{C}{A}\|n\|_{L^{\infty}}^{2}\|\partial_{y}\nabla c\|_{L^{2}}^{2}+\frac{C}{A}\|\partial_{y}n\|_{L^{4}}^{2}\|\nabla c\|_{L^{4}}^{2}\\\leq&\frac{1}{4A}\|\partial_{y}\nabla n\|_{L^{2}}^{2}+\frac{C}{A}\|\partial_{y}n\|_{L^{2}}^{2}+\frac{C}{A}\|\partial_{y}n\|_{L^{2}}^{\frac12}\|\partial_{y}\nabla n\|_{L^{2}}^{\frac32}\|n\|_{L^{\infty}}\|n\|_{L^{1}}\\\leq&\frac{1}{2A}\|\partial_{y}\nabla n\|_{L^{2}}^{2}+\frac{C}{A}\|\partial_{y}n\|_{L^{2}}^{2}.
		\end{aligned}
	\end{equation*}
	Collecting $ K_{1}-K_{3}, $ we get by $ (\ref{n'}) $ that
	\begin{equation*}
		\partial_{t}\|\partial_{y}n\|_{L^{2}}^{2}+\frac{1}{A}\|\partial_{y}\nabla n\|_{L^{2}}^{2}\leq C\left(\|\partial_{y}n\|_{L^{2}}^{2}+1 \right),
	\end{equation*}
	which gives (\ref{partial y n}).
	
	Moreover, recall \eqref{eq:higher order},  and  by using (\ref{eq:blow-up condi'}) and (\ref{partial y n}), we get 
	\begin{equation*}
		\partial_{t}\left(\|\nabla^{2}n\|_{L^{2}}^{2}+\|\nabla^{2}u\|_{L^{2}}^{2} \right)+\frac{1}{A}\left(\|\nabla^{3}n\|_{L^{2}}^{2}+\|\nabla^{3}u\|_{L^{2}}^{2} \right)\leq C\left(\|\nabla^{2}n\|_{L^{2}}^{2}+\|\nabla^{2}u\|_{L^{2}}^{2}+1\right),
	\end{equation*}
	which gives (\ref{n'' u''}).
	
To sum up, the proof is complete.
\end{proof}

	\section*{Acknowledgement}
	
	The authors would like to thank Professors Zhifei Zhang, Weiren Zhao and Ruizhao Zi for some helpful communications. W. Wang was supported by National Key R\&D Program of China (No. 2023YFA1009200) and NSFC under grant 12471219 and 12071054.
	Part of S. Cui's work was conducted during Ph.D studies
	while visiting McMaster University. He is grateful for discussions with Professor Dmitry Pelinovsky and thanks the CSC for financial support.


\end{document}